\documentclass[12pt]{article}%
\usepackage{float,enumerate}
\usepackage{longtable}
\usepackage{makeidx}
\usepackage{graphicx}
\usepackage{amsmath}
\usepackage{amssymb}
\usepackage{amsfonts}
\usepackage{rotating}
\usepackage{cite}
\usepackage{color}
\usepackage{comment}%
\usepackage{fullpage}

\newif\ifarxiv
\arxivtrue
\makeindex
\newtheorem{theorem}{Theorem}[section]

\newtheorem{corollary}[theorem]{Corollary}
\newenvironment{extraremark}[1]{\color{red}\textbf{Remark #1} \it }{ \color{black} }
\newenvironment{extraexample}[1]{\color{red} \textbf{Example #1} \it }{\color{black}}

\newtheorem{definition}[theorem]{Definition}
\newtheorem{example}[theorem]{Example}

\newtheorem{assumption}[theorem]{Assumption}

\newtheorem{lemma}[theorem]{Lemma}

\newtheorem{proposition}[theorem]{Proposition}
\newtheorem{remark}[theorem]{Remark}

\newenvironment{proof}[1][Proof]{\textbf{#1.} }{\ \hfill \rule{0.5em}{0.5em}}

\newcommand{\ned}{\boldsymbol{\mathcal N}^I}
\newcommand{\poly}{{\mathcal P}}
\newcommand{\boldpoly}{\boldsymbol{\mathcal{P}}}

\newcommand{\eremk}{\hbox{}\hfill\rule{0.8ex}{0.8ex}}

\newcommand{\hatPicurlcom}{\widehat \Pi^{curl,c}_p}

\newcommand{\hatPigradcom}{\widehat\Pi^{grad,c}_{p+1}}

\newcommand\innerprod[1]{{ \left(\!\left(\kern -0.05em #1\kern -0.1em \right)\!\right) }_k} 
\newcommand\binnerprod[1]{{ \left(\kern -0.2em \left(\kern -0.05em #1\kern -0.1em \right)\kern -0.2em\right)_k }} 
\newcommand\bdinnerprod[1]{{ \left(\kern -0.25em \left(\kern -0.05em #1\kern -0.1em \right)\kern -0.25em\right)_k }} 
\allowdisplaybreaks
\makeatletter\@addtoreset{equation}{section}\makeatother
\begin{document}

\title{Wavenumber-explicit $hp$-FEM analysis for
Maxwell's equations with impedance boundary conditions}
\author{J.M. Melenk\thanks{(melenk@tuwien.ac.at), Institut f\"{u}r Analysis und
Scientific Computing, Technische Universit\"{a}t Wien, Wiedner Hauptstrasse
8-10, A--1040 Wien, Austria.}
\and S.A. Sauter\thanks{(stas@math.uzh.ch), Institut f\"{u}r Mathematik,
Universit\"{a}t Z\"{u}rich, Winterthurerstr.~{190}, CH--8057 Z\"{u}rich,
Switzerland}}
\maketitle

\begin{abstract}
The time-harmonic Maxwell equations at high wavenumber $k$ in domains with an
analytic boundary and impedance boundary conditions are considered. A
wavenumber-explicit stability and regularity theory is developed that
decomposes the solution into a part with finite Sobolev regularity that is
controlled uniformly in $k$ and an analytic part. Using this regularity,
quasi-optimality of the Galerkin discretization based on N\'{e}d\'{e}lec
elements of order $p$ on a mesh with mesh size $h$ is shown under the
$k$-explicit scale resolution condition that a) $kh/p$ is sufficient small and
b) $p/\ln k$ is bounded from below.

\end{abstract}
\tableofcontents

\ifarxiv

\pagenumbering{roman} 
\section*{Symbols and Notation}
\addcontentsline{toc}{section}{\protect\numberline{}Symbols and Notation}%
\newcommand{\fs}{8mm}
\begin{longtable}[c]{lp{0.8\textwidth}}
{\bf general} & \\ \hline \\
$k \ge 1 > 0$
& wavenumber  \\
$\operatorname*{i} $
& imaginary unit $\sqrt{-1}$ \\
$A\lesssim B $
& there exists $C$ independent of $k$, $h$, $p$, and
independent of \\
& functions that possibly appear in $A$ and $B$ so that
$A\leq CB$ holds,  \\
$a_+$ & $a_+:= \max\{a,0\}$ for $a \in \mathbb{R}$ \\
$\mathbb{N}$ & positive integers: $\mathbb{N} = \{1,2,3\ldots\}$ \\
             & \emph{caveat:} in Appendix~\ref{AppAnalyticity}, we \\
             & follow the ``French'' convention 
               $\mathbb{N} = \{0,1,\ldots\}$ \\
$\mathbb{N}_0$ & $\mathbb{N}= \{0,1,\ldots\}$\\
${\mathbb{N}}_{\ge p}$ & $\{ n \in \mathbb{N}\,|\, n \ge p\}$ \\
${\mathbb{N}}_{\leq p}$ & $\{ n \in \mathbb{N}\,|\, n \leq p\}$ \\[\fs]
%
{\bf geometry} & \\ \hline \\
$B_1(0)$
& unit ball in ${\mathbb R}^3$ \\
$B_r^+$
& half-balls in ${\mathbb R}^3$ \\
$\Omega$
& domain in ${\mathbb R}^3$  \\
$\Gamma = \partial\Omega$
& boundary of $\Omega$ \\
$\mathbf{n}$
& unit normal vector on $\Gamma$ pointing into $\Omega^+$ \\
$\mathbf{n}^{\ast}$
&constant extension of ${\mathbf n}$ to tubular neighborhood of $\Gamma$\\[\fs]
{\bf spaces} & \\ \hline \\
${\mathbf X}:= {\mathbf H}(\Omega,\operatorname{curl})$
& (\ref{defXHcurl}) \\
${\mathbf X}_{\operatorname{imp}}$ & 
${\mathbf X}_{\operatorname{imp}}:= \{ \mathbf{u} \in {\mathbf H}_0(\Omega,\operatorname{curl})\,|\, \Pi_T \mathbf{u} \in \mathbf{L}^2(\Gamma)\}$; 
 (\ref{eq:Ximp}) \\
${\mathbf X}_{\operatorname{imp},0}$ & 
${\mathbf X}_{\operatorname{imp},0}:= \{ \mathbf{u} \in \mathbf{X}_{\operatorname{imp}}\,|\, \operatorname{curl} \mathbf{u} = 0\} = 
\{\nabla \varphi\,|\, \varphi \in H^1_{\operatorname{imp}}(\Omega)\}$;  \\
${\mathbf X}_{\operatorname{imp}}(\Omega)$ & 
${\mathbf X}_{\operatorname{imp}}(\Omega):= \left({\mathbf H}^1(\Omega)\right)^\prime \cap {\mathbf X}^\prime_{\operatorname{imp},0}$; 
 (\ref{XimpprimeOmega}) \\
${\mathbf X}_{\operatorname{imp}}(\Gamma)$ & 
${\mathbf X}_{\operatorname{imp}}(\Gamma):= \left({\mathbf H}^{-1/2}_T(\Gamma)\right)^\prime \cap {\mathbf H}^{-1}_T(\Gamma,\operatorname{div}_\Gamma)$ 
(\ref{XimpprimeGamma}) \\
${\mathbf H}(\Omega,\operatorname{curl})$, ${\mathbf H}(\Omega,\operatorname{div})$
& (\ref{defXHcurl}), (\ref{DefHOmegadiv}) \\
${\mathbf H}_0(\Omega,\operatorname{curl})$ 
& 
${\mathbf H}_0(\Omega,\operatorname{curl}) = \{
{\mathbf u} \in {\mathbf H}(\Omega,\operatorname{curl})\,|\, \gamma_T {\mathbf u} = 0 \quad \mbox{ on $\partial\Omega$}\}$  \\
${\mathbf{H}}(\Omega,\operatorname{div}_0)$ & divergence-free functions \\
${\mathbf L}^2(\Omega)$
& space of vector-valued $L^2$-functions \\ 
$H^s(\Omega)$, $H^s(\Gamma)$ 
& scalar-valued Sobolev spaces on $\Omega$, $\Gamma$, Sec.~\ref{SecGeometry} \\
${\mathbf H}^s(\Omega)$
& vector-valued Sobolev spaces on $\Omega$ \\
$H^1_{\operatorname{imp}}(\Omega)$ & $H^1_{\operatorname{imp}}(\Omega) = 
\{\varphi \in H^1(\Omega)\,|\, 
\varphi|_{\partial\Omega} \in H^1(\partial\Omega)\}$, Def.~\ref{def:Ximp} \\
${\mathbf H}(\operatorname{curl},\Omega)$ & norm $\|\operatorname{curl} {\mathbf u}\|_{\mathbf{L}^2(\Omega)}  + \|\mathbf{u}\|_{\mathbf{L}^2(\Omega)}$ \\
${\mathbf H}(\operatorname{div},\Omega)$ & norm $\|\operatorname{div} {\mathbf u}\|_{{L}^2(\Omega)}  + \|\mathbf{u}\|_{\mathbf{L}^2(\Omega)}$ \\
${\mathbf L}^2_T(\Gamma)$, ${\mathbf H}^s_T(\Gamma)$
& Sobolev space of tangential fields on $\Gamma$, (\ref{DefL2t}), (\ref{DefHsGammaTNorm})\\
${\mathbf H}_{\operatorname{div}}^{-1/2}(\Gamma)$ 
& (\ref{m1/2curldiv}) \\
${\mathbf H}_{\operatorname{curl}}^{-1/2}(\Gamma)$ 
& (\ref{m1/2curldiv}) \\
${\mathbf H}^{-1}(\Gamma,\operatorname{div}_\Gamma)$ 
& (\ref{XimpprimeGamma}) \\
%
${\mathbf V}_{k,0} $ 
& $\mathbf{u} \in \mathbf{V}_{k,0}$ $\iff$
$((\mathbf{u},\nabla\varphi))_{k} = 0$ for all $\varphi \in H^1_{\operatorname{imp}}(\Omega)$; (\ref{DefV0}) \\
${\mathbf V}_{k,0,h} $ 
& $\mathbf{u} \in \mathbf{V}_{k,0,h}$ $\iff$
$((\mathbf{u},\nabla\varphi_h))_{k} = 0$ for all $\varphi \in S_h$; 
(\ref{eq:Vk0h}) \\
${\mathcal A}(C_1,\gamma_1,\omega)$, & 
class of analytic fcts., Def.~\ref{DefClAnFct}; $C_1$, $\gamma_1$ are independent of $k$ \\[\fs]
{\bf functions} & \\ \hline \\
%
${\mathbf E}$, 
${\mathbf H}$, 
${\mathbf E}^+$, 
${\mathbf H}^+$ 
& electric and magnetic fields in $\Omega$ and in $\Omega^+$ \\
$Y^m_\ell$, $\lambda_\ell$ 
& eigenfunctions of Laplace-Beltrami, Remark~\ref{rem:Yml} \\
$\iota_\ell$ & index set of indices for eigenvalue 
$\lambda_\ell$, Rem.~\ref{rem:Yml} \\
$g_{k}$
& Helmholtz fundamental solution, (\ref{MaxwellFullFullc}) \\[\fs]
%
{\bf sesquilinear forms, norms} & \\ \hline \\
%
$\|\cdot\|_{\operatorname{imp},k}$ & 
$\|\mathbf{u}\|^2_{\operatorname{imp},k} = 
\|\operatorname{curl} \mathbf{u}\|^2 + 
k^2 \|\mathbf{u}\|^2 + |k| \|\mathbf{u}\|^2_{\mathbf{L}^2(\Gamma)}$, 
(\ref{eq:Himp}) \\
$\|\cdot\|_{k,+}$ & 
$\|\mathbf{u}\|^2_{k,+} = 
k^2 \|\mathbf{u}\|^2 + |k| \|\mathbf{u}\|^2_{\mathbf{L}^2(\Gamma)}$, 
(\ref{eq:Himp}) \\
$\|\cdot\|_{{\mathbf X}^\prime_{\operatorname{imp}}(\Omega),k}$ & 
$\sup_{{\mathbf v} \in {\mathbf X}_{\operatorname{imp}}} \frac{|({\mathbf f},{\mathbf v})|}{\|\mathbf{v}\|_{\operatorname{imp},k}}$; (\ref{eq:H-1div}); 
Lemma~\ref{lemma:norm-equivalence} \\
$\|\cdot\|_{{\mathbf X}^\prime_{\operatorname{imp}}(\Gamma),k}$ & 
$\sup_{{\mathbf v} \in {\mathbf X}_{\operatorname{imp}}} \frac{|({\mathbf g}_T,{\mathbf v}_T)_{{\mathbf L}^2(\Gamma)}|}{\|\mathbf{v}\|_{\operatorname{imp},k}}$; (\ref{eq:H-1/2diva}); Lemma~\ref{lemma:norm-equivalence} \\
$(\cdot,\cdot)$ 
& $(u,v) = \int_\Omega u \overline{v}$ is the $L^2(\Omega)$ innerproduct/duality pairing   \\
$\|\cdot\| = \|\cdot\|_{L^2(\Omega)}$ & $L^2(\Omega)$-norm; Sec.~\ref{SecGeometry} \\
$(\cdot,\cdot)_{\mathbf{L}^2(\Gamma)}$ & 
$L^2(\Gamma)$-inner prod. (or duality pairing) \\
$A_k$, 
&  sesquilinear form associated with \\
 &Maxwell's equations, (\ref{defAksesqui}), (\ref{eq:Ak-alternative}) \\
$A_k^+$, 
&  sesquilinear form associated with \\
 & Maxwell's equations with the ``good'' sign, (\ref{eq:Ak+}) \\
$ \left(\kern-.1em \left( \kern-.1em \cdot,\cdot \kern-.1em \right) \kern-.1em \right)_k$& 
$ \left(\kern-.1em \left( \kern-.1em \cdot,\cdot \kern-.1em \right) \kern-.1em \right)_k =
k^2 (\cdot,\cdot)_{L^2(\Omega)} + \operatorname*{i} k (\Pi_T (\cdot),\Pi_T (\cdot))_{\mathbf{L}^2(\Gamma)}$ ; 
see (\ref{eq:def:(())}) \\
$\|\cdot\|_{-1/2,\operatorname{curl}_\Gamma}$, & \\
$\|\cdot\|_{-1/2,\operatorname{div}_\Gamma}$ 
&  norms on ${\mathbf H}^{-1/2}_{\operatorname{curl}}(\Gamma)$, 
on ${\mathbf H}^{-1/2}_{\operatorname{div}}(\Gamma)$, (\ref{m1/2curldiv}) \\
$|{\mathbf v}|_{\mathbf{H}^\ell(\Omega),k}$& 
$|{\mathbf v}|_{\mathbf{H}^\ell(\Omega),k}:= k^{-\ell} |{\mathbf v}|_{\mathbf{H}^\ell(\Omega)}$;  see (\ref{eq:Hl-rho}) \\
$\|{\mathbf v}\|_{\mathbf{H}^m(\Omega),k}$& 
$\|{\mathbf v}\|_{\mathbf{H}^m(\Omega),k}:= \bigl(\sum_{\ell=0}^m |{\mathbf v}|^2_{\mathbf{H}^\ell(\Omega),k}\bigr)^{1/2}$;  see (\ref{eq:Hl-rho});  \\
& 
$\|{\mathbf v}\|^2_{\mathbf{H}^m(\Omega),k} = \|{\mathbf v}\|^2_{\mathbf{H}^0(\Omega)} + k^{-2} |\mathbf{v}|^2_{\mathbf{H}^1(\Omega)} + \cdots + k^{-2m} |{\mathbf v}|^2_{\mathbf{H}^m(\Omega)}$ \\ 
$\|{\mathbf v}\|_{\mathbf{H}^{-m}(\Omega),k}$&  
$\|{\mathbf v}\|_{\mathbf{H}^{-m}(\Omega),k}:= k^{m} \|{\mathbf v}\|_{\mathbf{H}^{-m}(\Omega)}$;  see (\ref{eq:Hl-rho-dual}) \\
%
%
%
$\|{\mathbf v}\|_{\mathbf{H}^{-1/2}(\Gamma,\operatorname{div}_\Gamma),k}$&  
$|k| \|\operatorname{div}_\Gamma {\mathbf g}_T \|_{{\mathbf H}^{-1/2}(\Gamma)} + k^2 \|{\mathbf g}_T\|_{{\mathbf X}^\prime_{\operatorname{imp}}(\Gamma),k}$; 
see (\ref{eq:H-1/2div}) and Lemma~\ref{lemma:H-1/2-div}
\\
$\|\mathbf{v} \|_{\mathbf{H}(\Omega,\operatorname{div}),k}$ &  
$\|\mathbf{v} \|_{\mathbf{H}(\Omega,\operatorname{div}),k} := 
 \bigl(k^{-2m} |\operatorname{div}{\mathbf v}|^2_{{H}^m(\Omega)} + k^2 \|\mathbf{v}\|^2_{\mathbf{H}^{m}(\Omega),k}
\bigr)^{1/2}$; 
 (\ref{eq:HD-norm})  \\
& $\|\mathbf{v} \|^2_{\mathbf{H}(\Omega,\operatorname{div}),k}  = k^{-2m} |\operatorname{div} \mathbf{v}|^2_{{H}^m(\Omega)}  
+ k^2 \|\mathbf{v}\|^2_{\mathbf{L}^2(\Omega)} + \cdots+k^{2-2m} |\mathbf{v}|^2_{\mathbf{H}^m(\Omega)}$ 
\\ 
$\|\mathbf{v} \|_{\mathbf{H}(\Omega,\operatorname{curl}),k}$ &  
$\|\mathbf{v} \|_{\mathbf{H}(\Omega,\operatorname{curl}),k} := 
 \bigl(k^{-2m} |\operatorname{curl}{\mathbf v}|^2_{\mathbf{H}^m(\Omega)} + k^2 \|\mathbf{v}\|^2_{\mathbf{H}^{m}(\Omega),k}
\bigr)^{1/2}$; 
 (\ref{eq:HD-norm})  \\
& 
$\|\mathbf{v} \|^2_{\mathbf{H}(\Omega,\operatorname{curl}),k}  = k^{-2m} |\operatorname{curl} \mathbf{v}|^2_{\mathbf{H}^m(\Omega)}  
+ k^2 \|\mathbf{v}\|^2_{\mathbf{L}^2(\Omega)} + \cdots+k^{2-2m} |\mathbf{v}|^2_{\mathbf{H}^m(\Omega)}$ 
\\ 
$\|\mathbf{g}_{T}\| _{\mathbf{H}^{\nu}(\Gamma),k}$ & 
$\|\mathbf{g}_{T}\| _{\mathbf{H}^{\nu}(\Gamma),k}
:=\left(
{\displaystyle\sum\limits_{\ell=0}^{2\nu}}
\left\vert k\right\vert ^{1-\ell}\left\vert \mathbf{g}_{T}\right\vert
_{\mathbf{H}^{\ell/2}\left(  \Gamma\right)  }^{2}\right)  ^{1/2}$;%
see (\ref{defgtGammaind}) \\
& 
$\|\mathbf{g}_{T}\| _{\mathbf{H}^{\nu}(\Gamma),k} \sim k^{1/2} \|\mathbf{g}_T\|_{\mathbf{L}^2(\Gamma)} + |\mathbf{g}_T|_{\mathbf{H}^{1/2}(\Gamma)} 
+ k^{1/2} |\mathbf{g}_T|_{\mathbf{H}^{1}(\Gamma)}$ \\
&\phantom{  
$\|\mathbf{g}_{T}\| _{\mathbf{H}^{\nu}(\Gamma),k} \sim $} $+ \cdots + 
k^{1/2-\nu} |\mathbf{g}_T|_{\mathbf{H}^{\nu}(\Gamma)}$ \\
$\|\mathbf{g}_{T}\| _{\mathbf{H}%
^{-1/2}(\Gamma)  ,k}$
& 
$\|\mathbf{g}_{T}\| _{\mathbf{H}%
^{-1/2}(\Gamma)  ,k}
:= k \| \mathbf{g}_{T}\| _{\mathbf{H}^{-1/2}(\Gamma)}$; see (\ref{defgtGammaind})
\\
$\|\mathbf{g}_{T}\| _{\mathbf{H}^{\nu}(\Gamma,\operatorname{div}_\Gamma),k}$ & 
$\|\mathbf{g}_{T}\| _{\mathbf{H}^{\nu}(\Gamma,\operatorname{div}),k}^2 = 
|\operatorname{div} \mathbf{g}_{T}|^2_{\mathbf{H}^{\nu}(\Gamma),k}+  
k |\mathbf{g}_{T}|^2 _{\mathbf{H}^{\nu}(\Gamma),k}$; see (\ref{eq:Hgamma-D}) \\
$\|\mathbf{g}_{T}\| _{\mathbf{H}^{\nu}(\Gamma,\operatorname{curl}_\Gamma),k}$ & 
$\|\mathbf{g}_{T}\| _{\mathbf{H}^{\nu}(\Gamma,\operatorname{curl}),k}^2 = 
|\operatorname{curl} \mathbf{g}_{T}|^2_{\mathbf{H}^{\nu}(\Gamma),k}+  
k |\mathbf{g}_{T}|^2 _{\mathbf{H}^{\nu}(\Gamma),k}$; see (\ref{eq:Hgamma-D}) \\
$\|\cdot\|_{{\mathcal H},\omega}$ & 
$\|\cdot\|^2_{{\mathcal H},\omega} = \|\nabla \cdot \|^2_{L^2(\omega)} + k^2 \|\cdot\|^2_{L^2(\omega)}$ \\
$\left\vert \cdot\right\vert $
& Euclidean norm \\
$\langle \cdot,\cdot \rangle$ & \emph{bilinear} form  on ${\mathbb C}^n$: 
$\langle {\mathbf a}, {\mathbf b}\rangle = \sum_{i=1}^n {\mathbf a}_i {\mathbf b}_i$; Sec.~\ref{SecGeometry} \\[\fs] 
$|\cdot|_{p,q,B_R^+}$, 
$[\![\cdot]\!]_{p,q,B_R^+}$,  & \\
$\rho^2_\ast[\![\cdot]\!]_{p,q,B_R^+}$, 
$\rho^{\frac{3}{2}}_\ast[\![\cdot]\!]_{p,\frac{1}{2},\Gamma_R}$, & \\
$\rho^{\frac{1}{2}}_\ast[\![\cdot]\!]_{p,\frac{3}{2},\Gamma_R}$, 
& seminorms to control high order derivatives, Appendix~\ref{AppAnalyticity}, p.~\pageref{eq:nicaise-seminorms}
\\[\fs]
{\bf discrete spaces, meshes} & \\ \hline \\
$\widehat K$ 
& reference tetrahedron \\
${\mathcal T}_h$, $F_K$, $F_K$, $A_K$
& triangulation, element maps, Sec.~\ref{sec:nedelec-elements}, Ass.~\ref{def:element-maps} \\
$S_h $
&  (discrete) subspace of $H^1(\Omega)$; \\
& we require $\nabla S_h \subset {\mathbf X}_h$
and exact seq. property (\ref{exdiscseq_rm}) \\
${\mathbf X}_h $
& (discrete) subspace of ${\mathbf H}(\Omega,\operatorname{curl})$ \\
$h$, $h_{K}$, $p$
& global and local meshwidth (Assumption~\ref{def:element-maps},  (\ref{defhkloc})),
polyn. deg. $p$ \\
$\poly_p$, $\boldpoly_p$
& space of ${\mathbb R}$-valued and ${\mathbb R}^3$-valued polynomials of degree $p$, (\ref{eq:Pp}) \\
$\ned_p(\widehat K)$
& N\'ed\'elec type I space on reference tetrahedron $\widehat K$, (\ref{eq:Np})  \\
$ S_{p+1}({\mathcal T}_h)$, 
$\ned_p({\mathcal T}_h)$,  
& polyn. spaces on ${\mathcal T}_h$: $H^1(\Omega)$-, ${\mathbf H}(\operatorname{curl},\Omega)$-conforming \\[\fs]
{\bf operators} & \\ \hline \\
$\operatorname*{curl}$, $\operatorname*{div}$
& 3D curl and divergence operators \\
$\operatorname*{curl}_\Gamma$, $\operatorname*{div}_\Gamma$
& 2D scalar curl and divergence operators on the surface $\Gamma$, (\ref{sccounttangcurl}) \\
$\overrightarrow{\operatorname*{curl}_\Gamma}$,
$\nabla_\Gamma$,
& 2D vectorial curl and surface gradient operators on $\Gamma$, (\ref{curlvec}) \\
$\Delta_\Gamma$
& surface Laplace-Beltrami operator, (\ref{defLaplBelt}) \\
${\mathcal E}_{\operatorname{curl}}$, ${\mathcal E}_{\operatorname{div}}$,
&  lifting operators (see Thm.~\ref{traceTHM1}) \\
${\mathcal F}$ & Fourier transformation, (\ref{eq:fourier-transform})  \\
$\gamma$ & standard trace operator: 
$\gamma u = u|_{\Gamma}$ and 
$\gamma {\mathbf u} = {\mathbf u}|_{\Gamma}$ \\
$\Pi_T$, $\gamma_T$ 
& trace operators (\ref{eq:trace-operators}); Thm.~\ref{traceTHM1};  \\
& 
$\Pi_T \mathbf{u} = \mathbf{n} \times (\mathbf{u} \times \mathbf{n})$,  
$\gamma_T \mathbf{u} = \mathbf{u} \times \mathbf{n}$,  \\
$(\cdot)_T$
& subscript $T$ indicates tangential trace: ${\mathbf u}_T = \Pi_\tau {\mathbf u}$ \\
$(\cdot)^{\operatorname{high}}$, 
$(\cdot)^{\operatorname{low}}$ 
& ${\mathbf v}^{\operatorname{high}} = H_\Omega {\mathbf v}$, 
 ${\mathbf v}^{\operatorname{low}} = L_\Omega {\mathbf v}$,  \\
$(\cdot)^\nabla$
& gradient part of Hodge decomp.\ of functions on $\Gamma$, (\ref{vFExp}), (\ref{defHodge})\\
$(\cdot)^{\operatorname{curl}}$
& curl part of Hodge decomp.\ of functions on $\Gamma$, (\ref{vFExp}), (\ref{defHodge})\\
${R}_1$, ${\mathbf R}_2$ & operators of order $-1$ of the regular decomp.\ in Lemma~\ref{LemRs};  \\
& by (\ref{RSrelation}), ${\mathbf R}_2$ is (up to a smoothing operator) \\
& a right inverse inverse of $\operatorname{curl}$ for divergence-free functions \\
${\mathbf K}$ & operator of order $-\infty$ of the regular decomp.\ in Lemma~\ref{LemRs} \\
${\mathbf S}$ & $\mathbf{f} = H^0_\Omega \mathbf{f} + L^0_\Omega \mathbf{f} + \mathbf{S}\mathbf{f}$; \\
              & for $\operatorname{div} \mathbf{f} = 0$ we have $\mathbf{S} \mathbf{f} \in \mathbf{C}^\infty(\overline{\Omega})$; see (\ref{eq:S}) \\
$\mathcal{L}_{\Omega,k}\mathbf{u}$ & $\mathcal{L}_{\Omega,k}\mathbf{u} = \operatorname{curl} \operatorname{curl} \mathbf{u} - k^2 \mathbf{u}$ \\
$\mathcal{B}_{\partial\Omega,k}\mathbf{u}$ & $\mathcal{B}_{\partial\Omega,k} \mathbf{u} = \gamma_T \operatorname{curl} \mathbf{u} - \operatorname{i} k \mathbf{u}_T$ \\
$g_k$ & Helmholtz Green's function: $g_k(\mathbf{r}) = e^{\operatorname{i} k |\mathbf{r}|}/(4\pi|{\mathbf r}|)$; see (\ref{MaxwellFullFullc})  \\
$\mathcal{N}_{\operatorname*{MW},k}^{\operatorname*{curl}}
$ &   
$\mathcal{N}_{\operatorname*{MW},k}^{\operatorname*{curl}}\left(
\mathbf{J}\right)
:=\int_{\mathbb{R}^{3}}g_{k}\left(  \left\Vert
\mathbf{\cdot}-\mathbf{y}\right\Vert \right)  \mathbf{J}\left(  \mathbf{y}%
\right)  d\mathbf{y}
$; see (\ref{MaxwellFullFullb}) 
\\
$\mathcal{N}_{-k}$ 
& solution operator for a dual problem, (\ref{eq:dual-problem-N})\\
$\mathcal{N}_{\operatorname*{MW},k}^{\nabla}
$ &   
$\mathcal{N}_{\operatorname*{MW},k}^{\nabla}\left(
\mathbf{J}\right)
:=k^{-2} \nabla \int_{\mathbb{R}^{3}}g_{k}\left(  \left\Vert
\mathbf{\cdot}-\mathbf{y}\right\Vert \right) \operatorname{div} \mathbf{J}\left(  \mathbf{y}%
\right)  d\mathbf{y}
$; see (\ref{MaxwellFullFullb}) 
\\
$\mathcal{N}_{\operatorname*{MW},k} $ &   
$\mathcal{N}_{\operatorname*{MW},k}\left(
\mathbf{J}\right)
:=\mathcal{N}_{\operatorname{MW},k}^{\operatorname{curl}} ( \mathbf{J}) + 
\mathcal{N}_{\operatorname{MW},k}^{\nabla} ( \mathbf{J})$  
; see (\ref{MaxwellFullFulla}) \\
$\mathcal{S}_{\operatorname*{MW},k}^{\operatorname*{curl}}
$ &   
$\mathcal{S}_{\operatorname*{MW},k}^{\operatorname*{curl}}\left(
\mathbf{J}\right)
:=\int_{\Gamma}g_{k}\left(  \left\Vert
\mathbf{\cdot}-\mathbf{y}\right\Vert \right)  \mathbf{J}\left(  \mathbf{y}%
\right)  d\mathbf{y}
$; see (\ref{eq:S-potential}) 
\\
$\mathcal{S}_{\operatorname*{MW},k}^{\nabla}
$ &   
$\mathcal{S}_{\operatorname*{MW},k}^{\nabla}\left(
\mathbf{J}\right)
:=k^{-2} \nabla \int_{\Gamma}g_{k}\left(  \left\Vert
\mathbf{\cdot}-\mathbf{y}\right\Vert \right) \operatorname{div} \mathbf{J}\left(  \mathbf{y}%
\right)  d\mathbf{y}
$; see (\ref{eq:S-potential}) 
\\
$\mathcal{S}_{\mathbb{R}^3,k}^{\operatorname*{MW}} $ &   
$\mathcal{S}_{\mathbb{R}^3,k}^{\operatorname*{MW}}\left(
\mathbf{J}\right)
:=\mathcal{S}_{\operatorname{MW},k}^{\operatorname{curl}} ( \mathbf{J}) + 
\mathcal{S}_{\operatorname{MW},k}^{\nabla} ( \mathbf{J})$  
; see (\ref{eq:S-potential}) 
\\
$\mathcal{S}_{\Omega,k}^+ $ & solution operator with the ``good'' sign; (\ref{weak_plus})  \\
 & $\mathcal{L}_{\Omega,\operatorname{i}k} \mathcal{S}_{\Omega,k}^+ \mathbf{g}_T = 0$ in $\Omega$, 
  $\mathcal{B}_{\partial\Omega,k} \mathcal{S}_{\Omega,k}^+ \mathbf{g}_T = \mathbf{g}_G$ on $\partial\Omega$,  \\
$\mathcal{S}^{\operatorname{MW}}_{\Omega,k} $ & solution operator for the Maxwell problem on $\Omega$; Section~\ref{SecMW_Imp} \\
$\mathcal{E}_{\operatorname{Stein}}$ & Stein's extension operator for $\Omega$ \\
$\mathcal{E}_{\operatorname{div}}$ & $H(\operatorname{div})$-stable extension operator, (\ref{DefBRR}) \\ 
%
%
%
$L_{\mathbb{R}^3}$, $H_{\mathbb{R}^3}$ & frequency splitting on ${\mathbb{R}}^3$ with parameter $\lambda > 1$; see (\ref{deffreqslitR3}) \\ 
$H_\Omega$, $L_\Omega$ & $H_\Omega = H_{\mathbb{R}^3} \mathcal{E}_{\operatorname{stein}}$, (\ref{DefLOmegaHOmega}) \\
$H^0_\Omega$, $L^0_\Omega$ & 
$H^0_\Omega \mathbf{f} = \operatorname{curl} H_\Omega \mathbf{R}_2 \mathbf{f}$, (\ref{DefHOmega0}) \\
${\mathcal{L}}^\nabla_{\operatorname{imp}}$, 
${\mathcal{L}}^{\operatorname{curl}}_{\operatorname{imp}}$ & 
part of the Hodge decomposition on $\Gamma$, (\ref{nablaliftings}) \\ 
$\Pi^\nabla_{\operatorname{imp}}$, 
$\Pi^{\operatorname{curl}}_{\operatorname{imp}}$ & 
part of the Hodge decomposition on $\Gamma$, (\ref{nablaliftings}), (\ref{defHodge}) \\ 
$\mathcal{E}^\Delta_\Omega$ & harmonic lifting from the boundary, (\ref{LaplaceDiriProbl}) \\
$L_{\partial\Omega}$, $H_{\partial\Omega}$ & 
$L_{\partial\Omega} = (L_\Omega \mathcal{E}^\Delta_\Omega)|_{\partial\Omega}$ 
$H_{\partial\Omega} = (H_\Omega \mathcal{E}^\Delta_\Omega)|_{\partial\Omega}$, (\ref{DefDiriLift}) \\
$\mathbf{H}^\nabla_\Gamma$, 
$\mathbf{L}^\nabla_\Gamma$ & $\mathbf{H}^\nabla_\Gamma = \nabla_\Gamma (H_{\partial\Omega} \mathcal{L}^\nabla_{\operatorname{imp}})$, 
                             $\mathbf{L}^\nabla_\Gamma = \nabla_\Gamma (L_{\partial\Omega} \mathcal{L}^\nabla_{\operatorname{imp}})$, 
(\ref{gt4}) \\ %
$\mathbf{H}^{\operatorname{curl}}_\Gamma$, 
$\mathbf{L}^{\operatorname{curl}}_\Gamma$ & $\mathbf{H}^{\operatorname{curl}}_\Gamma = \overrightarrow{\operatorname{curl}}_\Gamma  H_{\partial\Omega} \mathcal{L}^{\operatorname{curl}}_{\operatorname{imp}}$, 
$\mathbf{L}^{\operatorname{curl}}_\Gamma = \overrightarrow{\operatorname{curl}}_\Gamma  L_{\partial\Omega} \mathcal{L}^{\operatorname{curl}}_{\operatorname{imp}}$, (\ref{gt4}) \\
$\mathbf{H}_\Gamma$ & $\mathbf{H}_\Gamma = \mathbf{H}^\nabla_\Gamma + \mathbf{H}^{\operatorname{curl}}_\Gamma$, (\ref{DefHGammaLGamma2}) \\
$\mathbf{L}_\Gamma$ & $\mathbf{L}_\Gamma = \mathbf{L}^\nabla_\Gamma + \mathbf{L}^{\operatorname{curl}}_\Gamma$, (\ref{DefHGammaLGamma2}) \\
%
%
%
%
$\Pi^E_h$
& abstract form of a commuting diagram operator acting on 
$\mathbf{X}_{\operatorname{imp}}$  \\
$\Pi^F_h$
& abstract form of a commuting diagram operator \\
 & acting on a subspace of $\mathbf{H}(\operatorname{div})$ \\
%
$\Pi_p^{\operatorname{curl},s}$ & approximation operator 
of \cite[Lem.~{8.5}(i)]{melenk2018wavenumber} \\
& ${\mathbf H}(\operatorname{curl})$-conforming approx. operator, \\
& optimal $p$-rates
\emph{simultaneously} in ${\mathbf L}^2$ and ${\mathbf H}(\operatorname{curl})$ \\
%
$\widehat\Pi_p^{\operatorname{curl}, 3d}$, 
$\widehat\Pi_p^{\operatorname{grad}, 3d}$
& commuting diagram operators on reference tetrahedron $\widehat K$ 
from \cite{melenk_rojik_2018}.  \\
$\Pi^{\nabla}_k$, $\Pi^{\nabla}_{k,h}$, 
& projection onto $\nabla H^1(\Omega)$
or $S_h$ w.r.t. $((\cdot,\cdot))_k$ (Definition~\ref{DefHelmsplit})  \\
$\Pi^{\operatorname*{curl}}_k$, 
$\Pi^{\operatorname*{curl}}_{k,h}$, 
& $\operatorname{I} - \Pi^{\nabla}_k$ and $\operatorname{I} - \Pi^{\nabla}_{k,h}$,
see Def.~\ref{DefHelmsplit}\\[\fs]
%
{\bf constants} & \\ \hline \\
$C_{\operatorname{stab}}$, $\theta$ & constants characterizing 
assumed stability of the Maxwell problem, 
(\ref{DefStabConst}) \\
$C_{\operatorname{cont}}$ & $C_{\operatorname{cont}} = 1$ is the continuity 
constant of $A_k$ \\
$C^\Omega_{\operatorname{stab}}(k)$, 
$C^{\operatorname{imp}}_{\operatorname{stab}}(k)$
& stability constants for continuous problem, (\ref{eq:Cstab}), (\ref{AssumptionAlgGrowth}) \\
$C_{\operatorname*{affine}}$, $C_{\operatorname{metric}}$
&constants measuring the quality of the mesh (Assumption \ref{def:element-maps}) \\
$C_{\mathbf{V}_0}$ 
& an embedding constant ${\mathbf V}_0 \subset {\mathbf H}^1(\Omega)$, 
see Prop.~\ref{PropEmbVk0};  \\
$C_{tr,R}$, $C^\prime_{tr,R}$ & trace constants for $\Gamma_R$; Lemma~\ref{lemma:A6} and its proof \\[\fs] 
%
%
%
%
{\bf dual problems and} & \\ 
{\bf approximation } & \\ 
{\bf properties} & \\ \hline \\
${\mathcal N}_{-k}$ & solution operator for an adjoint problem; (\ref{eq:dual-problem-N}), \\
${\eta}_6^{\operatorname{alg}}$ & (\ref{eq:eta-6}) \\
$\tilde{\eta}_2^{\operatorname{alg}}$
& (\ref{eq:eta2-alg}) \\
%
 & a tilde indicates that an adjoint sol. operator ${\mathcal N}$ is involved; \\
& $\eta$ indicates a pure approximation property, \\
& superscript ``$\operatorname{alg}$'' indicates that algebraic convergence \\
& of $hp$-FEM is expected 
\end{longtable}

 \setcounter{page}{-1} \pagenumbering{arabic}

\fi


\section{Introduction}

The time-harmonic Maxwell equations at high wavenumber $k$ are a fundamental
component of high-frequency computational electromagnetics. Computationally,
these equations are challenging for several reasons. The solutions are highly
oscillatory so that fine discretizations are necessary and correspondingly
large computational resources are required. While conditions to resolve the
oscillatory nature of the solution appear unavoidable, even more stringent
conditions on the discretizations have to be imposed for stability reasons: In
many numerical methods based on the variational formulation of Maxwell's
equations, the gap between the actual error and the best approximation error
widens as the wavenumber $k$ becomes large. This \textquotedblleft pollution
effect\textquotedblright\ is a manifestation of a lack of coercivity of the
problem, as is typical in time-harmonic wave propagation problems.
Mathematically understanding this \textquotedblleft pollution
effect\textquotedblright\ in terms of the wavenumber $k$ and the
discretization parameters for the model problem (\ref{eq:model}) is the
purpose of the present work.

The \textquotedblleft pollution effect\textquotedblright, i.e., the fact that
discretizations of time-harmonic wave propagation problems are prone to
dispersion errors, is probably best studied for the Helmholtz equation at
large wavenumbers. The beneficial effect of using high order methods was
numerically observed very early and substantiated for translation-invariant
meshes \cite{Ainsworth2004,ainsworth04b}; a rigorous mathematical analysis for
unstructured meshes was developed in the last decade only in
\cite{MelenkSauterMathComp, mm_stas_helm2, MelenkHelmStab2010}. These works
analyze high order FEM ($hp$-FEM) for the Helmholtz equation in a G{\aa }rding
setting using duality techniques. This technique, often called
\textquotedblleft Schatz argument\textquotedblright, crucially hinges on the
regularity of the dual problem, which is again a Helmholtz problem. The key
new insight of the line of work \cite{MelenkSauterMathComp, mm_stas_helm2,
MelenkHelmStab2010} is a refined wavenumber-explicit regularity theory for
Helmholtz problems that takes the following form (\textquotedblleft regularity
by decomposition\textquotedblright): given data, the solution $u$ is written
as $u_{H^{2}}+u_{\mathcal{A}}$ where $u_{H^{2}}$ has the regularity expected
of elliptic problems and is controlled in terms of the data with constants
\emph{independent} of $k$. The part $u_{\mathcal{A}}$ is a (piecewise)
analytic function whose regularity is described explicitly in terms of $k$.
Employing \textquotedblleft regularity by decomposition\textquotedblright\ for
the analysis of discretizations has been successfully applied to other
Helmholtz problems and discretizations such DG methods \cite{MPS13}, BEM
\cite{MelenkLoehndorf}, FEM-BEM coupling
\cite{mascotto-melenk-perugia-rieder20}, and heterogeneous Helmholtz problems
\cite{chaumont-frelet-nicaise20,bernkopf-chaumont-melenk21,lafontaine2021wavenumberexplicit,lafontaine2021decompositions}%
.

In this paper, we consider the following time-harmonic Maxwell equations with
impedance boundary conditions as our model problem:%
\begin{subequations}
\label{eq:model}
\begin{align}
\operatorname{curl}\operatorname{curl}{\mathbf{E}}-k^{2}{\mathbf{E}}  &
={\mathbf{f}}\quad\mbox {in $\Omega$},
\label{eq:modela}\\
(\operatorname{curl}{\mathbf{E}})\times{\mathbf{n}}-\operatorname{i}%
k{\mathbf{E}}_{T}  &  ={\mathbf{g}}_{T}\quad\mbox {on $\partial\Omega$} 
\label{eq:modelb}%
\end{align}
\end{subequations}
on a bounded Lipschitz domain $\Omega\subset\mathbb{R}^{3}$ with simply
connected boundary $\partial\Omega$. We study an ${\mathbf{H}}%
(\operatorname{curl})$-conforming Galerkin method with elements of degree $p$
on a mesh of size $h$ and show quasi-optimality of the method under
the
\emph{scale resolution condition}
\begin{equation}
\frac{\left\vert k\right\vert h}{p}\leq c_{1}\qquad\mbox{ and }\qquad p\geq
c_{2}\ln\left\vert k\right\vert , \label{eq:scale-resolution}%
\end{equation}
where $c_{2}$%
$>0$
is arbitrary and $c_{1}$%
$>0$
is sufficiently small (Theorem~\ref{thm:quasi-optimality}). The resolution
condition $\left\vert k\right\vert h/p\leq c_{1}$ is a natural condition to
resolve the oscillatory behavior of the solution, and the side constraint
$p\geq c_{2}\ln\left\vert k\right\vert $ is a rather weak condition that
suppresses the \textquotedblleft pollution effect\textquotedblright.

Compared to the scalar Helmholtz case, where the compact 
embedding $H^1 \subset L^2$ underlies the success of the duality argument, 
the convergence analysis of discretizations of Maxwell's equations is 
hampered by the fact that the embedding 
$\mathbf{H}(\operatorname{curl}) \subset \mathbf{L}^2$ is not compact so 
that a duality argument is not immediately applicable. This issue arises
even in the context of convergence analyses that are not explicit in the 
wavenumber $k$. An analysis can be based on the observation that 
$\mathbf{H}(\operatorname{curl}) \cap 
\mathbf{H}(\operatorname{div})$ endowed with appropriate boundary conditions
is compactly embedded in  $\mathbf{L}^2$. This approach, which 
is structurally described in \cite[Sec.~{1.2}]{MelenkSauterMaxwell_I}, 
involves as a first ingredient the ability to decompose discrete functions 
into gradient parts and (discrete) solenoidal parts in two ways, namely, 
on the continuous level and the discrete level. The solenoidal part of 
the decomposition on the continuous level 
is in 
$\mathbf{H}(\operatorname{curl}) \cap 
\mathbf{H}(\operatorname{div})$ 
and admits a duality argument. 
Galerkin orthogonalities are invoked to then reduce the analysis to that 
of the difference between the solenoidal parts of the continuous 
and the discrete level. For the analysis of this difference, 
a second ingredient is vital, namely, special interpolation operators 
with a commuting diagram property. These two ingredients underlie 
many duality arguments for Maxwell problems in the literature, see, 
e.g.,  \cite[Sec.~{7.2}]{Monkbook}, 
\cite{zhong-shu-wittum-xu09,chaumont-frelet-nicaise-pardo18,ern-guermond18,chaumont19} and references therein. The present work 
follows \cite[Sec.~{7.2}]{Monkbook} and the path outlined in
\cite[Sec.~{1.1--1.3}]{MelenkSauterMaxwell_I}. 

At the heart of the $k$-explicit convergence analysis for 
(\ref{eq:model}) is a $k$-explicit regularity theory for the above 
mentioned dual problem. Similarly to the Helmholtz case discussed 
above, it takes the form of a ``regularity by decomposition'' 
(Theorem~\ref{TheoMainIt}). 
Such a regularity theory was developed for
Maxwell's equations in full space in
the recent
paper \cite{MelenkSauterMaxwell_I}, where the decomposition is directly
accessible in terms of the Newton potential and layer potentials. 
For the present
bounded domain case, however, an explicit construction of the decomposition is
not available, and the iterative construction as in the Helmholtz case of
\cite{mm_stas_helm2} has to be brought to bear. For this, a significant
complication in the Maxwell case compared to the Helmholtz case arises from
the requirement that the frequency filters used in the construction be such
that they produce solenoidal fields if the argument is solenoidal.

While our wavenumber-explicit regularity result Theorem~\ref{TheoMainIt}
underlies our proof of quasioptimal convergence of the high order Galerkin
method (cf.\ Theorem~\ref{thm:quasi-optimality}), it also proves useful for
wavenumber-explicit interpolation error estimates as worked out in
Corollary~\ref{cor:convergence}.

The present paper analyzes an ${\mathbf{H}}(\operatorname{curl})$-conforming
discretization based on high order N{\'{e}}d{\'{e}}lec elements. Various other
high order methods for Maxwell's equations that are explicit in the wavenumber
can be found in the literature. Closest to our work are
\cite{nicaise-tomezyk19,chaumont-frelet-vega21}. The work
\cite{nicaise-tomezyk19} studies the same problem (\ref{eq:model}) but uses an
${\mathbf{H}}^{1}$-based instead of an ${\mathbf{H}}(\operatorname{curl}%
)$-based variational formulation involving both the electric and the magnetic
field. The proof of quasi-optimality in \cite{nicaise-tomezyk19} is based on a
\textquotedblleft regularity by decomposition\textquotedblright\ technique
similar to the present one. \cite{nicaise-tomezyk17} studies the same
${\mathbf{H}}^{1}$-based variational formulation and ${\mathbf{H}}^{1}%
$-conforming discretizations for (\ref{eq:model}) on certain polyhedral
domains and obtains $k$-explicit conditions on the discretization for
quasi-optimality. Key to this is a description of the solution regularity in
\cite{nicaise-tomezyk17} in terms
of corner and edge singularities. {The work \cite{chaumont-frelet-vega21}
studies fixed (but arbitrary) 
order ${\mathbf{H}}(\operatorname{curl})$-conforming discretizations of
heterogeneous Maxwell problems and shows a similar quasi-optimality result by
generalizing the corresponding Helmholtz result
\cite{chaumont-frelet-nicaise20}; the restriction to finite order methods
compared to the present work appears to be due to the difference in which the
decomposition of solutions of Maxwell problems is obtained.} High order
Discontinuous Galerkin (DG) and Hybridizable DG (HDG) methods for
(\ref{eq:model}) have been presented in \cite{feng-wu14} and
\cite{lu-chen-qiu17} together with a stability analysis that is explicit in
$h$, $k$, and $p$. {A dispersion analysis of high order methods on
tensor-product meshes is given in \cite{ainsworth04b}.}

The outline of the paper is as follows. Section~\ref{sec:setting} introduces
the notation and tools such as regular decompositions (see
Section~\ref{sec:regular-decomposition}) that are indispensable for the
analysis of Maxwell problems. Section~\ref{sec:stability}
(Theorem~\ref{TheoNonStarShaped}) shows that the solution of (\ref{eq:model})
depends only polynomially on the wavenumber $k$. This stability result is
obtained using layer potential techniques in the spirit of earlier work
\cite[Thm.~{2.4}]{MelenkHelmStab2010} for the analogous Helmholtz equation.
While earlier stability estimates for (\ref{eq:model}) in
\cite{hiptmair-moiola-perugia11, feng-wu14,verfuerth19}, and \cite[Thm.~{5.2}%
]{nicaise-tomezyk17} are obtained by judicious choices of test functions and
rely on star-shapedness of the geometry, Theorem~\ref{TheoNonStarShaped} does
not require star-shapedness. It is worth mentioning that at least 
in the analogous case of the Helmholtz equation, alternatives to the 
use of suitable test functions or layer potential exist, which can lead
to better $k$-dependencies; we refer to \cite{spence14} for results and 
a discussion. 
Section~\ref{sec:good-sign} analyzes a
\textquotedblleft sign definite\textquotedblright\ Maxwell problem and
presents $k$-explicit regularity assertions for it
(Theorem~\ref{lemma:apriori-with-good-sign}). The motivation for studying this
particular boundary value problem is that, since the principal parts of our
sign-definite Maxwell operator and that of (\ref{eq:model}) coincide, a
contraction argument can be brought to bear in the proof of
Theorem~\ref{TheoMainIt}. A similar technique has recently been used for
heterogeneous Helmholtz problems in \cite{bernkopf-chaumont-melenk21}.
Section~\ref{SecHRP} collects $k$-explicit regularity assertions for
(\ref{eq:model}) (Lemma~\ref{lemma:MW-regularity} for finite regularity data
and Theorem~\ref{ThmAnaRegSum} for analytic data). The contraction argument in
the proof of Theorem~\ref{TheoMainIt} relies on certain frequency splitting
operators (both in the volume and on the boundary), which are provided in
Section~\ref{SecFreqSplit}. Section~\ref{SecReg} presents the main analytical
result, Theorem~\ref{TheoMainIt}, where the solution of (\ref{eq:model}) with
finite regularity data $\mathbf{f}$, $\mathbf{g}$ is decomposed into a part
with finite regularity but $k$-uniform bounds, a gradient field, and an
analytic part. Section~\ref{sec:discretization} presents the discretization of
(\ref{eq:model}) based on high order N\'{e}d\'{e}lec elements 
and presents $hp$-approximation operators that map into 
N\'{e}d\'{e}lec spaces. These operators are the same ones as used 
in \cite{MelenkSauterMaxwell_I} but we work out their approximation properties 
on the skeleton of the mesh since stronger approximation properties on the boundary 
$\partial\Omega$ are required in the present case of impedance boundary 
conditions. 
Section~\ref{SecStabConv} shows quasi-optimality
(Theorem~\ref{thm:quasi-optimality}) under the scale resolution condition
(\ref{eq:scale-resolution}). Section~\ref{sec:numerics} concludes the paper
with numerical results.


\section{Setting}

\label{sec:setting}
\subsection{Geometric setting and Sobolev spaces on Lipschitz
domains\label{SecGeometry}}


Let $\Omega\subset\mathbb{R}^{3}$ be a bounded Lipschitz domain which we
assume throughout the paper to have a simply connected and sufficiently smooth
boundary $\Gamma:=\partial\Omega$; if less regularity is required, we will
specify this. {We flag already at this point that the main quasi-optimal
convergence result, Theorem~\ref{thm:quasi-optimality} will require
analyticity of $\Gamma$.} The outward unit normal vector field is denoted by
$\mathbf{n}:\Gamma\rightarrow\mathbb{S}_{2}$.

The Maxwell problem in the frequency domain involves the wavenumber (denoted
by $k$) and we assume that\footnote{We exclude here a neighborhood of $0$
since we are interested in the high-frequency behavior -- to simplify notation
we have fixed $k_{0}=1$ while any other positive choice $k_{0}\in\left(
0,1\right)  $ leads to qualitatively the same results while constants then
depend continuously on $k_{0}\in\left(  0,1\right)  $ and, possibly,
deteriorate as $k_{0}\rightarrow0$.}%
\begin{equation}
k\in\mathbb{R}\backslash\left(  -k_{0},k_{0}\right)  \quad\text{for }k_{0}=1.
\label{loweromega}%
\end{equation}

Let $L^{2}(\Omega)$ denote the usual Lebesgue space on $\Omega$ with scalar
product $\left(  \cdot,\cdot\right)  _{L^{2}(\Omega)}$ and norm $\left\Vert
\cdot\right\Vert _{L^{2}(\Omega)}:=\left(  \cdot,\cdot\right)  _{L^{2}%
(\Omega)}^{1/2}$.
Recall that the complex conjugation is applied to the second argument in $\left(
\cdot,\cdot\right)  _{L^{2}(\Omega)}$. If the domain $\Omega$ is clear from
the context we write short $\left(  \cdot,\cdot\right)  $, $\left\Vert
\cdot\right\Vert $ for $\left(  \cdot,\cdot\right)  _{L^{2}\left(
\Omega\right)  }$, $\left\Vert \cdot\right\Vert _{L^{2}(\Omega)}$. 
For Sobolev spaces, we follow the notation of \cite{Mclean00}.
For $s \ge 0$ we denote 
by $H^{s}(\Omega)$ the usual Sobolev spaces of index $s$ with norm
$\left\Vert \cdot\right\Vert _{H^{s}\left(  \Omega\right)  }$ 
and by $\widetilde{H}^s(\Omega) = H^s_{\overline{\Omega}}({\mathbb R}^3)$ 
the space of Sobolev functions on ${\mathbb R}^3$ with support in $\overline{\Omega}$.  
For $s \ge 0$, $H^{-s}(\Omega)$ denotes the dual of $\widetilde{H}^s(\Omega)$. 
The space $\mathbf{H}^s(\Omega)$ of vector-valued functions is characterized by 
componentwise membership in $H^s(\Omega)$. 
We write $(\cdot,\cdot)$ also for the vectorial $\mathbf{L}^{2}(\Omega)$ inner
product given by $(\mathbf{f},\mathbf{g})=\int_{\Omega}\langle\mathbf{f}%
,\overline{\mathbf{g}}\rangle$. 
Here, we introduce for vectors $\mathbf{a,b}\in\mathbb{C}^{3}$ with
$\mathbf{a}=(a_{j})_{j=1}^{3}$, $\mathbf{b=}(b_{j})_{j=1}^{3}$ the bilinear
form $\langle\cdot,\cdot\rangle$ by $\langle\mathbf{a},\mathbf{b}\rangle
:=\sum_{j=1}^{3}a_{j}b_{j}$. For $m\in\mathbb{N}_{0}$, we introduce the
seminorms
\begin{equation}
\left\vert \mathbf{f}\right\vert _{\mathbf{H}^{m}(\Omega)}:=\left(
\sum_{\alpha\in\mathbb{N}_{0}^{3}\colon|\alpha|=m}\frac{|\alpha|!}{\alpha
!}\left(  \partial^{\alpha}\mathbf{f},\partial^{\alpha}\mathbf{f}\right)
\right)  ^{1/2}%
\end{equation}
and the full norms $\Vert\mathbf{f}\Vert_{\mathbf{H}^{m}(\Omega)}^{2}%
:=\sum_{n=0}^{m}|\mathbf{f}|_{\mathbf{H}^{n}(\Omega)}^{2}$. For the Maxwell
problem the space $\mathbf{H}(\operatorname*{curl})$ is
the key to describe the energy of the electric field. For $m\in\mathbb{N}_{0}$ we set%

\begin{subequations}
\label{defXHcurl}%
\begin{align}
\mathbf{H}^{m}\left(  \operatorname{curl},\Omega\right)  & :=\left\{
\mathbf{u}\in\mathbf{H}^{m}(  \Omega)  \mid\operatorname{curl}%
\mathbf{u}\in\mathbf{H}^{m}(  \Omega)  \right\}  \quad\text{and} \\ 
\mathbf{X} & :=\mathbf{H}(  \operatorname*{curl},\Omega)  :=\mathbf{H}%
^{0}(  \operatorname*{curl},\Omega)  . 
\end{align}
\end{subequations}
The
space $\mathbf{H}^{m}(\operatorname*{div},\Omega)$  
is given for $m\in\mathbb{N}_{0}$ by%
\begin{equation}
\mathbf{H}^{m}(  \operatorname{div},\Omega)  :=\left\{
\mathbf{u}\in\mathbf{H}^{m}(  \Omega)  \mid\operatorname{div}%
\mathbf{u}\in H^{m}(  \Omega)  \right\}  \label{DefHOmegadiv}%
\end{equation}
with $\mathbf{H}(\operatorname{div},\Omega):=\mathbf{H}^{0}(\operatorname{div}%
,\Omega)$. We introduce
\begin{equation}
\mathbf{H}(\operatorname*{div}\nolimits_{0},\Omega):=\left\{  \mathbf{u}%
\in\mathbf{H}(  \operatorname*{div},\Omega)  \mid\operatorname{div}%
\mathbf{u}=0\right\}  . \label{Hdiv0}%
\end{equation}
For $\rho\in\mathbb{R}\setminus \{  0\}  $ and $m$, $\ell
\in\mathbb{N}_{0}$ we define the indexed norms and seminorms by%
\begin{equation}
\left\vert \mathbf{v}\right\vert _{\mathbf{H}^{\ell}(\Omega),\rho}:=\left\vert
\rho\right\vert ^{-\ell}\left\vert \mathbf{v}\right\vert _{\mathbf{H}^{\ell
}(\Omega)}\quad\text{and\quad}\left\Vert \mathbf{v}\right\Vert _{\mathbf{H}%
^{m}(\Omega),\rho}:=\left(  \sum_{\ell=0}^{m}\left\vert \mathbf{v}\right\vert
_{\mathbf{H}^{\ell}(\Omega),\rho}^{2}\right)  ^{1/2} \label{eq:Hl-rho}%
\end{equation}
and corresponding dual norms%
\begin{equation}
\left\Vert \mathbf{v}\right\Vert _{\mathbf{H}^{-m}(\Omega),\rho}:=\left\vert
\rho\right\vert ^{m}\left\Vert \mathbf{v}\right\Vert _{\mathbf{H}^{-m}%
(\Omega)}. \label{eq:Hl-rho-dual}%
\end{equation}
We define for $\operatorname*{D}\in\left\{  \operatorname*{curl}%
,\operatorname*{div}\right\}  $%
\begin{align*}
\left\Vert \mathbf{f}\right\Vert _{\mathbf{H}^{m}(\operatorname*{D}%
,\Omega),\rho}  &  :=\left(  \rho^{-2m}\left\vert \operatorname*{D}%
\mathbf{f}\right\vert _{\mathbf{H}^{m}(\Omega)}^{2}+\rho^{2}\left\Vert
\mathbf{f}\right\Vert _{\mathbf{H}^{m}(\Omega),\rho}^{2}\right)  ^{1/2}\\
&  =\left(  \rho^{-2m}\left\vert \operatorname*{D}\mathbf{f}\right\vert
_{\mathbf{H}^{m}\left(  \Omega\right)  }^{2}+\sum_{\ell=0}^{m}\rho^{2-2\ell
}\left\vert \mathbf{f}\right\vert _{\mathbf{H}^{\ell}(\Omega)}^{2}\right)
^{1/2}%
\end{align*}
and introduce the shorthands:%
\begin{align}
\left\Vert \mathbf{f}\right\Vert _{\mathbf{H}^{m}\left(  \operatorname*{D}%
,\Omega\right)  }  &  :=\left\Vert \mathbf{f}\right\Vert _{\mathbf{H}%
^{m}\left(  \operatorname*{D},\Omega\right)  ,1},\nonumber\label{eq:HD-norm}\\
\left\Vert \mathbf{f}\right\Vert _{\mathbf{H}\left(  \operatorname*{D}%
,\Omega\right)  ,\rho}  &  :=\left\Vert \mathbf{f}\right\Vert _{\mathbf{H}%
^{0}\left(  \operatorname*{D},\Omega\right)  ,\rho}=\left(  \left\Vert
\operatorname*{D}\mathbf{f}\right\Vert ^{2}+\rho^{2}\left\Vert \mathbf{f}%
\right\Vert ^{2}\right)  ^{1/2},\\
\left\Vert \mathbf{f}\right\Vert _{\mathbf{H}\left(  \operatorname*{D}%
,\Omega\right)  }  &  :=\left\Vert \mathbf{f}\right\Vert _{\mathbf{H}%
^{0}\left(  \operatorname*{D},\Omega\right)  }.
\end{align}
We close this section with the introduction of the spaces of analytic functions:

\begin{definition}
\label{DefClAnFct}For an open set $\omega\subset\mathbb{R}^{3}$, constants
$C_{1}$, $\gamma_{1}>0$, and wavenumber $\left\vert k\right\vert \geq1$, we
set
\[
\mathcal{A}(C_{1},\gamma_{1},\omega):=\left\{  {\mathbf{u}}\in(C^{\infty
}(\omega))^{3}\mid\left\vert {\mathbf{u}}\right\vert _{\mathbf{H}^{n}(\omega
)}\leq C_{1}\gamma_{1}^{n}\max\left\{  n+1,\left\vert k\right\vert \right\}
^{n}\;\forall n\in\mathbb{N}_{0}\right\}  .
\]

\end{definition}


\subsection{Sobolev spaces on a sufficiently smooth surface $\Gamma$}


The Sobolev spaces on the boundary $\Gamma$ are denoted by $H^{s}( \Gamma) $
for scalar-valued functions and by $\mathbf{H}^{s}( \Gamma) $ for
vector-valued functions with norms $\left\Vert \cdot\right\Vert _{H^{s}(
\Gamma) }$, $\left\Vert \cdot\right\Vert _{\mathbf{H}^{s}( \Gamma) }$ (see,
e.g., \cite[p.~{98}]{Mclean00}). Note that the range of $s$ for which $H^{s}(
\Gamma) $ is defined may be limited, depending on the global smoothness of the
surface $\Gamma$; for Lipschitz surfaces, $s$ can be chosen in the range
$\left[  0,1\right]  $. For $s<0$, the space $H^{s}( \Gamma) $ is the dual of
$H^{-s}( \Gamma) $.

For a scalar-valued function $u$ and vector-valued function $\mathbf{v}$ on
$\Gamma$, both sufficiently smooth, 
the constant extensions (along the normal direction) into a sufficiently
small three-dimensional neighborhood ${\mathcal{U}}$ of $\Gamma$ is denoted by
$u^{\ast}$ and $\mathbf{v}^{\ast}$. The \textit{surface gradient}
$\nabla_{\Gamma}$, the \textit{tangential curl} $\overrightarrow
{\operatorname*{curl}\nolimits_{\Gamma}}$, and the \textit{surface divergence}
$\operatorname*{div}_{\Gamma}$ are defined by (cf., e.g., \cite{Nedelec01},
\cite{BuffaCostabelSheen})%
\begin{equation}
\nabla_{\Gamma}u:=\left.  \left(  \nabla u^{\star}\right)  \right\vert
_{\Gamma}\text{,\quad}\overrightarrow{\operatorname*{curl}\nolimits_{\Gamma}%
}u:=\nabla_{\Gamma}u\times\mathbf{n}\text{,\quad and\quad}\operatorname*{div}%
\nolimits_{\Gamma}\mathbf{v}=\left.  \left(  \operatorname*{div}%
\mathbf{v}^{\ast}\right)  \right\vert _{\Gamma}\qquad\text{on }\Gamma\text{.}
\label{curlvec}%
\end{equation}
The scalar counterpart of the tangential curl is the \textit{surface curl}%
\begin{equation}
\operatorname*{curl}\nolimits_{\Gamma}\mathbf{v}:=\left\langle \left.  \left(
\operatorname*{curl}\mathbf{v}^{\ast}\right)  \right\vert _{\Gamma}%
,\mathbf{n}\right\rangle \qquad\text{on }\Gamma. 
\label{sccounttangcurl}%
\end{equation}
The composition of the surface divergence and surface gradient leads to the
\textit{scalar Laplace-Beltrami operator} (see \cite[(2.5.191)]{Nedelec01})%
\begin{equation}
\Delta_{\Gamma}u=\operatorname*{div}\nolimits_{\Gamma}\nabla_{\Gamma
}u=-\operatorname*{curl}\nolimits_{\Gamma}\overrightarrow{\operatorname*{curl}%
\nolimits_{\Gamma}}u. 
\label{defLaplBelt}%
\end{equation}
{}From \cite[(2.5.197)]{Nedelec01} it follows that 
\begin{equation*}
\operatorname*{div}\nolimits_{\Gamma}\left(  \mathbf{v}\times\mathbf{n}%
\right)  =\operatorname*{curl}\nolimits_{\Gamma}\mathbf{v}. 
\end{equation*}

Next, we introduce Hilbert spaces of tangential fields on the compact and
\textit{simply connected} manifold $\Gamma$ and corresponding norms and refer
for their definitions and properties to \cite[Sec.~{5.4.1}]{Nedelec01}. We start
with the definition of the space $\mathbf{L}_{T}^{2}(\Gamma)$ of tangential
vector fields given by
\begin{equation}
\mathbf{L}_{T}^{2}(\Gamma):=\left\{  \mathbf{v}\in\mathbf{L}^{2}(\Gamma
)\mid\left\langle \mathbf{n},\mathbf{v}\right\rangle =0\text{ on }%
\Gamma\right\}  . \label{DefL2t}%
\end{equation}
Any tangential field $\mathbf{v}_{T}$ on $\Gamma$ then can be represented in
terms of the \textit{Hodge decomposition}\footnote{Throughout the paper we use
the convention that if $\mathbf{v}_{T}$, $\mathbf{v}_{T}^{\nabla}$,
$\mathbf{v}_{T}^{\operatorname*{curl}}$, $V^{\nabla}$,
$V^{\operatorname*{curl}}$ appear in the same context they are related by
(\ref{vFExp}).}as
\begin{equation}
\mathbf{v}_{T}=\mathbf{v}_{T}^{\nabla}+\mathbf{v}_{T}^{\operatorname*{curl}%
}\quad\text{with\quad}\mathbf{v}_{T}^{\nabla}:=\nabla_{\Gamma}V^{\nabla}%
\quad\text{and\quad}\mathbf{v}_{T}^{\operatorname*{curl}}:=\overrightarrow
{\operatorname*{curl}\nolimits_{\Gamma}}V^{\operatorname*{curl}} \label{vFExp}%
\end{equation}
for some scalar potentials%
\[
V^{\nabla}\in H^{1}(\Gamma)\quad\text{and\quad}V^{\operatorname*{curl}}\in
H(\overrightarrow{\operatorname*{curl}\nolimits_{\Gamma}},\Gamma):=\left\{
\phi\in L^{2}(\Gamma)\mid\overrightarrow{\operatorname*{curl}\nolimits_{\Gamma
}}\phi\in\mathbf{L}_{T}^{2}(\Gamma)\right\}  .
\]
In particular, this decomposition is $\mathbf{L}_{T}^{2}$-orthogonal:
\[
\left(  \mathbf{v}_{T}^{\nabla},\mathbf{v}_{T}^{\operatorname*{curl}}\right)
_{\mathbf{L}^{2}(\Gamma)}=\left(  \nabla_{\Gamma}V^{\nabla},\overrightarrow
{\operatorname*{curl}\nolimits_{\Gamma}}V^{\operatorname*{curl}}\right)
_{\mathbf{L}^{2}(\Gamma)}=0\quad\forall\,\mathbf{v}_{T} = \mathbf{v}_{T}^{\nabla} + \mathbf{v}_{T}^{\operatorname{curl}} \text{ as in
(\ref{vFExp}).}%
\]
Hence, the splitting (\ref{vFExp}) is stable:
\begin{align*}
\left\Vert \mathbf{v}_{T}\right\Vert _{\mathbf{L}^{2}(\Gamma)}  &  =\left(
\left\Vert \nabla_{\Gamma}V^{\nabla}\right\Vert _{\mathbf{L}^{2}(\Gamma)}%
^{2}+\left\Vert \overrightarrow{\operatorname*{curl}\nolimits_{\Gamma}%
}V^{\operatorname*{curl}}\right\Vert _{\mathbf{L}^{2}(\Gamma)}^{2}\right)
^{1/2},\\
\left\Vert \nabla_{\Gamma}V^{\nabla}\right\Vert _{\mathbf{L}^{2}(\Gamma)}  &
\leq\left\Vert \mathbf{v}_{T}\right\Vert _{\mathbf{L}^{2}\left(
\Gamma\right)  }\quad\text{and\quad}\left\Vert \overrightarrow
{\operatorname*{curl}\nolimits_{\Gamma}}V^{\operatorname*{curl}}\right\Vert
_{\mathbf{L}^{2}(\Gamma)}\leq\left\Vert \mathbf{v}_{T}\right\Vert
_{\mathbf{L}^{2}(\Gamma)}.
\end{align*}
Higher order spaces are defined for $s>0$ by
\begin{equation}
\mathbf{H}_{T}^{s}(\Gamma):=\left\{  \mathbf{v}_{T}\in\mathbf{L}_{T}%
^{2}(\Gamma)\mid\left\Vert \mathbf{v}_{T}\right\Vert _{\mathbf{H}^{s}(
\Gamma)  }<\infty\right\}  \label{DefHsGammaTNorm}%
\end{equation}
and for negative $s$ by duality.

The $H^{s}(\Gamma)$-norm of $\operatorname*{curl}\nolimits_{\Gamma}\left(
\cdot\right)  $ and $\operatorname*{div}_{\Gamma}\left(  \cdot\right)  $ can
be expressed by using the Hodge decomposition: 
\begin{align*}
\left\Vert \operatorname*{curl}\nolimits_{\Gamma}\mathbf{v}_{T}\right\Vert
_{H^{s}(\Gamma)}  &  =\left\Vert \operatorname*{curl}\nolimits_{\Gamma
}\mathbf{v}_{T}^{\operatorname*{curl}}\right\Vert _{H^{s}(\Gamma)}\!=\!\left\Vert
\operatorname*{curl}\nolimits_{\Gamma}\overrightarrow{\operatorname*{curl}%
\nolimits_{\Gamma}}V^{\operatorname*{curl}}\right\Vert _{H^{s}(\Gamma
)}\!\!=\!\left\Vert \Delta_{\Gamma}V^{\operatorname*{curl}}\right\Vert
_{H^{s}(\Gamma)},
\\
\left\Vert \operatorname*{div}\nolimits_{\Gamma}\mathbf{v}_{T}\right\Vert
_{H^{s}(\Gamma)}  &  =\left\Vert \operatorname*{div}\nolimits_{\Gamma
}\mathbf{v}_{T}^{\nabla}\right\Vert _{H^{s}(\Gamma)}=\left\Vert
\operatorname*{div}\nolimits_{\Gamma}\nabla_{\Gamma}V^{\nabla}\right\Vert
_{H^{s}(\Gamma)}=\left\Vert \Delta_{\Gamma}V^{\nabla}\right\Vert
_{H^{s}(\Gamma)}.
\end{align*}
{We define}%
\begin{subequations}
\label{m1/2curldiv}
\begin{align}
\nonumber 
\left\Vert \mathbf{v}_{T}\right\Vert _{\mathbf{H}^{-1/2}(
\operatorname*{curl}\nolimits_{\Gamma},\Gamma)  }  &  :=\left(
\left\Vert \operatorname*{curl}\nolimits_{\Gamma}\mathbf{v}_{T}%
^{\operatorname*{curl}}\right\Vert _{H^{-1/2}(  \Gamma)  }%
^{2}+\left\Vert \mathbf{v}_{T}\right\Vert _{\mathbf{H}^{-1/2}(
\Gamma)  }^{2}\right)  ^{1/2} \\
& =\left(  \left\Vert \Delta_{\Gamma
}V^{\operatorname*{curl}}\right\Vert _{H^{-1/2}(  \Gamma)  }%
^{2}+\left\Vert \mathbf{v}_{T}\right\Vert _{\mathbf{H}^{-1/2}(
\Gamma)  }^{2}\right)  ^{1/2},
\label{m1/2curldiva}\\
\nonumber 
\left\Vert \mathbf{v}_{T}\right\Vert _{\mathbf{H}^{-1/2}(
\operatorname{div}_{\Gamma},\Gamma)  }  &  :=\left(  \left\Vert
\operatorname*{div}\nolimits_{\Gamma}\mathbf{v}_{T}^{\nabla}\right\Vert
_{H^{-1/2}(  \Gamma)  }^{2}+\left\Vert \mathbf{v}_{T}\right\Vert
_{\mathbf{H}^{-1/2}(  \Gamma)  }^{2}\right)  ^{1/2}
\\
& =\left(
\left\Vert \Delta_{\Gamma}V^{\nabla}\right\Vert _{H^{-1/2}\left(
\Gamma\right)  }^{2}+\left\Vert \mathbf{v}_{T}\right\Vert _{\mathbf{H}%
^{-1/2}(  \Gamma)  }^{2}\right)  ^{1/2}. 
\label{m1/2curldivb}%
\end{align}
\end{subequations}

The corresponding spaces $\mathbf{H}_{T}^{-1/2}(\operatorname*{curl}%
\nolimits_{\Gamma},\Gamma)$ and $\mathbf{H}_{T}^{-1/2}(\operatorname*{div}%
\nolimits_{\Gamma},\Gamma)$ are characterized by%
\begin{equation}%
\begin{array}
[c]{ccc}%
\mathbf{v}_{T}\in\mathbf{H}_{T}^{-1/2}(\operatorname*{div}\nolimits_{\Gamma
},\Gamma) & \iff & \mathbf{v}_{T}\text{ has form (\ref{vFExp}) and
}\left\Vert \mathbf{v}_{T}\right\Vert _{\mathbf{H}^{-1/2}\left(
\operatorname{div}_{\Gamma},\Gamma\right)  }<\infty,\\
\mathbf{v}_{T}\in\mathbf{H}_{T}^{-1/2}(\operatorname*{curl}\nolimits_{\Gamma
},\Gamma) & \iff & \mathbf{v}_{T}\text{ has form (\ref{vFExp}) and
}\left\Vert \mathbf{v}_{T}\right\Vert _{\mathbf{H}^{-1/2}\left(
\operatorname*{curl}\nolimits_{\Gamma},\Gamma\right)  }<\infty.
\end{array}
\label{Hm1/2divspace}%
\end{equation}
We also introduce indexed norms for functions in Sobolev spaces on the
boundary: for $\nu\in\mathbb{R}$ with $2\nu\in\mathbb{N}_{0}$, we formally set%
\begin{subequations}
\label{defgtGammaind}%
\begin{align}
\left\Vert \mathbf{g}_{T}\right\Vert _{\mathbf{H}^{\nu}(\Gamma),k}
& :=\left(
\sum\limits_{\ell=0}^{2\nu}
\left\vert k\right\vert ^{1-\ell}\left\Vert \mathbf{g}_{T}\right\Vert
_{\mathbf{H}^{\ell/2}(\Gamma)}^{2}\right)  ^{1/2}\quad\text{and}  \\
\left\Vert \mathbf{g}_{T}\right\Vert _{\mathbf{H}^{-\nu}(\Gamma),k} 
& :=\left\vert k\right\vert ^{\nu+1/2}\left\Vert \mathbf{g}_{T}\right\Vert
_{\mathbf{H}^{-\nu}(\Gamma)}.
\end{align}
\end{subequations}
For $\operatorname*{D}_{\Gamma}\in\left\{  \operatorname*{curl}%
\nolimits_{\Gamma},\operatorname*{div}_{\Gamma}\right\}  $, we
introduce\footnote{%
We always write $\left\vert k\right\vert $ in the estimates also if the
exponent is even for the sake of clarity.}%
\begin{equation}
\left\Vert \mathbf{g}_{T}\right\Vert _{\mathbf{H}^{\nu}(\operatorname*{D}%
_{\Gamma},\Gamma),k}:=\left(  \left\Vert \operatorname*{D}\nolimits_{\Gamma
}\mathbf{g}_{T}\right\Vert _{\mathbf{H}^{\nu}(\Gamma),k}^{2}+\left\vert
k\right\vert ^{2}\left\Vert \mathbf{g}_{T}\right\Vert _{\mathbf{H}^{\nu
}(\Gamma),k}^{2}\right)  ^{1/2}.\label{eq:Hgamma-D}%
\end{equation}
In particular, we have
\begin{equation*}
\left\Vert \mathbf{g}_{T}\right\Vert _{\mathbf{H}^{0}(  \Gamma)
,k}=\left\vert k\right\vert ^{1/2}\left\Vert \mathbf{g}_{T}\right\Vert
_{\mathbf{L}^2(\Gamma)}\quad\text{and\quad}\left\Vert \mathbf{g}_{T}\right\Vert
_{\mathbf{H}^{\nu}\left(  \Gamma\right)  }\leq C\left\vert k\right\vert
^{-1/2+\nu}\left\Vert \mathbf{g}_{T}\right\Vert _{\mathbf{H}^{\nu}\left(
\Gamma\right)  ,k}.
\end{equation*}
We remark that the special dual norms $\Vert\cdot\Vert_{\mathbf{H}%
^{-1/2}(\operatorname{div}_{\Gamma},\Gamma),k}$ and $\Vert\cdot\Vert
_{\mathbf{X}_{\operatorname*{imp}}^{\prime}(  \Gamma)  ,k}$ on the
boundary will be defined later in (\ref{eq:H-1/2diva}) and in
(\ref{eq:H-1/2div}). By using standard interpolation inequalities for Sobolev
spaces we obtain the following lemma.

\begin{lemma}
\label{LemBoundNorm}For $m\in\mathbb{N}_{0}$, there holds%
\begin{align}
\nonumber 
\left\Vert \mathbf{g}_{T}\right\Vert _{\mathbf{H}^{m+1/2}(\Gamma),k}%
&  
\leq C\left(  \left\vert k\right\vert \left\Vert \mathbf{g}_{T}\right\Vert
_{\mathbf{L}^{2}(\Gamma)}^{2}+%
\sum\limits_{r=1}^{m+1}
\left\vert k\right\vert ^{2-2r}\left\Vert \mathbf{g}_{T}\right\Vert
_{\mathbf{H}^{r-1/2}(\Gamma)}^{2}\right)  ^{1/2} 
\\%
\nonumber 
&  
\leq C\left(
\displaystyle\sum\limits_{r=0}^{m+1}
\left\vert k\right\vert ^{2-2r}\left\Vert \mathbf{g}_{T}\right\Vert
_{\mathbf{H}^{r-1/2}(\Gamma)}^{2}\right)  ^{1/2},%
\\
\left\Vert \mathbf{g}_{T}\right\Vert _{\mathbf{H}^{m+1/2}(\Gamma),k}  &  \leq
C\left(  \left\vert k\right\vert \left\Vert \mathbf{g}_{T}\right\Vert
_{\mathbf{L}^{2}(\Gamma)}^{2}+\left\vert k\right\vert ^{-2m}\left\Vert
\mathbf{g}_{T}\right\Vert _{\mathbf{H}^{m+1/2}(\Gamma)}^{2}\right)
^{1/2}\label{bnormests}\\
&  \leq C\left(  \left\vert k\right\vert ^{2}\left\Vert \mathbf{g}%
_{T}\right\Vert _{\mathbf{H}^{-1/2}(\Gamma)}^{2}+\left\vert k\right\vert
^{-2m}\left\Vert \mathbf{g}_{T}\right\Vert _{\mathbf{H}^{m+1/2}(\Gamma)}%
^{2}\right)  ^{1/2}.\nonumber
\end{align}

\end{lemma}


\subsection{Trace operators and energy spaces for Maxwell's equations}


We introduce tangential trace operators $\Pi_{T}$ and $\gamma_{T}$, which map
sufficiently smooth functions $\mathbf{u}$ in $\overline{\Omega}$ to
tangential fields on $\Gamma$, by 
\begin{equation}
\Pi_{T}:\mathbf{u}\mapsto\mathbf{n}\times\left(  \mathbf{u}|_{\Gamma}%
\times\mathbf{n}\right)  ,\quad\gamma_{T}:\mathbf{u}\mapsto\mathbf{u}%
|_{\Gamma}\times\mathbf{n}. \label{eq:trace-operators}%
\end{equation}

The following theorem shows that $\mathbf{H}_{T}^{-1/2}(\operatorname*{div}%
\nolimits_{\Gamma},\Gamma)$ and $\mathbf{H}_{T}^{-1/2}(\operatorname*{curl}%
\nolimits_{\Gamma},\Gamma)$ are the correct spaces for the continuous
extension of the tangential trace operators to Hilbert spaces.

\begin{proposition}
[{{\cite{Cessenat_book}, \cite[Thm.~{5.4.2}]{Nedelec01}}}]\label{traceTHM1}The
trace mappings $\Pi_{T}$ and $\gamma_{T}$ in (\ref{eq:trace-operators}) extend
to continuous and surjective operators%
\[
\Pi_{T}:\mathbf{X}\rightarrow\mathbf{H}_{T}^{-1/2}(\operatorname*{curl}%
\nolimits_{\Gamma},\Gamma),\qquad\gamma_{T}:\mathbf{X}\rightarrow
\mathbf{H}_{T}^{-1/2}(\operatorname*{div}\nolimits_{\Gamma},\Gamma).
\]
Moreover, for theses trace spaces there exist continuous divergence-free
liftings $\mathcal{E}_{\operatorname*{curl}}^{\Gamma}:\mathbf{H}_{T}%
^{-1/2}(  \operatorname*{curl}\nolimits_{\Gamma},\Gamma)
\rightarrow\mathbf{X}$ and $\mathcal{E}_{\operatorname*{div}}^{\Gamma
}:\mathbf{H}_{T}^{-1/2}(  \operatorname*{div}\nolimits_{\Gamma}%
,\Gamma)  \rightarrow\mathbf{X}$.
\end{proposition}

For a vector field $\mathbf{u}\in\mathbf{X}$, we will employ frequently the
notation%
\[
\mathbf{u}_{T}:=\Pi_{T}\mathbf{u}.
\]
From \cite[(2.5.161), (2.5.208)]{Nedelec01} and the relation $\Pi_{T}\nabla
u\mathbf{=n\times}\left(  \left.  \nabla u\right\vert _{\Gamma}\times
\mathbf{n}\right)  =\left.  \nabla u\right\vert _{\Gamma}-\left(
\partial_{\mathbf{n}}u\right)  \mathbf{n}$ we conclude%
\begin{align}
\Pi_{T}\nabla u  &  =\nabla_{\Gamma}(  u\vert _{\Gamma
}), \label{PiTgradProp}\\
\gamma_{T}\nabla u  &  =(  \Pi_{T}\nabla u)  \times\mathbf{n}%
=\nabla_{\Gamma}(  u\vert _{\Gamma})  \times
\mathbf{n}. \label{gmmaTgradProp}%
\end{align}

\begin{remark}
\label{rem:decomposition-of-nabla-phi}{For gradient fields $\nabla\varphi$ we
have $(\nabla\varphi)_{T}^{\operatorname*{curl}}=0$ and $(\nabla\varphi
)_{T}^{\nabla}=\nabla_{\Gamma}\varphi$.} \hbox{}\hfill%
\eremk
\end{remark}

\begin{definition}
\label{def:Ximp} Let $\Omega\subset\mathbb{R}^{3}$ be a bounded domain with
sufficiently smooth Lipschitz boundary $\Gamma$ as described in 
Section~\ref{SecGeometry}. The \emph{energy space} for Maxwell's equations with
\emph{impedance boundary conditions} on $\Gamma$
and
real wavenumber $k\in\mathbb{R}\backslash\left(  -k_{0},k_{0}\right)  $ is%
\begin{equation}
\mathbf{X}_{\operatorname*{imp}}:=\left\{  \mathbf{u}\in\mathbf{X}\colon
\Pi_{T}\mathbf{u\in L}_{T}^{2}(\Gamma)\right\}  \label{eq:Ximp}%
\end{equation}
with corresponding norm%
\begin{equation*}
\left\Vert \mathbf{u}\right\Vert _{\operatorname*{imp},k}:=\left[  \left\Vert
\operatorname*{curl}\mathbf{u}\right\Vert ^{2}+\left\Vert \mathbf{u}%
\right\Vert _{k,+}^{2}\right]  ^{1/2}\text{ with }\left\Vert
\mathbf{u}\right\Vert _{k,+}:=\left[  k^{2}\left\Vert \mathbf{u}\right\Vert
^{2}+\left\vert k\right\vert \left\Vert \mathbf{u}_{T}\right\Vert
_{\mathbf{L}^{2}(\Gamma)}^{2}\right]  ^{1/2}. 
\label{eq:Himp}%
\end{equation*}
Its \emph{companion space} of scalar potentials is%
\begin{equation}
H_{\operatorname*{imp}}^{1}(\Omega):=\left\{  \varphi\in H^{1}(\Omega
)\mid\left.  \varphi\right\vert _{\Gamma}\in H^{1}(\Gamma)\right\}  .
\label{H1imp}%
\end{equation}

\end{definition}


\subsection{Regular decompositions}

\label{sec:regular-decomposition}
We will rely on various decompositions of functions into regular parts and
gradient parts.
The decompositions may not be orthogonal but must be stable. We refer to
\cite[\S 4.4]{hiptmair-acta} and the bibliographic notes therein for some
early contributions. Many variants have been introduced since then, and the
results in this section are essentially taken from the literature: Lemma~\ref{LemRs} 
is a consequence of \cite[Thm.~{4.6}]{costabel-mcintosh10}; 
Lemma~\ref{lemma:helmholtz-a-la-schoeberl} relies on {\cite[Thm.~{3.38}]{Monkbook}
and Lemma~\ref{LemRs}; Lemma~\ref{lemma:schoeberl-apriori} is based on
}\cite{costabel90} while closely related results can be found in
\cite{amrouche-bernardi-dauge-girault98}. Finally, 
Lemma~\ref{LemHelmDecompVar1} is a consequence of \cite[Thm.~{4.2(2)}]%
{Schweizer-Friedrichs-2016} and \cite{costabel90}. For newer overview
articles, we refer to, e.g., \cite{hipt_reg_decomp}, \cite{hipt_pech}.%
The following Lemma~\ref{LemRs} collects a key result from the seminal paper
\cite{costabel-mcintosh10}. The operator $\mathbf{R}_{2}$, which is
essentially a right inverse of the curl operator, will frequently be employed
in the present paper.

\begin{lemma}
\label{LemRs}Let $\Omega$ be a bounded Lipschitz domain. There exist
pseudodifferential operators $R_{1}$,
$\mathbf{R}_{2}$ 
of order $-1$ and $\mathbf{K}$, $\mathbf{K}_{2}$ of order $-\infty$ on
${\mathbb{R}}^{3}$ with the following properties: For each $m\in{\mathbb{Z}}$
they have the mapping properties $R_{1}:\mathbf{H}^{-m}(\Omega)\rightarrow
H^{1-m}(  \Omega)  $, $\mathbf{R}_{2}:\mathbf{H}^{-m}%
(\Omega)\rightarrow\mathbf{H}^{1-m}(  \Omega)  $, and $\mathbf{K}$,
$\mathbf{K}_{2}:\mathbf{H}^{m}(  \Omega)  \rightarrow({C}^{\infty
}(\overline{\Omega}))^{3}$, and for any $\mathbf{u}\in\mathbf{H}^{m}(
\operatorname*{curl},\Omega)  $ there holds%
\begin{equation}
\mathbf{u}=\nabla R_{1}(  \mathbf{u}-\mathbf{R}_{2}(
\operatorname*{curl}\mathbf{u})  )  +\mathbf{R}_{2}(
\operatorname*{curl}\mathbf{u})  +\mathbf{Ku}. \label{repu}%
\end{equation}
For $\mathbf{u}$ with $\operatorname{div}\mathbf{u}=0$ {on $\Omega$} there
holds
\begin{equation}
\operatorname{curl}\mathbf{R}_{2}\mathbf{u}=\mathbf{u}-\mathbf{K}%
_{2}\mathbf{u}. \label{eq:S=K2}%
\end{equation}
\end{lemma}

\begin{proof}
In \cite[Thm.~{4.6}]{costabel-mcintosh10}, operators $R_{1}$, $\mathbf{R}_{2}%
$, $\mathbf{R}_{3}$, $\mathbf{K}_{1}$, $\mathbf{K}_{2}$ with the mapping
properties%
\begin{equation}
\label{RSrelation-neu}%
\begin{array}
[c]{cl}%
R_{1}:\mathbf{H}^{-m}(\Omega)\rightarrow H^{1-m}\left(  \Omega\right)  , & \\
\mathbf{R}_{2}:\mathbf{H}^{-m}(\Omega)\rightarrow\mathbf{H}^{1-m}\left(
\Omega\right)  , & \\
\mathbf{R}_{3}:H^{-m}(\Omega)\rightarrow\mathbf{H}^{1-m}\left(  \Omega\right)
, & \\
\mathbf{K}_{\ell}:\mathbf{H}^{m}\left(  \Omega\right)  \rightarrow({C}%
^{\infty}(\overline{\Omega}))^{3}, & \ell=1,2, 
\end{array}
\end{equation}
are constructed with
\begin{subequations}
\label{RSrelation}%
\begin{align}
\label{RSrelation-a}%
\nabla R_{1}\mathbf{v}+\mathbf{R}_{2}(\operatorname*{curl}\mathbf{v})  &
=\mathbf{v}-\mathbf{K}_{1}\mathbf{v},
\\
\label{RSrelation-b}%
\operatorname*{curl}\mathbf{R}_{2}\mathbf{v}+\mathbf{R}_{3}\left(
\operatorname*{div}\mathbf{v}\right)   
&  =\mathbf{v}-\mathbf{K}_{2}%
\mathbf{v}. 
\end{align}
\end{subequations}
We note that (\ref{RSrelation}) implies (\ref{eq:S=K2}). It is worth stressing
that the mapping properties given in (\ref{RSrelation-neu}) express a locality of
the operators, which are pseudodifferential operators on $\mathbb{R}^{3}$: on
$\Omega$, the operators depend only on the argument restricted to $\Omega$ and
not on the values on $\mathbb{R}^{3}\setminus\overline{\Omega}$.

Selecting $\mathbf{v}=\mathbf{u}-\mathbf{R}_{2}(\operatorname*{curl}%
\mathbf{u})$ in (\ref{RSrelation-a}) we obtain%
\begin{equation}
\begin{split}
\nabla R_{1}(\mathbf{u}-\mathbf{R}_{2}(\operatorname*{curl}\mathbf{u}%
))+\mathbf{R}_{2}(\operatorname*{curl}(\mathbf{u}-\mathbf{R}_{2}%
(\operatorname*{curl}\mathbf{u})))= 
\\
\mathbf{u}-\mathbf{R}_{2}%
(\operatorname*{curl}\mathbf{u})-\mathbf{K}_{1}(\mathbf{u}-\mathbf{R}%
_{2}(\operatorname*{curl}\mathbf{u})). \label{gradR1}%
\end{split}
\end{equation}
Since $\operatorname*{curl}\mathbf{u}$ is divergence free, we obtain from 
(\ref{RSrelation-b}) 
\begin{align*}
\mathbf{R}_{2}(\operatorname*{curl}(\mathbf{u}-\mathbf{R}_{2}\left(
\operatorname*{curl}\mathbf{u}))\right)   &  =\mathbf{R}_{2}\left(
\operatorname*{curl}\mathbf{u}\right)  -\mathbf{R}_{2}(\operatorname*{curl}%
\mathbf{u}-\mathbf{K}_{2}\operatorname*{curl}\mathbf{u})\\
&  =\mathbf{R}_{2}(\mathbf{K}_{2}(\operatorname*{curl}\mathbf{u}%
))=:\mathbf{K}_{3}\mathbf{u},
\end{align*}
where, again, $\mathbf{K}_{3}$ is a smoothing operator of order $-\infty$.
Inserting this into (\ref{gradR1}) leads to%
\[
\nabla R_{1}\left(  \mathbf{u}-\mathbf{R}_{2}\left(  \operatorname*{curl}%
\mathbf{u}\right)  \right)  +\mathbf{R}_{2}\left(  \operatorname*{curl}%
\mathbf{u}\right)  =\mathbf{u}-\mathbf{K}_{1}\left(  \mathbf{u}-\mathbf{R}%
_{2}\left(  \operatorname*{curl}\mathbf{u}\right)  \right)  -\mathbf{K}%
_{3}\mathbf{u}.
\]
By choosing $\mathbf{K}\mathbf{u}:=\mathbf{K}_{1}\left(  \mathbf{u}%
-\mathbf{R}_{2}\operatorname*{curl}\mathbf{u}\right)  +\mathbf{K}%
_{3}\mathbf{u}$ the representation (\ref{repu}) is proved.%
\end{proof}

\begin{lemma}
\label{lemma:helmholtz-a-la-schoeberl} Let $\Omega$ be a bounded, connected
Lipschitz domain.

\begin{enumerate}
[(i)]

\item \label{item:lemma:helmholtz-a-la-schoeberl-i} There is $C>0$ such that
for every $\mathbf{u}\in\mathbf{X}$ there is a decomposition $\mathbf{u}%
=\nabla\varphi+\mathbf{z}$ with%
\begin{equation}
\operatorname{div}\mathbf{z}=0,\quad\left\Vert \mathbf{z}\right\Vert
_{\mathbf{H}^{1}(\Omega)}\leq C\left\Vert \operatorname*{curl}\mathbf{u}%
\right\Vert ,\quad\left\Vert \varphi\right\Vert _{H^{1}(\Omega)}\leq
C\left\Vert \mathbf{u}\right\Vert _{\mathbf{H}\left(  \operatorname*{curl}%
,\Omega\right)  }. \label{eq:lemma:helmholtz-a-la-schoeberl-i}%
\end{equation}

\item \label{item:lemma:helmholtz-a-la-schoeberl-ii} Let $m\in\mathbb{Z}$. For
each $\mathbf{u}\in\mathbf{H}^{m}(  \operatorname*{curl},\Omega)  $
there is a splitting independent of $m$ of the form $\mathbf{u}=\nabla
\varphi+\mathbf{z}$ with $\varphi\in H^{m+1}(  \Omega)  $,
$\mathbf{z}\in\mathbf{H}^{m+1}(  \Omega)  $ satisfying%
\begin{subequations}
\label{eq:lemma:helmholtz-a-la-schoeberl-ii}%
\begin{align}
\label{eq:lemma:helmholtz-a-la-schoeberl-ii-a}%
\left\Vert \mathbf{z}\right\Vert _{\mathbf{H}^{m+1}(\Omega)}& \leq C\left\Vert
\mathbf{u}\right\Vert _{\mathbf{H}^{m}(  \operatorname*{curl}%
,\Omega)  }\quad\text{and\quad}
\\
\label{eq:lemma:helmholtz-a-la-schoeberl-ii-b}%
\left\Vert \mathbf{z}\right\Vert
_{\mathbf{H}^{m}(\Omega)}+\left\Vert \varphi\right\Vert _{H^{m+1}(\Omega)}
& \leq
C\left\Vert \mathbf{u}\right\Vert _{\mathbf{H}^{m}(  \Omega)  }.
\end{align}
\end{subequations}

\item \label{item:lemma:helmholtz-a-la-schoeberl-iii} There is $C > 0$
depending only on $\Omega$ such that each $\mathbf{u} \in\mathbf{X}%
_{\operatorname{imp}}$ can be written as $\mathbf{u} = \nabla\varphi+
\mathbf{z}$ with $\varphi\in H^{1}_{\operatorname{imp}}(\Omega)$, $\mathbf{z}
\in\mathbf{H}^{1}(\Omega)$ and 
\begin{equation}
\|\nabla\varphi\|_{\operatorname{imp},k} + \|\mathbf{z}\|_{\mathbf{H}%
^{1}(\Omega)} + |k| \|\mathbf{z}\|_{\mathbf{L}^{2}(\Omega)} + |k|^{1/2}
\|\mathbf{z}\|_{\mathbf{L}^{2}(\Gamma)} \leq C \|\mathbf{u}%
\|_{\operatorname{imp},k}.
\end{equation}

\end{enumerate}
\end{lemma}

\begin{proof}
\emph{Proof of (\ref{item:lemma:helmholtz-a-la-schoeberl-i}):} Let
$\mathbf{u}\in\mathbf{X}$. The point is to choose $\mathbf{z}$ in the
splitting $\mathbf{u}=\nabla\varphi+\mathbf{z}$ such that it can be controlled
by $\operatorname*{curl}\mathbf{u}$. To this end, we set $\mathbf{v}%
=\operatorname*{curl}\mathbf{u}\in\mathbf{L}^{2}(\Omega)$ and observe that
$\operatorname*{div}\mathbf{v}=0$ {and therefore $(1,\langle{\mathbf{n}%
},{\mathbf{v}}\rangle)_{L^{2}(\Gamma)}=0$. By \cite[Thm.~{3.38}]{Monkbook},}
this allows\ us to conclude the existence of $\mathbf{z}\in\mathbf{H}%
^{1}(\Omega)$ with $\operatorname*{div}\mathbf{z}=0$, $\mathbf{v}%
=\operatorname*{curl}\mathbf{z}$ and%
\[
\left\Vert \mathbf{z}\right\Vert _{\mathbf{H}^{1}(\Omega)}\leq C\left\Vert
\mathbf{v}\right\Vert .
\]
Since $\operatorname*{curl}(\mathbf{u}-\mathbf{z})=0$, we have $\mathbf{u}%
-\mathbf{z}=\nabla\varphi$ for a $\varphi\in H^{1}(\Omega)$, which trivially
satisfies
\[
\left(  \nabla\varphi,\nabla\psi\right)  =\left(  \mathbf{u}-\mathbf{z}%
,\nabla\psi\right)  \qquad\forall\psi\in H^{1}(\Omega).
\]
By fixing $\varphi$ such that $\int_{\Omega}\varphi=0$, the estimate of
$\varphi$ follows by a Poincar\'{e} inequality.

\emph{Proof of (\ref{item:lemma:helmholtz-a-la-schoeberl-ii}):} With the
operators of Lemma~\ref{LemRs}, we define
\[
\mathbf{z}:=\mathbf{R}_{2}(\operatorname*{curl}\mathbf{u})+\mathbf{K}%
\mathbf{u},\qquad\nabla\varphi:=\nabla R_{1}(\mathbf{u}-\mathbf{R}%
_{2}(\operatorname{curl}\mathbf{u})).
\]
Lemma~\ref{LemRs} implies $\mathbf{u}=\mathbf{z}+\nabla\varphi$ as well as the
bounds by the mapping properties given in Lemma~\ref{LemRs}.

\emph{Proof of (\ref{item:lemma:helmholtz-a-la-schoeberl-iii}):} Multiplying
estimate (\ref{eq:lemma:helmholtz-a-la-schoeberl-ii-b}) for the
decomposition of (\ref{item:lemma:helmholtz-a-la-schoeberl-ii}) and $m=0$ by
$\left\vert k\right\vert $ leads to $\left\vert k\right\vert \Vert
\mathbf{z}\Vert_{\mathbf{L}^{2}(\Omega)}+\left\vert k\right\vert \Vert
\nabla\varphi\Vert_{\mathbf{L}^{2}(\Omega)}\leq C\left\vert k\right\vert
\Vert\mathbf{u}\Vert_{\mathbf{L}^{2}(\Omega)}$. (\ref{eq:lemma:helmholtz-a-la-schoeberl-ii-a}) gives 
$\Vert\mathbf{z} \Vert_{\mathbf{H}^{1}(\Omega)}\leq C\Vert\mathbf{u}\Vert_{\mathbf{H}%
(\operatorname{curl},\Omega)}$. 
The multiplicative trace inequality gives
$|k|\Vert\mathbf{z}\Vert_{\mathbf{L}^{2}(\Gamma)}^{2}\leq C\left\vert
k\right\vert \Vert\mathbf{z}\Vert_{\mathbf{L}^{2}(\Omega)}\Vert\mathbf{z}%
\Vert_{\mathbf{H}^{1}(\Omega)}\leq C\left\vert k\right\vert ^{2}%
\Vert\mathbf{z}\Vert_{\mathbf{L}^{2}(\Omega)}^{2}+C\Vert\mathbf{z}%
\Vert_{\mathbf{H}^{1}(\Omega)}^{2}$. Hence follows $\Vert\mathbf{z}\Vert
_{\operatorname{imp},k}\leq C\Vert\mathbf{u}\Vert_{\mathbf{H}%
(\operatorname{curl},\Omega),k}\leq C\Vert\mathbf{u}\Vert_{\operatorname{imp}%
,k}$. The triangle inequality then provides the bound $\Vert\nabla\varphi
\Vert_{\operatorname{imp},k}\leq\Vert\mathbf{u}\Vert_{\operatorname{imp}%
,k}+\Vert\mathbf{z}\Vert_{\operatorname{imp},k}\leq C\Vert\mathbf{u}%
\Vert_{\operatorname{imp},k}$.
\end{proof}

The following result relates the space $\mathbf{H}(\operatorname{curl}%
,\Omega)\cap\mathbf{H}(\operatorname{div},\Omega)$ to classical Sobolev
spaces. The statement (\ref{ucurldivrandest2}) is from \cite{costabel90};
closely related results can be found in
\cite{amrouche-bernardi-dauge-girault98}.

\begin{lemma}
\label{lemma:schoeberl-apriori} Let $\partial\Omega$ be smooth and simply
connected. Then there is $C>0$ such that for every $\mathbf{u}\in
\mathbf{H}\left(  \operatorname*{curl},\Omega\right)  \cap\mathbf{H}\left(
\operatorname{div},\Omega\right)  $ there holds%
\begin{subequations}
\label{ucurldivrandest}
\begin{align}
\left\Vert \mathbf{u}\right\Vert  &  \leq C\left(  \left\Vert
\operatorname*{curl}\mathbf{u}\right\Vert +\left\Vert \operatorname{div}%
\mathbf{u}\right\Vert +\left\Vert \left\langle \mathbf{u},\mathbf{n}%
\right\rangle \right\Vert _{H^{-1/2}(  \Gamma)  }\right)  ,
\label{ucurldivrandesta}\\
\left\Vert \mathbf{u}\right\Vert  &  \leq C\left(  \left\Vert
\operatorname*{curl}\mathbf{u}\right\Vert +\left\Vert \operatorname{div}%
\mathbf{u}\right\Vert +\left\Vert \gamma_{T}\mathbf{u}\right\Vert
_{\mathbf{H}^{-1/2}(  \Gamma)  }\right)  . 
\label{ucurldivrandestb}%
\end{align}%
\end{subequations}
Under the assumption\footnote{In \cite{costabel90}, it is shown that these
conditions are equivalent for ${\mathbf{u}}$ with ${\mathbf{u}}\in{\mathbf{H}%
}(\Omega,\operatorname{curl})\cap H(\Omega,\operatorname{div})$.} that
$\left\langle \mathbf{u},\mathbf{n}\right\rangle \in L^{2}(
\Gamma)  $ or $\gamma_{T}\mathbf{u\in L}_{T}^{2}(  \Gamma)
$, there holds%
\begin{subequations}
\label{ucurldivrandest2}
\begin{align}
\left\Vert \mathbf{u}\right\Vert _{\mathbf{H}^{1/2}(\Omega)}  &  \leq C\left(
\left\Vert \operatorname*{curl}\mathbf{u}\right\Vert +\left\Vert
\operatorname{div}\mathbf{u}\right\Vert +\left\Vert \left\langle
\mathbf{u},\mathbf{n}\right\rangle \right\Vert _{L^{2}\left(  \Gamma\right)
}\right)  ,
\label{ucurldivrandest2a}\\
\Vert\mathbf{u}\Vert_{\mathbf{H}^{1/2}(\Omega)}  &  \leq C\left(  \left\Vert
\operatorname*{curl}\mathbf{u}\right\Vert +\left\Vert \operatorname{div}%
\mathbf{u}\right\Vert +\left\Vert \gamma_{T}\mathbf{u}\right\Vert
_{\mathbf{L}^{2}\left(  \Gamma\right)  }\right)  . 
\label{ucurldivrandest2b}%
\end{align}
\end{subequations}

\end{lemma}

\begin{proof}
We use the regular decomposition
$\displaystyle\mathbf{u}=\nabla\varphi+\mathbf{z},
$ of Lemma~\ref{lemma:helmholtz-a-la-schoeberl}%
(\ref{item:lemma:helmholtz-a-la-schoeberl-i})
where $\mathbf{z}\in\mathbf{H}^{1}( \Omega) $ satisfies
\begin{equation}
\left\Vert \mathbf{z}\right\Vert _{\mathbf{H}^{1}( \Omega) }\leq C\left\Vert
\operatorname*{curl}\mathbf{u}\right\Vert . \label{uHelmSchoeberl2}%
\end{equation}
Since $\operatorname{div}\mathbf{z}=0$ we have $\displaystyle\Delta
\varphi=\operatorname{div}\mathbf{u}.$ Concerning the boundary conditions for
$\varphi$, we consider two cases corresponding to (\ref{ucurldivrandesta}),
(\ref{ucurldivrandest2a}) and (\ref{ucurldivrandestb}),
(\ref{ucurldivrandest2b}) separately as Case 1 and Case 2. 

\textbf{Case 1.}The function $\varphi$ satisfies the Neumann problem
\[
\Delta\varphi=\operatorname{div}\mathbf{u},\qquad\partial_{n}\varphi
=\left\langle \mathbf{n},\nabla\varphi\right\rangle =\left\langle
\mathbf{n},\mathbf{u}-\mathbf{z}\right\rangle ,
\]
and we note that the condition $\operatorname{div}\mathbf{z}=0$ implies that
the solvability condition for this Neumann problem is satisfied. We estimate%
\begin{align*}
\left\Vert \left\langle \mathbf{n},\mathbf{u}-\mathbf{z}\right\rangle
\right\Vert _{H^{-1/2}(  \Gamma)  }
& \leq\Vert\left\langle
\mathbf{n},\mathbf{u}\right\rangle \Vert_{H^{-1/2}(  \Gamma)
}+\left\Vert \left\langle \mathbf{n},\mathbf{z}\right\rangle \right\Vert
_{H^{-1/2}(  \Gamma)  }
\\
&\leq\Vert\left\langle \mathbf{n}%
,\mathbf{u}\right\rangle \Vert_{H^{-1/2}\left(  \Gamma\right)  }+
C \left\Vert
\mathbf{z}\right\Vert _{\mathbf{H}^{1}(\Omega)}.
\end{align*}
An energy estimate for $\varphi$ provides
$\displaystyle 
\left\Vert \nabla\varphi\right\Vert \leq C(  \left\Vert
\operatorname*{div}\mathbf{u}\right\Vert +\left\Vert \partial_{n}%
\varphi\right\Vert _{H^{-1/2}\left(  \Gamma\right)  })  .
$
The combination of these estimates lead to (\ref{ucurldivrandesta}). We also
note that if $\left\langle \mathbf{u},\mathbf{n}\right\rangle \in L^{2}(
\Gamma)  $, then we get by the smoothness of $\Gamma$ that $\varphi\in
H^{3/2}(\Omega)$ with $\Vert\varphi\Vert_{H^{3/2}(\Omega)}\leq C(\Vert
\operatorname{div}{\mathbf{u}}\Vert+\Vert\partial_{n}\varphi\Vert
_{L^{2}(\Gamma)})$, which shows (\ref{ucurldivrandest2a}).

\textbf{Case 2.} We obtain regularity assertions for $\varphi$ by using that $\varphi$ satisfies 
$\Delta \varphi = \operatorname{div}\mathbf{u}$ and determine the boundary regularity $\varphi|_\Gamma$.
We observe
\[
\overrightarrow{\operatorname*{curl}\nolimits_{\Gamma}}\varphi=\gamma
_{T}\nabla\varphi=\gamma_{T}\left(  \mathbf{u}-\mathbf{z}\right)
\]
and therefore%
\[
\Delta_{\Gamma}\varphi=-\operatorname{curl}_{\Gamma}\overrightarrow
{\operatorname{curl}_{\Gamma}}\varphi=-\operatorname{curl}_{\Gamma}\left(  \gamma
_{T}\left(  \mathbf{u}-\mathbf{z}\right)  \right)  .
\]
Hence, by smoothness of $\partial\Omega$ {(and the fact that $\partial\Omega$
is connected)} we get
\begin{align*}
\left\Vert \varphi\right\Vert _{H^{1/2}(\Gamma)}  &  \leq C\left\Vert
\Delta_{\Gamma}\varphi\right\Vert _{H^{-3/2}(\Gamma)}=C\left\Vert
\operatorname{curl}_{\Gamma}\left(  \gamma_{T}\left(  \mathbf{u}%
-\mathbf{z}\right)  \right)  \right\Vert _{H^{-3/2}(\Gamma)}
\\
& \leq C\left\Vert
\gamma_{T}\left(  \mathbf{u}-\mathbf{z}\right)  \right\Vert _{\mathbf{H}%
^{-1/2}(\Gamma)}
  \leq C\left(  \left\Vert \gamma_{T}\mathbf{u}\right\Vert _{\mathbf{H}%
^{-1/2}(\Gamma)}+\left\Vert \mathbf{z}\right\Vert _{\mathbf{H}^{1}(\Omega
)}\right)  .
\end{align*}
Since $\Vert\varphi\Vert_{H^{1}(\Omega)}\leq C(\Vert\operatorname{div}%
\mathbf{u}\Vert+\Vert\varphi\Vert_{H^{1/2}(\Gamma)})$ we get 
(\ref{ucurldivrandestb}). By similar reasoning, $\gamma_{T}\mathbf{u}%
\in\mathbf{L}^{2}(\Gamma)$ implies $\varphi|_{%
 \Gamma 
}\in H^{1}(%
 \Gamma 
)$ with $\Vert\varphi\Vert_{H^{1}(\Gamma)}\leq C(\Vert\gamma_{T}%
\mathbf{u}\Vert_{L^{2}(\Gamma)}+\Vert\mathbf{z}\Vert_{\mathbf{H}^{1}(\Omega
)})$ so that $\varphi\in H^{3/2}(\Omega)$ and thus (\ref{ucurldivrandest2b}).
\end{proof}

The following lemma introduces some variants of Helmholtz decompositions.

\begin{lemma}
\label{LemHelmDecompVar1}Let $\Omega$ be a bounded sufficiently smooth
Lipschitz domain with simply connected boundary. For any $\mathbf{u}%
\in\mathbf{X}_{\operatorname*{imp}}\cap\mathbf{H}(  \operatorname*{div}%
,\Omega)  $, there exist $\varphi\in H_{0}^{1}(\Omega)%
\cap H^{3/2}(  \Omega)  $ 
and $\mathbf{z}\in\mathbf{H}^{1}(  \Omega)  $ with
$\operatorname*{div}\mathbf{z}=0$ such that $\mathbf{u}=\nabla\varphi
+\operatorname*{curl}\mathbf{z}$. The function $\mathbf{u}$ belongs to
$\mathbf{H}^{1/2}(\Omega)$, and we have the estimates%
\begin{subequations}
\label{HelmDecompVar1est}
\begin{align}
\left\Vert \nabla\varphi\right\Vert _{\mathbf{H}^{1/2}(\Omega)} &  \leq
C\left\Vert \mathbf{u}\right\Vert _{\mathbf{H}^{1/2}(\Omega)},
\label{HelmDecompVar1esta}\\
\left\Vert \operatorname*{curl}\mathbf{z}\right\Vert _{\mathbf{H}^{1/2}%
(\Omega)} &  \leq C\left(  \left\Vert \operatorname*{curl}\mathbf{u}%
\right\Vert +\left\Vert \gamma_{T}\mathbf{u}\right\Vert _{\mathbf{L}%
^{2}(\Gamma)}\right)  .
\label{HelmDecompVar1estb}%
\end{align}%
If $\mathbf{u}$ admits a decomposition of the form $\mathbf{u}=\mathbf{r}%
+\operatorname*{curl} {\boldsymbol \psi} $ with $\mathbf{r}\in\mathbf{H}^{1/2}(  \Omega)  $, then
the decomposition $\mathbf{u} = \nabla \varphi + \operatorname{curl} \mathbf{z}$ satisfies 
\begin{equation}
\left\Vert \nabla\varphi\right\Vert _{\mathbf{H}^{1/2}(\Omega)}\leq
C\left\Vert \mathbf{r}\right\Vert _{\mathbf{H}^{1/2}(\Omega)}.
\label{HelmDecompVar1estc}%
\end{equation}
\end{subequations}
\end{lemma}%

\begin{proof}
The Helmholtz decomposition was considered in \cite[Thm.~{4.2(2)}%
]{Schweizer-Friedrichs-2016}, \cite[Thm.~{28(i)}]{schoeberlscript}. Since
$\operatorname*{div}\operatorname*{curl}=0$ and we require $\varphi\in H_{0}^{1}(\Omega
)$, we have
\begin{equation}
\Delta\varphi=\operatorname{div}\mathbf{u}\quad\text{and\quad}\varphi
|_{\partial\Omega}=0.\label{difphiu0SF}%
\end{equation}
Lemma~\ref{lemma:schoeberl-apriori} implies for 
$\mathbf{u} \in\mathbf{X}_{\operatorname*{imp}}\cap\mathbf{H}(  \operatorname*{div},\Omega)$
that $\mathbf{u}\in\mathbf{H}^{1/2}(\Omega)$. A standard shift theorem for the
Poisson equation leads to%
\begin{equation}
\label{difphiu0SF-foo}%
\left\Vert \varphi\right\Vert _{H^{3/2}(\Omega)}\leq C\left\Vert
\operatorname{div}{\mathbf{u}}\right\Vert _{H^{-1/2}(\Omega)}\leq C\left\Vert
{\mathbf{u}}\right\Vert _{\mathbf{H}^{1/2}(\Omega)}.
\end{equation}
Next, we estimate $\mathbf{z}$. Note that $\varphi\in H_{0}^{1}(\Omega)$
implies $\nabla_{\Gamma}\varphi=0$ so that also $\gamma_{T}\nabla\varphi=0$ on
$\Gamma$. Lemma~\ref{lemma:schoeberl-apriori} then implies
\begin{align}
\Vert\operatorname*{curl}\mathbf{z}\Vert_{\mathbf{H}^{1/2}(\Omega)} &  \leq
C\left(  \left\Vert \operatorname*{curl}\operatorname*{curl}\mathbf{z}%
\right\Vert +\left\Vert \gamma_{T}\operatorname*{curl}\mathbf{z}\right\Vert
_{\mathbf{L}^{2}(\Gamma)}\right)  \nonumber\\
&  \leq C\left(  \left\Vert \operatorname*{curl}\mathbf{u}\right\Vert
+\left\Vert \gamma_{T}\mathbf{u}\right\Vert _{\mathbf{L}^{2}(\Gamma
)}+\left\Vert \gamma_{T}\nabla\varphi\right\Vert _{\mathbf{L}^{2}(\Gamma
)}\right)  \nonumber\\
&  \overset{\gamma_{T}\nabla\varphi=0}{=}C\left(  \left\Vert
\operatorname*{curl}\mathbf{u}\right\Vert +\left\Vert \gamma_{T}%
\mathbf{u}\right\Vert _{\mathbf{L}^{2}(\Gamma)}\right)  .\nonumber
\end{align}%
The estimate (\ref{HelmDecompVar1estc}) follows from (\ref{difphiu0SF-foo}) via
$\operatorname{div}\mathbf{u}=\operatorname*{div}\mathbf{r}$.
This finishes the proof of (\ref{HelmDecompVar1est}).

\end{proof}


\subsection{Maxwell's equations with impedance boundary
conditions\label{SecMW_Imp}}

%

We have introduced all basic ingredients to formulate the electric Maxwell
equations for constant wavenumber $k\in\mathbb{R}\backslash\left(  -k_{0}%
,k_{0}\right)  $ with impedance boundary conditions on $\Gamma$. We define the
sesquilinear form $A_{k}:\mathbf{X}_{\operatorname*{imp}}\times\mathbf{X}%
_{\operatorname*{imp}}\rightarrow\mathbb{C}$ by%
\begin{equation}
A_{k}(\mathbf{u},\mathbf{v}):=\left(  \operatorname*{curl}\mathbf{u,}%
\operatorname*{curl}\mathbf{v}\right)  -k^{2}\left(  \mathbf{u,v}\right)
-\operatorname*{i}k\left(  \mathbf{u}_{T},\mathbf{v}_{T}\right)
_{\mathbf{L}^{2}(  \Gamma)  }. \label{defAksesqui}%
\end{equation}
The variational formulation is: Given an electric current density $\mathbf{j}$ 
and boundary data $\mathbf{g}_T$ with 
\begin{equation}
\mathbf{j}\in\mathbf{L}^{2}(\Omega),\qquad\mathbf{g}_{T}\in\mathbf{H}%
_{T}^{-1/2}\left(  \operatorname*{div}\nolimits_{\Gamma},\Gamma\right)
\cap\mathbf{L}_{T}^{2}(\Gamma) \label{eq:condition-on-j-g}%
\end{equation}
find $\mathbf{E}\in\mathbf{X}_{\operatorname*{imp}}$ such that%
\begin{equation}
A_{k}(\mathbf{E},\mathbf{v})=\left(  \mathbf{j,v}\right)  +\left(
\mathbf{g}_{T},\mathbf{v}_{T}\right)  _{\mathbf{L}^{2}(  \Gamma)
}\quad\forall\mathbf{v}\in\mathbf{X}_{\operatorname*{imp}}.
\label{weakformulation}%
\end{equation}
Note that the assumptions (\ref{eq:condition-on-j-g}) on the data are not the
most general ones (see (\ref{XimpprimeOmega}), (\ref{XimpprimeGamma}) and
(\ref{eq:item:lemma:apriori-with-good-sign-iia}) below) but they reduce
technicalities in some places. By integration by parts it is easy to see that
the classical strong form of this equation is given by%

\begin{equation}%
\begin{array}
[c]{ll}%
\mathcal{L}_{\Omega,k}\mathbf{E}=\mathbf{j} & \text{in }\Omega,\\
\mathcal{B}_{%
 \Gamma 
,k}\mathbf{E}=\mathbf{g}_{T} & \text{on }\Gamma
\end{array}
\label{MWEq}%
\end{equation}
with the volume and boundary differential operators $\mathcal{L}_{\Omega,k}$
and $\mathcal{B}_{\Omega,k}$, defined by%
\[
\mathcal{L}_{\Omega,k}\mathbf{v}:=\operatorname*{curl}\operatorname*{curl}%
\mathbf{v}-k^{2}\mathbf{v}\text{ in }\Omega\quad\text{and\quad}\mathcal{B}_{%
 \Gamma 
,k}\mathbf{v}:=\gamma_{T}%
   \operatorname*{curl}\mathbf{v} 
-\operatorname*{i}k\Pi_{T}\mathbf{v}\text{ on }\Gamma\mathbf{.}%
\]
We denote by
\begin{equation}
\mathcal{S}_{\Omega,k}^{\operatorname{MW}}:\mathbf{X}_{\operatorname{imp}%
}^{\prime}\rightarrow\mathbf{X}_{\operatorname{imp}} \label{defMaxSolOp}%
\end{equation}
the solution operator that maps the linear functional $\mathbf{X}%
_{\operatorname*{imp}}\ni\mathbf{v}\rightarrow\left(  \mathbf{j,v}\right)
+\left(  \mathbf{g}_{T},\mathbf{v}\right)  _{\mathbf{L}^{2}\left(
\Gamma\right)  }$ to the solution $\mathbf{E}$ of (\ref{MWEq}) and whose
existence follows from Proposition~\ref{lemma:apriori-homogeneous-rhs} below.

In our analysis, the sesquilinear form
\begin{equation}
\innerprod{
\mathbf{u},\mathbf{v}%
}
:=k^{2}\left(  \mathbf{u},\mathbf{v}\right)  +\operatorname*{i}%
k\left(  \mathbf{u}_{T},\mathbf{v}_{T}\right)  _{\mathbf{L}^{2}(
\Gamma)  } \label{eq:def:(())}%
\end{equation}
will play an important role. We note
\begin{align}
A_{k}( \mathbf{u},\mathbf{v})  &  =\left(  \operatorname*{curl}\mathbf{u,}%
\operatorname*{curl}\mathbf{v}\right)  -
\innerprod{\mathbf{u},\mathbf{v}}, 
\label{eq:Ak-alternative}\\
A_{k}( \mathbf{u},\nabla\varphi)  &  =-
\innerprod{ \mathbf{u},\nabla\varphi}
\qquad\forall\mathbf{u}\in\mathbf{X}_{\operatorname*{imp}}%
,\quad\forall\varphi\in H_{\operatorname*{imp}}^{1}( \Omega) ,\\
\innerprod{ \mathbf{u},\mathbf{v} }
  &  =\overline{
\innerprod{ \mathbf{v},\mathbf{u}}}
. \label{Akcommute}%
\end{align}


\section{Stability analysis of the continuous Maxwell problem}

\label{sec:stability}
In this section we show that the model problem (\ref{weakformulation}) is
well-posed and that the norm of the solution operator is $O(|k|^{\theta})$ for
suitable choices of norms and some $\theta%
\geq0 $%
.


\subsection{Well-posedness}


The continuity of the sesquilinear form $A_{k}( \cdot,\cdot) $ is obvious: it
holds%
\[
\left\vert A_{k}( \mathbf{u},\mathbf{v}) \right\vert \leq
C_{\operatorname*{cont}}\left\Vert \mathbf{u}\right\Vert _{\operatorname*{imp}%
,k}\left\Vert \mathbf{v}\right\Vert _{\operatorname*{imp},k}\qquad
\text{with\qquad}C_{\operatorname*{cont}}:=1\text{.}%
\]

Well-posedness of the Maxwell problem with impedance condition is proved in
\cite[Thm.~{4.17}]{Monkbook}. Here we recall the statement and give a sketch
of the proof.

\begin{proposition}
\label{lemma:apriori-homogeneous-rhs}Let $\Omega\subset\mathbb{R}^{3}$ be a
bounded Lipschitz domain with simply connected and sufficiently smooth
boundary. Then there exists $\gamma_{k}>0$ such that
\[
\gamma_{k}\leq\inf_{\mathbf{u}\in\mathbf{X}_{\operatorname*{imp}}%
\backslash\left\{
\mathbf{0}%
\right\}  }\sup_{\mathbf{v}\in\mathbf{X}_{\operatorname*{imp}}\backslash
\left\{
\mathbf{0}%
\right\}  }\frac{\left\vert A_{k}(\mathbf{u},\mathbf{v})\right\vert
}{\left\Vert \mathbf{u}\right\Vert _{\operatorname*{imp},k}\left\Vert
\mathbf{v}\right\Vert _{\operatorname*{imp},k}}.
\]

\end{proposition}

\begin{proof}
\textbf{Step 1.} We show uniqueness. If $A_{k}(\mathbf{u},\mathbf{v})=0$ for
all $\mathbf{v}\in\mathbf{X}_{\operatorname*{imp}}$ then
\[
0=\operatorname{Im}A_{k}(\mathbf{u},\mathbf{u})=-k\left\Vert \mathbf{u}%
_{T}\right\Vert _{\mathbf{L}^{2}\left(  \Gamma\right)  }^{2}.
\]
Hence, $\mathbf{u}_{T}=\mathbf{0}$ on $\Gamma$ and the extension of
$\mathbf{u}$ by zero outside of $\Omega$ (denoted $\widetilde{\mathbf{u}}$) is
in $\mathbf{H}(\operatorname{curl},\widetilde{\Omega})$ for any bounded domain
$\widetilde{\Omega}\subset\mathbb{R}^{3}$. This zero extension $\widetilde
{\mathbf{u}}$ solves the homogeneous Maxwell equations on $\mathbb{R}^{3}$. An
application of the operator \textquotedblleft$\operatorname{div}%
$\textquotedblright\ shows that $\operatorname{div}\widetilde{\mathbf{u}}=0$
and thus $\widetilde{\mathbf{u}}\in\mathbf{H}^{1}(\mathbb{R}^{3})$. Using
$\operatorname{curl}\operatorname{curl}=-\Delta+\nabla\operatorname{div}$ we
see that each component of $\widetilde{\mathbf{u}}$ solves the homogeneous
Helmholtz equation. Since $\widetilde{\mathbf{u}}$ vanishes outside $\Omega$,
the unique continuation principle asserts $\widetilde{\mathbf{u}}=0$.
%

\textbf{Step 2.} From \cite[Thm.~{4.8}]{gatica_MW_imped} or the technique developed in \cite{Buffa2005}
it follows that the operator induced by $A_{k}$ is a compact perturbation of
an isomorphism and the Fredholm alternative shows well-posedness of the problem.
\end{proof}


\subsection{Wavenumber-explicit stability estimates}


Proposition~\ref{lemma:apriori-homogeneous-rhs} does not give any insight how
the (positive) inf-sup constant $\gamma_{k}$ depends on the wavenumber $k$. In
this section, we introduce the stability constant $C_{\operatorname*{stab}}(k)$ 
and estimate its dependence on $k$ under certain assumptions.

\begin{definition}
\label{DefStabConst}Let $\Omega\subset\mathbb{R}^{3}$ be a bounded Lipschitz
domain with simply connected and sufficiently smooth boundary. 
Let $C_{\operatorname*{stab}}(k)$ denote the smallest constant such that
for each $\mathbf{j}\in\mathbf{L}^{2}(\Omega)$ and $\mathbf{g}_{T}\in
\mathbf{L}_{T}^{2}(\Gamma)$ the solution $\mathbf{E} = 
\mathcal{S}_{\Omega,k}^{\operatorname{MW}} ({\mathbf j}, {\mathbf{g}}_T)$ of 
(\ref{weakformulation}) satisfies%
\begin{align}
\Vert\mathbf{E}\Vert_{\operatorname{imp},k}
 \leq C_{\operatorname*{stab}}\left(  k\right)  \left(  \left\Vert
\mathbf{j}\right\Vert _{\mathbf{L}^{2}(\Omega)}+\left\Vert \mathbf{g}%
_{T}\right\Vert _{\mathbf{L}^{2}(\Gamma)}\right)  . 
\label{eq:Cstab}%
\end{align}
\end{definition}%
The behavior of the constant
$C_{\operatorname*{stab}}\left(  k\right)  $ 
with respect to the wavenumber typically depends on the geometry of the domain
$\Omega$. Our stability and convergence theory for conforming Galerkin finite
element discretization as presented in Sections~\ref{SecStabConv} requires
that this constant grow at most algebraically in $k$, i.e.,%
\begin{equation}
\exists\theta%
 \geq0 
,C_{\operatorname{stab}}>0\quad\text{such that }%
 C_{\operatorname*{stab}}\left(  k\right) 
\leq C_{\operatorname{stab}}|k|^{\theta}\qquad\forall k\in\mathbb{R}%
\backslash\left(  -k_{0},k_{0}\right)  . \label{AssumptionAlgGrowth}%
\end{equation}
\begin{remark}%
 For the 
$hp$-FEM application below, the term $\left\vert k\right\vert ^{\theta}$ will
be mitigated by an exponentially converging approximation term so that
any finite value $\theta\geq0$ leads to an exponential convergence of the
discretization. 
\eremk
\end{remark}%

Next we present an example\footnote{We thank an anonymous referee for suggesting 
this example.} which shows that in general the exponent $\theta$ in
(\ref{AssumptionAlgGrowth}) cannot be negative.

\begin{example}
Let $\Omega=\left(  -2,2\right)  ^{3}$, $\omega:=\left(  -1,1\right)  ^{2}%
$. Define the cutoff function $\chi:\Omega\rightarrow\mathbb{C}$ by%
\[
\chi\left(  \mathbf{x}\right)  :=\left\{
\begin{array}
[c]{ll}%
\left(  1-x_{1}^{2}\right)  ^{2}\left(  1-x_{2}^{2}\right)  ^{2}\left(
1-x_{3}^{2}\right)  ^{2} & \mathbf{x}=\left(  x_{j}\right)  _{j=1}^{3}%
\in\omega,\\
0 & \text{otherwise.}%
\end{array}
\right.
\]
Note that $\chi\in H_{0}^{2}\left(  \Omega\right)  $. With 
$\mathbf{e}_{j} \in \mathbb{R}^3$, $1\leq j\leq3$, denoting the\ $j$-th canonical unit vector 
we define%
\[
\mathbf{j}:=\operatorname*{e}\nolimits^{\operatorname*{i}kx_{1}}\left(
-\left(  \Delta\chi+\operatorname*{i}k\partial_{1}\chi\right)  \mathbf{e}%
_{2}+\nabla\partial_{2}\chi+\operatorname*{i}k\operatorname*{curl}\left(
\chi\mathbf{e}_{3}\right)  \right)  .
\]
Then, $\mathbf{E:}=\chi\operatorname*{e}^{\operatorname*{i}kx_{1}}%
\mathbf{e}_{2}$ is the unique weak solution of%
\begin{align*}
\mathcal{L}_{\Omega,k}\mathbf{E}=\mathbf{j} & \text{ in }\Omega, 
& 
\mathcal{B}_{\Gamma,k}\mathbf{E}=\mathbf{0} & \text{ on }\Gamma.
\end{align*}
Using the symbolic computer algebra program {\sc MATHEMATICA} we obtain explicitly%
\begin{align*}
\left\Vert \mathbf{j}\right\Vert _{\mathbf{L}^{2}\left(  \Omega\right)  }%
^{2}& =\frac{16777216\left(  5k^{2}+33\right)  }{10418625}, \\
\left\Vert \mathbf{E}\right\Vert _{\mathbf{H}(\operatorname*{curl},\Omega
),k}^{2}& =2\frac{16777216\left(  k^{2}+3\right)  }{31255875},\quad\left\Vert
\mathbf{E}_{T}\right\Vert _{\mathbf{L}^{2}(\Gamma)}^{2}=0,
\end{align*}
which shows that in general $\theta \ge 0$ in (\ref{eq:Cstab}). 
\eremk 
\end{example}
\begin{remark}%
Let $\Omega\subset\mathbb{R}^{3}$ be a bounded Lipschitz domain with simply
connected and sufficiently smooth boundary. The sesquilinear form $A_{k}$
satisfies the inf-sup condition
\begin{equation}
\inf_{\mathbf{u}\in\mathbf{H}(\operatorname{curl},\Omega)}\sup_{\mathbf{u}%
\in\mathbf{H}(\operatorname{curl},\Omega)}\frac{|A_{k}(\mathbf{u}%
,\mathbf{v})|}{\Vert\mathbf{u}\Vert_{\operatorname{imp},k}\Vert\mathbf{v}%
\Vert_{\operatorname{imp},k}}\geq\frac{1}{1+|k|%
C_{\operatorname{stab}}%
\left(  k\right)  }. \label{eq:inf-sup}%
\end{equation}
This result is shown in the same way as in the Helmholtz case, see, e.g.,
\cite[Thm.~{2.5}]{MelenkHelmStab2010}, \cite[Prop.~{8.2.7}]{MelenkDiss}.
If assumption~(\ref{AssumptionAlgGrowth}) holds, then
\[
\inf_{\mathbf{u}\in\mathbf{H}(\operatorname{curl},\Omega)}\sup_{\mathbf{u}%
\in\mathbf{H}(\operatorname{curl},\Omega)}\frac{|A_{k}(\mathbf{u}%
,\mathbf{v})|}{\Vert\mathbf{u}\Vert_{\operatorname{imp},k}\Vert\mathbf{v}%
\Vert_{\operatorname{imp},k}}\geq\frac{1}{1+C_{\operatorname{stab}}%
|k|^{\theta+1}}.
\]
\hbox{}\hfill%
\eremk
\end{remark}%

\color{black}%
In the remaining part of this section, we prove estimate
(\ref{AssumptionAlgGrowth}) for certain classes of domains. The following
result removes the assumption in \cite{hiptmair-moiola-perugia11} for the
right-hand side to be solenoidal.

\begin{proposition}
\label{PropSmoothStarshaped}Let $\Omega\subset\mathbb{R}^{3}$ be a bounded
$C^{2}$ domain that is star-shaped with respect to a ball. Then, assumption
(\ref{AssumptionAlgGrowth}) holds with $\theta=0$.
\end{proposition}

\begin{proof}
Let $\varphi\in H_{0}^{1}(\Omega)$ satisfy
\[
-\Delta\varphi=\operatorname{div}\mathbf{j}\qquad\mbox{ in $\Omega$}\text{.}%
\]
Then,
\begin{align*}
\Vert k^{-2}\nabla\varphi\Vert_{\operatorname{imp},k}  &  =|k|^{-1}\Vert
\nabla\varphi\Vert_{\mathbf{L}^{2}(\Omega)}\leq C|k|^{-1}\Vert
\operatorname{div}\mathbf{j}\Vert_{H^{-1}(\Omega)}\leq C|k|^{-1}%
\Vert\mathbf{j}\Vert_{\mathbf{L}^{2}(\Omega)},\\
\Vert\mathbf{j}+\nabla\varphi\Vert_{\mathbf{L}^{2}(\Omega)}  &  \leq
\Vert\mathbf{j}\Vert_{\mathbf{L}^{2}(\Omega)}+C\Vert\operatorname{div}%
\mathbf{j}\Vert_{H^{-1}(\Omega)}\leq C\Vert\mathbf{j}\Vert_{\mathbf{L}%
^{2}(\Omega)}.
\end{align*}
Noting that $\varphi$ vanishes on $\Gamma$, the difference $\mathcal{S}%
_{\Omega,k}^{\operatorname{MW}}(\mathbf{j},\mathbf{g}_{T})-\left\vert
k\right\vert ^{-2}\nabla\varphi$ satisfies
\[
\mathcal{L}_{\Omega,k}\left(  \mathcal{S}_{\Omega,k}^{\operatorname{MW}%
}(\mathbf{j},\mathbf{g}_{T})-\left\vert k\right\vert ^{-2}\nabla
\varphi\right)  =\mathbf{j}+\nabla\varphi,\quad\mathcal{B}_{%
\Gamma 
,k}\!\left(  \mathcal{S}_{\Omega,k}^{\operatorname{MW}}(\mathbf{j}%
,\mathbf{g}_{T})-\left\vert k\right\vert ^{-2}\nabla\varphi\right)
=\mathbf{g}_{T},
\]
and $\operatorname{div}(\mathbf{j}+\nabla\varphi)=0$. \cite[Thm.~{3.1}%
]{hiptmair-moiola-perugia11} implies
\begin{align*}
\Vert\mathcal{S}_{\Omega,k}^{\operatorname{MW}}(\mathbf{j},\mathbf{g}%
_{T})-\left\vert k\right\vert ^{-2}\nabla\varphi\Vert_{\operatorname{imp}%
,k}& \leq C\left(  \Vert\mathbf{j}+\nabla\varphi\Vert_{\mathbf{L}^{2}(\Omega
)}+\Vert\mathbf{g}_{T}\Vert_{\mathbf{L}^{2}(\Gamma)}\right)  
\\
& \leq C\left(
\left\Vert \mathbf{j}\right\Vert _{\mathbf{L}^{2}(  \Gamma)
}+\Vert\mathbf{g}_{T}\Vert_{\mathbf{L}^{2}(\Gamma)}\right)  .
\end{align*}
The estimate for $\mathcal{S}_{\Omega,k}^{\operatorname{MW}}(\mathbf{j}%
,\mathbf{g}_{T})$ follows from a triangle inequality. 
\end{proof}

For the more general situation of domains that are not necessarily star-shaped we
require some preliminaries. A bounded domain $\Omega$ with smooth boundary
admits, e.g., by \cite[Cor.~{4.1}]{hiptmair-li-zou12} a 
continuous extension operator $\mathcal{E}_{\operatorname*{div}}%
:\mathbf{H}^{m}(  \operatorname*{div},\Omega)  \rightarrow
\mathbf{H}^{m}(  \operatorname*{div},\mathbb{R}^{3})  $ for any
$m\in\mathbb{N}_{0}$. In particular this extension can be chosen such that for
a ball $B_{R}$ of radius $R$ with $\Omega\subset B_{R}$ there holds
\begin{equation}
\operatorname*{supp}\left(  \mathcal{E}_{\operatorname*{div}}\mathbf{h}%
\right)  \subset B_{R}\quad\forall\,\mathbf{h}\in\mathbf{H}%
(\operatorname*{div},\Omega).\label{DefBRR}%
\end{equation}%
Since the right-hand side $\mathbf{j}$ in (\ref{MWEq}), in general, does not
belong to $\mathbf{H}(\operatorname*{div},\Omega)$ we subtract an appropriate
gradient field: Let $\psi\in H_{0}^{1}(\Omega)$ be the weak solution of
$-\Delta\psi=\operatorname{div}\mathbf{j}$. As in the proof of
Proposition~\ref{PropSmoothStarshaped}, we write $\mathcal{S}_{\Omega
,k}^{\operatorname{MW}}(\mathbf{j},\mathbf{g}_{T})=\mathcal{S}_{\Omega
,k}^{\operatorname{MW}}(\widetilde{\mathbf{j}},\mathbf{g}_{T})-k^{-2}%
\nabla\psi$ with
\begin{equation}
\widetilde{\mathbf{j}}:=\mathbf{j}+\nabla\psi\in\mathbf{H}(\operatorname*{div}%
,\Omega).\label{Defjtilde}%
\end{equation}
The operator $\mathcal{E}_{\operatorname*{div}}$ allows us to extend
$\widetilde{\mathbf{j}}$ to a compactly supported function $\mathbf{J}%
:=\mathcal{E}_{\operatorname*{div}}(\widetilde{\mathbf{j}})%
\in\mathbf{H}(\operatorname*{div},\mathbb{R}^{3})$. 
Next we introduce the solution operator for the full space problem%
\begin{equation}%
\begin{array}
[c]{cl}%
\operatorname*{curl}\operatorname*{curl}\mathbf{Z}-k^{2}\mathbf{Z}=\mathbf{J}
& \text{in }\mathbb{R}^{3},\\
\left\vert \partial_{r}\mathbf{Z}\left(  \mathbf{x}\right)  -\operatorname*{i}%
k\mathbf{Z}\left(  \mathbf{x}\right)  \right\vert \leq c/r^{2} & \text{as
}r=\left\Vert \mathbf{x}\right\Vert \rightarrow\infty, 
\end{array}
\label{FullHelmProbl}%
\end{equation}
via the Maxwell potential%
\begin{subequations}
\label{MaxwellFullFull}
\begin{equation}
\mathbf{Z}=\mathcal{N}_{\operatorname*{MW},k}\left(  \mathbf{J}\right)
:=\mathcal{N}_{\operatorname*{MW},k}^{\nabla}\left(  \mathbf{J}\right)
+\mathcal{N}_{\operatorname*{MW},k}^{\operatorname*{curl}}\left(
\mathbf{J}\right)  ,
\label{MaxwellFullFulla}%
\end{equation}
where\footnote{With a slight abuse of notation we write $\mathcal{N}%
_{\operatorname*{MW},k}^{\operatorname*{curl}}\left(  v\right)  :=\int
_{\mathbb{R}^{3}}g_{k}\left(  \left\Vert \mathbf{\cdot}-\mathbf{y}\right\Vert
\right)  v\left(  \mathbf{y}\right)  d\mathbf{y}$ also for scalar functions
$v$. This is the classical acoustic Newton potential.}%
\begin{equation}
\left.
\begin{array}
[c]{l}%
\mathcal{N}_{\operatorname*{MW},k}^{\operatorname*{curl}}\left(
\mathbf{J}\right)  :=%
{\displaystyle \int_{\mathbb{R}^{3}}} 
g_{k}\left(  \left\Vert \mathbf{\cdot}-\mathbf{y}\right\Vert \right)
\mathbf{J}\left(  \mathbf{y}\right)  d\mathbf{y}\\
\\
\mathcal{N}_{\operatorname*{MW},k}^{\nabla}\left(  \mathbf{J}\right)
:=k^{-2}\nabla\mathcal{N}_{\operatorname*{MW},k}^{\operatorname*{curl}}\left(
\operatorname*{div}\mathbf{J}\right)
\end{array}
\right\}  \quad\text{in }\mathbb{R}^{3}
\label{MaxwellFullFullb}%
\end{equation}
with the fundamental solution of the Helmholtz equation in $\mathbb{R}^{3}$%
\begin{equation}
g_{k}\left(  r\right)  :=\frac{\operatorname*{e}\nolimits^{\operatorname*{i}%
kr}}{4\pi r}\text{.}
\label{MaxwellFullFullc}%
\end{equation}
\end{subequations}
Note that the adjoint full space problem is given by replacing $k$ in
(\ref{FullHelmProbl}) by $-k$ with solution operator $\mathcal{N}%
_{\operatorname*{MW},-k}\left(  \mathbf{J}\right)  $.

The layer operators $\mathcal{S}_{\operatorname*{MW},k}^{\nabla}$,
$\mathcal{S}_{\operatorname*{MW},k}^{\operatorname*{curl}}$ map densities defined on
the boundary $\Gamma$ to functions defined in $\Omega$ by%
\begin{equation}
\left.
\begin{array}
[c]{l}%
\mathcal{S}_{\operatorname*{MW},k}^{\operatorname*{curl}}\left(
{\boldsymbol \mu}%
\right)  :=%
{\displaystyle \int_{\Gamma}}
g_{k}\left(  \left\Vert \mathbf{\cdot}-\mathbf{y}\right\Vert \right)
{\boldsymbol \mu}%
\left(  \mathbf{y}\right)  d%
\Gamma_{\mathbf{y}}\\
\\%
\mathcal{S}_{\operatorname*{MW},k}^{\nabla}\left(
{\boldsymbol \mu}%
\right)  :=k^{-2}\nabla\mathcal{S}_{\operatorname*{MW},k}%
^{\operatorname*{curl}}\left(  \operatorname*{div}_{\Gamma}%
{\boldsymbol \mu}%
\right)
\end{array}%
\right\}  \quad\text{in }\mathbb{R}^{3}\backslash\Gamma.
\label{eq:S-potential}%
\end{equation}
We set $\mathcal{S}_{\mathbb{R}^{3},k}^{\operatorname*{MW}}:=\mathcal{S}%
_{\operatorname*{MW},k}^{\nabla}+\mathcal{S}_{\operatorname*{MW}%
,k}^{\operatorname*{curl}}$.

\begin{theorem}
\label{TheoNonStarShaped}Let $\Omega\subset\mathbb{R}^{3}$ be a bounded
Lipschitz domain with simply connected,{ analytic boundary}. Then, there is $C > 0$ 
depending only on $\Omega$ such that 
\[
C_{\operatorname*{stab}}\left(  k\right)  \leq C \left\vert k\right\vert
^{7/2}\sqrt{1+\ln\left\vert k\right\vert }.
\]%
\end{theorem}
\begin{remark}
The analyticity requirement of $\partial\Omega$ can be relaxed. It is due to
our citing \cite{MelenkStab}, which assumes analyticity. 
\eremk
\end{remark}

\begin{proof}
We estimate $\mathcal{S}_{\Omega,k}^{\operatorname{MW}}(\mathbf{j}%
,\mathbf{g}_{T})$ (see (\ref{defMaxSolOp})) for given $(\mathbf{j}%
,\mathbf{g}_{T})\in\mathbf{L}^{2}(\Omega)\times\mathbf{L}_{T}^{2}(\Gamma)$.

\textbf{Step 1.} \textbf{(reduction to solenoidal right-hand side)}
Let $\psi$, $\widetilde{\mathbf{j}}$ be as in (\ref{Defjtilde}) so that
$\mathcal{S}_{\Omega,k}^{\operatorname{MW}}(\mathbf{j},\mathbf{g}%
_{T})=\mathcal{S}_{\Omega,k}^{\operatorname{MW}}(\widetilde{\mathbf{j}%
},\mathbf{g}_{T})-k^{-2}\nabla\psi$ and $\operatorname*{div}\widetilde
{\mathbf{j}}=0$.
As in the proof of Proposition~\ref{PropSmoothStarshaped}, we have%
\begin{equation}
\Vert k^{-2}\nabla\psi\Vert_{\operatorname{imp},k}\leq C%
 \left\vert k\right\vert ^{-1} 
\Vert\mathbf{j}\Vert_{\mathbf{L}^{2}(\Omega)},
\quad 
\Vert\widetilde{\mathbf{j}}%
\Vert_{\mathbf{L}^{2}(\Omega)}\leq C\Vert\mathbf{j}\Vert_{\mathbf{L}%
^{2}(\Omega)},
\quad 
\operatorname{div}\widetilde{\mathbf{j}}%
=0.\label{eq:stability-estimate-psi}%
\end{equation}
In particular,
\[
\Vert\widetilde{\mathbf{j}}\Vert_{\mathbf{H}(\operatorname{div},\Omega)}\leq
C\Vert\mathbf{j}\Vert_{\mathbf{L}^{2}(\Omega)}.
\]

\textbf{Step 2.} \textbf{(reduction to homogeneous volume right-hand side)} We set%

\begin{equation}
\mathbf{\tilde{g}}_{T}:=\mathbf{g}_{T}-\mathcal{B}_{%
 \Gamma 
,k}\mathbf{u}_{\mathbf{j}}\quad\text{with\quad}\mathbf{u}_{\mathbf{j}%
}:=\left.  \left(  \mathcal{N}_{\operatorname*{MW},k}\mathcal{E}%
_{\operatorname*{div}}\mathbf{\tilde{j}}\right)  \right\vert _{\Omega
}\label{defgtildeT}%
\end{equation}
so that $\mathbf{\tilde{E}}%
:=\mathcal{S}_{\Omega,k}^{\operatorname{MW}}(\widetilde{\mathbf{j}%
},\mathbf{\tilde{g}}_{T})%
=\mathbf{u}_{0}+\mathbf{u}_{\mathbf{j}}$ with $\mathbf{u}_{0}$ being the
solution of the homogeneous problem%
\begin{equation}%
\begin{array}
[c]{ll}%
\operatorname*{curl}\operatorname*{curl}\mathbf{u}_{0}-k^{2}\mathbf{u}%
_{0}=\mathbf{0} & \text{in }\Omega,\\
\gamma_{T}\left(  \operatorname*{curl}\mathbf{u}_{0}\right)
-\operatorname*{i}k\left(  \mathbf{u}_{0}\right)  _{T}=\mathbf{\tilde{g}}_{T}
& \text{on }\Gamma.
\end{array}
\label{u02}%
\end{equation}
To estimate $\mathbf{u}_{\mathbf{j}}$, we rely on the following estimate from
\cite[Lem.~{3.5}]{MelenkSauterMathComp}
\begin{equation}
|k|\Vert\mathcal{N}_{\operatorname{MW},k}^{\operatorname{curl}}(f)\Vert
_{L^{2}(\Omega)}+\Vert\mathcal{N}_{\operatorname{MW},k}^{\operatorname{curl}%
}(f)\Vert_{H^{1}(\Omega)}+|k|^{-1}\Vert\mathcal{N}_{\operatorname{MW}%
,k}^{\operatorname{curl}}(f)\Vert_{H^{2}(\Omega)}\leq C\Vert f\Vert
_{L^{2}(\mathbb{R}^{3})}
\label{eq:lemma3.5}%
\end{equation}
for all $f\in L^{2}(\mathbb{R}^3)$.  Abbreviate $\mathbf{N}^{\operatorname{curl}}:=\mathcal{N}_{\operatorname{MW}%
,k}^{\operatorname{curl}}(\mathcal{E}_{\operatorname{div}}\widetilde
{\mathbf{j}})$ and $N^{\nabla}:=\mathcal{N}_{\operatorname{MW},k}%
^{\operatorname*{curl}}(\operatorname{div}\mathcal{E}_{\operatorname{div}%
}\widetilde{\mathbf{j}})$. Estimates (\ref{eq:lemma3.5})
and (\ref{eq:stability-estimate-psi})  
imply%
\begin{align*}
& |k|\Vert\mathbf{N}^{\operatorname{curl}}\Vert_{\mathbf{L}^{2}(\Omega)}%
+\Vert\mathbf{N}^{\operatorname{curl}}\Vert_{\mathbf{H}^{1}(\Omega)}%
+|k|^{-1}\Vert\mathbf{N}^{\operatorname{curl}}\Vert_{\mathbf{H}^{2}(\Omega)}
  \leq C\Vert\mathbf{j}\Vert_{\mathbf{L}^{2}(\Omega)},\\
& |k|\Vert N^{\nabla}\Vert_{L^{2}(\Omega)}+\Vert N^{\nabla}\Vert_{H^{1}(\Omega
)}+|k|^{-1}\Vert N^{\nabla}\Vert_{H^{2}(\Omega)}   \leq C\Vert
\operatorname{div}\mathcal{E}_{\operatorname{div}}\mathbf{\tilde{j}}%
\Vert_{L^{2}(\mathbb{R}^{3})}\leq C\Vert\mathbf{j}\Vert
_{\mathbf{L}^{2}(\Omega)}.
\end{align*}
For $\mathbf{u}_{\mathbf{j}}=\mathbf{N}^{\operatorname{curl}}+k^{-2}\nabla
N^{\nabla}$ we get by a multiplicative trace inequality:%
\begin{align}
\Vert\mathbf{u}_{\mathbf{j}}\Vert_{\operatorname{imp},k} &  \leq
C\bigl(|k|\Vert\mathbf{N}^{\operatorname{curl}}\Vert_{\mathbf{L}^{2}(\Omega
)}+\Vert\mathbf{N}^{\operatorname{curl}}\Vert_{\mathbf{H}^{1}(\Omega
)}+|k|^{1/2}\Vert\mathbf{N}^{\operatorname{curl}}\Vert_{\mathbf{L}^{2}%
(\Omega)}^{1/2}\Vert\mathbf{N}^{\operatorname{curl}}\Vert_{\mathbf{H}%
^{1}(\Omega)}^{1/2}\nonumber\\
&  \quad\mbox{}+|k|^{-1}\Vert{N}^{\nabla}\Vert_{{H}^{1}(\Omega)}%
+|k|^{-3/2}\Vert\nabla_{\Gamma}{N}^{\nabla}\Vert_{\mathbf{L}^{2}(\Gamma
)}\bigr)\nonumber\\
&  \leq C\left(  \Vert\mathbf{j}\Vert_{\mathbf{L}^{2}(\Omega)}+|k|^{-1}%
\Vert\mathbf{j}\Vert_{\mathbf{L}^{2}(\Omega)}+|k|^{-3/2}\Vert N^{\nabla}%
\Vert_{H^{1}(\Omega)}^{1/2}\Vert N^{\nabla}\Vert_{H^{2}(\Omega)}^{1/2}\right)
\nonumber\\
&  \leq C\Vert\mathbf{j}\Vert_{\mathbf{L}^{2}(\Omega)}%
.\label{eq:stability-estimate-uj}%
\end{align}
Arguing similarly, we get for $\widetilde{\mathbf{g}}_{T}=\mathbf{g}%
_{T}-\mathcal{B}_{%
 \Gamma 
,k}\mathbf{u}_{\mathbf{j}}$%
\begin{align}
\nonumber 
\Vert\widetilde{\mathbf{g}}_{T}\Vert_{\mathbf{L}^{2}(\Gamma)}& \leq
\Vert{\mathbf{g}}_{T}\Vert_{\mathbf{L}^{2}(\Gamma)}+\Vert\mathcal{B}_{\Gamma,k}\mathbf{u}_{\mathbf{j}}\Vert_{\mathbf{L}^{2}(\Gamma)} \\
& \leq C\left(
\Vert\mathbf{g}_{T}\Vert_{\mathbf{L}^{2}(\Gamma)}+|k|^{1/2}\Vert
\mathbf{j}\Vert_{\mathbf{L}^{2}(\Omega)}\right)
.\label{eq:stability-estimate-gtilde}%
\end{align}
\textbf{Step 3}. \textbf{(Estimate of} $\gamma_{T}\operatorname*{curl}%
\mathbf{u}_{0}$, $\gamma_{T}\mathbf{u}_{0}$\textbf{.)} To estimate the
function $\mathbf{u}_{0}$, we employ the Stratton-Chu formula (see, e.g.,
\cite[Thm.~{6.2}]{coltonkress_inverse}, \cite[(5.5.3)-(5.5.6)]{Nedelec01})%
\[
\mathbf{u}_{0}=\operatorname*{curl}\mathcal{S}_{\operatorname*{MW}%
,k}^{\operatorname*{curl}}\left(  \gamma_{T}\mathbf{u}_{0}\right)
+\mathcal{S}_{{\mathbb{R}}^{3},k}^{\operatorname*{MW}}\left(  \gamma
_{T}\operatorname*{curl}\mathbf{u}_{0}\right)  \quad\text{in }\Omega.
\]

The weak formulation (\ref{weakformulation}) implies%
\[
\left\Vert \operatorname*{curl}\mathbf{u}_{0}\right\Vert ^{2}-k^{2}\left\Vert
\mathbf{u}_{0}\right\Vert ^{2}-\operatorname*{i}k\left\Vert \left(
\mathbf{u}_{0}\right)  _{T}\right\Vert _{\mathbf{L}^{2}(\Gamma)}^{2}=\left(
\mathbf{\tilde{g}}_{T},\left(  \mathbf{u}_{0}\right)  _{T}\right)
_{\mathbf{L}^{2}(\Gamma)}%
\]
from which we obtain by considering the imaginary part
\begin{equation}
\left\vert k\right\vert \left\Vert \left(  \mathbf{u}_{0}\right)
_{T}\right\Vert _{\mathbf{L}^{2}\left(  \Gamma\right)  }\leq\left\Vert
\mathbf{\tilde{g}}_{T}\right\Vert _{\mathbf{L}^{2}(\Gamma)}. \label{estu0T}%
\end{equation}
{}From the real part, we then obtain by a Cauchy-Schwarz inequality%
\begin{equation}
\left\vert \Vert\operatorname*{curl}\mathbf{u}_{0}\Vert^{2}-\left\vert
k\right\vert ^{2}\Vert\mathbf{u}_{0}\Vert^{2}\right\vert \leq C\left\vert
k\right\vert ^{-1}\Vert\mathbf{\tilde{g}}_{T}\Vert_{\mathbf{L}^{2}(\Gamma
)}^{2}. \label{eq:apriori-from-real}%
\end{equation}
Next, we estimate the traces $\gamma_{T}\operatorname*{curl}\mathbf{u}_{0}$
and $\gamma_{T}\mathbf{u}_{0}$. Since $\Pi_{T}\mathbf{u}_{0}\in\mathbf{L}%
_{T}^{2}(\Gamma)$ we may employ $\gamma_{T}\mathbf{u}_{0}=\left(  \Pi
_{T}\mathbf{u}_{0}\right)  \times\mathbf{n}$ and (\ref{estu0T}) to obtain%
\begin{equation}
\Vert\gamma_{T}\mathbf{u}_{0}\Vert_{\mathbf{L}^{2}(\Gamma)}=\left\Vert \Pi
_{T}\mathbf{u}_{0}\right\Vert _{\mathbf{L}^{2}(  \Gamma)  }%
\leq\frac{1}{\left\vert k\right\vert }\Vert\mathbf{\tilde{g}}_{T}%
\Vert_{\mathbf{L}^{2}(\Gamma)}. \label{eq:estimate-j}%
\end{equation}%
The boundary conditions (second equation in (\ref{u02})) lead to%
\begin{equation}
\left\Vert \gamma_{T}\operatorname*{curl}\mathbf{u}_{0}\right\Vert
_{\mathbf{L}^{2}(  \Gamma)  }\leq\left\Vert \mathbf{\tilde{g}}%
_{T}\right\Vert _{\mathbf{L}^{2}\left(  \Gamma\right)  }+\left\vert
k\right\vert \left\Vert \left(  \mathbf{u}_{0}\right)  _{T}\right\Vert
_{\mathbf{L}^{2}(  \Gamma)  }\leq2\left\Vert \mathbf{\tilde{g}}%
_{T}\right\Vert _{\mathbf{L}^{2}(  \Gamma)  }. \label{gammatu0}%
\end{equation}
The estimate (\ref{gammatu0}) also implies%
\begin{equation}
\left\Vert \operatorname*{div}\nolimits_{\Gamma}\gamma_{T}\operatorname*{curl}%
\mathbf{u}_{0}\right\Vert _{\mathbf{H}^{-1}(  \Gamma)  }%
\leq\left\Vert \gamma_{T}\operatorname*{curl}\mathbf{u}_{0}\right\Vert
_{\mathbf{L}^{2}(  \Gamma)  }\leq2\left\Vert \mathbf{\tilde{g}}%
_{T}\right\Vert _{\mathbf{L}^{2}(  \Gamma)  }.
\label{eq:estimate-div-j}%
\end{equation}

\textbf{Step 4. (Mapping properties of Maxwell Layer Potentials.)}

The mapping properties of $\operatorname*{curl}\mathcal{S}_{\operatorname*{MW}%
,k}^{\operatorname*{curl}}$, $\mathcal{S}_{\operatorname*{MW},k}%
^{\operatorname*{curl}}$, and $\mathcal{S}_{\operatorname*{MW},k}^{\nabla}$
are well understood due to their relation with the acoustic single layer
potential. We conclude from \cite[Lem.~{3.4}, Thm.~{5.3}]{MelenkStab}:%
\[
\left\Vert \mathcal{S}_{\operatorname*{MW},k}^{\operatorname*{curl}}%
{\boldsymbol \mu}%
\right\Vert _{\mathbf{H}^{s}(  \Omega)  }\leq C_{s}\left\vert
k\right\vert ^{s+1}\left\Vert
{\boldsymbol \mu}%
\right\Vert _{\mathbf{H}^{s-3/2}(  \Gamma)  }\quad\text{for }%
s\geq0.
\]
Hence, the Stratton-Chu formula leads to the estimate%
\begin{align}
\label{u0H-1/2}
&  \left\Vert \mathbf{u}_{0}\right\Vert _{\mathbf{H}^{-1/2}(\Omega)} \\
& \leq \left\Vert \operatorname*{curl}\mathcal{S}_{\operatorname*{MW}%
,k}^{\operatorname*{curl}}(\gamma_{T}\mathbf{u}_{0})\right\Vert _{\mathbf{H}%
^{-1/2}(\Omega)}+\left\Vert \mathcal{S}_{\operatorname*{MW},k}%
^{\operatorname*{curl}}(\gamma_{T}\operatorname*{curl}\mathbf{u}%
_{0})\right\Vert _{\mathbf{H}^{-1/2}(\Omega)}
+\left\Vert \mathcal{S}_{\operatorname*{MW},k}^{\nabla}%
(\gamma_{T}\operatorname*{curl}\mathbf{u}_{0})\right\Vert _{\mathbf{H}%
^{-1/2}(\Omega)}
\nonumber\\
&  \leq\left\Vert \mathcal{S}_{\operatorname*{MW},k}%
^{\operatorname*{curl}}(\gamma_{T}\mathbf{u}_{0})\right\Vert _{\mathbf{H}%
^{1/2}(\Omega)}+\left\Vert \mathcal{S}_{\operatorname*{MW},k}%
^{\operatorname*{curl}}(\gamma_{T}\operatorname*{curl}\mathbf{u}%
_{0})\right\Vert +\left\vert k\right\vert ^{-2}\left\Vert \mathcal{S}%
_{\operatorname*{MW},k}^{\operatorname*{curl}}(\operatorname*{div}%
\nolimits_{\Gamma}\gamma_{T}\operatorname*{curl}\mathbf{u}_{0})\right\Vert
_{\mathbf{H}^{1/2}(\Omega)}\nonumber\\
&  \leq C\left(  \left\vert k\right\vert ^{3/2}\left\Vert \gamma
_{T}\mathbf{u}_{0}\right\Vert _{\mathbf{L}^{2}(\Gamma)}+C\left\vert
k\right\vert \left\Vert \gamma_{T}\operatorname*{curl}\mathbf{u}%
_{0}\right\Vert _{\mathbf{H}^{-3/2}(\Gamma)}+\left\vert k\right\vert
^{-1/2}\left\Vert \operatorname*{div}\nolimits_{\Gamma}\gamma_{T}%
\operatorname*{curl}\mathbf{u}_{0}\right\Vert _{\mathbf{H}^{-1}(\Gamma
)}\right)  .\nonumber
\end{align}
Inserting (\ref{eq:estimate-j}), (\ref{gammatu0}), (\ref{eq:estimate-div-j})
in (\ref{u0H-1/2}), we get%
\begin{equation}
\left\Vert \mathbf{u}_{0}\right\Vert _{\mathbf{H}^{-1/2}\left(  \Omega\right)
}\leq C\left\vert k\right\vert \left\Vert \mathbf{\tilde{g}}_{T}\right\Vert
_{\mathbf{L}^{2}(\Gamma)}. \label{u0est1/2}%
\end{equation}

\textbf{Step 5}. Let $\mathbf{R}_{2}$ and $\mathbf{K}_{2}$ be as in
Lemma~\ref{LemRs} and consider
\begin{equation}
\widetilde{\mathbf{u}}:=\mathbf{r}_{0}-\operatorname*{curl}\mathbf{u}_{0}%
\quad\text{for\quad}\mathbf{r}_{0}:={\mathbf{R}}_{2}(
\operatorname*{curl}\operatorname*{curl}\mathbf{u}_{0})
.\label{defutilde}%
\end{equation}
Since $\mathbf{u}_{0}\in\mathbf{X}_{\operatorname*{imp}}$, we have
$\mathbf{u}_{0}\in\mathbf{L}^{2}(\Omega)$, and the relation
$\operatorname*{curl}\operatorname*{curl}\mathbf{u}_{0}-k^{2}\mathbf{u}%
_{0}=\mathbf{0}$ implies $\operatorname*{curl}\operatorname*{curl}%
\mathbf{u}_{0}\in\mathbf{L}^{2}(\Omega)$. Hence, ${\mathbf{r}}_{0}%
\in\mathbf{H}^{1}(\Omega)$ together with the $k$-explicit bound
\begin{align}
\nonumber 
\left\Vert {\mathbf{r}}_{0}\right\Vert _{\mathbf{H}^{1/2}(\Omega)}%
& =k^{2}\left\Vert {\mathbf{R}}_{2}(\mathbf{u}_{0})\right\Vert _{\mathbf{H}%
^{1/2}(\Omega)}\leq C\left\vert k\right\vert ^{2}\left\Vert \mathbf{u}%
_{0}\right\Vert _{\mathbf{H}^{-1/2}(\Omega)}
\\
& \overset{\text{(\ref{u0est1/2})}%
}{\leq}C\left\vert k\right\vert ^{3}%
\left\Vert \mathbf{\tilde{g}}_{T}\right\Vert _{\mathbf{L}^{2}(\Gamma
)}.\label{Rcurlcurlu0}%
\end{align}
By the same reasoning and the mapping properties of ${\mathbf{R}}_{2}$, we
obtain
\begin{equation}
\left\Vert {\mathbf{r}}_{0}\right\Vert _{\mathbf{H}^{1}(\Omega)}\leq
C\left\vert k\right\vert ^{2}\left\Vert \mathbf{u}_{0}\right\Vert
.\label{H1estr0u0}%
\end{equation}
Furthermore, we compute with Lemma~\ref{LemRs}
\begin{equation}
\operatorname{curl}\widetilde{\mathbf{u}}\overset{(\ref{defutilde})}%
{=}\operatorname{curl}\mathbf{R}_{2}(\operatorname{curl}\operatorname{curl}%
{\mathbf{u}}_{0})-\operatorname{curl}\operatorname{curl}\mathbf{u}_{0}%
\overset{\text{Lem.~\ref{LemRs}}}{=}-\operatorname{curl}{\mathbf{K}}%
_{2}\operatorname{curl}{\mathbf{u}}_{0}.\label{eq:rep-curlutilde}%
\end{equation}
We employ the Helmholtz decomposition of $\widetilde{\mathbf{u}}$ in the form
$\widetilde{\mathbf{u}}=\nabla\varphi+\operatorname*{curl}\mathbf{z}$ given in
Lemma~\ref{LemHelmDecompVar1} with $\varphi\in H_{0}^{1}\left(  \Omega\right)
\cap H^{3/2}(  \Omega)   $%
, $\mathbf{z\in H}^{1}(\Omega)$, and $\operatorname*{div}\mathbf{z}=0$. Since
$\operatorname*{div}\widetilde{\mathbf{u}}=\operatorname*{div}\mathbf{r}_{0}$
the function $\varphi$ does not depend on $\operatorname*{curl}\mathbf{u}_{0}$
(see (\ref{difphiu0SF})), and we obtain from (\ref{HelmDecompVar1esta})%
\begin{equation}
\left\Vert \varphi\right\Vert _{H^{3/2}(\Omega)}\overset
{\text{(\ref{HelmDecompVar1estc})}}{\leq}C\left\Vert {\mathbf{r}}%
_{0}\right\Vert _{\mathbf{H}^{1/2}\left(  \Omega\right)  }\overset
{\text{(\ref{Rcurlcurlu0})}}{\leq}C\left\vert k\right\vert^{3}
\left\Vert \mathbf{\tilde{g}}_{T}\right\Vert _{\mathbf{L}^{2}(\Gamma
)}.\label{phiest3/2}%
\end{equation}
Next, we estimate $\mathbf{z}$. The definition of $\mathbf{r}_{0}$ in
(\ref{defutilde}) gives
\begin{equation}
\gamma_{T}\tilde{\mathbf{u}}=\gamma_{T}{\mathbf{r}}_{0}-\gamma_{T}%
\operatorname*{curl}\mathbf{u}_{0}=\gamma_{T}\mathbf{r}_{0}-\mathbf{\tilde{g}%
}_{T}+\operatorname*{i}k(\mathbf{u}_{0})_{T}\in\mathbf{L}_{T}^{2}%
(\Gamma).\label{gammaTutilderep}%
\end{equation}
Lemma~\ref{LemHelmDecompVar1} then implies
\begin{align}
& \Vert\operatorname*{curl}\mathbf{z}\Vert_{\mathbf{H}^{1/2}( \Omega)  }   \leq C\left(  \left\Vert \operatorname*{curl}%
\widetilde{\mathbf{u}}\right\Vert +\left\Vert \gamma_{T}\widetilde{\mathbf{u}%
}\right\Vert _{\mathbf{L}^{2}(  \Gamma)  }\right)  \nonumber\\
&  \overset{(\ref{eq:rep-curlutilde})}{=}C\left(  \left\Vert
\operatorname{curl}{\mathbf{K}}_{2}\left(  \operatorname*{curl}\mathbf{u}%
_{0}\right)  \right\Vert +\left\Vert \gamma_{T}\widetilde{\mathbf{u}%
}\right\Vert _{\mathbf{L}^{2}(  \Gamma)  }\right)  \nonumber\\
&  \overset{\text{(\ref{gammaTutilderep}), \text{Lem.~\ref{LemRs}}}}{\leq
}C\left(  \left\Vert \mathbf{u}_{0}\right\Vert _{\mathbf{H}^{-1/2}(\Omega
)}+\left\Vert \gamma_{T}\mathbf{r}_{0}\right\Vert _{\mathbf{L}^{2}(
\Gamma)  }+\left\vert k\right\vert \left\Vert \left(  \mathbf{u}%
_{0}\right)  _{T}\right\Vert _{\mathbf{L}^{2}\left(  \Gamma\right)
}+\left\Vert \mathbf{\tilde{g}}_{T}\right\Vert _{\mathbf{L}^{2}(
\Gamma)  }\right)  \nonumber\\
&  \overset{\text{(\ref{u0est1/2}), (\ref{estu0T})}}{\leq}C\left(  \left\vert
k\right\vert \left\Vert {\mathbf{\tilde{g}}}_{T}\right\Vert _{\mathbf{L}%
^{2}(\Gamma)}+\Vert\gamma_{T}\mathbf{r}_{0}\Vert_{\mathbf{L}^{2}(
\Gamma)  }+\left\Vert \mathbf{\tilde{g}}_{T}\right\Vert _{\mathbf{L}%
^{2}(  \Gamma)  }\right)  .\label{curlz1/2}%
\end{align}

\textbf{Step 6. }The combination of Step~5 with a trace inequality leads to%
\begin{align}
\nonumber 
\left\Vert \widetilde{\mathbf{u}}\right\Vert _{\mathbf{H}^{1/2}(
\Omega)  } &\leq\left\Vert \nabla\varphi\right\Vert _{\mathbf{H}%
^{1/2}(  \Omega)  }+\left\Vert \operatorname*{curl}\mathbf{z}%
\right\Vert _{\mathbf{H}^{1/2}(  \Omega)  } 
\\ & \overset
{\text{(\ref{phiest3/2}), (\ref{curlz1/2})}}{\leq}C\left(  \left\vert
k\right\vert ^{ 3 }%
\left\Vert \mathbf{\tilde{g}}_{T}\right\Vert _{\mathbf{L}^{2}\left(
\Gamma\right)  }+\left\Vert \gamma_{T}\mathbf{r}_{0}\right\Vert _{\mathbf{L}%
^{2}(  \Gamma)  }\right)  . 
\label{eq:utilde-10}%
\end{align}
Let $\mathbf{B}_{2,1}^{1/2}(\Omega)$ denote the Besov space as defined, e.g.,
in \cite{triebel95}. Then the trace map $\displaystyle\gamma_{T}%
:\mathbf{B}_{2,1}^{1/2}(\Omega)\rightarrow\mathbf{L}_{T}^{2}(\Gamma)$ is a
continuous mapping (see \cite[Thm.~{2.9.3}]{triebel95}), and we obtain from
(\ref{eq:utilde-10})
\begin{equation}
\left\Vert \widetilde{\mathbf{u}}\right\Vert _{\mathbf{H}^{1/2}(
\Omega)  }\leq C\left(  \left\vert k\right\vert ^{ 3 }%
\left\Vert \mathbf{\tilde{g}}_{T}\right\Vert _{\mathbf{L}^{2}(
\Gamma)  }+\left\Vert \mathbf{r}_{0}\right\Vert _{\mathbf{B}_{2,1}%
^{1/2}(  \Omega)  }\right)  . 
\label{eq:utilde-20}%
\end{equation}
This allows us to estimate%
\begin{align}
\nonumber 
\left\Vert \operatorname*{curl}\mathbf{u}_{0}\right\Vert _{\mathbf{H}%
^{1/2}(  \Omega)  } &\overset{\text{(\ref{defutilde})}}{\leq}C\left(
\left\Vert \mathbf{r}_{0}\right\Vert _{\mathbf{H}^{1/2}(  \Omega)
}+\left\Vert \widetilde{\mathbf{u}}\right\Vert _{\mathbf{H}^{1/2}(
\Omega)  }\right) 
\\
&  \overset{\text{(\ref{Rcurlcurlu0}),
(\ref{eq:utilde-20})}}{\leq}C\left(  \left\vert k\right\vert ^{ 3 }%
\Vert\mathbf{\tilde{g}}_{T}\Vert_{\mathbf{L}^{2}(  \Gamma)
}+\left\Vert \mathbf{r}_{0}\right\Vert _{\mathbf{B}_{2,1}^{1/2}(
\Omega)  }\right)  . 
\label{curlu0est1/2}%
\end{align}
To estimate $\left\Vert \mathbf{r}_{0}\right\Vert _{\mathbf{B}_{2,1}%
^{1/2}\left(  \Omega\right)  }$ we use the fact (see \cite{triebel95}) that
the Besov space is an interpolation space $B_{2,1}^{1/2}(  \Omega)
=\left(  L^{2}(  \Omega)  ,H^{1}(  \Omega)  \right)
_{1/2,1}$ (via the so-called real method of interpolation). For $t\in(0,1]$
select $(\mathbf{r}_{0})_{t}\in{\mathbf{H}}^{1}(\Omega)$ as given by
Lemma~\ref{lemma:interpolation-inequality}
and estimate with the interpolation inequality (by using the notation as in
Lemma~\ref{lemma:interpolation-inequality})%
\begin{align*}
& \left\Vert \mathbf{r}_{0}\right\Vert _{\mathbf{B}_{2,1}^{1/2}(
\Omega)  }    \leq\left\Vert \mathbf{r}_{0}-\left(  \mathbf{r}%
_{0}\right)  _{t}\right\Vert _{\mathbf{B}_{2,1}^{1/2}(  \Omega)
}+\left\Vert \left(  \mathbf{r}_{0}\right)  _{t}\right\Vert _{\mathbf{B}%
_{2,1}^{1/2}(  \Omega)  }\\
&  \leq C\left(  \left\Vert \mathbf{r}_{0}-\left(  \mathbf{r}_{0}\right)
_{t}\right\Vert ^{1/2}\left\Vert \mathbf{r}_{0}-\left(  \mathbf{r}_{0}\right)
_{t}\right\Vert _{\mathbf{H}^{1}\left(  \Omega\right)  }^{1/2}+\left\Vert
\left(  \mathbf{r}_{0}\right)  _{t}\right\Vert _{\mathbf{B}_{2,1}^{1/2}(
\Omega)  }\right) \\
&  \overset{\text{Lem.~\ref{lemma:interpolation-inequality}}}{\leq}C\left(
t^{1/4}\left\Vert \mathbf{r}_{0}\right\Vert _{\mathbf{H}^{1/2}(
\Omega)  }^{1/2}\left(  \left\Vert \mathbf{r}_{0}\right\Vert
_{\mathbf{H}^{1}(  \Omega)  }^{1/2}+t^{-1/4}\left\Vert
\mathbf{r}_{0}\right\Vert _{\mathbf{H}^{1/2}(  \Omega)  }%
^{1/2}\right)  +\left\Vert \left(  \mathbf{r}_{0}\right)  _{t}\right\Vert
_{\mathbf{B}_{2,1}^{1/2}(  \Omega)  }\right) \\
&  \overset{\text{(\ref{Lemintineqa}), (\ref{Lemintineqb})}}{\leq}C\left(
\left\Vert \mathbf{r}_{0}\right\Vert _{\mathbf{H}^{1/2}(  \Omega)
}+t^{1/2}\Vert\mathbf{r}_{0}\Vert_{\mathbf{H}^{1}(  \Omega)
}+\sqrt{1+|\ln t|}\Vert\mathbf{r}_{0}\Vert_{\mathbf{H}^{1/2}(
\Omega)  }\right) \\
&  \leq C\left(  t^{1/2}\left\Vert \mathbf{r}_{0}\right\Vert _{\mathbf{H}%
^{1}(  \Omega)  }+\sqrt{1+\left\vert \ln t\right\vert }\left\Vert
\mathbf{r}_{0}\right\Vert _{\mathbf{H}^{1/2}\left(  \Omega\right)  }\right) \\
&  \overset{\text{(\ref{Rcurlcurlu0}), (\ref{H1estr0u0})}}{\leq}C\left(
t^{1/2}\left\vert k\right\vert ^{2}\left\Vert \mathbf{u}_{0}\right\Vert
+\sqrt{1+\left\vert \ln t\right\vert }\left\vert k\right\vert ^{2}%
\Vert\mathbf{u}_{0}\Vert_{\mathbf{H}^{-1/2}(  \Omega)  }\right)  .
\end{align*}%
Using (\ref{eq:apriori-from-real}) we get%
\begin{align*}
& \left\vert k\right\vert \left\Vert \mathbf{u}_{0}\right\Vert    \leq C\left(
\left\Vert \operatorname*{curl}\mathbf{u}_{0}\right\Vert ^{2}+\left\vert
\left(  \left\vert k\right\vert \left\Vert \mathbf{u}_{0}\right\Vert \right)
^{2}-\left\Vert \operatorname*{curl}\mathbf{u}_{0}\right\Vert ^{2}\right\vert
\right)  ^{1/2}\\
&  \leq C\left(  \left\vert k\right\vert ^{-1/2}\left\Vert \mathbf{\tilde{g}%
}_{T}\right\Vert _{\mathbf{L}^{2}\left(  \Gamma\right)  }+\left\Vert
\operatorname*{curl}\mathbf{u}_{0}\right\Vert \right)  \overset
{\text{(\ref{curlu0est1/2})}}{\leq}C\left(  \left\vert k\right\vert ^{ 3 }%
\left\Vert \mathbf{\tilde{g}}_{T}\right\Vert _{\mathbf{L}^{2}\left(
\Gamma\right)  }+\left\Vert \mathbf{r}_{0}\right\Vert _{\mathbf{B}_{2,1}%
^{1/2}\left(  \Omega\right)  }\right) \\
&  \leq C\left(  \left\vert k\right\vert ^{ 3 }%
\left\Vert \mathbf{\tilde{g}}_{T}\right\Vert _{\mathbf{L}^{2}\left(
\Gamma\right)  }+t^{1/2}\left\vert k\right\vert ^{2}\left\Vert \mathbf{u}%
_{0}\right\Vert +\sqrt{1+\left\vert \ln t\right\vert }\left\vert k\right\vert
^{2}\left\Vert \mathbf{u}_{0}\right\Vert _{\mathbf{H}^{-1/2}\left(
\Omega\right)  }\right) \\
&  \overset{\text{(\ref{u0est1/2})}}{\leq}C\left(  \sqrt{1+\left\vert \ln
t\right\vert }\left\vert k\right\vert ^{ 3 }  
\left\Vert \mathbf{\tilde{g}}_{T}\right\Vert _{\mathbf{L}^{2}\left(
\Gamma\right)  }+t^{1/2}\left\vert k\right\vert ^{2}\left\Vert \mathbf{u}%
_{0}\right\Vert \right)  .
\end{align*}
Selecting $t\sim1/\left\vert k\right\vert ^{2}$ sufficiently small implies%
\[
\left\vert k\right\vert \left\Vert \mathbf{u}_{0}\right\Vert \leq C\left\vert
k\right\vert ^{ 3 }%
\sqrt{1+\ln\left\vert k\right\vert }\left\Vert \mathbf{\tilde{g}}%
_{T}\right\Vert _{\mathbf{L}^{2}(  \Gamma)  }.
\]
We conclude from this and (\ref{curlu0est1/2})
\begin{equation}
\left\Vert \operatorname*{curl}\mathbf{u}_{0}\right\Vert _{\mathbf{H}%
^{1/2}\left(  \Omega\right)  }+\left\vert k\right\vert \left\Vert
\mathbf{u}_{0}\right\Vert \leq C\left\vert k\right\vert ^{3 }%
\sqrt{1+\ln\left\vert k\right\vert }\left\Vert \mathbf{\tilde{g}}%
_{T}\right\Vert _{\mathbf{L}^{2}(  \Gamma)  }.
\label{kOmegacurlest}%
\end{equation}
Combining (\ref{kOmegacurlest}) and (\ref{estu0T}) yields
\begin{align}
\Vert\mathbf{u}_{0}\Vert_{\operatorname{imp},k}  &  \leq C\left\vert
k\right\vert ^{ 3 }%
\sqrt{1+\ln\left\vert k\right\vert }\left\Vert \mathbf{\tilde{g}}%
_{T}\right\Vert _{\mathbf{L}^{2}(  \Gamma)  }\nonumber\\
&  \overset{(\ref{eq:stability-estimate-gtilde})}{\leq}C\left\vert
k\right\vert ^{ 3 }%
\sqrt{1+\ln\left\vert k\right\vert }\left(  \left\Vert \mathbf{g}%
_{T}\right\Vert _{\mathbf{L}^{2}(  \Gamma)  }+|k|^{1/2}%
\Vert\mathbf{j}\Vert_{\mathbf{L}^{2}(\Omega)}\right)  .
\label{eq:stability-estimate-u0}%
\end{align}

\textbf{Step 7. }Combining (\ref{eq:stability-estimate-psi}),
(\ref{eq:stability-estimate-uj}), and (\ref{eq:stability-estimate-u0}), we
have arrived at
\begin{align*}%
\left\Vert \mathcal{S}_{\Omega,k}^{\operatorname{MW}}(\mathbf{j}%
,\mathbf{g}_{T})\right\Vert _{\operatorname{imp},k}%
&  \leq\Vert k^{-2}\nabla\psi\Vert_{\operatorname{imp},k}+\Vert\mathbf{u}%
_{\mathbf{j}}\Vert_{\operatorname{imp},k}+\Vert\mathbf{u}_{0}\Vert
_{\operatorname{imp},k}\\
&  \leq C\left\vert k\right\vert ^{ 3 }%
\sqrt{1+\ln\left\vert k\right\vert }\left(  \left\Vert \mathbf{g}%
_{T}\right\Vert _{\mathbf{L}^{2}(  \Gamma)  }+|k|^{1/2}%
\Vert\mathbf{j}\Vert_{\mathbf{L}^{2}(\Omega)}\right)  ,
\end{align*}
which is the claimed estimate.
\end{proof}

\begin{lemma}
[{{\cite[Prop.~{4.14}]{melenk-rieder21}}}]%
\label{lemma:interpolation-inequality}Let $\Omega\subset\mathbb{R}^{3}$ be a
bounded Lipschitz domain. Then there is $C>0$ such that for every $w\in
H^{1/2}(\Omega)$ and every $t\in\left(  0,1\right]  $ there exists some
$w_{t}\in H^{1}(\Omega)$ such that
\begin{align}
\left\Vert w-w_{t}\right\Vert +t\left\Vert w_{t}\right\Vert _{H^{1}(\Omega)}
&  \leq Ct^{1/2}\left\Vert w\right\Vert _{H^{1/2}\left(  \Omega\right)
},\label{Lemintineqa}\\
\Vert w_{t}\Vert_{B_{2,1}^{1/2}(\Omega)}  &  \leq C\sqrt{1+\left\vert \ln
t\right\vert }\left\Vert w\right\Vert _{H^{1/2}(\Omega)}. \label{Lemintineqb}%
\end{align}

\end{lemma}


\section{Maxwell's equations with the \textquotedblleft
good\textquotedblright\label{sec:good-sign} sign}


\subsection{Norms}

We consider Maxwell's equations with the \textquotedblleft
good\textquotedblright\ sign and first describe the spaces for the given
data.
Since the sesquilinear form $A_{k}\left(  \cdot,\cdot\right)  $ is considered
in the space $\mathbf{X}_{\operatorname*{imp}}$, the natural space for the
right-hand side is its dual $\mathbf{X}_{\operatorname*{imp}}^{\prime}$. In
our setting, the right-hand side is given in $\Omega$ via the volume data
$\mathbf{j}$ and on $\Gamma$ via the boundary data $\mathbf{g}_{T}$. 
In order to view $\mathbf{j}$ and $\mathbf{g}_T$ as elements of 
$\mathbf{X}_{\operatorname*{imp}}^\prime$, we introduce the spaces 
$\mathbf{X}_{\operatorname*{imp}}^{\prime}(  \Omega)$
and $\mathbf{X}_{\operatorname*{imp}}^{\prime}(  \Gamma)  $.
By using the usual notation $V^{\prime}$ for the dual space of a normed vector
space $V$ we define%
\begin{align}
\mathbf{X}_{\operatorname*{imp},0}  &  :=\left\{  \mathbf{w}\in\mathbf{X}%
_{\operatorname*{imp}}:\operatorname*{curl}\mathbf{w}=0\right\}
={\{\nabla\varphi\,|\,\varphi\in H_{\operatorname{imp}}^{1}(\Omega
)\}},\nonumber\\
\mathbf{X}_{\operatorname*{imp}}^{\prime}(\Omega)  &  :=\left(  \mathbf{H}%
^{1}(\Omega)\right)  ^{\prime}\cap\mathbf{X}_{\operatorname*{imp},0}^{\prime
},\label{XimpprimeOmega}\\
\mathbf{H}_{T}^{-1}(\operatorname*{div}\nolimits_{\Gamma},\Gamma)  &
:=\left\{  \mathbf{w}\in\mathbf{H}_{T}^{-1}(\Gamma)\mid\operatorname*{div}%
\nolimits_{\Gamma}\mathbf{w}\in\mathbf{H}_{T}^{-1}(\Gamma)\right\}
,\nonumber\\
\mathbf{X}_{\operatorname*{imp}}^{\prime}(\Gamma)  &  :=\mathbf{H}_{T}%
^{-1/2}(\Gamma)\cap\mathbf{H}_{T}^{-1}\left(  \operatorname*{div}%
\nolimits_{\Gamma},\Gamma\right)  . \label{XimpprimeGamma}%
\end{align}
and equip the spaces $\mathbf{X}_{\operatorname{imp}}^{\prime}(\Omega)$
and $\mathbf{X}_{\operatorname{imp}}^{\prime}(\Gamma)$ with the norms
(cf.\ also Lemma~\ref{lemma:norm-equivalence} below)%
\begin{align}
\Vert\mathbf{f}\Vert_{\mathbf{X}_{\operatorname{imp}}^{\prime}(\Omega),k}  &
:=\sup_{{\mathbf{v}}\in{\mathbf{X}}_{\operatorname{imp}}\setminus\{0\}}%
\frac{|({\mathbf{f}},{\mathbf{v}})|}{\Vert{\mathbf{v}}\Vert
_{\operatorname{imp},k}},\label{eq:H-1div}\\
\Vert\mathbf{g}_{T}\Vert_{\mathbf{X}_{\operatorname{imp}}^{\prime}(\Gamma),k}
&  :=\sup_{{\mathbf{v}}\in{\mathbf{X}}_{\operatorname{imp}}}\frac
{|({\mathbf{g}}_{T},{\mathbf{v}}_{T})_{{\mathbf{L}}^{2}(\Gamma)}|}%
{\Vert{\mathbf{v}}\Vert_{\operatorname{imp},k}}. \label{eq:H-1/2diva}%
\end{align}
We also introduce for $\mathbf{g}_{T}\in\mathbf{H}_{T}^{-1/2}\left(
\operatorname*{div}\nolimits_{\Gamma},\Gamma\right)  $ (cf.
(\ref{Hm1/2divspace}))%
\begin{equation}
\Vert\mathbf{g}_{T}\Vert_{\mathbf{H}^{-1/2}(\operatorname{div}_{\Gamma}%
,\Gamma),k}:=|k|\Vert\operatorname{div}_{\Gamma}\mathbf{g}_{T}\Vert
_{H^{-1/2}(\Gamma)}+\left\vert k\right\vert ^{2}\Vert\mathbf{g}_{T}%
\Vert_{\mathbf{X}_{\operatorname*{imp}}^{\prime}\left(  \Gamma\right)  ,k}.
\label{eq:H-1/2div}%
\end{equation}
An equivalent norm that is more naturally associated with the intersection
spaces $\mathbf{X}_{\operatorname{imp}}^{\prime}(\Omega)$ and $\mathbf{X}%
_{\operatorname{imp}}^{\prime}(\Gamma)$ is given in the following lemma.

\begin{lemma}
\label{lemma:norm-equivalence}The spaces $\mathbf{X}_{\operatorname{imp}%
}^{\prime}(\Omega)$ and $\mathbf{X}_{\operatorname{imp}}^{\prime}(\Gamma)$ can
be viewed in a canonical way as subspaces of $\mathbf{X}_{\operatorname{imp}%
}^{\prime}$, and there holds the norm equivalences
\begin{align}
\left\Vert \mathbf{f}\right\Vert _{\mathbf{X}_{\operatorname*{imp}}^{\prime
}\left(  \Omega\right)  ,k} &  \sim\!\!\sup_{\varphi\in H_{\operatorname*{imp}%
}^{1}\left(  \Omega\right)  :\nabla\varphi\neq\mathbf{0}}\frac{|(\mathbf{f}%
,\nabla\varphi)|}{\Vert\nabla\varphi\Vert_{\operatorname{imp},k}}%
+\sup_{\mathbf{z}\in\mathbf{H}^{1}\left(  \Omega\right)  \backslash\left\{
\mathbf{0}\right\}  }\frac{|(\mathbf{f},\mathbf{z})|}{%
\left\vert k\right\vert \left\Vert \mathbf{z}\right\Vert _{\mathbf{H}%
^{1}(\Omega),k}}%
,\label{eq:lemma:norm-equivalence-10}\\
\left\Vert \mathbf{g}_{T}\right\Vert _{\mathbf{X}_{\operatorname*{imp}%
}^{\prime}\left(  \Gamma\right)  ,k} &  \sim
\sup_{\varphi\in
H_{\operatorname*{imp}}^{1}\left(  \Omega\right)  :\nabla\varphi\neq
\mathbf{0}}\frac{|(\mathbf{g}_{T},\nabla_{\Gamma}\varphi)_{\mathbf{L}%
^{2}(\Gamma)}|}{\Vert\nabla\varphi\Vert_{\operatorname{imp},k}}+\sup
_{\mathbf{z}\in\mathbf{H}^{1}\left(  \Omega\right)  \backslash\left\{
\mathbf{0}\right\}  }\frac{|(\mathbf{g}_{T},\mathbf{z}_{T})_{\mathbf{L}%
^{2}(\Gamma)}|}{%
\left\vert k\right\vert \left\Vert \mathbf{z}\right\Vert _{\mathbf{H}%
^{1}(\Omega),k}}\label{eq:lemma:norm-equivalence-20}%
\end{align}
with constants implied in $\sim$ that are independent of $\left\vert
k\right\vert \geq k_{0}$.
\end{lemma}

\begin{proof}
\emph{Proof of (\ref{eq:lemma:norm-equivalence-10}):} Since $\nabla\varphi
\in\mathbf{X}_{\operatorname{imp}}$ for $\varphi\in H_{\operatorname{imp}}%
^{1}(\Omega)$ and ${\mathbf{H}}^{1}(\Omega)\subset{\mathbf{X}}%
_{\operatorname{imp}}$, the right-hand side of
(\ref{eq:lemma:norm-equivalence-10}) is easily bounded by the left-hand side.
For the reverse estimate, we decompose any element $\mathbf{v}\in
\mathbf{X}_{\operatorname{imp}}$ with the aid of
Lemma~\ref{lemma:helmholtz-a-la-schoeberl} as $\mathbf{v}=\nabla
\varphi+\mathbf{z}$ with $\Vert\nabla\varphi\Vert_{\mathbf{L}^{2}(\Omega
)}+\Vert\mathbf{z}\Vert_{\mathbf{L}^{2}(\Omega)}\leq C\Vert\mathbf{v}%
\Vert_{\mathbf{L}^{2}(\Omega)}$ and $\Vert\mathbf{z}\Vert_{\mathbf{H}%
^{1}(\Omega)}\leq C\Vert\mathbf{v}\Vert_{\mathbf{H}(\operatorname{curl}%
,\Omega)}$. Hence,
\begin{align}
\nonumber 
\Vert\nabla\varphi\Vert_{\operatorname{imp},k}+|k|\Vert\mathbf{z}%
\Vert_{\mathbf{H}^{1}(\Omega),k}
& \leq C\left(  |k|^{1/2}(\Vert\mathbf{v}%
_{T}\Vert_{\mathbf{L}^{2}(\Gamma)}+\Vert\mathbf{z}_{T}\Vert_{\mathbf{L}%
^{2}(\Gamma)})+\Vert\mathbf{v}\Vert_{\operatorname{imp},k}\right)  
\\
& \leq
C\Vert\mathbf{v}\Vert_{\operatorname{imp},k},
\label{eq:k-norm-stable-decomposition-10}%
\end{align}
where, in the last step we used the multiplicative trace estimate
$\Vert\mathbf{z}_{T}\Vert_{\mathbf{L}^{2}(\Gamma)}^{2}\leq C\Vert
\mathbf{z}\Vert_{\mathbf{L}^{2}(\Omega)}\Vert\mathbf{z}\Vert_{\mathbf{H}%
^{1}(\Omega)}$. This implies that the left-hand side of
(\ref{eq:lemma:norm-equivalence-10}) can be bounded by the right-hand side.

\emph{Proof of (\ref{eq:lemma:norm-equivalence-20}):} The proof is analogous
to that of (\ref{eq:lemma:norm-equivalence-10}).
\end{proof}

Note that $\mathbf{L}^{2}(\Omega)\subset\mathbf{X}_{\operatorname*{imp}%
}^{\prime}(\Omega)$ and ${\mathbf{L}}_{T}^{2}(\Gamma)\subset\mathbf{X}%
_{\operatorname*{imp}}^{\prime}(\Gamma)$ with continuous embeddings as can be
seen from the following reasoning. For $m\in\mathbb{N}_{0}$ and $\mathbf{f}%
\in\mathbf{L}^{2}(\Omega)$ or $\mathbf{f}\in\mathbf{H}^{m}(\Omega)$ or
$\mathbf{f}\in\mathbf{H}^{m}(\operatorname{div},\Omega)$ and for
$\mathbf{g}_{T}\in\mathbf{L}_{T}^{2}(\Gamma)$ or $\mathbf{g}_{T}\in
\mathbf{H}_{T}^{m+1/2}(\Gamma)$ or $\mathbf{g}_{T}\in\mathbf{H}_{T}%
^{m+1/2}(\operatorname*{div}\nolimits_{\Gamma},\Gamma)$, we have by direct
estimations%
\begin{align}
& \Vert\mathbf{f}\Vert_{\mathbf{X}_{\operatorname*{imp}}^{\prime}(
\Omega)  ,k}    \!\leq \! C|k|^{-1}\Vert\mathbf{f}\Vert_{\mathbf{L}%
^{2}(\Omega)}\!\leq \! C|k|^{-1}\Vert\mathbf{f}\Vert_{\mathbf{H}^{m}(\Omega),k}\!\leq\!
C|k|^{-2}\Vert\mathbf{f}\Vert_{\mathbf{H}^{m}(\operatorname{div},\Omega
),k},\!\!
\label{eq:H-1div-vs-L2}\\
\nonumber 
& \Vert\mathbf{g}_{T}\Vert_{\mathbf{X}_{\operatorname*{imp}}^{\prime}(
\Gamma)  ,k}    \leq C|k|^{-1/2}\Vert\mathbf{g}_{T}\Vert_{\mathbf{L}%
^{2}(\Gamma)}\leq C|k|^{-1}\Vert\mathbf{g}_{T}\Vert_{\mathbf{H}^{m+1/2}%
(\Gamma),k} \\
& \phantom{ 
\Vert\mathbf{g}_{T}\Vert_{\mathbf{X}_{\operatorname*{imp}}^{\prime}( \Gamma)  ,k}
          } 
\leq C\left\vert k\right\vert ^{-2}\Vert\mathbf{g}_{T}%
\Vert_{\mathbf{H}^{m+1/2}(\operatorname{div}_{\Gamma},\Gamma),k}%
,\label{eq:H-1/2div-vs-L2}\\
& \Vert\mathbf{g}_{T}\Vert_{\mathbf{H}^{-1/2}(\operatorname{div}_{\Gamma}%
,\Gamma),k}    \leq C|k|\Vert\mathbf{g}_{T}\Vert_{\mathbf{H}^{1/2}(\Gamma
),k}. \label{eq:H-1div-vs-L2-b}%
\end{align}
We also have the following result for $\Vert\mathbf{g}_{T}\Vert_{\mathbf{H}%
^{-1/2}(\operatorname{div}_{\Gamma},\Gamma),k}$:

\begin{lemma}
\label{lemma:H-1/2-div} There is $C>0$ depending only on $\Omega$ such that%
\[
\Vert\mathbf{g}_{T}\Vert_{\mathbf{H}^{-1/2}(\operatorname{div}_{\Gamma}%
,\Gamma),k}\leq C\Vert\operatorname{div}_{\Gamma}\mathbf{g}_{T}\Vert
_{H^{-1/2}(\Gamma),k}+|k|\Vert\mathbf{g}_{T}\Vert_{%
 \mathbf{H}^{-1/2}(\Gamma),k}%
.
\]

\end{lemma}

\begin{proof}
We use the minimum norm lifting $\mathcal{E}_{\Omega}^{\Delta}$ from
(\ref{LaplaceDiriProbl}) with the property that $\Vert\nabla\varphi\Vert
_{\mathbf{L}^{2}(\Omega)}\geq\Vert\nabla\mathcal{E}_{\Omega}^{\Delta}%
(\varphi|_{\Gamma})\Vert_{\mathbf{L}^{2}(\Omega)}$ for arbitrary $\varphi\in
H_{\operatorname*{imp}}^{1}(\Omega)$. By continuity of the trace mapping, we
get $\inf_{c\in\mathbb{R}}\Vert\varphi-c\Vert_{H^{1/2}(\Gamma)}\leq
C\Vert\nabla\mathcal{E}_{\Omega}^{\Delta}(\varphi|_{\Gamma})\Vert
_{\mathbf{L}^{2}(\Omega)}\leq C\Vert\nabla\varphi\Vert_{\mathbf{L}^{2}%
(\Omega)}$. An integration by parts shows for arbitrary $\varphi\in
H_{\operatorname*{imp}}^{1}(\Omega)$ and arbitrary $c\in\mathbb{R}$
\[
|(\mathbf{g}_{T},\nabla_{\Gamma}\varphi)_{\mathbf{L}^{2}(\Gamma)}%
|=|(\operatorname{div}_{\Gamma}\mathbf{g}_{T},\varphi-c)_{\mathbf{L}%
^{2}(\Gamma)}|\leq\Vert\operatorname{div}_{\Gamma}\mathbf{g}_{T}%
\Vert_{H^{-1/2}(\Gamma)}\Vert\varphi-c\Vert_{H^{1/2}(\Gamma)}.
\]
Taking the infimum over all $c\in\mathbb{R}$ yields, for arbitrary $\varphi\in
H_{\operatorname*{imp}}^{1}(\Omega)$,
\[
|(\mathbf{g}_{T},\nabla_{\Gamma}\varphi)_{\mathbf{L}^{2}(\Gamma)}|\leq
C\Vert\operatorname{div}_{\Gamma}\mathbf{g}_{T}\Vert_{H^{-1/2}(\Gamma)}%
\Vert\nabla\varphi\Vert_{\mathbf{L}^{2}(\Omega)},
\]
and we conclude
\[
\sup_{\varphi\in H_{\operatorname*{imp}}^{1}(\Omega):\nabla_{\Gamma}%
\varphi\neq\mathbf{0}}\frac{|(\mathbf{g}_{T},\nabla_{\Gamma}\varphi
)_{\mathbf{L}^{2}(\Gamma)}|}{|k|^{1/2}\Vert\nabla_{\Gamma}\varphi
\Vert_{\mathbf{L}^{2}(\Gamma)}+|k|\Vert\nabla\varphi\Vert_{\mathbf{L}%
^{2}(\Omega)}}\leq C|k|^{-1}\Vert\operatorname{div}_{\Gamma}\mathbf{g}%
_{T}\Vert_{H^{-1/2}(\Gamma)}.
\]
Similarly, for $\mathbf{z}\in\mathbf{H}^{1}(\Omega)$ we estimate
$|(\mathbf{g}_{T},\mathbf{z}_{T})_{\mathbf{L}^{2}(\Gamma)}|\leq C\Vert
\mathbf{g}_{T}\Vert_{\mathbf{H}^{-1/2}(\Gamma)}\Vert\mathbf{z}\Vert
_{\mathbf{H}^{1}(\Omega)}$. Hence,
\begin{equation}
\Vert\mathbf{g}_{T}\Vert_{\mathbf{X}_{\operatorname*{imp}}^{\prime}\left(
\Gamma\right)  ,k}\leq C\left(  \left\vert k\right\vert ^{-1}\Vert
\operatorname{div}_{\Gamma}\mathbf{g}_{T}\Vert_{H^{-1/2}(\Gamma)}%
+\Vert\mathbf{g}_{T}\Vert_{\mathbf{H}^{-1/2}(\Gamma)}\right)  .
\label{estgTprime}%
\end{equation}
The result follows.
\end{proof}


\subsection{The Maxwell problem with the good sign}

\textbf{The Maxwell problem with the good sign reads:} Given $\displaystyle\mathbf{f}%
\in\mathbf{X}_{\operatorname*{imp}}^{\prime}(\Omega)$ and $\mathbf{g}_{T}%
\in\mathbf{X}_{\operatorname*{imp}}^{\prime}(\Gamma),$ find $\mathbf{v}%
\in\mathbf{X}_{\operatorname{imp}}$ such that
\begin{equation}
\mathcal{L}_{\Omega,\operatorname*{i}k}\mathbf{v}=\mathbf{f}\text{ in }%
\Omega\quad\text{and\quad}\mathcal{B}_{ \Gamma ,k}\mathbf{v}=\mathbf{g}_{T}\text{ on }\Gamma. \label{weak_plus}%
\end{equation}
The weak formulation is:%
\begin{equation}
\text{find }\mathbf{z}\in\mathbf{X}_{\operatorname*{imp}}\quad\text{s.t.\quad
}A_{k}^{+}(\mathbf{z}, \mathbf{v} )=\left(  \mathbf{f},\mathbf{v}\right)  +\left(  \mathbf{g}_{T},\mathbf{v}%
_{T}\right)  _{\mathbf{L}^{2}\left(  \Gamma\right)  }\quad\forall\mathbf{v}%
\in\mathbf{X}_{\operatorname*{imp}}, \label{Maxwellgoodsign-weak}%
\end{equation}
where the sesquilinear form $A_{k}^{+}$ is given by
\begin{equation}
A_{k}^{+}(\mathbf{u},\mathbf{v}):=\left(  \operatorname*{curl}\mathbf{u,}%
\operatorname*{curl}\mathbf{v}\right)  +k^{2}\left(  \mathbf{u,v}\right)
-\operatorname*{i}k\left(  \mathbf{u}_{T},\mathbf{v}_{T}\right)
_{\mathbf{L}^{2}(\Gamma)}. \label{eq:Ak+}%
\end{equation}

The solution operator is denoted $(\mathbf{f},\mathbf{g}_{T})\mapsto
\mathcal{S}_{\Omega,k}^{+}(\mathbf{f},\mathbf{g}_{T}$). In this section, we
develop the regularity theory for problem (\ref{weak_plus}). Indeed, as the
following Theorem~\ref{lemma:apriori-with-good-sign} shows,
(\ref{Maxwellgoodsign-weak}) is uniquely solvable.

\begin{theorem}
\label{lemma:apriori-with-good-sign} Let $\Omega$ be a bounded Lipschitz
domain with simply connected boundary. Then there is $C>0$ independent of $k$
such that the following holds:

\begin{enumerate}
[(i)]

\item \label{item:lemma:apriori-with-good-sign-0} The sesquilinear form
$A_{k}^{+}$ satisfies $\operatorname{Re} A_{k}^{+}(\mathbf{v},\sigma
\mathbf{v}) = 2^{-1/2} \|\mathbf{v}\|^{2}_{\operatorname{imp},k}$ for all
$\mathbf{v} \in\mathbf{X}_{\operatorname{imp}}$, where $\sigma=\exp\left(
\frac{\pi\operatorname*{i}}{4}\operatorname*{sign}k\right)  $.

\item \label{item:lemma:apriori-with-good-sign-00} The sesquilinear form is
continuous: $|A_{k}^{+}(\mathbf{u},\mathbf{v})| \leq\|\mathbf{u}%
\|_{\operatorname{imp},k} \|\mathbf{v}\|_{\operatorname{imp},k}$ for all
$\mathbf{u}$, $\mathbf{v} \in\mathbf{X}_{\operatorname{imp}}$.

\item \label{item:lemma:apriori-with-good-sign-i} The solution $\mathbf{u}%
\in\mathbf{X}_{\operatorname{imp}}$ of (\ref{weak_plus})
satisfies%
\begin{align}
\left\Vert \mathbf{u}\right\Vert _{\operatorname*{imp},k}  &  \leq C\left(
|k|^{-1}\left\Vert \mathbf{f}\right\Vert _{\mathbf{L}^{2}(\Omega)}+\left\vert
k\right\vert ^{-1/2}\left\Vert \mathbf{g}_{T}\right\Vert _{\mathbf{L}%
^{2}(\Gamma)}\right)  ,\label{eq:item:lemma:apriori-with-good-sign-i}\\
\left\Vert \mathbf{u}\right\Vert _{\operatorname{imp},k}  &  \leq C\left(
\left\Vert \mathbf{f}\right\Vert _{\mathbf{X}_{\operatorname*{imp}}^{\prime
}(  \Omega)  ,k}+\left\Vert \mathbf{g}_{T}\right\Vert
_{\mathbf{X}_{\operatorname*{imp}}^{\prime}(  \Gamma)  ,k}\right)
, \label{eq:item:lemma:apriori-with-good-sign-iia}%
\end{align}
provided $\left(  \mathbf{f},\mathbf{g}_{T}\right)  \in\mathbf{L}^{2}(
\Omega)  \times\mathbf{L}_{T}^{2}(  \Gamma)  $ for
(\ref{eq:item:lemma:apriori-with-good-sign-i}) and $\left(  \mathbf{f}%
,\mathbf{g}_{T}\right)  \in\mathbf{X}_{\operatorname*{imp}}^{\prime}(
\Omega)  \times\mathbf{X}_{\operatorname*{imp}}^{\prime}(
\Gamma)  $ for (\ref{eq:item:lemma:apriori-with-good-sign-iia}).

\item \label{item:lemma:apriori-with-good-sign-iibegin}Let $m\in\mathbb{N}%
_{0}$. If $\Gamma$ is sufficiently smooth and $\mathbf{f}\in\mathbf{H}%
^{m}(\operatorname{div},\Omega)$, $\mathbf{g}_{T}\in\mathbf{H}_{T}%
^{m+1/2}(  \Gamma)  $, then
\begin{subequations}
\label{item:lemma:apriori-with-good-sign-ii}
\begin{align}
\left\Vert \mathbf{u}\right\Vert _{\mathbf{H}^{m+1}\left(  \Omega\right)  ,k}
&  \leq C\left\vert k\right\vert ^{-3}\left(  \left\Vert \mathbf{f}\right\Vert
_{\mathbf{H}^{m}(\operatorname{div},\Omega),k}+\left\Vert \mathbf{g}%
_{T}\right\Vert _{\mathbf{H}^{m-1/2}(\operatorname{div}_{\Gamma},\Gamma
),k}\right)  ,
\label{eq:lemma:apriori-with-good-sign-25b}\\
\left\Vert \mathbf{u}\right\Vert _{\mathbf{H}^{m+1}\left(
\operatorname*{curl},\Omega\right)  ,k}  &  \leq C\left\vert k\right\vert
^{-2}\left(  \left\Vert \mathbf{f}\right\Vert _{\mathbf{H}^{m}%
(\operatorname{div},\Omega),k}+|k|\left\Vert \mathbf{g}_{T}\right\Vert
_{\mathbf{H}^{m+1/2}\left(  \Gamma\right)  ,k}\right)  . 
\label{eq:lemma:apriori-with-good-sign-25}%
\end{align}
\end{subequations}
\end{enumerate}
\end{theorem}

\begin{proof}
\emph{Proof of (\ref{item:lemma:apriori-with-good-sign-0}),
(\ref{item:lemma:apriori-with-good-sign-00}):} For
(\ref{item:lemma:apriori-with-good-sign-0}) we compute
\[
\operatorname{Re}\left(  A_{k}^{+}(\mathbf{v},\sigma\mathbf{v})\right)
=\operatorname{Re}\left(  \bar{\sigma}\left\Vert \mathbf{v}\right\Vert
_{\mathbf{H}\left(  \operatorname*{curl},\Omega\right)  ,k}^{2}%
+\operatorname*{i}\bar{\sigma}k\left\Vert \mathbf{v}_{T}\right\Vert
_{\mathbf{L}^{2}\left(  \Gamma\right)  }^{2}\right)  =\frac{\sqrt{2}}%
{2}\left\Vert \mathbf{v}\right\Vert _{\operatorname*{imp},k}^{2}.
\]
The continuity assertion (\ref{item:lemma:apriori-with-good-sign-00}) follows
by the Cauchy-Schwarz inequality.

\emph{Proof of (\ref{item:lemma:apriori-with-good-sign-i}):} The estimate
(\ref{item:lemma:apriori-with-good-sign-i}) follows directly from a variant of
the Lax-Milgram lemma: We choose $\mathbf{v}=\mathbf{u}$ in the weak form
(\ref{Maxwellgoodsign-weak}) and estimate
\begin{align}
\frac{\sqrt{2}}{2}\left\Vert \mathbf{u}\right\Vert _{\operatorname{imp}%
,k}^{2}  &  =\operatorname{Re}A^{+}(\mathbf{u},\sigma\mathbf{u}%
)=\operatorname{Re}\left(  \left(  \mathbf{f},\sigma\mathbf{u}\right)
+\left(  \mathbf{g}_{T},\sigma\mathbf{u}_{T}\right)  _{\mathbf{L}^{2}(\Gamma
)}\right) \nonumber\\
&  \leq\left(  \Vert\mathbf{f}\Vert_{{\mathbf{X}}_{\operatorname{imp}}%
^{\prime}(\Omega),k}+\Vert\mathbf{g}_{T}\Vert_{{\mathbf{X}}%
_{\operatorname{imp}}^{\prime}(\Gamma),k}\right)  \Vert{\mathbf{u}}%
\Vert_{\operatorname{imp},k} \label{LaxMilg1}%
\end{align}
from which (\ref{eq:item:lemma:apriori-with-good-sign-iia}) follows. Estimate
(\ref{eq:item:lemma:apriori-with-good-sign-i}) is then obtained from
(\ref{eq:item:lemma:apriori-with-good-sign-iia}) and (\ref{eq:H-1div-vs-L2}),
(\ref{eq:H-1/2div-vs-L2}).

\emph{Proof of (\ref{item:lemma:apriori-with-good-sign-iibegin}):} From now
on, we assume that $\Gamma$ is sufficiently smooth. We proceed by induction on
$m\in\mathbb{N}_{0}$ and show that if the solution $\mathbf{u}\in
\mathbf{H}^{m}(\operatorname{curl},\Omega)$, then $\mathbf{u}\in
\mathbf{H}^{m+1}(\operatorname{curl},\Omega)$. Specifically, after the
preparatory Step~1, we will show $\mathbf{u}\in\mathbf{H}^{m+1}(\Omega)$ in
Step~2 and $\operatorname{curl}\mathbf{u}\in\mathbf{H}^{m+1}(\Omega)$ in
Step~3. Step~4 shows the induction hypothesis for $m=0$ including the norm
bounds. Step~5 completes the induction argument for the norm bounds.

\textbf{Step 1.} Taking the surface divergence of the boundary conditions we
get by using the differential equation
\begin{align}
-\operatorname*{i}k\operatorname{div}_{\Gamma}\mathbf{u}_{T}  &
=\operatorname{div}_{\Gamma}\mathbf{g}_{T}-\operatorname{div}_{\Gamma}\left(
\gamma_{T}\operatorname*{curl}\mathbf{u}\right)  \overset
{\text{\cite[(2.5.197)]{Nedelec01}}}{=}\operatorname{div}_{\Gamma}%
\mathbf{g}_{T}+\operatorname{curl}_{\Gamma}\operatorname*{curl}\mathbf{u}%
\nonumber\\
&  =\operatorname{div}_{\Gamma}\mathbf{g}_{T}+\langle
\operatorname*{curl}\operatorname*{curl}\mathbf{u},\mathbf{n}\rangle
=\operatorname{div}_{\Gamma}\mathbf{g}_{T}+\langle \mathbf{f}%
-k^{2}\mathbf{u},\mathbf{n}\rangle .
\label{eq:lemma:apriori-with-good-sign-12}%
\end{align}
We note that $\operatorname{div}(\mathbf{f}-k^{2}\mathbf{u})=0$ so that
\begin{equation}
\Vert\langle\mathbf{f}-k^{2}\mathbf{u},\mathbf{n}\rangle\Vert_{\mathbf{H}%
^{m-1/2}(\Gamma)}\leq C\Vert\mathbf{f}-k^{2}\mathbf{u}\Vert_{\mathbf{H}%
^{m}(\Omega)}. \label{eq:div-trace-estimate}%
\end{equation}
Inserting this in (\ref{eq:lemma:apriori-with-good-sign-12}) yields
\begin{equation}
\left\Vert \operatorname{div}_{\Gamma}\mathbf{u}_{T}\right\Vert _{\mathbf{H}%
^{m-1/2}(  \Gamma)  }\leq 
C|k|^{-1}\left[  \left\Vert
\operatorname{div}_{\Gamma}\mathbf{g}_{T}\right\Vert _{\mathbf{H}%
^{m-1/2}(  \Gamma)  }+\left\Vert \mathbf{f}\right\Vert
_{\mathbf{H}^{m}(\Omega)}+\left\vert k\right\vert ^{2}\left\Vert
\mathbf{u}\right\Vert _{\mathbf{H}^{m}\left(  \Omega\right)  }\right]  .
\label{difuT_divgT}%
\end{equation}
It will be convenient to abbreviate
\begin{equation}
\begin{split} 
R_{m}:=|k|^{-1}\Bigl[  \left\Vert \operatorname{div}_{\Gamma}\mathbf{g}%
_{T}\right\Vert _{\mathbf{H}^{m-1/2}\left(  \Gamma\right)  }+\left\Vert
\mathbf{f}\right\Vert _{\mathbf{H}^{m}(\Omega)}+|k|^{-1}\left\Vert
\mathbf{f}\right\Vert _{\mathbf{H}^{m}(\operatorname{div},\Omega)} 
\\
 + \left\vert
k\right\vert ^{2}\left\Vert \mathbf{u}\right\Vert _{\mathbf{H}^{m}\left(
\Omega\right)  }+|k|\left\Vert \mathbf{u}\right\Vert _{\mathbf{H}^{m}\left(
\operatorname{curl},\Omega\right)  }\Bigr]  . \label{eq:Rell}%
\end{split} 
\end{equation}

\textbf{Step 2. } ($\mathbf{H}^{m+1}(\Omega)$-estimate) With the aid of
Lemma~\ref{lemma:helmholtz-a-la-schoeberl}%
(\ref{item:lemma:helmholtz-a-la-schoeberl-ii}), we write $\mathbf{u}%
=\nabla\varphi+\mathbf{z}$ with $\varphi\in H^{m+1}(\Omega)$ and
$\mathbf{z}\in\mathbf{H}^{m+1}(\Omega)$ and
\begin{align}
\Vert\varphi\Vert_{\mathbf{H}^{m+1}(\Omega)}+\Vert\mathbf{z}\Vert
_{\mathbf{H}^{m}(\Omega)}  &  \leq C\Vert\mathbf{u}\Vert_{\mathbf{H}%
^{m}(\Omega)}\overset{%
 \text{(\ref{eq:Rell})} 
}{\leq}C|k|^{-1}R_{m},\label{eq:good-sign-100}\\
\Vert\mathbf{z}\Vert_{\mathbf{H}^{m+1}(\Omega)}  &  \leq C\Vert\mathbf{u}%
\Vert_{\mathbf{H}^{m}(\operatorname{curl},\Omega)}\overset{%
 \text{(\ref{eq:Rell})}  
}{\leq}CR_{m}. \label{eq:good-sign-101}%
\end{align}
\textbf{Step 2a:} We bound
\begin{equation}
\left\Vert \operatorname*{div}\nolimits_{\Gamma}\mathbf{z}_{T}\right\Vert
_{\mathbf{H}^{m-1/2}\left(  \Gamma\right)  }\leq C\Vert\mathbf{z}_{T}%
\Vert_{\mathbf{H}^{m+1/2}(\Gamma)}\leq C\Vert\mathbf{z}\Vert_{\mathbf{H}%
^{m+1}(\Omega)}\overset{(\ref{eq:good-sign-101})}{\leq}CR_{m}.
\label{traceestzt}%
\end{equation}

\textbf{Step 2b:} Applying $\operatorname*{div}_{\Gamma}\Pi_{T}$ to the
decomposition of $\mathbf{u}$ leads to%
\begin{equation}
\Delta_{\Gamma}\varphi|_{\Gamma}=\operatorname{div}_{\Gamma}\nabla_{\Gamma
}\varphi=\operatorname{div}_{\Gamma}\mathbf{u}_{T}-\operatorname{div}_{\Gamma
}\mathbf{z}_{T} \label{eq:lemma:apriori-with-good-sign-100}%
\end{equation}
with
\begin{equation}
\left\Vert \operatorname{div}_{\Gamma}\mathbf{u}_{T}-\operatorname{div}%
_{\Gamma}\mathbf{z}_{T}\right\Vert _{\mathbf{H}^{m-1/2}\left(  \Gamma\right)
}\overset{(\ref{difuT_divgT}),(\ref{traceestzt}),(\ref{eq:Rell})}{\leq}CR_{m}.
\label{eq:apriori-with-good-sign-220}%
\end{equation}
Together with (\ref{eq:lemma:apriori-with-good-sign-100}), we infer
$\varphi|_{\Gamma}\in H^{m+3/2}(\Gamma)$. Since $\Gamma$ is connected,
$\varphi|_{\Gamma}$ is unique up to a constant. We select this constant such
that $\varphi|_{\Gamma}$ has zero mean. Elliptic regularity implies
\[
\left\Vert \varphi\right\Vert _{H^{3/2+m}\left(  \Gamma\right)  }\leq
C\left\Vert \operatorname{div}_{\Gamma}\mathbf{u}_{T}-\operatorname{div}%
_{\Gamma}\mathbf{z}_{T}\right\Vert _{H^{-1/2+m}\left(  \Gamma\right)
}\overset{(\ref{eq:apriori-with-good-sign-220})}{\leq}CR_{m}.
\]
The function $\varphi$ satisfies the following Dirichlet problem:
\[
\Delta\varphi=\operatorname{div}\mathbf{u}-\operatorname{div}\mathbf{z}%
=k^{-2}\operatorname{div}\mathbf{f}-\operatorname{div}\mathbf{z}\in
H^{m}(\Omega),\qquad\varphi|_{\Gamma}\in H^{3/2+m}(  \Gamma), 
\]
from which we get by elliptic regularity
\[
\left\Vert \varphi\right\Vert _{H^{2+m}(  \Omega)  }\leq C\left(
\Vert\varphi\Vert_{H^{3/2+m}(  \Gamma)  }+\left\vert k\right\vert
^{-2}\Vert\operatorname{div}\mathbf{f}\Vert_{\mathbf{H}^{m}(\Omega
)}+\left\Vert \operatorname{div}\mathbf{z}\right\Vert _{H^{m}(\Omega)}\right)
\leq CR_{m}.
\]
We conclude
\begin{equation}
\Vert\mathbf{u}\Vert_{\mathbf{H}^{m+1}(\Omega)}\leq CR_{m}. \label{eq:uinHm+1}%
\end{equation}

\textbf{Step 3.} ($\mathbf{H}^{m+1}(  \operatorname*{curl},\Omega)
$-estimate) We set $\mathbf{w}:=\operatorname*{curl}\mathbf{u}$. Since
$\mathbf{u}\in\mathbf{H}^{m+1}(  \Omega)  $ (cf.\ (\ref{eq:uinHm+1}%
)) we know that $\mathbf{w}\in\mathbf{H}^{m}(  \Omega)  $. As in
Step~2 we write $\mathbf{w}=\nabla\tilde{\varphi}+\mathbf{\tilde{z}}$ and
obtain%
\begin{align*}
\Vert\tilde{\varphi}\Vert_{H^{m+1}(\Omega)}+\Vert\mathbf{\tilde{z}}%
\Vert_{\mathbf{H}^{m}(\Omega)} &  \leq C\Vert\mathbf{w}\Vert_{\mathbf{H}%
^{m}(\Omega)}\leq C\Vert\mathbf{u}\Vert_{\mathbf{H}^{m+1}(\Omega)}%
\overset{(\ref{eq:uinHm+1})}{\leq}CR_{m},\\
\Vert\mathbf{\tilde{z}}\Vert_{\mathbf{H}^{m+1}(\Omega)} &  \leq C\Vert
\mathbf{w}\Vert_{\mathbf{H}^{m}(\operatorname{curl},\Omega)}\leq C\left(
\Vert\operatorname{curl}\mathbf{w}\Vert_{\mathbf{H}^{m}(\Omega)}%
+\Vert\mathbf{w}\Vert_{\mathbf{H}^{m}(\Omega)}\right)  \\
&  \leq C\left(  \Vert\operatorname{curl}\operatorname{curl}\mathbf{u}%
\Vert_{\mathbf{H}^{m}(\Omega)}+\Vert\mathbf{u}\Vert_{\mathbf{H}^{m+1}(\Omega
)}\right)  \\
& \leq C\left(  \Vert\mathbf{f}-k^{2}\mathbf{u}\Vert_{\mathbf{H}%
^{m}(\Omega)}+R_{m}\right)  .
\end{align*}
To estimate the norm of  $\tilde{\varphi}$, we employ the boundary condition satisfied by $\mathbf{u}$,
i.e.,%
\[
\nabla_{\Gamma}\tilde{\varphi}=\mathbf{n}\times\gamma_{T}\nabla\tilde{\varphi
}=\mathbf{n}\times\left(  \gamma_{T}\mathbf{w}-\gamma_{T}\mathbf{\tilde{z}%
}\right)  =\mathbf{n}\times\left(  \mathbf{g}_{T}+\operatorname*{i}%
k\mathbf{u}_{T}-\gamma_{T}\mathbf{\tilde{z}}\right)  .
\]
In view of $\mathbf{g}_{T}\in\mathbf{H}_{T}^{m+1/2}\left(  \Gamma\right)  $,
this implies $\tilde{\varphi}|_{\Gamma}\in H^{m+3/2}\left(  \Gamma\right)  $
with%
\begin{align*}
\left\Vert \nabla_{\Gamma}\tilde{\varphi}\right\Vert _{\mathbf{H}%
^{m+1/2}(  \Gamma)  } &  \leq C\left(  \Vert\mathbf{g}_{T}%
\Vert_{\mathbf{H}^{m+1/2}(  \Gamma)  }+\left\vert k\right\vert
\left\Vert \mathbf{u}_{T}\right\Vert _{\mathbf{H}^{m+1/2}(
\Gamma)  }+\left\Vert \gamma_{T}\mathbf{\tilde{z}}\right\Vert
_{\mathbf{H}^{m+1/2}(  \Gamma)  }\right)  \\
&  \leq C\left(  \Vert\mathbf{g}_{T}\Vert_{\mathbf{H}^{m+1/2}(\Gamma
)}+|k|\Vert\mathbf{u}\Vert_{\mathbf{H}^{m+1}(\Omega)}+\Vert\mathbf{\tilde{z}%
}\Vert_{\mathbf{H}^{m+1}(\Omega)}\right)  \\
&  \leq C\left(  \Vert\mathbf{g}_{T}\Vert_{\mathbf{H}^{m+1/2}(\Gamma
)}+|k|R_{m}\right)  .
\end{align*}
The function $\tilde{\varphi}$ solves the Dirichlet problem%
\begin{equation}
\Delta\tilde{\varphi}=\operatorname{div}(\mathbf{w}-\mathbf{\tilde{z}%
})=-\operatorname{div}\mathbf{\tilde{z}}\quad\text{in }\Omega,\qquad
\tilde{\varphi}|_{\Gamma}\in H^{m+3/2}\left(  \Gamma\right)
.\label{eq:tildevarphi-dirichlet-problem}%
\end{equation}
Since $\tilde{\varphi}|_{\Gamma}$ is determined up to a constant, we may
assume that $\tilde{\varphi}|_{\Gamma}$ has vanishing mean. Elliptic
regularity theory for (\ref{eq:tildevarphi-dirichlet-problem}) tells us that%
\begin{align*}
\left\Vert \nabla\tilde{\varphi}\right\Vert _{\mathbf{H}^{m+1}(\Omega)}
& \leq
C\left(  \left\Vert \nabla_{\Gamma}\tilde{\varphi}\right\Vert _{\mathbf{H}%
^{m+1/2}(\Gamma)}^{2}+\left\Vert \operatorname*{div}\mathbf{\tilde{z}%
}\right\Vert _{H^{m}(\Omega)}^{2}\right)  ^{1/2}
\\
& \leq C\left(  \Vert
\mathbf{g}_{T}\Vert_{\mathbf{H}^{m+1/2}(\Gamma)}+|k|R_{m}\right)  .
\end{align*}
We obtain $\mathbf{w}\in\mathbf{H}^{m+1}(\Omega)$ with%
\begin{align}
\nonumber 
\left\vert \operatorname*{curl}\mathbf{u}\right\vert _{\mathbf{H}^{m+1}%
(\Omega)}& =\left\vert \mathbf{w}\right\vert _{\mathbf{H}^{m+1}(\Omega)}%
\leq C\left(  \left\vert \nabla\tilde{\varphi}\right\vert _{\mathbf{H}%
^{m+1}(\Omega)}+\left\vert \nabla\mathbf{\tilde{z}}\right\vert
_{\mathbf{H}^{m}(\Omega)}\right)  
\\
& \leq C\left(  \Vert\mathbf{g}_{T}%
\Vert_{\mathbf{H}^{m+1/2}(\Gamma)}+|k|R_{m}\right)  
.\label{eq:uinHm+1curl}%
\end{align}

\textbf{Step 4:} We ascertain the bounds
(\ref{eq:lemma:apriori-with-good-sign-25b}),
(\ref{eq:lemma:apriori-with-good-sign-25}) for $m=0$. We have%
\begin{align*}
\Vert\mathbf{u}\Vert_{\operatorname{imp},k}  &  \overset
{\text{(\ref{eq:item:lemma:apriori-with-good-sign-iia}),
(\ref{eq:H-1div-vs-L2}), (\ref{eq:H-1/2div})}}{\leq}C\left\vert k\right\vert
^{-2}\left(  \Vert\mathbf{f}\Vert_{\mathbf{H}(\operatorname{div},\Omega
),k}+\Vert\mathbf{g}_{T}\Vert_{\mathbf{H}^{-1/2}(\operatorname{div}_{\Gamma
},\Gamma),k}\right)  ,\\
\Vert\operatorname{div}_{\Gamma}\mathbf{g}_{T}\Vert_{\mathbf{H}^{-1/2}%
(\Gamma)}  &  \overset{(\ref{eq:H-1/2div})}{\leq}|k|^{-1}\Vert\mathbf{g}%
_{T}\Vert_{\mathbf{H}^{-1/2}\left(  \operatorname{div}_{\Gamma},\Gamma\right)
,k}.
\end{align*}

This implies for $R_{0}$ from (\ref{eq:Rell})%
\begin{equation}
R_{0}\leq C\left\vert k\right\vert ^{-2}\left(  \Vert\mathbf{f}\Vert
_{\mathbf{H}(\operatorname{div},\Omega),k}+\Vert\mathbf{g}_{T}\Vert
_{\mathbf{H}^{-1/2}(\operatorname{div}_{\Gamma},\Gamma),k}\right)
\label{eq:R0}%
\end{equation}
and in turn from (\ref{eq:uinHm+1})
\begin{align*}
\Vert\mathbf{u}\Vert_{\mathbf{H}^{1}(\Omega),k}  &  \leq C\left(
\Vert\mathbf{u}\Vert_{\mathbf{L}^{2}(\Omega)}+|k|^{-1}\Vert\mathbf{u}%
\Vert_{\mathbf{H}^{1}(\Omega)}\right)  \overset{(\ref{eq:uinHm+1}%
),(\ref{eq:Rell})}{\leq}C|k|^{-1}R_{0}\\
&  \leq C|k|^{-3}\left(  \Vert\mathbf{f}\Vert_{\mathbf{H}(\operatorname{div}%
,\Omega),k}+\Vert\mathbf{g}_{T}\Vert_{\mathbf{H}^{-1/2}(\operatorname{div}%
_{\Gamma},\Gamma),k}\right)  ,
\end{align*}
which is formula (\ref{eq:lemma:apriori-with-good-sign-25b}) for $m=0$. Next,
\begin{align*}
\Vert\mathbf{u}\Vert_{\mathbf{H}^{1}(\operatorname{curl},\Omega),k}  &  \leq
C\left(  |k|^{-1}\Vert\operatorname{curl}\mathbf{u}\Vert_{\mathbf{H}%
^{1}(\Omega)}+|k|\Vert\mathbf{u}\Vert_{\mathbf{H}^{1}(\Omega),k}\right) \\
& 
\overset{(\ref{eq:uinHm+1curl})}{\leq}C\left(  |k|^{-1}\Vert\mathbf{g}%
_{T}\Vert_{\mathbf{H}^{1/2}(\Gamma)}+R_{0}\right) \\
&  \overset{(\ref{eq:R0}), (\ref{eq:H-1div-vs-L2-b})}{\leq}C\left(  |k|^{-1}\Vert\mathbf{g}_{T}%
\Vert_{\mathbf{H}^{1/2}(\Gamma),k}+\left\vert k\right\vert ^{-2}%
\Vert\mathbf{f}\Vert_{\mathbf{H}(\operatorname{div},\Omega),k}\right)  ,
\end{align*}
which is formula (\ref{eq:lemma:apriori-with-good-sign-25}) for $m=0$.

We now assume that the estimates (\ref{eq:lemma:apriori-with-good-sign-25b}),
(\ref{eq:lemma:apriori-with-good-sign-25}) holds up to $m$ and show that they
hold for $m+1$. Introduce the abbreviations
\begin{align*}
T_{1}(m)  &  :=\Vert\mathbf{f}\Vert_{\mathbf{H}^{m}(\operatorname{div}%
,\Omega),k}+\Vert\mathbf{g}_{T}\Vert_{\mathbf{H}^{m-1/2}(\operatorname{div}%
_{\Gamma},\Gamma),k},\\
T_{2}(m)  &  :=\Vert\mathbf{f}\Vert_{\mathbf{H}^{m}(\operatorname{div}%
,\Omega),k}+\left\vert k\right\vert \Vert\mathbf{g}_{T}\Vert_{\mathbf{H}%
^{m+1/2}(\Gamma),k}.
\end{align*}
It is easy to verify that (using (\ref{eq:H-1div-vs-L2-b}) for the case $m = 0$)
\begin{equation}
T_{1}\left(  m\right)  \leq CT_{2}\left(  m\right)  \leq CT_{1}\left(
m+1\right)  . \label{T1T2est}%
\end{equation}
By the induction hypothesis, we have
\begin{align}
& |k|\Vert\mathbf{u}\Vert_{\mathbf{H}^{m+1}(\Omega)}+\Vert\mathbf{u}%
\Vert_{\mathbf{H}^{m+1}(\operatorname{curl},\Omega)}    \leq C|k|^{m+1}%
\Vert\mathbf{u}\Vert_{\mathbf{H}^{m+1}(\operatorname{curl},\Omega
),k}\nonumber\\
& \qquad \qquad  \stackrel{\text{Ind.\ hyp.}}{\leq} C|k|^{m-1}T_{2}(m)\stackrel{(\ref{T1T2est})}{\leq} C|k|^{m-1}T_{1}(m+1). \label{kuT2}%
\end{align}
Hence,
\begin{align}
\nonumber 
|k|^{-(m+2)}R_{m+1}  &  = |k|^{-(m+2)}\left(  |k|^{-1}\Vert
\operatorname{div}_{\Gamma}\mathbf{g}_{T}\Vert_{\mathbf{H}^{m+1/2}(\Gamma
)}
+|k|^{-1}\Vert\mathbf{f}\Vert_{\mathbf{H}^{m+1}(\Omega)}\right. \\
&  \quad\left.  
+|k|^{-2}\Vert\mathbf{f}\Vert_{\mathbf{H}^{m+1}%
(\operatorname{div},\Omega)}+|k|\Vert\mathbf{u}\Vert_{\mathbf{H}^{m+1}%
(\Omega)}+\Vert\mathbf{u}\Vert_{\mathbf{H}^{m+1}(\operatorname{curl},\Omega
)}\right) \nonumber\\
\nonumber 
&  \stackrel{(\ref{kuT2})}{\leq} C|k|^{-3}\left(  \Vert\operatorname{div}_{\Gamma}\mathbf{g}_{T}%
\Vert_{\mathbf{H}^{m+1/2}(\Gamma),k}+\Vert\mathbf{f}\Vert_{\mathbf{H}%
^{m+1}(\operatorname{div},\Omega),k}+T_{1}(m+1)\right)  \\
& \leq C|k|^{-3} T_{1}(m+1)
\label{estkRm+1}
\end{align}
and therefore {by the induction hypothesis and (\ref{eq:uinHm+1})}
\begin{align}
\Vert\mathbf{u}\Vert_{\mathbf{H}^{m+2}(\Omega),k}  &  \leq C\left(
\Vert\mathbf{u}\Vert_{\mathbf{H}^{m+1}(\Omega),k}+|k|^{-(m+2)}|\mathbf{u}%
|_{\mathbf{H}^{m+2}(\Omega)}\right) \nonumber\\
&  \overset{\text{ind.\ hyp., (\ref{eq:uinHm+1})}}{\leq}{C\left(
|k|^{-3}T_{1}(m)+|k|^{-(m+2)}R_{m+1}\right)  }\nonumber\\
&  \overset{(\ref{estkRm+1})}{\leq}C|k|^{-3}T_{1}(m+1), \label{esthHm+2_k}%
\end{align}
which completes the induction step for formula
(\ref{eq:lemma:apriori-with-good-sign-25b}).

Again from the definition of $R_{m+1}$, the induction hypothesis, and
(\ref{T1T2est}), we have
\begin{align*}
|k|^{-(m+2)}R_{m+1}& \leq C\left[  |k|^{-2}\Vert\mathbf{g}_{T}\Vert
_{\mathbf{H}^{m+3/2}(\Gamma),k}+|k|^{-3}\Vert\mathbf{f}\Vert_{\mathbf{H}%
^{m+1}(\operatorname{div},\Omega),k}+|k|^{-3}T_{2}(m)\right]  \\
& \leq
C|k|^{-3}T_{2}(m+1).
\end{align*}
The combination of this with (\ref{esthHm+2_k}) and (\ref{T1T2est}) leads to%
\begin{align*}
& \Vert\mathbf{u}\Vert_{\mathbf{H}^{m+2}(\operatorname{curl},\Omega),k}    \leq
C\left(  \left\vert k\right\vert \Vert\mathbf{u}\Vert_{\mathbf{H}^{m+2}%
(\Omega),k}+|k|^{-(m+2)}|\operatorname{curl}\mathbf{u}|_{\mathbf{H}%
^{m+2}(\Omega)}\right) \\
&  \qquad \overset{(\ref{esthHm+2_k}),(\ref{eq:uinHm+1curl})}{\leq}C\left(
\left\vert k\right\vert ^{-2}T_{2}(m+1)+|k|^{-(m+2)}\left(  \Vert
\mathbf{g}_{T}\Vert_{\mathbf{H}^{m+3/2}(\Gamma)}+|k|R_{m+1}\right)  \right) \\
&  \qquad \leq C\left(  \left\vert k\right\vert ^{-2}T_{2}(m+1)+|k|^{-1}%
\Vert\mathbf{g}_{T}\Vert_{\mathbf{H}^{m+3/2}(\Gamma),k}+\left\vert
k\right\vert ^{-2}T_{2}(m+1)\right)  \\
& \qquad \leq C\left\vert k\right\vert ^{-2}%
T_{2}(m+1),
\end{align*}
which completes the induction argument for
(\ref{eq:lemma:apriori-with-good-sign-25}).%
\end{proof}


\section{Regularity theory for Maxwell's equations\label{SecHRP}}

In this section, we collect regularity assertions for the Maxwell model
problem (\ref{MWEq}). In particular, the case of analytic data studied in
Section~\ref{sec:analytic-regularity} will be a building block for the
regularity by decomposition studied in Section~\ref{SecReg}.

\subsection{Finite regularity theory}

\label{sec:finite-regularity}
The difference between Maxwell's equations with the ``good'' sign and the
time-harmonic Maxwell equations lies in a lower order term. Therefore, higher
regularity statements for the solution of Maxwell's equations can be inferred
from those for with the ``good'' sign, i.e., from
Theorem~\ref{lemma:apriori-with-good-sign}. The following result makes this precise.

\begin{lemma}
\label{lemma:MW-regularity}Let $\Omega$ be a bounded Lipschitz domain with
simply connected, sufficiently smooth boundary $\Gamma$. Let $m\in
\mathbb{N}_{0}$. Then there is $C>0$ (depending only on $m$ and $\Omega$) such
that for $\mathbf{f}\in\mathbf{H}^{m}(\operatorname{div},\Omega)$,
$\mathbf{g}_{T}\in\mathbf{H}_{T}^{m+1/2}(\Gamma)$ the solution $\mathbf{u}$ of
(\ref{MWEq}) (for $\mathbf{j}:=\mathbf{f}$) satisfies $\mathbf{u}\in
\mathbf{H}^{m+1}(\operatorname{curl},\Omega)$ and%
\begin{align}
\Vert\mathbf{u}\Vert_{\mathbf{H}^{m+1}(\Omega),k}  & \! \leq C\left[ |k|^{-3}\left(
\Vert\mathbf{f}\Vert_{\mathbf{H}^{m}(\operatorname{div},\Omega),k}%
+\Vert\mathbf{g}_{T}\Vert_{\mathbf{H}^{m-1/2}(\operatorname{div}_{\Gamma
},\Gamma),k}\right)  +\Vert\mathbf{u}\Vert_{\mathbf{L}^{2}(\Omega
)}\!\right],\label{eq:lemma:MW-regularity-10}\\
\Vert\mathbf{u}\Vert_{\mathbf{H}^{m+1}(\operatorname{curl},\Omega),k}  &  \leq\!
C\left[ \left\vert k\right\vert ^{-2}\left(  \Vert\mathbf{f}\Vert_{\mathbf{H}%
^{m}(\operatorname{div},\Omega),k}+|k|\Vert\mathbf{g}_{T}\Vert_{\mathbf{H}%
^{m+1/2}(\Gamma),k}\right)  +\left\vert k\right\vert \Vert\mathbf{u}%
\Vert_{\mathbf{L}^{2}(\Omega)}\!\right].\! \label{eq:lemma:MW-regularity-20}%
\end{align}
If assumption~(\ref{AssumptionAlgGrowth}) holds, then $\Vert\mathbf{u}%
\Vert_{\mathbf{L}^{2}(\Omega)}\leq C|k|^{\theta-1}\left(  \Vert\mathbf{f}%
\Vert_{\mathbf{L}^{2}(\Omega)}+\Vert\mathbf{g}_{T}\Vert_{\mathbf{L}^{2}%
(\Gamma)}\right)  $. In particular,%
\begin{align}
\Vert\mathbf{u}\Vert_{\mathbf{H}^{1}(\operatorname{curl},\Omega),k}  &  \leq
C\left\vert k\right\vert ^{-2}\Bigl\{\Vert\operatorname{div}\mathbf{f}%
\Vert_{\mathbf{L}^{2}(\Omega)}+|k|^{%
 \theta+2 }%
\Vert\mathbf{f}\Vert_{\mathbf{L}^{2}(\Omega)}
\label{eq:lemma:MW-regularity-30}\\
&  \qquad\qquad\mbox{}+|k|\Vert\mathbf{g}_{T}\Vert_{\mathbf{H}^{1/2}(\Gamma
)}+|k|^{%
 \theta+2 }%
\Vert\mathbf{g}_{T}\Vert_{\mathbf{L}^{2}(\Gamma)}\Bigr\}.\nonumber
\end{align}

\end{lemma}

\begin{proof}
The weak solution $\mathbf{u}$ of (\ref{MWEq}) exists by
Proposition~\ref{lemma:apriori-homogeneous-rhs} and depends continuously on
the data. In particular, $\mathbf{u}\in\mathbf{L}^{2}(\Omega)$. From the
equation $\mathcal{L}_{\Omega,k}\mathbf{u}=\mathbf{f}$, we have $-k^{2}%
\operatorname{div}\mathbf{u}=\operatorname{div}\mathbf{f}$ so that
$\mathbf{u}\in\mathbf{H}(\operatorname{div},\Omega)$. The function
$\mathbf{u}$ solves%
\begin{equation}
\mathcal{L}_{\Omega,\operatorname*{i}k}\mathbf{u}=\mathbf{f}+2k^{2}%
\mathbf{u},\qquad\mathcal{B}_{ \Gamma ,k}\mathbf{u}=\mathbf{g}_{T}.
\end{equation}
It is easy to see that Theorem~\ref{lemma:apriori-with-good-sign} is
inductively applicable. We get
\begin{align}
& \Vert\mathbf{u}\Vert_{\mathbf{H}^{m+1}(\Omega),k}    \leq C\left(
|k|^{-3}\Vert\mathbf{f}+2k^{2}\mathbf{u}\Vert_{\mathbf{H}^{m}%
(\operatorname{div},\Omega),k}+|k|^{-3}\Vert\mathbf{g}_{T}\Vert_{\mathbf{H}%
^{m-1/2}(\operatorname{div}_{\Gamma},\Gamma),k}\right) \nonumber\\
& \quad  \leq C\left(  |k|^{-3}\Vert\mathbf{f}\Vert_{\mathbf{H}^{m}%
(\operatorname{div},\Omega),k}+|k|^{-3}\Vert\mathbf{g}_{T}\Vert_{\mathbf{H}%
^{m-1/2}(\operatorname{div}_{\Gamma},\Gamma),k}+\Vert\mathbf{u}\Vert
_{\mathbf{H}^{m}(\Omega),k}\right)  . \label{eq:induction-u}%
\end{align}
We may successively insert (\ref{eq:induction-u}) into itself to
arrive at (\ref{eq:lemma:MW-regularity-10}). The statement
(\ref{eq:lemma:MW-regularity-20}) follows from
(\ref{eq:lemma:MW-regularity-10}) and
Theorem~\ref{lemma:apriori-with-good-sign} and the observation $\Vert
\mathbf{g}_{T}\Vert_{\mathbf{H}^{m-1/2}(\operatorname{div}_{\Gamma},\Gamma
),k}\leq C|k|\Vert\mathbf{g}_{T}\Vert_{\mathbf{H}^{m+1/2}(\Gamma),k}$.
\end{proof}


\subsection{Analytic regularity theory}

\label{sec:analytic-regularity}
In this section, we consider the Maxwell problem (\ref{MWEq}), i.e.,
\begin{equation}
\mathcal{L}_{\Omega,k}\mathbf{E}=\mathbf{f}\text{ in }\Omega,\qquad
\qquad\mathcal{B}_{ \Gamma ,k}\mathbf{E}=\mathbf{g}_{T}\text{ on }\Gamma\label{eq:problem-analytic-data}%
\end{equation}
with \emph{analytic} data $\mathbf{f}$ and $\mathbf{g}_{T}$ and analytic
boundary $\Gamma$. We show in Theorem~\ref{ThmAnaRegSum} that the solution is
analytic, making  the dependence on $k$ explicit.
\ifarxiv In Appendix~\ref{AppAnalyticity} we generalize the theory in
\cite{nicaise-tomezyk19,tomezyk19} to the case of inhomogeneous boundary data.
\else In \cite[Appendix~A]{MelenkSauterMaxwell_II} we generalize the theory in
\cite{nicaise-tomezyk19,tomezyk19} to the case of inhomogeneous boundary data.
\fi The key idea there is to reformulate the problem
(\ref{eq:problem-analytic-data}) as an elliptic system and then to apply the
regularity theory for elliptic systems with analytic data to this problem (see
\cite{CoDaNi}). Here, we summarize the main results.

Problem (\ref{eq:problem-analytic-data}) can be formulated as an elliptic
system for $\mathbf{U}=\left(  \mathbf{E},\mathbf{H}\right)  $, where
$\mathbf{E}$ is the electric and $\mathbf{H}:=-\frac{\operatorname*{i}}%
{k}\operatorname*{curl}\mathbf{E}$ the magnetic field \ifarxiv
(see Appendix~\ref{AppAnalyticity}):\else
: \fi
\begin{equation}%
\begin{array}
[c]{cclc}%
L\left(  \mathbf{U}\right)  := & \left(
\begin{array}
[c]{c}%
\operatorname*{curl}\operatorname*{curl}\mathbf{E}-\nabla\operatorname*{div}%
\mathbf{E}\\
\operatorname*{curl}\operatorname*{curl}\mathbf{H}-\nabla\operatorname*{div}%
\mathbf{H}%
\end{array}
\right)  & =\mathbf{F}+k^{2}\mathbf{U} & \text{in }\Omega,\\
T\left(  \mathbf{U}\right)  := & \mathbf{H}\times\mathbf{n}-\mathbf{E}_{T} &
=-\frac{\operatorname*{i}}{k}\mathbf{g}_{T} & \text{on }\Gamma,\\
B\left(  \mathbf{U}\right)  := & \left(
\begin{array}
[c]{c}%
\operatorname*{div}\mathbf{E}\\
\operatorname*{div}\mathbf{H}\\
\gamma_{T}\operatorname*{curl}\mathbf{H}+\left(  \operatorname*{curl}%
\mathbf{E}\right)  _{T}%
\end{array}
\right)  & =kG\mathbf{U}+\mathbf{G}_{\Gamma} & \text{on }\Gamma
\end{array}
\label{MwEllSys}%
\end{equation}
for
\[
\mathbf{F}:=\left(
\begin{array}
[c]{c}%
\mathbf{f}+\frac{1}{k^{2}}\nabla\operatorname*{div}\mathbf{f}\\
-\frac{\operatorname*{i}}{k}\operatorname*{curl}\mathbf{f}%
\end{array}
\right)  ,\text{\quad}G\mathbf{U}:=\left(
\begin{array}
[c]{c}%
0\\
0\\
\operatorname*{i}\left(  \mathbf{H}_{T}-\gamma_{T}\mathbf{E}\right)
\end{array}
\right)  \text{,\quad}\mathbf{G}_{\Gamma}:=\left(
\begin{array}
[c]{c}%
-\frac{1}{k^{2}}\left.  \left(  \operatorname*{div}\mathbf{f}\right)
\right\vert _{\Gamma}\\
0\\
-\frac{\operatorname*{i}}{k}\gamma_{T}\mathbf{f}%
\end{array}
\right)  .
\]

\ifarxiv
In Appendix~\ref{AppAnalyticity}, \else
In \cite[Appendix~A]{MelenkSauterMaxwell_II} \fi
we show that this system is elliptic in the sense of \cite{CoDaNi}. For the
special case $\mathbf{G}_{\Gamma}=\mathbf{0}$ and $\mathbf{g}_{T}=\mathbf{0}$,
the analytic regularity theory for this problem has been developed in
\cite{nicaise-tomezyk19,tomezyk19}. The following Theorem~\ref{ThmAnaRegSum}
generalizes their result to the case of inhomogeneous boundary data
$\mathbf{g}_{T}$, $\mathbf{G}_{\Gamma}$. To describe the analyticity of
$\mathbf{g}_{T}$ and $\mathbf{G}_{\Gamma}$, we assume that these functions are
restrictions of analytic functions $\mathbf{g}^{\ast}$ and $\mathbf{G}^{\ast}$
on an open neighborhood $\mathcal{U}_{\Gamma}$ of $\Gamma$ and satisfy
$\mathbf{g}_{T}=\gamma\mathbf{g}^{\ast}$ and $\mathbf{G}_{\Gamma}%
=\gamma\mathbf{G}^{\ast}$ for the standard trace operator $\gamma$ \ifarxiv
(see (\ref{eq:A1a}), (\ref{eq:A1b})). \else
, i.e., the restriction to $\Gamma$. \fi
We write $\mathbf{g}_{T}\in\mathcal{A}(C_{\mathbf{g}},\lambda_{\mathbf{g}%
},\mathcal{U}_{\Gamma}\cap\Omega)$ if{ }$\mathbf{g}^{\ast}\in\mathcal{A}%
(C_{\mathbf{g}},\lambda_{\mathbf{g}},\mathcal{U}_{\Gamma}\cap\Omega)$.

\begin{theorem}
\label{ThmAnaRegSum}Let $\Omega\subset{\mathbb{R}}^{3}$ be a bounded Lipschitz
domain with a simply connected, analytic boundary. Let $\mathcal{U}_{\Gamma}$
be an open neighborhood of $\Gamma$. Let $\mathbf{f}\in\mathcal{A}%
(C_{\mathbf{f}},\lambda_{\mathbf{f}},\Omega)$ and $\mathbf{g}_{T}%
\in\mathcal{A}(C_{\mathbf{g}},\lambda_{\mathbf{g}},\mathcal{U}_{\Gamma}%
\cap\Omega)$. Then there are constants $B$, $C>0$ (depending only on $\Omega$,
$\mathcal{U}_{\Gamma}$, and $\lambda_{\mathbf{f}}$, $\lambda_{\mathbf{g}}$)
such that the solution $\mathbf{E}$ of (\ref{eq:problem-analytic-data})
satisfies
\begin{equation}
\mathbf{E}\in\mathcal{A}(CC_{\mathbf{E}},B,\Omega), \label{eq:ThmAnaRegSum-10}%
\end{equation}
where $C_{\mathbf{E}}=C_{\mathbf{f}}\left\vert k\right\vert ^{-2}%
+C_{{\mathbf{g}}}\left\vert k\right\vert ^{-1}+\frac{1}{\left\vert
k\right\vert }\left\Vert {\mathbf{E}}\right\Vert _{\mathbf{H}^{1}\left(
\operatorname*{curl},\Omega\right)  ,k}$. If
assumption~(\ref{AssumptionAlgGrowth}) holds, then
\begin{equation}
C_{\mathbf{E}}\leq C\left(  C_{\mathbf{f}}|k|^{ \theta-1 }%
+C_{\mathbf{g}}|k|^{ \theta-1/2} 
\right)  . \label{eq:ThmAnaRegSum-20}%
\end{equation}

\end{theorem}

\begin{proof}
\ifarxiv The statement of the theorem follows from Corollary~\ref{CorA.2} and more details can be found there. 
\else 
The statement of the theorem follows from \cite[Cor.~{A.2}]{MelenkSauterMaxwell_II} and more details can be found
there. 
\fi 
The existence $\mathbf{u}\in\mathbf{X}_{\operatorname*{imp}}$ is
implied by Proposition~\ref{lemma:apriori-homogeneous-rhs}, and finite
regularity assertions for $\mathbf{E}$ are provided in
Lemma~\ref{lemma:MW-regularity}. In particular, $\mathbf{E}\in\mathbf{H}%
^{2}(\Omega)$. In turn, $\mathbf{U}=(\mathbf{E},\mathbf{H})\in\mathbf{H}%
^{1}(\operatorname{curl},\Omega)$ solves the elliptic system (\ref{MwEllSys}).
This makes \ifarxiv Theorem~\ref{thm:A1} \else\cite[Thm.~{A.1}]%
{MelenkSauterMaxwell_II} \fi applicable, which shows the corresponding result
for $\mathbf{U}$ by a boot-strapping argument and an explicit tracking of the
wavenumber $k$ to arrive at the result of \ifarxiv Corollary~\ref{CorA.2},
which reads \else\cite[Cor.~{A.2}]{MelenkSauterMaxwell_II} \fi%
\begin{equation}
|{\mathbf{E}}|_{\mathbf{H}^{p}\left(  \Omega\right)  }\leq CC_{\mathbf{E}%
}B^{p}\max(p,|k|)^{p},\quad\forall p\in\mathbb{N}_{\geq2}
\label{eq:ThmAnaRegSum-100}%
\end{equation}
with $C_{\mathbf{E}}$ as given in the statement. A direct calculation shows
$\Vert\mathbf{E}\Vert_{\mathbf{H}^{1}(\Omega)}\leq C_{\mathbf{E}}$ and
$\Vert\mathbf{E}\Vert_{\mathbf{L}^{2}(\Omega)}\leq|k|C_{\mathbf{E}}$ so that
(\ref{eq:ThmAnaRegSum-100}) also holds for $p=0$ and $p=1$. This shows
(\ref{eq:ThmAnaRegSum-10}).

The estimate (\ref{eq:ThmAnaRegSum-20}) follows from
(\ref{eq:lemma:MW-regularity-30}) of Lemma~\ref{lemma:MW-regularity} and the
definition of the analyticity classes together with the trace estimates
$\|\mathbf{g}_{T}\|_{\mathbf{H}^{1/2}(\Gamma)} \leq C C_{\mathbf{g}} |k|$ and
$\|\mathbf{g}_{T}\|_{\mathbf{L}^{2}(\Gamma)} \leq C C_{\mathbf{g}} |k|^{1/2}%
$.
\end{proof}


\section{Frequency splittings\label{SecFreqSplit}}


As in \cite{MelenkSauterMathComp, mm_stas_helm2, MPS13, MelenkStab,
MelenkHelmStab2010, MelenkSauterMaxwell_I} we analyze the regularity of 
Maxwell's equations (\ref{MWEq}) via a decomposition of the right-hand side 
into high and low frequency parts.

\subsection{Frequency splittings in $\Omega$: $H_{\mathbb{R}^{3}}$,
$L_{\mathbb{R}^{3}}$, $H_{\Omega}$, $L_{\Omega}$, ${H}_{\Omega}^{0}$,
${L}_{\Omega}^{0}$}

In order to construct the splitting, we start by recalling the definition of
the Fourier transform for {sufficiently smooth} functions with compact
support
\begin{equation}
\hat{u}( {\boldsymbol \xi})  =\mathcal{F}(  u)  (
{\boldsymbol \xi}%
)  =\left(  2\pi\right)  ^{-3/2}%
{\displaystyle\int_{\mathbb{R}^{3}}}
\operatorname{e}^{-\operatorname*{i}\langle
{\boldsymbol \xi}%
,\mathbf{x}\rangle }u(  \mathbf{x})  d\mathbf{x}\qquad\forall%
{\boldsymbol \xi}%
\in\mathbb{R}^{3} \label{eq:fourier-transform}%
\end{equation}
and the inversion formula
\[
u(  \mathbf{x})  =\mathcal{F}^{-1}(  u)(
\mathbf{x})  =\left(  2\pi\right)  ^{-3/2}%
{\displaystyle\int_{\mathbb{R}^{3}}}
\operatorname{e}^{\operatorname*{i}\langle \mathbf{x},%
{\boldsymbol \xi}%
\rangle }\hat{u}(
{\boldsymbol \xi})  d%
{\boldsymbol \xi}%
\qquad\forall\mathbf{x}\in\mathbb{R}^{3}.
\]
These formulas extend to tempered distributions and in particular to functions
in $L^{2}(\mathbb{R}^{3})$. Next, we introduce a frequency splitting for
functions in $\mathbb{R}^{3}$ that depends on $k$ and a parameter $\lambda>1$
by using the Fourier transformation. The low and high frequency part is given by%
\begin{equation}
L_{\mathbb{R}^{3}}u:=\mathcal{F}^{-1}\left(  \chi_{\lambda\left\vert
k\right\vert }\mathcal{F}(  u)  \right)  \quad\text{and\quad
}H_{\mathbb{R}^{3}}u:=\mathcal{F}^{-1}\left(  \left(  1-\chi_{\lambda
\left\vert k\right\vert }\right)  \mathcal{F}(  u)  \right)  ,
\label{deffreqslitR3}%
\end{equation}
where $\chi_{\delta}$ is the characteristic function of the open ball with
radius $\delta>0$ centered at the origin. We note the splitting $H_{\mathbb{R}%
^{3}}+L_{\mathbb{R}^{3}}=I$. By using Stein's extension operator
$\mathcal{E}_{\operatorname*{Stein}}$, \cite[Chap.~{VI}]{emstein} this
splitting induces a frequency splitting for functions in Sobolev spaces in
$\Omega$ via%
\begin{equation}
L_{\Omega}\mathbf{f}:=\left.  \left(  L_{\mathbb{R}^{3}}\mathcal{E}%
_{\operatorname*{Stein}}\mathbf{f}\right)  \right\vert _{\Omega}%
\quad\text{and\quad}H_{\Omega}\mathbf{f}:=\left.  \left(  H_{\mathbb{R}^{3}%
}\mathcal{E}_{\operatorname*{Stein}}\mathbf{f}\right)  \right\vert _{\Omega},
\label{DefLOmegaHOmega}%
\end{equation}
where, again, $L_{\Omega}\mathbf{f}+H_{\Omega}\mathbf{f}=\mathbf{f}$ in
$\Omega$.

In general, the condition $\operatorname{div}\mathbf{f}=0$ neither implies
$\operatorname*{div}L_{\Omega}\mathbf{f}=0$ nor $\operatorname*{div}H_{\Omega
}\mathbf{f}=0$. We therefore introduce another lifting (instead of
$\mathcal{E}_{\operatorname*{Stein}}$) for functions in Sobolev spaces on
$\Omega$ that passes on the divergence-free property to the lifting to the full space and
allows for alternative frequency splittings $L_{\Omega}^{0}$, $H_{\Omega}^{0}$
at the expense that $L_{\Omega}^{0}+H_{\Omega}^{0}$ is not the identity but
the \textit{identity plus a smoothing operator}.
With the operator $\mathbf{R}_{2}$ of Lemma~\ref{LemRs}, which has been
constructed in \cite{costabel-mcintosh10}, we set%
\begin{equation}
H_{\Omega}^{0}\mathbf{f}:=\operatorname*{curl}H_{\mathbb{R}^{3}}%
\mathcal{E}_{\operatorname*{Stein}}\mathbf{R}_{2}\mathbf{f}\quad
\text{and\quad}L_{\Omega}^{0}\mathbf{f}:=\operatorname*{curl}L_{\mathbb{R}%
^{3}}\mathcal{E}_{\operatorname*{Stein}}\mathbf{R}_{2}\mathbf{f}%
\label{DefHOmega0}%
\end{equation}
and define the operator $\mathbf{S}$ by
\begin{equation}
\mathbf{S}\mathbf{f}:\mathbf{=f}-\left.  \left(  H_{\Omega}^{0}\mathbf{f}%
+L_{\Omega}^{0}\mathbf{f}\right)  \right\vert _{\Omega}\mathbf{.}\label{eq:S}%
\end{equation}
In view of (\ref{eq:S=K2}), we have for $\mathbf{f}$ with $\operatorname{div}%
\mathbf{f}=0$ that $\mathbf{S}\mathbf{f}=\mathbf{K}_{2}\mathbf{f}|_{\Omega}$
so that in particular {for all $s$, $s^{\prime}$}
\begin{equation}
\Vert\mathbf{S}\mathbf{f}\Vert_{\mathbf{H}^{s}(\Omega)}\leq C_{s,s^{\prime}%
}\Vert\mathbf{f}\Vert_{\mathbf{H}^{s^{\prime}}(\Omega)}\qquad\forall
\mathbf{f}\in\mathbf{H}^{s^{\prime}}(\Omega)\colon\operatorname{div}%
\mathbf{f}=0.\label{eq:mapping-property-S}%
\end{equation}

\subsection{Frequency splittings on $\Gamma$}

For the definition of the Hodge decompositions and frequency splittings of
this section, we recall that $\Omega$ has a simply connected, analytic boundary.

\begin{remark}
\label{rem:Yml} The Laplace-Beltrami operator $\Delta_{\Gamma}$ is
self-adjoint with respect to the $L^{2}(\Gamma)$ scalar product $\left(
\cdot,\cdot\right)  _{L^{2}(\Gamma)}$ and positive semidefinite. It admits a
countable sequence of eigenfunctions in $L^{2}(\Gamma)$ denoted by $Y_{\ell
}^{m}$ such that%
\begin{equation}
-\Delta_{\Gamma}Y_{\ell}^{m}=\lambda_{\ell}Y_{\ell}^{m}\qquad\text{\ for }%
\ell=0,1,\ldots\text{and }m\in\iota_{\ell}. \label{eigvalabsLapBelt}%
\end{equation}
Here, $\iota_{\ell}$ is a finite index set whose cardinality equals the
multiplicity of the eigenvalue $\lambda_{\ell}$, and we always assume that the
eigenvalues $\lambda_{\ell}$ are distinct and ordered increasingly. We have
$\lambda_{0}=0$ and for $\ell\geq1$, they are real and positive and accumulate
at infinity. Since we assumed that $\Gamma$ is simply connected we know that
$\lambda_{0}=0$ is a simple eigenvalue. 
\eremk
\end{remark}

According to \cite[Sec.~{5.4.1}]{Nedelec01}, any tangential field
$\mathbf{h}_{T}\in\mathbf{L}_{T}^{2}(\Gamma)$ on the bounded, simply connected
manifold $\Gamma$ admits an expansion%
\begin{equation}
\mathbf{h}_{T}=\sum_{\ell=1}^{\infty}\sum_{m\in\iota_{\ell}}
\alpha_{\ell}^{m}\nabla_{\Gamma}Y_{\ell}^{m}+\beta_{\ell}^{m}\left(
\overrightarrow{\operatorname*{curl}\nolimits_{\Gamma}}Y_{\ell}^{m} \right)  .
\label{orthosurface_exp}%
\end{equation}
The functions $\left\{  \nabla_{\Gamma}Y_{\ell}^{m},\overrightarrow
{\operatorname*{curl}\nolimits_{\Gamma}}Y_{\ell}^{m}:\ell\in\mathbb{N}_{\geq
1},\,m\in\iota_{\ell}\right\}  $ constitute an orthogonal basis of 
$\mathbf{L}_{T}^{2}(\Gamma)$ and hence the coefficients $\alpha_{\ell}^{m}$,
$\beta_{\ell}^{m}$ are uniquely determined via (\ref{orthosurface_exp}). We
set%
\begin{equation}%
\begin{array}
[c]{ll}%
\mathcal{L}_{\operatorname*{imp}}^{\nabla}\mathbf{h}_{T}:=%
{\displaystyle\sum\limits_{\ell=1}^{\infty}}
{\displaystyle\sum\limits_{m\in\iota_{\ell}}}
\alpha_{\ell}^{m}Y_{\ell}^{m}, & \qquad \mathcal{L}_{\operatorname*{imp}%
}^{\operatorname*{curl}}\mathbf{h}_{T}:=%
{\displaystyle\sum\limits_{\ell=1}^{\infty}}
{\displaystyle\sum\limits_{m\in\iota_{\ell}}}
\beta_{\ell}^{m}Y_{\ell}^{m},\\
\Pi_{\operatorname*{imp}}^{\nabla}:=\nabla_{\Gamma}\mathcal{L}%
_{\operatorname*{imp}}^{\nabla}, &\qquad  \Pi_{\operatorname*{imp}}%
^{\operatorname*{curl}}:=I-\Pi_{\operatorname*{imp}}^{\nabla}=\overrightarrow
{\operatorname*{curl}\nolimits_{\Gamma}}\mathcal{L}_{\operatorname*{imp}%
}^{\operatorname*{curl}},
\end{array}
\label{nablaliftings}%
\end{equation}
where $I$ denotes the identity operator.

\begin{remark}
\label{remk:characterize-L} $\mathcal{L}_{\operatorname*{imp}}^{\nabla
}\mathbf{h}_{T}$ and $\mathcal{L}_{\operatorname*{imp}}^{\operatorname{curl}%
}\mathbf{h}_{T}$ are characterized by%
\begin{align}
\left(  \nabla_{\Gamma}\mathcal{L}_{\operatorname{imp}}^{\nabla}\mathbf{h}%
_{T},\nabla_{\Gamma}\psi\right)  _{\mathbf{L}^{2}(  \Gamma)  }  &
=\left(  \mathbf{h}_{T},\nabla_{\Gamma}\psi\right)  _{\mathbf{L}^{2}(
\Gamma)  }\quad\forall\psi\in C^{\infty}(\Gamma),\label{defphhigT}\\
\left(  \overrightarrow{\operatorname{curl}_{\Gamma}}\mathcal{L}%
_{\operatorname{imp}}^{\operatorname{curl}}\mathbf{h}_{T},\overrightarrow
{\operatorname{curl}_{\Gamma}}\psi\right)  _{\mathbf{L}^{2}(
\Gamma)  }  &  =\left(  \mathbf{h}_{T},\overrightarrow
{\operatorname{curl}_{\Gamma}}\psi\right)  _{\mathbf{L}^{2}(
\Gamma)  }\quad\forall\psi\in C^{\infty}(\Gamma), \label{defphhicT}%
\end{align}
and the conditions $(\mathcal{L}_{\operatorname{imp}}^{\nabla}\mathbf{h}%
_{T},1)_{ L^{2}  (\Gamma)}=0$ and $(\mathcal{L}_{\operatorname{imp}}^{\operatorname{curl}%
}\mathbf{h}_{T},1)_{ L^{2}  (\Gamma)}=0$. In strong form, we have in view of $\operatorname{curl}_{\Gamma
}\overrightarrow{\operatorname{curl}_{\Gamma}}=-\Delta_{\Gamma}$ that
$\Delta_{\Gamma}\mathcal{L}_{\operatorname{imp}}^{\nabla}\mathbf{h}%
_{T}=\operatorname{div}_{\Gamma}\mathbf{h}_{T}$ and $\Delta_{\Gamma
}\mathcal{L}_{\operatorname{imp}}^{\operatorname*{curl}}\mathbf{h}%
_{T}=-\operatorname{curl}_{\Gamma}\mathbf{h}_{T}$. 
\eremk
\end{remark}

In summary, we have introduced a Hodge decomposition:%
\begin{equation}
\mathbf{h}_{T}=\Pi_{\operatorname*{imp}}^{\nabla}\mathbf{h}_{T}+\Pi
_{\operatorname*{imp}}^{\operatorname*{curl}}\mathbf{h}_{T}=\nabla_{\Gamma
}\varphi+\overrightarrow{\operatorname*{curl}\nolimits_{\Gamma}}\psi
\quad\text{for }\varphi=\mathcal{L}_{\operatorname*{imp}}^{\nabla}%
\mathbf{h}_{T}\text{ and }\psi=\mathcal{L}_{\operatorname*{imp}}%
^{\operatorname*{curl}}\mathbf{h}_{T} \label{defHodge}%
\end{equation}
(for further details see \cite[Sec.~{5.4.1}]{Nedelec01}).


Next, we introduce the harmonic extension $\mathcal{E}_{\Omega}^{\Delta
}:H^{1/2}\left(  \Gamma\right)  \rightarrow H^{1}\left(  \Omega\right)  $ of
Dirichlet boundary data defined by%
\begin{equation}
\Delta\left(  \mathcal{E}_{\Omega}^{\Delta}\varphi\right)  =0\quad\text{in
}\Omega\text{, \qquad}\left.  \mathcal{E}_{\Omega}^{\Delta}\varphi\right\vert
_{\Gamma}=\varphi.\label{LaplaceDiriProbl}%
\end{equation}
(Later, we will use that $\mathcal{E}_{\Omega}^{\Delta}$ extends to a
continuous operator $H^{s}(\Gamma)\rightarrow H^{1/2+s}(\Omega)$ for $s\geq
0$.) This allows us to define boundary frequency filters $L_{ \Gamma }$ and $H_{ \Gamma  }$ 
based on this Dirichlet lifting by
\begin{equation}
L_{ \Gamma }\varphi:=\left.  \left(  L_{\Omega}\mathcal{E}_{\Omega}^{\Delta}%
\varphi\right)  \right\vert _{\Gamma}\quad\text{and\quad}H_{ \Gamma }\varphi:=\left.  
\left(  H_{\Omega}\mathcal{E}_{\Omega}^{\Delta}%
\varphi\right)  \right\vert _{\Gamma}.\label{DefDiriLift}%
\end{equation}
The vector-valued versions for tangential fields on the surface are used to define%

\begin{equation}%
\begin{array}
[c]{ll}%
\mathbf{H}_{\Gamma}^{\nabla}\left(  \mathbf{h}_{T}\right)  :=\nabla_{\Gamma
}\left(  H_{ \Gamma }\mathcal{L}_{\operatorname{imp}}^{\nabla}\mathbf{h}_{T}\right)  , &
\mathbf{L}_{\Gamma}^{\nabla}\left(  \mathbf{h}_{T}\right)  :=\nabla_{\Gamma
}\left(  L_{ \Gamma }\mathcal{L}_{\operatorname{imp}}^{\nabla}\mathbf{h}_{T}\right)  ,\\
\mathbf{H}_{\Gamma}^{\operatorname*{curl}}\left(  \mathbf{h}_{T}\right)
:=\overrightarrow{\operatorname*{curl}\nolimits_{\Gamma}}\left(  H_{ \Gamma }\mathcal{L}_{\operatorname{imp}}^{\operatorname{curl}}\mathbf{h}_{T}\right)
, & \mathbf{L}_{\Gamma}^{\operatorname*{curl}}\left(  \mathbf{h}_{T}\right)
:=\overrightarrow{\operatorname*{curl}\nolimits_{\Gamma}}\left(  L_{ \Gamma }\mathcal{L}_{\operatorname{imp}}^{\operatorname{curl}}\mathbf{h}_{T}\right)
,
\end{array}
\label{gt4}%
\end{equation}
and we set
\begin{equation}
\mathbf{H}_{\Gamma}:=\mathbf{H}_{\Gamma}^{\nabla}+\mathbf{H}_{\Gamma
}^{\operatorname*{curl}}\quad\text{and\quad}\mathbf{L}_{\Gamma}:=\mathbf{L}%
_{\Gamma}^{\nabla}+\mathbf{L}_{\Gamma}^{\operatorname*{curl}}%
.\label{DefHGammaLGamma2}%
\end{equation}


\subsection{Estimates for the frequency splittings}


\begin{lemma}
\label{LemScLift}Let $\Omega$ be a bounded Lipschitz domain with simply
connected, analytic boundary. The operators $\mathcal{L}_{\operatorname{imp}%
}^{\nabla}$ and $\mathcal{L}_{\operatorname{imp}}^{\operatorname{curl}}$ can
be extended (uniquely) to bounded linear operators $\mathbf{H}_{T}^{s}%
(\Gamma)\rightarrow H^{s+1}(\Gamma)$ for any $s\in\mathbb{R}$ and
\begin{subequations}
\label{eq:LemScLift}%
\begin{align}
\Vert\mathcal{L}_{\operatorname{imp}}^{\nabla}\mathbf{h}_{T}\Vert
_{H^{s+1}(\Gamma)}&\leq C_{s}\Vert\operatorname{div}_{\Gamma}\mathbf{h}%
_{T}\Vert_{H^{s-1}(\Gamma)}, \\
\Vert\mathcal{L}_{\operatorname{imp}%
}^{\operatorname{curl}}\mathbf{h}_{T}\Vert_{H^{s+1}(\Gamma)}& \leq C_{s}%
\Vert\operatorname{curl}_{\Gamma}\mathbf{h}_{T}\Vert_{H^{s-1}(\Gamma
)}.
\end{align}
\end{subequations}
For every $s>-1$, there is $C_{s}>0$ independent of $\lambda>1$ such that for
any $\mathbf{h}_{T}\in\mathbf{H}_{T}^{s}(\Gamma)$ there holds
 $\mathbf{H}_{\Gamma}  
\mathbf{h}_{T}=\mathbf{H}_{\Gamma}^{\nabla}\mathbf{h}_{T}+\mathbf{H}_{\Gamma
}^{\operatorname{curl}}\mathbf{h}_{T}$ together with
\[
\left\Vert \mathbf{H}_{\Gamma}^{\nabla}\mathbf{h}_{T}\right\Vert
_{\mathbf{H}^{s}(\Gamma)}+\left\Vert \mathbf{H}_{\Gamma}^{\operatorname*{curl}%
}\mathbf{h}_{T}\right\Vert _{\mathbf{H}^{s}\left(  \Gamma\right)  }\leq C_{ s  }\left\Vert \mathbf{h}_{T}\right\Vert _{\mathbf{H}^{s}\left(  \Gamma\right)
}.
\]

\end{lemma}

\begin{proof}
The mapping properties for $\mathcal{L}_{\operatorname{imp}}^{\nabla}$,
$\mathcal{L}_{\operatorname{imp}}^{\operatorname{curl}}$, follow directly from
elliptic regularity theory on smooth manifolds in view of
Remark~\ref{remk:characterize-L}. For the stability of the operators
$\mathbf{H}_{\Gamma}^{\nabla}$, $\mathbf{H}_{\Gamma}^{\operatorname{curl}}$ we
use the stability of the operator $H_{\Omega}:H^{s^{\prime}}(\Omega
)\rightarrow H^{s^{\prime}}(\Omega)$ for $s^{\prime}\geq0$ and the stability
of the trace operator $\gamma:H^{1/2+s^{\prime}}(\Omega)\rightarrow
H^{s^{\prime}}(\Gamma)$ for $s^{\prime}>0$ as in \cite[Lem.~{4.2}%
]{mm_stas_helm2} to get that $\mathbf{h}_{T}\mapsto\gamma H_{\Omega
}\mathcal{E}_{\Omega}^{\Delta}\mathcal{L}_{\operatorname{imp}}^{\nabla
}\mathbf{h}_{T}$ maps continuously $H^{-1+\varepsilon}(\Gamma)\rightarrow
H^{\varepsilon}(\Gamma)$ for any $\varepsilon>0$ with continuity constant
independent of $\lambda>1$. Since $\nabla_{\Gamma}:H^{\varepsilon}%
(\Gamma)\rightarrow\mathbf{H}_{T}^{-1+\varepsilon}(\Gamma)$, the result
follows. The case of $\mathbf{H}_{\Gamma}^{\operatorname{curl}}$ is handled analogously.
\end{proof}

We recall some properties of the high frequency splittings that are proved in
\cite[Lem.~{4.2}]{mm_stas_helm2}.

\begin{proposition}
\label{lemma:properties-of-H-via-lifting} Let $\lambda>1$ be the parameter
appearing in the definition of $H_{\mathbb{R}^{3}}$ in (\ref{deffreqslitR3}).
Let $H_{\Omega}$, $H_{\Omega}^{0}$, and $\mathbf{H}_{\Gamma}$ be the operators
of (\ref{DefLOmegaHOmega}), (\ref{DefHOmega0}), and (\ref{DefHGammaLGamma2}).
There are constants $C_{s^{\prime},s}$ independent of $\lambda>1$ such that
the following holds.

\begin{enumerate}
[(i)]

\item \label{item:lemma:properties-of-H-via-lifting-a} The frequency splitting
(\ref{deffreqslitR3}) satisfies for all $0\leq s^{\prime}\leq s$ the
estimates
\begin{align}
\left\Vert H_{\mathbb{R}^{3}}f\right\Vert _{H^{s^{\prime}}(\mathbb{R}^{3})}
&  \leq C_{s^{\prime},s}\left(  \lambda\left\vert k\right\vert \right)
^{s^{\prime}-s}\left\Vert f\right\Vert _{H^{s}(\mathbb{R}^{3})}\qquad\forall
f\in H^{s}(\mathbb{R}^{3}),\label{hfsplitdomain}\\
\left\Vert H_{\Omega}f\right\Vert _{H^{s^{\prime}}(\Omega)}  &  \leq
C_{s^{\prime},s}\left(  \lambda\left\vert k\right\vert \right)  ^{s^{\prime
}-s}\left\Vert f\right\Vert _{H^{s}(\Omega)}\qquad\forall f\in H^{s}%
(\Omega),\label{4.3}\\
\left\Vert H_{\Omega}^{0}\mathbf{f}\right\Vert _{\mathbf{H}^{s^{\prime}%
}(\Omega)}  &  \leq C_{s^{\prime},s}\left(  \lambda\left\vert k\right\vert
\right)  ^{s^{\prime}-s}\left\Vert \mathbf{f}\right\Vert _{\mathbf{H}%
^{s}(\Omega)}\qquad\forall\mathbf{f}\in\mathbf{H}^{s}(\Omega).
\label{kurbelHOmega0}%
\end{align}
These estimates hold also for Lipschitz domains.

\item \label{item:lemma:properties-of-H-via-lifting-b} Let $0\leq s^{\prime
}<s$ or $0<s^{\prime}\leq s$. Then the operator $H_{ \Gamma }$ satisfies
\begin{equation}
\left\Vert H_{ \Gamma }g\right\Vert _{H^{s^{\prime}}( \Gamma  )}\leq C_{s^{\prime},s}\left(  \lambda\left\vert k\right\vert \right)
^{s^{\prime}-s}\Vert g\Vert_{H^{s}( \Gamma )}. 
\label{eq:lemma:properties-of-H-via-lifting-1000}%
\end{equation}

\item \label{item:lemma:properties-of-H-via-lifting-c} Let $-1\leq s^{\prime
}<s$ or $-1<s^{\prime}\leq s$. Then the operator $\mathbf{H}_{\Gamma}$
satisfies for $\mathbf{g}_{T}\in\mathbf{H}_{T}^{s}(\Gamma)$%
\begin{subequations}
\label{Hboldfilter}
\begin{align}
\left\Vert \mathbf{H}_{\Gamma}\mathbf{g}_{T}\right\Vert _{\mathbf{H}%
^{s^{\prime}}(\Gamma)}  &  \leq C_{s^{\prime},s}\left(  \lambda\left\vert
k\right\vert \right)  ^{s^{\prime}-s}\Vert\mathbf{g}_{T}\Vert_{\mathbf{H}%
^{s}(\Gamma)},
\label{Hboldfiltera}\\
\left\Vert \operatorname*{div}\nolimits_{\Gamma}\mathbf{H}_{\Gamma}%
\mathbf{g}_{T}\right\Vert _{H^{s^{\prime}-1}(\Gamma)}  &  \leq C_{s^{\prime
},s}\left(  \lambda\left\vert k\right\vert \right)  ^{s^{\prime}-s}%
\Vert\operatorname*{div}\nolimits_{\Gamma}\mathbf{g}_{T}\Vert_{\mathbf{H}%
^{s-1}(\Gamma)},
\label{Hboldfilterb}\\
\left\Vert \operatorname*{curl}\nolimits_{\Gamma}\mathbf{H}_{\Gamma}%
\mathbf{g}_{T}\right\Vert _{H^{s^{\prime}-1}(\Gamma)}  &  \leq C_{s^{\prime
},s}\left(  \lambda\left\vert k\right\vert \right)  ^{s^{\prime}-s}%
\Vert\operatorname*{curl}\nolimits_{\Gamma}\mathbf{g}_{T}\Vert_{\mathbf{H}%
^{s-1}(\Gamma)}. 
\label{Hboldfilterc}%
\end{align}
\end{subequations}
\end{enumerate}

\end{proposition}

\begin{proof}
\emph{Proof of (\ref{item:lemma:properties-of-H-via-lifting-a}):} Estimates
(\ref{hfsplitdomain}) and (\ref{4.3}) are shown in \cite[Lem.~{4.2}%
]{mm_stas_helm2}. To see (\ref{kurbelHOmega0}), we bound $H_{\Omega}^{0}$ as
follows%
\begin{align*}
\left\Vert H_{\Omega}^{0}\mathbf{f}\right\Vert _{\mathbf{H}^{s^{\prime}%
}(\Omega)} &  =\left\Vert \operatorname*{curl}H_{\mathbb{R}^{3}}%
\mathcal{E}_{\operatorname*{Stein}}\mathbf{R}_{2}\mathbf{f}\right\Vert
_{\mathbf{H}^{s^{\prime}}(\Omega)}\leq \left\Vert H_{\mathbb{R}^{3}}%
\operatorname*{curl}\mathcal{E}_{\operatorname*{Stein}}\mathbf{R}%
_{2}\mathbf{f}\right\Vert _{\mathbf{H}^{s^{\prime}}(%
\mathbb{R}^{3}  )}\\
&  \overset{\text{(\ref{hfsplitdomain})}}{\leq}C_{s^{\prime},s}\left(
\lambda\left\vert k\right\vert \right)  ^{s^{\prime}-s}\left\Vert
\operatorname*{curl}\mathcal{E}_{\operatorname*{Stein}}\mathbf{R}%
_{2}\mathbf{f}\right\Vert _{\mathbf{H}^{s}(\mathbb{R}^{3})}\\
&  \leq C_{s^{\prime},s}\left(  \lambda\left\vert k\right\vert \right)
^{s^{\prime}-s}\left\Vert \mathcal{E}_{\operatorname*{Stein}}\mathbf{R}%
_{2}\mathbf{f}\right\Vert _{\mathbf{H}^{s+1}(\mathbb{R}^{3})} \\
& \leq
C_{s^{\prime},s}\left(  \lambda\left\vert k\right\vert \right)  ^{s^{\prime
}-s}\left\Vert \mathbf{R}_{2}\mathbf{f}\right\Vert _{\mathbf{H}^{s+1}(\Omega
)}
  \leq C_{s^{\prime},s}\left(  \lambda\left\vert k\right\vert \right)
^{s^{\prime}-s}\left\Vert \mathbf{f}\right\Vert _{\mathbf{H}^{s}(\Omega)}.
\end{align*}

\emph{Proof of (\ref{item:lemma:properties-of-H-via-lifting-b}):} For
$s^{\prime}>0$ the definition of $H_{ \Gamma }$ in (\ref{DefDiriLift}) implies%
\begin{align*}
\left\Vert H_{ \Gamma  }g\right\Vert _{H^{s^{\prime}}( \Gamma  )}
& =\left\Vert H_{\Omega}\mathcal{E}_{\Omega}^{\Delta}g\right\Vert
_{H^{s^{\prime}}( \Gamma  )}\leq C\left\Vert H_{\Omega}\mathcal{E}_{\Omega}^{\Delta}g\right\Vert
_{H^{s^{\prime}+1/2}(\Omega)} \\
& \overset{\text{(\ref{4.3})}}{\leq}\tilde
{C}_{s^{\prime},s}\left(  \lambda\left\vert k\right\vert \right)  ^{s^{\prime
}-s}\left\Vert \mathcal{E}_{\Omega}^{\Delta}g\right\Vert _{H^{s+1/2}(\Omega
)}.
\end{align*}
The regularity theory for the Laplace problem (\ref{LaplaceDiriProbl}) leads
to (\ref{eq:lemma:properties-of-H-via-lifting-1000}). For the case $s^{\prime
}=0$, we have $s>0$, and the multiplicative trace inequality
\[
\Vert H_{ \Gamma  
}g\Vert_{L^{2}( \Gamma )}\leq C\Vert H_{\Omega}\mathcal{E}_{\Omega}^{\Delta}g\Vert_{L^{2}(\Omega
)}^{1-1/(2s+1)}\Vert H_{\Omega}\mathcal{E}_{\Omega}^{\Delta}g\Vert
_{H^{s+1/2}(\Omega)}^{1/(2s+1)},
\]
together with the properties of $H_{\Omega}$ lead to the result.

\emph{Proof of (\ref{item:lemma:properties-of-H-via-lifting-c}):} We have
$\displaystyle 
\mathbf{H}_{\Gamma}\mathbf{g}_{T}=\nabla_{\Gamma}(  H_{ \Gamma  
}\mathcal{L}_{\operatorname*{imp}}^{\nabla}\mathbf{g}_{T})
+\overrightarrow{\operatorname*{curl}\nolimits_{\Gamma}}(  H_{ \Gamma  
}\mathcal{L}_{\operatorname*{imp}}^{\operatorname*{curl}}\mathbf{g}%
_{T})  .
$
A triangle inequality leads to%
\begin{align*}
\left\Vert \mathbf{H}_{\Gamma}\mathbf{g}_{T}\right\Vert _{\mathbf{H}%
^{s^{\prime}}\left(  \Gamma\right)  }  &  \leq\left\Vert H_{ \Gamma  }\mathcal{L}_{\operatorname*{imp}}^{\nabla}\mathbf{g}_{T}\right\Vert
_{\mathbf{H}^{s^{\prime}+1}\left(  \Gamma\right)  }+\left\Vert H_{ \Gamma  }\mathcal{L}_{\operatorname*{imp}}^{\operatorname*{curl}}\mathbf{g}%
_{T}\right\Vert _{\mathbf{H}^{s^{\prime}+1}\left(  \Gamma\right)  }\\
&  \overset{\text{(\ref{eq:lemma:properties-of-H-via-lifting-1000})}}{\leq
}C_{s^{\prime},s}\left(  \lambda\left\vert k\right\vert \right)  ^{s^{\prime
}-s}\left(  \left\Vert \mathcal{L}_{\operatorname*{imp}}^{\nabla}%
\mathbf{g}_{T}\right\Vert _{\mathbf{H}^{s+1}\left(  \Gamma\right)
}+\left\Vert \mathcal{L}_{\operatorname*{imp}}^{\operatorname*{curl}%
}\mathbf{g}_{T}\right\Vert _{\mathbf{H}^{s+1}\left(  \Gamma\right)  }\right)
\\
&  \overset{\text{(\ref{eq:LemScLift})}}{\leq}C_{s^{\prime},s}\left(
\lambda\left\vert k\right\vert \right)  ^{s^{\prime}-s}\left\Vert
\mathbf{g}_{T}\right\Vert _{\mathbf{H}^{s}\left(  \Gamma\right)  },
\end{align*}
which shows (\ref{Hboldfiltera}). For (\ref{Hboldfilterb}) we start from
\[
\operatorname*{div}\nolimits_{\Gamma}\mathbf{H}_{\Gamma}\mathbf{g}%
_{T}=\operatorname*{div}\nolimits_{\Gamma}\left(  \nabla_{\Gamma}(  H_{ \Gamma  
}\mathcal{L}_{\operatorname*{imp}}^{\nabla}\mathbf{g}_{T})
+\overrightarrow{\operatorname*{curl}\nolimits_{\Gamma}}(  H_{ \Gamma }\mathcal{L}_{\operatorname*{imp}}^{\operatorname*{curl}}\mathbf{g}%
_{T})  \right)  =\Delta_{\Gamma}(  H_{ \Gamma  
}\mathcal{L}_{\operatorname*{imp}}^{\nabla}\mathbf{g}_{T})  .
\]
We apply the previous estimate
(\ref{eq:lemma:properties-of-H-via-lifting-1000}) to get
\begin{align*}
\left\Vert \operatorname*{div}\nolimits_{\Gamma}\mathbf{H}_{\Gamma}%
\mathbf{g}_{T}\right\Vert _{H^{s^{\prime}-1}\left(  \Gamma\right)  }  &
=\left\Vert \Delta_{\Gamma}(  H_{ \Gamma  }\mathcal{L}_{\operatorname*{imp}}^{\nabla}\mathbf{g}_{T})  \right\Vert
_{H^{s^{\prime}-1}\left(  \Gamma\right)  }\leq C\left\Vert H_{ \Gamma  }\mathcal{L}_{\operatorname*{imp}}^{\nabla}\mathbf{g}_{T}\right\Vert
_{H^{s^{\prime}+1}\left(  \Gamma\right)  }\\
&  \leq C_{s^{\prime},s}\left(  \lambda\left\vert k\right\vert \right)
^{s^{\prime}-s}\left\Vert \mathcal{L}_{\operatorname*{imp}}^{\nabla}%
\mathbf{g}_{T}\right\Vert _{H^{s+1}\left(  \Gamma\right)  }.
\end{align*}
From (\ref{defphhigT}), we obtain $\Delta_{\Gamma}\mathcal{L}%
_{\operatorname*{imp}}^{\nabla}\mathbf{g}_{T}=\operatorname*{div}%
\nolimits_{\Gamma}\mathbf{g}_{T}$ so that%
\begin{align*}
\left\Vert \operatorname*{div}\nolimits_{\Gamma}\mathbf{H}_{\Gamma}%
\mathbf{g}_{T}\right\Vert _{H^{s^{\prime}-1}\left(  \Gamma\right)  } 
&\leq
\tilde{C}_{s^{\prime},s}\left(  \lambda\left\vert k\right\vert \right)
^{s^{\prime}-s}\left\Vert \Delta_{\Gamma}\mathcal{L}_{\operatorname*{imp}%
}^{\nabla}\mathbf{g}_{T}\right\Vert _{H^{s-1}\left(  \Gamma\right)  } \\
& =\tilde{C}_{s^{\prime},s}\left(  \lambda\left\vert k\right\vert \right)
^{s^{\prime}-s}\left\Vert \operatorname*{div}\nolimits_{\Gamma}\mathbf{g}%
_{T}\right\Vert _{H^{s-1}\left(  \Gamma\right)  }.
\end{align*}
This shows (\ref{Hboldfilterb}). The proof of (\ref{Hboldfilterc}) follows
along the same lines by using $\operatorname*{curl}\nolimits_{\Gamma
}\overrightarrow{\operatorname*{curl}\nolimits_{\Gamma}}=-\Delta_{\Gamma}$ and
$\operatorname*{curl}\nolimits_{\Gamma}\nabla_{\Gamma}=0$.%
\end{proof}

The following lemma concerns the parameter-explicit bounds for the low
frequency operators.

\begin{lemma}
\label{LemLowFreqEst}Let $\lambda>1$ be fixed in the definition
(\ref{deffreqslitR3}) of $L_{\mathbb{R}^{3}}$. There exists a constant $C>0$
independent of $\lambda$ such that for all $p\in\mathbb{N}_{0}$,
$\mathbf{v}\in\mathbf{L}^{2}(\mathbb{R}^{3})$, $\mathbf{w}\in\mathbf{L}%
^{2}(\Omega)$, there holds%
\begin{subequations}
\label{freqfiltest}
\begin{align}
\bigl\vert L_{\mathbb{R}^{3}}\mathbf{v}\bigr\vert _{\mathbf{H}^{p}%
(\mathbb{R}^{3})}  &  \leq C\left(  \lambda|k|\right)  ^{p}\left\Vert
\mathbf{v}\right\Vert _{\mathbf{L}^{2}(\mathbb{R}^{3})},
\label{freqfiltesta}\\
\bigl\vert L_{\Omega}\mathbf{w}\bigr\vert _{\mathbf{H}^{p}(\Omega
)}+\bigl\vert L_{\Omega}^{0}\mathbf{w}\bigr\vert _{\mathbf{H}^{p}(\Omega)}
&  \leq C\left(  \lambda|k|\right)  ^{p}\left\Vert \mathbf{w}\right\Vert .
\label{freqfiltestb}%
\end{align}
\end{subequations}%
For the boundary frequency filter we have, due to the analyticity of $\Gamma$,
the existence of $C>0$ and a neighborhood $\mathcal{U}_{\Gamma}\subset
\mathbb{R}^{3}$ of $\Gamma$ (depending only on $\Omega$) and some $\gamma>0$
(depending additionally on $\lambda$) such that for each $\mathbf{z}_{T}%
\in\mathbf{H}_{T}^{-1/2}(\Gamma)$ there exists a function
\[
\mathbf{Z}\in\mathcal{A}\bigl(  C\left\Vert \mathbf{z}_{T}\right\Vert
_{\mathbf{H}^{-1/2}(\Gamma)},\gamma,\mathcal{U}_{\Gamma}\bigr)
\]
such that $\mathbf{Z}|_{\Gamma}=L_{\Gamma}\mathbf{z}_{T}$.
\end{lemma}

\begin{proof}
From \cite[(3.32b)]{MelenkSauterMathComp} for the full space and from
\cite[Lem.~{4.3}]{mm_stas_helm2} for bounded domains the estimates
(\ref{freqfiltesta}) and the first one in (\ref{freqfiltestb}) follow. For the
operator $L_{\Omega}^{0}$, recall the lifting operator $\mathbf{R}_{2}$ of 
(\ref{LemRs}). Then, for any $\mathbf{w}\in\mathbf{L}^{2}( \Omega)  $ and $p\in\mathbb{N}_{0}$, the second estimate in
(\ref{freqfiltestb}) follows from%
\begin{align*}
\bigl\vert L_{\Omega}^{0}\mathbf{w}\bigr\vert _{\mathbf{H}^{p}(\Omega)}  &
\leq\bigl\vert L_{\Omega}^{0}\mathbf{w}\bigr\vert _{\mathbf{H}^{p}%
(\mathbb{R}^{3})}=\left\vert \operatorname*{curl}L_{\mathbb{R}^{3}}%
\mathcal{E}_{\operatorname{Stein}}\mathbf{R}_{2}\mathbf{w}\right\vert
_{\mathbf{H}^{p}(\mathbb{R}^{3})}=\left\vert L_{\mathbb{R}^{3}}%
\operatorname*{curl}\mathcal{E}_{\operatorname{Stein}}\mathbf{R}_{2}%
\mathbf{w}\right\vert _{\mathbf{H}^{p}(\mathbb{R}^{3})}\\
&  \overset{(\ref{freqfiltesta})}{\leq}C(\lambda|k|)^{p}\left\Vert
\operatorname{curl}\mathcal{E}_{\operatorname{Stein}}\mathbf{R}_{2}%
\mathbf{w}\right\Vert \leq C(\lambda|k|)^{p}\left\Vert \mathbf{w}\right\Vert .
\end{align*}
Finally, we consider the boundary low frequency operator. For $\mathbf{z}%
_{T}\in\mathbf{H}^{-1/2}(\Gamma)$, we define functions in the volume $\Omega$
via%
\[
\Phi:=L_{\Omega}\mathcal{E}_{\Omega}^{\Delta}\mathcal{L}_{\operatorname*{imp}%
}^{\nabla}\mathbf{z}_{T},\qquad\Psi:=L_{\Omega}\mathcal{E}_{\Omega}^{\Delta
}\mathcal{L}_{\operatorname*{imp}}^{\operatorname*{curl}}\mathbf{z}_{T}%
\]
so that, 
for $%
{\boldsymbol \phi}%
:=\nabla\Phi$ and $%
{\boldsymbol \psi}%
:=\nabla\Psi$, 
\begin{align*}
\mathbf{L}_{\Gamma}\mathbf{z}_{T}  &  =\Pi_{T}\nabla\Phi+\Pi_{T}\nabla
\Psi\times\mathbf{n}
=\mathbf{n}\times\left(  \left.
{\boldsymbol \phi}%
\right\vert _{\Gamma}\times\mathbf{n}\right)  +\left.
{\boldsymbol \psi}%
\right\vert _{\Gamma}\times\mathbf{n}.
\end{align*}
Let $\mathbf{n}^{\ast}$ denote an analytic extension of the
normal vector field into the domain $\Omega$; due to the analyticity of the
domain we may assume that there are constants $C_{\mathbf{n}}$, $\gamma
_{\mathbf{n}}>0$ and a tubular neighborhood $\mathcal{U}_{ \Gamma  }\subset\Omega$ with $ \Gamma \subset\overline{\mathcal{U}_{ \Gamma  }}$ 
such that $\mathbf{n}^{\ast}\in\mathcal{A}^{\infty}\left(  C_{\mathbf{n}%
},\gamma_{\mathbf{n}},\mathcal{U}_{ \Gamma  }\right)  $. Let%
\[
\mathbf{N}_{\ast}:=\left[
\begin{array} [c]{ccc}%
0 & n_{3}^{\ast} & -n_{2}^{\ast}\\
-n_{3}^{\ast} & 0 & n_{1}^{\ast}\\
n_{2}^{\ast} & -n_{1}^{\ast} & 0
\end{array}
\right]  .
\]
\quad Then,%
\[
\left.
\mbox{\boldmath$ \psi$}%
\right\vert _{\Gamma}\times\mathbf{n=}\left(  \mathbf{N}_{\ast}%
\mbox{\boldmath$ \psi$}%
\right)  _{\Gamma}\quad\text{and\quad}\mathbf{n}\times\left(  \left.
\mbox{\boldmath$ \phi$}%
\right\vert _{\Gamma}\times\mathbf{n}\right)  =-\left(  \mathbf{N}_{\ast}^{2}%
\mbox{\boldmath$ \phi$}%
\right)  _{\Gamma},
\]
i.e.,%
\begin{equation}
\mathbf{L}_{\Gamma}\mathbf{z}_{T}=\left.  \mathbf{G}_{\mathbf{z}}\right\vert
_{\Gamma}\quad\text{for }\mathbf{G}_{\mathbf{z}}:=\mathbf{N}_{\ast}\left(
{\boldsymbol \psi}%
-\mathbf{N}_{\ast}%
{\boldsymbol \phi}%
\right)  . \label{DefGz}%
\end{equation}
We further have for $p\in\mathbb{N}_{0}$%
\begin{align*}
\left\vert
{\boldsymbol \phi}%
\right\vert _{\mathbf{H}^{p}(\mathcal{U}_{ \Gamma  })}  
&  =\bigl\vert \nabla L_{\Omega}\mathcal{E}_{\Omega}^{\Delta}%
\mathcal{L}_{\operatorname*{imp}}^{\nabla}\mathbf{z}_{T}\bigr\vert
_{\mathbf{H}^{p}(\mathcal{U}_{ \Gamma  })}
\leq\bigl\vert L_{\mathbb{R}^{3}}\nabla\mathcal{E}_{\operatorname*{Stein}%
}\mathcal{E}_{\Omega}^{\Delta}\mathcal{L}_{\operatorname*{imp}}^{\nabla
}\mathbf{z}_{T}\bigr\vert _{\mathbf{H}^{p}(\mathcal{U}_{ \Gamma  })}
\\
&  \leq C(\lambda|k|)^{p}\bigl\Vert \nabla\mathcal{E}_{\operatorname*{Stein}%
}\mathcal{E}_{\Omega}^{\Delta}\mathcal{L}_{\operatorname*{imp}}^{\nabla
}\mathbf{z}_{T}\bigr\Vert _{\mathbf{L}^{2}(\mathbb{R}^{3})}\leq C(\lambda
|k|)^{p}\bigl\Vert \mathcal{E}_{\Omega}^{\Delta}\mathcal{L}%
_{\operatorname*{imp}}^{\nabla}\mathbf{z}_{T}\bigr\Vert _{\mathbf{H}%
^{1}(\Omega)}\\
&  \leq C\left(  \lambda|k|\right)  ^{p}\bigl\Vert \mathcal{L}%
_{\operatorname*{imp}}^{\nabla}\mathbf{z}_{T}\bigr\Vert _{\mathbf{H}%
^{1/2}(\Gamma)}\leq C_{1}\left(  \lambda|k|\right)  ^{p}\bigl\Vert
\mathbf{z}_{T}\bigr\Vert _{\mathbf{H}^{-1/2}(  \Gamma)  }.
\end{align*}
The proof of the estimate
\[
\left\vert
{\boldsymbol \psi}%
\right\vert _{\mathbf{H}^{p}\left(  \mathcal{U}_{ \Gamma  }\right)  }\leq C_{2}\left(  \lambda\left\vert k\right\vert \right)
^{p}\left\Vert \mathbf{z}_{T}\right\Vert _{\mathbf{H}^{-1/2}\left(
\Gamma\right)  }%
\]
follows along the same lines. Next we use \cite[Lem.~{4.3.1}]{MelenkHabil} (an
inspection of the proof shows that this lemma also holds for $d=3$) to deduce
that $\mathbf{N}_{\ast}{\boldsymbol\psi}$, $-\mathbf{N}_{\ast}^{2}%
{\boldsymbol\phi}\in\mathcal{A}\bigl(  C\left\Vert \mathbf{z}_{T}\right\Vert
_{\mathbf{H}^{-1/2}(  \Gamma)  },\gamma,\mathcal{U}_{ \Gamma  }\bigr)  $, 
where $C$ depends only on $C_{1}$, $C_{2}$, $C_{\mathbf{n}}$,
$\gamma_{\mathbf{n}}$, while $\gamma$ depends additionally on $\lambda$.
\end{proof}


\section{$k$-Explicit regularity by decomposition\label{SecReg}}


In this section, we always assume that the bounded Lipschitz domain
$\Omega\subset\mathbb{R}^{3}$ has a simply connected, analytic boundary
$\Gamma=\partial\Omega$. We consider the Maxwell problem (\ref{MWEq}) with
data $\mathbf{f}$, $\mathbf{g}_{T}$ with finite regularity.

For the regularity analysis of the operator $\mathcal{S}_{\Omega
,k}^{\operatorname*{MW}}$ it is key to understand that the solutions for high 
frequency right-hand sides have low order regularity but well-behaved
stability constant (with respect to the wavenumber) while solutions
corresponding to low-frequency right-hand sides are analytic but with possibly
growing stability constant. This different behavior is reflected in the
regularity theory, which decomposes the solution $\mathbf{z}=\mathcal{S}%
_{\Omega,k}^{\operatorname{MW}}(\mathbf{f},\mathbf{g}_{T})$ into a part with
finite regularity that can be controlled uniformly in $k$ and an analytic part
that can be controlled explicitly in $k$. This is achieved in
Theorem~\ref{TheoMainIt}. The main idea of the proof is to exploit that the
operators $\mathcal{L}_{\Omega,k}$ and $\mathcal{L}_{\Omega,\operatorname{i}%
k}$ have the same leading order differential operator. With the filter
operators of the preceding Section~\ref{SecFreqSplit} and recalling
$I=H_{\Omega}^{0}+L_{\Omega}^{0}+\mathbf{S}=H_{\Omega}^{0}+L_{\Omega}%
^{0}+H_{\Omega}\mathbf{S}+L_{\Omega}\mathbf{S}$ as well as $I=\mathbf{H}%
_{\Gamma}+\mathbf{L}_{\Gamma}$ one can write
\[
\mathcal{S}_{\Omega,k}^{\operatorname{MW}}(\mathbf{f},\mathbf{g}%
_{T})=\mathcal{S}_{\Omega,k}^{+}(H_{\Omega}^{0}\mathbf{f}+H_{\Omega}%
\mathbf{S}\mathbf{f},\mathbf{H}_{\Gamma}\mathbf{g}_{T})+\mathcal{S}_{\Omega
,k}^{\operatorname{MW}}(L_{\Omega}^{0}\mathbf{f}+L_{\Omega}\mathbf{S}%
\mathbf{f},\mathbf{L}_{\Gamma}\mathbf{g}_{T})+\mathbf{z}^{\prime}%
\]
for a remainder $\mathbf{z}^{\prime}$. One then makes the following observations:

\begin{enumerate}
\item If $\operatorname{div}\mathbf{f} = 0$ (which may be achieved by
subtracting a suitable gradient field), then the operator $\mathbf{S}$ is
smoothing by (\ref{eq:mapping-property-S}).

\item The term $\mathcal{S}_{\Omega,k}^{+}(H_{\Omega}^{0} \mathbf{f} +
H_{\Omega}\mathbf{S}\mathbf{f},\mathbf{H}_{\Gamma}\mathbf{g}_{T})$ has finite
regularity properties given by Theorem~\ref{lemma:apriori-with-good-sign}. The 
effect of the high frequency filters $H_{\Omega}^{0}$ and $\mathbf{H}_{\Gamma
}$ is that they improve the $k$-dependence of lower-order terms in the indexed
norms such as $\|\cdot\|_{\mathbf{H}^{m}(\Omega),k}$ (see
Lemma~\ref{lemma:estimateS+H} below).

\item $\mathcal{S}_{\Omega,k}^{\operatorname{MW}}(L_{\Omega}^{0} \mathbf{f} +
L_{\Omega}\mathbf{S} \mathbf{f},\mathbf{L}_{\Gamma}\mathbf{g}_{T})$ is an
analytic function and can be estimated with the aid of
Theorem~\ref{ThmAnaRegSum}.

\item The function $\mathbf{z}^{\prime}$ satisfies
\[
\mathcal{L}_{\Omega,k}^{\operatorname{MW}}\mathbf{z}^{\prime}=\mathbf{r}%
,\qquad\mathcal{B}_{  \Gamma ,k}=0, 
\]
where, by suitably choosing the cut-off parameters $\lambda$ in the frequency
operators, the residual $\mathbf{r}$ satisfies $\Vert\mathbf{r}\Vert_{\ast
}\leq q\Vert\mathbf{f}\Vert_{\ast}$ for some $q\in(0,1)$ and a suitable norm
$\Vert\cdot\Vert_{\ast}$. Hence, the arguments can be repeated for
$\mathbf{z}^{\prime}$ and the decomposition can be obtained by a geometric
series argument.
\end{enumerate}


\subsection{The concatenation of $\mathcal{S}_{\Omega,k}^{+}$ with high
frequency filters}

The following lemma analyzes the mapping properties of the concatenation of
the solution operator $\mathcal{S}_{\Omega,k}^{+}$ with the high frequency
filter operators $H_{\Omega}^{0}$ and $\mathbf{H}_{\Gamma}$.

\begin{lemma}
\label{lemma:estimateS+H} Let $m\in\mathbb{N}_{0}$, $\ell\in\mathbb{N}_{0}$.
Provided the right-hand sides are finite, the following estimates hold:%
\begin{align}
| k| ^{m}\Vert H_{\Omega}^{0}\mathbf{f}\Vert
_{\mathbf{H}^{m}(  \operatorname{div},\Omega)  ,k} &  \leq
C\left\vert k\right\vert \left\Vert \mathbf{f}\right\Vert _{\mathbf{H}%
^{m}(\Omega)},\label{eq:lemma:estimateS+H-3}\\
|k|^{m-1}\Vert H_{\Omega}^{0}\mathbf{f}\Vert_{\mathbf{H}^{m
-1}(\operatorname{div},\Omega),k} &  \leq C\lambda^{-1/2}\Vert\mathbf{f}%
\Vert_{\mathbf{H}^{m}(\Omega)},\qquad m \ge 1, \label{eq:lemma:estimateS+H-5}\\
{|k|}\Vert H_{\Omega}^{0}\mathbf{f}\Vert_{\mathbf{X}_{\operatorname*{imp}%
}^{\prime}(  \Omega)  ,k} &  \leq C\lambda^{-1/2}\Vert
\mathbf{f}\Vert_{\mathbf{L}^{2}(\Omega)}%
,%
\label{eq:lemma:estimateS+H-7}\\
|k|^{m+2}\Vert\mathcal{S}_{\Omega,k}^{+}({H}_{\Omega}^{0}\mathbf{f}%
,0)\Vert_{\mathbf{H}^{m}(\Omega),k} &  \leq C\lambda^{-1/2}\Vert
\mathbf{f}\Vert_{\mathbf{H}^{m}(\Omega)},\label{eq:lemma:estimateS+H-10}
\end{align}
as well as 
\begin{align}
\label{eq:lemma:estimateS+H-15}
& |k|^{m-1}\Vert\mathbf{H}_{\Gamma}\mathbf{g}_{T}\Vert_{\mathbf{H}%
^{m-1/2}(\operatorname{div}_{\Gamma},\Gamma),k}   
\leq 
\\ 
\nonumber 
& \quad 
C(\lambda|k|)^{-\ell
}\left(  |k|\Vert\mathbf{g}_{T}\Vert_{\mathbf{H}^{m-1/2+\ell}(\Gamma)}%
+\Vert\operatorname{div}_{\Gamma}\mathbf{g}_{T}\Vert_{{H}^{m-1/2+\ell}%
(\Gamma)}\right)  , \\
\label{eq:lemma:estimateS+H-20}
& |k|^{m+2}\Vert\mathcal{S}_{\Omega,k}^{+}(0,\mathbf{H}_{\Gamma}\mathbf{g}%
_{T})\Vert_{\mathbf{H}^{m}(\Omega),k}   
\leq   \\
\nonumber 
& \quad  C(\lambda|k|)^{-\ell} 
\begin{cases}
\lambda^{-1}\left(  |k|\Vert\mathbf{g}_{T}\Vert_{\mathbf{H}^{m-1/2+\ell
}(\Gamma)}+\Vert\operatorname{div}_{\Gamma}\mathbf{g}_{T}\Vert_{\mathbf{H}%
^{m-1/2+\ell}(\Gamma)}\right)  , & m\geq1,\\
\left(  |k|\Vert\mathbf{g}_{T}\Vert_{\mathbf{H}^{m-1/2+\ell}(\Gamma)}%
+\Vert\operatorname{div}_{\Gamma}\mathbf{g}_{T}\Vert_{\mathbf{H}^{m-1/2+\ell
}(\Gamma)}\right)  , & m=0.
\end{cases}
\nonumber
\end{align}
For $\mathbf{f}\in\mathbf{L^{2}}(\Omega)$ with $\operatorname{div}%
\mathbf{f}=0$ and the operator $\mathbf{S}$ of (\ref{eq:S}) we have for any
$n\in\mathbb{N}_{0}$%
\begin{align}
|k|^{n}\Vert H_{\Omega}\mathbf{S}\mathbf{f}\Vert_{\mathbf{H}^{m}(\Omega),k} &
\leq C_{n}\lambda^{-n}\Vert\mathbf{f}\Vert_{\mathbf{L}^{2}(\Omega
)},\label{eq:lemma:estimateS+H-25}\\
|k|^{n}\Vert\mathcal{S}_{\Omega,k}^{+}(H_{\Omega}\mathbf{S}\mathbf{f}%
,0)\Vert_{\mathbf{H}^{m+1}(\Omega),k} &  \leq C_{n}\lambda^{-n}\Vert
\mathbf{f}\Vert_{\mathbf{L}^{2}(\Omega)}.\label{eq:lemma:estimateS+H-30}%
\end{align}

\end{lemma}

\begin{proof}
\emph{Proof of (\ref{eq:lemma:estimateS+H-3}):} 
(\ref{eq:lemma:estimateS+H-3})
follows from the fact that $\operatorname{div} H_{\Omega}^{0} \mathbf{f} = 0$
and Proposition~\ref{lemma:properties-of-H-via-lifting}.

\emph{Proof of (\ref{eq:lemma:estimateS+H-5}):}
For $m \geq1$, we estimate
\begin{align*}
& \Vert{H}_{\Omega}^{0}\mathbf{f}\Vert_{\mathbf{H}^{m-1}(\Omega),k}    \leq
C\left(  \Vert{H}_{\Omega}^{0}\mathbf{f}\Vert_{\mathbf{L}^{2}(\Omega
)}+|k|^{-(m-1)}|{H}_{\Omega}^{0}\mathbf{f}|_{\mathbf{H}^{m-1}(\Omega
)}\right) \\
&  \qquad \overset{\text{Prop.~\ref{lemma:properties-of-H-via-lifting}}}{\leq
}C\left(  (\lambda|k|)^{-m}+|k|^{-(m-1)}(\lambda|k|)^{-1}\right)
\Vert\mathbf{f}\Vert_{\mathbf{H}^{\ell}(\Omega)}\leq C\lambda^{-1}|k|^{-m
}\Vert\mathbf{f}\Vert_{\mathbf{H}^{m}(\Omega)}.
\end{align*}
Noting that $\operatorname{div}H_{\Omega}^{0}\mathbf{f}=0$, the estimate
(\ref{eq:lemma:estimateS+H-5}) follows. 

\emph{Proof of (\ref{eq:lemma:estimateS+H-7}):} 
Recall
the definition of $\Vert\cdot\Vert_{\mathbf{X}_{\operatorname*{imp}}^{\prime
}(\Omega),k}$ in (\ref{eq:H-1div}) and observe with the multiplicative trace
inequality
\begin{equation}
\Vert\gamma_{T}H_{\mathbb{R}^{3}}\mathcal{E}_{\operatorname{Stein}}%
\mathbf{R}_{2}\mathbf{f}\Vert_{\mathbf{L}^{2}(\Gamma)}\leq C(\lambda
|k|)^{-1/2}\Vert\mathcal{E}_{\operatorname{Stein}}\mathbf{R}_{2}%
\mathbf{f}\Vert_{\mathbf{H}^{1}(\mathbb{R}^{3})}\leq C(\lambda|k|)^{-1/2}%
\Vert\mathbf{f}\Vert_{\mathbf{L}^{2}(\Omega)}.
\label{eq:H0omega-trace-estimate}%
\end{equation}
Since $H_{\Omega}^{0}\mathbf{f}=\operatorname{curl}H_{\mathbb{R}^{3}%
}\mathcal{E}_{\operatorname{Stein}}\mathbf{R}_{2}\mathbf{f}$, we get with an
integration by parts (\cite[Thm.~{3.29}]{Monkbook}) for ${\mathbf{v}}%
\in{\mathbf{X}}_{\operatorname{imp}}$
\begin{align*}
\left\vert (H_{\Omega}^{0}{\mathbf{f}},{\mathbf{v}})\right\vert  &
=\left\vert (\gamma_{T}H_{{\mathbb{R}}^{3}}{\mathcal{E}}_{\operatorname{Stein}%
}{\mathbf{R}}_{2}{\mathbf{f}},{\mathbf{v}}_{T})_{\mathbf{L}^{2}(\Gamma
)}+(H_{\mathbb{R}^{3}}{\mathcal{E}}_{\operatorname{Stein}}{\mathbf{R}}%
_{2}{\mathbf{f}},\operatorname{curl}{\mathbf{v}})\right\vert \\
&  \leq C(\lambda|k|)^{-1/2}\Vert{\mathbf{f}}\Vert_{\mathbf{L}^{2}(\Omega
)}\Vert{\mathbf{v}}_{T}\Vert_{\mathbf{L}^{2}(\Gamma)}+(\lambda|k|)^{-1}%
\Vert{\mathbf{f}}\Vert_{\mathbf{L}^{2}(\Omega)}\Vert\operatorname{curl}%
{\mathbf{v}}\Vert_{\mathbf{L}^{2}(\Omega)}\\
&  \leq C\lambda^{-1/2}|k|^{-1}\|{\mathbf{f}}\|_{{\mathbf{L}}^{2}(\Omega)}
\Vert{\mathbf{v}}\Vert_{\operatorname{imp},k}.
\end{align*}
We conclude
$\displaystyle\left\vert k\right\vert ^{1}\Vert H_{\Omega}^{0}%
\mathbf{f}\Vert _{\mathbf{X}_{\operatorname*{imp}}^{\prime}(
\Omega)  ,k}\leq C\lambda^{-1/2}\Vert\mathbf{f}\Vert_{\mathbf{L}%
^{2}(\Omega)}$ by the definition (\ref{eq:H-1div}) 
of $\Vert\cdot\Vert_{\mathbf{X} _{\operatorname*{imp}}^{\prime}(  \Omega)  ,k}$, 
and the statement (\ref{eq:lemma:estimateS+H-7}) is shown.

\emph{Proof of (\ref{eq:lemma:estimateS+H-10}): }For $m\geq1$, we obtain
(\ref{eq:lemma:estimateS+H-10}) from (\ref{eq:lemma:estimateS+H-5}) and
Theorem~\ref{lemma:apriori-with-good-sign},
(\ref{eq:lemma:apriori-with-good-sign-25b}). For $m=0$,
we observe that 
Theorem~\ref{lemma:apriori-with-good-sign},
(\ref{eq:item:lemma:apriori-with-good-sign-iia}) and
(\ref{eq:lemma:estimateS+H-7}) imply $\Vert\mathcal{S}_{\Omega,k}%
^{+}(H_{\Omega}^{0}\mathbf{f},0)\Vert_{\mathbf{L}^{2}(\Omega)}\leq
|k|^{-1}\Vert\mathcal{S}_{\Omega,k}^{+}(H_{\Omega}^{0}\mathbf{f}%
,0)\Vert_{\operatorname{imp},k}\leq C{|k|^{-1}}\Vert H_{\Omega}^{0}%
\mathbf{f}\Vert_{\mathbf{X}_{\operatorname*{imp}}^{\prime}(
\Omega)  ,k} \leq 
C\lambda^{-1/2}\Vert \mathbf{f}\Vert_{\mathbf{L}^{2}(\Omega)}$. 
%
%

\emph{Proof of (\ref{eq:lemma:estimateS+H-15}):} We distinguish the cases
{$m=0$ and $m\geq1$. For $m\geq1$, the statement follows from the
estimates
of ${\mathbf{H}}_{\Gamma}$ given in
Proposition~\ref{lemma:properties-of-H-via-lifting}. For $m=0$, in addition to
Proposition~\ref{lemma:properties-of-H-via-lifting} one invokes
Lemma~\ref{lemma:H-1/2-div}. }

\emph{Proof of (\ref{eq:lemma:estimateS+H-20}):} For $m\geq1$, the estimate
(\ref{eq:lemma:estimateS+H-20}) follows from combining
Proposition~\ref{lemma:properties-of-H-via-lifting} with
(\ref{eq:lemma:apriori-with-good-sign-25b}) 
and (\ref{eq:lemma:estimateS+H-15}) (taking $\ell+1$ for $\ell$ there) to get
\begin{align*}
& \Vert\mathcal{S}_{\Omega,k}^{+}(0,\mathbf{H}_{\Gamma}\mathbf{g}_{T}%
)\Vert_{\mathbf{H}^{m}(\Omega),k}   \leq C|k|^{-3}\Vert\mathbf{H}_{\Gamma
}\mathbf{g}_{T}\Vert_{\mathbf{H}^{m-1-1/2}(\operatorname{div}_{\Gamma}%
,\Gamma),k}\\
&  \qquad \overset{(\ref{eq:lemma:estimateS+H-15})}{\leq}C|k|^{-3-m+2}(\lambda
|k|)^{-1-\ell}\left(  |k|\Vert\mathbf{g}_{T}\Vert_{\mathbf{H}^{m-1/2+\ell
}(\Gamma)}+\Vert\operatorname{div}_{\Gamma}\mathbf{g}_{T}\Vert_{H^{m-1/2+\ell
}(\Gamma)}\right)  .
\end{align*}
For $m=0$, we use Theorem~\ref{lemma:apriori-with-good-sign},
(\ref{eq:item:lemma:apriori-with-good-sign-iia}) to get
\begin{align*}
& \Vert\mathcal{S}_{\Omega,k}^{+}(0,\mathbf{H}_{\Gamma}\mathbf{g}_{T}%
)\Vert_{\mathbf{L}^{2}(\Omega)} \leq C|k|^{-1}\Vert\mathcal{S}_{\Omega
,k}^{+}(0,\mathbf{H}_{\Gamma}\mathbf{g}_{T})\Vert_{\operatorname{imp},k}\leq
C|k|^{-1}\Vert\mathbf{H}_{\Gamma}\mathbf{g}_{T}\Vert_{\mathbf{X}%
_{\operatorname*{imp}}^{\prime}\left(  \Gamma\right)  ,k}\\
&  \qquad \overset{\text{(\ref{estgTprime})}}{\leq}C\left\vert k\right\vert
^{-2}\left(  |k|\Vert\mathbf{H}_{\Gamma}\mathbf{g}_{T}\Vert_{\mathbf{H}%
^{-1/2}(\Gamma)}+\Vert\operatorname{div}_{\Gamma}\mathbf{H}_{\Gamma}%
\mathbf{g}_{T}\Vert_{{H}^{-1/2}(\Gamma)}\right)  \\
&  \qquad \overset{\text{Prop.~\ref{lemma:properties-of-H-via-lifting}}}{\leq
}C\left\vert k\right\vert ^{-2}(\lambda|k|)^{-\ell}\left(  |k|\Vert%
\mathbf{g}_{T}%
\Vert_{\mathbf{H}^{-1/2+\ell}(\Gamma)}+\Vert\operatorname{div}_{\Gamma}%
\mathbf{g}_{T}%
\Vert_{{H}^{-1/2+\ell}(\Gamma)}\right)  .
\end{align*}
This completes the proof of (\ref{eq:lemma:estimateS+H-20}).

\emph{Proof of (\ref{eq:lemma:estimateS+H-25}), (\ref{eq:lemma:estimateS+H-30}%
):} For $\mathbf{f}$ with $\operatorname{div}\mathbf{f}=0$ we have
$\mathbf{S}\mathbf{f}\in C^{\infty}( \Omega) $ by (\ref{eq:mapping-property-S}%
). Hence, (\ref{eq:lemma:estimateS+H-25}), (\ref{eq:lemma:estimateS+H-30})
follow from Proposition~\ref{lemma:properties-of-H-via-lifting} and
(\ref{eq:mapping-property-S}) and again
Theorem~\ref{lemma:apriori-with-good-sign}.
\end{proof}

\begin{lemma}
\label{lemma:make-div-free} Let $m\in\mathbb{N}_{0}$, $\operatorname{div}%
\mathbf{f}\in H^{m-1}(\Omega)$, $\operatorname{div}_{\Gamma}\mathbf{g}_{T}\in
H^{m-3/2}(\Gamma)$. Then there exist $\varphi_{\mathbf{f}}\in H^{m+1}%
(\Omega)\cap H_{0}^{1}(\Omega)$ and $\varphi_{\mathbf{g}}\in H^{m+1}(\Omega)$
such that for $\ell=0,\ldots,m$
\begin{align*}
\Vert\varphi_{\mathbf{f}}\Vert_{H^{\ell+1}(\Omega)} &  \leq C\Vert
\operatorname{div}\mathbf{f}\Vert_{\mathbf{H}^{\ell-1}(\Omega)}, & -  
\operatorname{div}\nabla\varphi_{\mathbf{f}} &  =\operatorname{div}\mathbf{f},
& \mathcal{B}_{ \Gamma ,k}\nabla\varphi_{\mathbf{f}} &  =0,\\
\Vert\varphi_{\mathbf{g}}\Vert_{H^{\ell+1}(\Omega)} &  \leq C\Vert
\operatorname{div}_{\Gamma}\mathbf{g}_{T}\Vert_{\mathbf{H}^{\ell-3/2}(\Gamma
)}, & \operatorname{div}\nabla\varphi_{\mathbf{g}} &  =0, & \operatorname{div}%
_{\Gamma}\nabla_{\Gamma}\varphi_{\mathbf{g}} &  =\operatorname{div}_{\Gamma
}\mathbf{g}_{%
T}%
.
\end{align*}

\end{lemma}

\begin{proof}
Define $\varphi_{\mathbf{f}}\in H_{0}^{1}(\Omega)$ as the weak solution of
\[%
 - \Delta\varphi_{\mathbf{f}}=\operatorname{div}\mathbf{f}.
\]
By elliptic regularity, we have $\varphi_{\mathbf{f}}\in H^{m+1}(\Omega)$
and the stated bounds. The function $\varphi_{\mathbf{g}}$ is defined in two
steps: First, let $\varphi\in H^{1}(\Gamma)$ with $\int_{\Gamma}\varphi=0$
denote the weak solution of
\[
\Delta_{\Gamma}\varphi=\operatorname{div}_{\Gamma}\mathbf{g}_{T},
\]
which satisfies $\Vert\varphi\Vert_{H^{m+1/2}(\Gamma)}\leq C\Vert
\operatorname{div}\mathbf{g}_{T}\Vert_{H^{m-3/2}(\Gamma)}$. Then,
$\varphi_{\mathbf{g}}$ is defined on $\Omega$ as the harmonic extension from
$\Gamma$, i.e., $\varphi_{\mathbf{g}}\in H^{m+1}(\Omega)$ solves
\[
\Delta\varphi_{\mathbf{g}}=0,\qquad\varphi_{\mathbf{g}}|_{\Gamma}=\varphi.
\]
Again, the bounds follow from elliptic regularity theory.
\end{proof}

\subsection{Regularity by decomposition: the main result}

\begin{theorem}
\label{TheoMainIt} Let $\Omega\subset\mathbb{R}^{3}$ be a bounded Lipschitz
domain with a simply connected, analytic boundary $\Gamma=\partial\Omega$. Let
the stability Assumption~\ref{AssumptionAlgGrowth} be satisfied. Then there is
a linear mapping $\mathbf{L}^{2}(\Omega)\times\mathbf{H}_{T}^{-1/2}%
(\operatorname{div}_{\Gamma},\Gamma)\ni(\mathbf{f},\mathbf{g}_{T}%
)\mapsto(\mathbf{z}_{H^{2}},\mathbf{z}_{\mathcal{A}},\varphi_{\mathbf{f}%
},\varphi_{\mathbf{g}})$ such that the solution $\mathbf{z}:=\mathcal{S}%
_{\Omega,k}^{\operatorname{MW}}(\mathbf{f},\mathbf{g}_{T})\in\mathbf{X}%
_{\operatorname{imp}}$ of (\ref{MWEq}) can be written as $\mathbf{z}%
=\mathbf{z}_{H^{2}}+\mathbf{z}_{\mathcal{A}}+k^{-2}\nabla\varphi_{\mathbf{f}%
}+\operatorname{i}k^{-1}\nabla\varphi_{\mathbf{g}}$.

The linear mapping has the following properties: For any $m$, $m^{\prime}%
\in\mathbb{N}_{0}$, there are constants $C$, $B>0$ (depending only on $\Omega$
and $m$, $m^{\prime}$) such that for $(\mathbf{f},\mathbf{g}_{T})\in
\mathbf{H}^{m}(\Omega)\times\mathbf{H}_{T}^{m-1/2}(\Gamma)$ with
$(\operatorname{div}\mathbf{f},\operatorname{div}_{\Gamma}\mathbf{g}_{T})\in
H^{m^{\prime}-1}(\Omega)\times H^{m^{\prime}-3/2}(\Gamma)$ the following holds:

\begin{enumerate}
[(i)]

\item \label{item:TheoMainIt-i} The function $\mathbf{z}_{H^{2}}$ satisfies
\begin{equation}
\Vert\mathbf{z}_{H^{2}}\Vert_{\mathbf{H}^{m+1}(\Omega),k}\leq C|k|^{-m-2}%
\left(  |k|\Vert\mathbf{g}_{T}\Vert_{\mathbf{H}^{m-1/2}(\Gamma)}%
+\Vert\mathbf{f}\Vert_{\mathbf{H}^{m}(\Omega)}\right)  . \label{zH2estimate0}%
\end{equation}
If $\mathbf{g}_{T}\in\mathbf{H}_{T}^{m+1/2}(\Gamma)$, then%
\begin{equation}
\left\Vert \mathbf{z}_{H^{2}}\right\Vert _{\mathbf{H}^{m+1}\left(
\operatorname*{curl},\Omega\right)  ,k}\leq C\left\vert k\right\vert
^{-m-1}\left(  \left\Vert \mathbf{g}_{T}\right\Vert _{\mathbf{H}%
^{m+1/2}\left(  \Gamma\right)  }+\left\Vert \mathbf{f}\right\Vert
_{\mathbf{H}^{m}\left(  \Omega\right)  }\right)  \label{zH2estimate1}%
\end{equation}
and in (\ref{zH2estimate0}) the term $|k|\Vert\mathbf{g}_{T}\Vert
_{\mathbf{H}^{m-1/2}(\Gamma)}$ can be replaced with $\Vert\mathbf{g}_{T}%
\Vert_{\mathbf{H}^{m+1/2}(\Gamma)}$.

\item \label{item:TheoMainIt-ii} The gradient fields $\nabla\varphi
_{\mathbf{f}}$ and $\nabla\varphi_{\mathbf{g}}$ are given by
Lemma~\ref{lemma:make-div-free} and satisfy, for $\ell=0,\ldots,m^{\prime}$:
\begin{align}
\left\Vert \varphi_{\mathbf{f}}\right\Vert _{\mathbf{H}^{\ell+1}(\Omega)}  &
\leq C\left\Vert \operatorname*{div}\mathbf{f}\right\Vert _{H^{\ell-1}\left(
\Omega\right)  },\\
\left\Vert \varphi_{\mathbf{g}}\right\Vert _{\mathbf{H}^{\ell+1}(\Omega)}  &
\leq C\left\Vert \operatorname{div}_{\Gamma}\mathbf{g}_{T}\right\Vert
_{H^{\ell-3/2}\left(  \Gamma\right)  }.
\end{align}

\item \label{item:TheoMainIt-iv} The analytic part $\mathbf{z}_{\mathcal{A}}$
satisfies
\begin{equation}
\mathbf{z}_{\mathcal{A}}\in\mathcal{A}(C|k|^{ \theta-1} 
\{\Vert\mathbf{f}\Vert+|k|\Vert\mathbf{g}_{T}\Vert_{\mathbf{H}^{-1/2}(\Gamma
)}\},B,\Omega). \label{zaregsplitest}%
\end{equation}

\end{enumerate}
\end{theorem}

\begin{proof}
By linearity of the solution operator $\mathcal{S}_{\Omega,k}%
^{\operatorname{MW}}$, we consider the cases $\mathcal{S}_{\Omega
,k}^{\operatorname{MW}}(\mathbf{f},0)$ and $\mathcal{S}_{\Omega,k}%
^{\operatorname{MW}}(0,\mathbf{g}_{T})$ separately. The fact that the
right-hand sides in (\ref{zH2estimate0}), (\ref{zH2estimate1}),
(\ref{zaregsplitest}) do not contain the divergence of $\mathbf{f}$ or
$\mathbf{g}_{T}$ is due to the fact that we suitably choose the functions
$\varphi_{\mathbf{f}}$, $\varphi_{\mathbf{g}}$ in the course of the proof.

\textbf{Step 1:} (reduction to divergence-free data) Let the functions
$\varphi_{\mathbf{f}}$, $\varphi_{\mathbf{g}}$ be given by
Lemma~\ref{lemma:make-div-free}. These functions have the regularity
properties given in (\ref{item:TheoMainIt-ii}). The function $\mathbf{z}%
^{\prime}:=\mathbf{z}-k^{-2}\nabla\varphi_{\mathbf{f}}-\operatorname{i}%
k^{-1}\nabla\varphi_{\mathbf{g}}$ satisfies
\begin{align*}
\mathcal{L}_{\Omega,k}\mathbf{z}^{\prime} &  =\mathbf{f} + 
\nabla\varphi_{\mathbf{f}} +  
\operatorname{i}k\nabla\varphi_{\mathbf{g}}=:\mathbf{f}^{\prime}%
\quad\mbox{in $\Omega$},\\
\mathcal{B}_{ \Gamma  ,k}\mathbf{z}^{\prime} &  =\mathbf{g}_{T}-\nabla_{\Gamma}\varphi_{\mathbf{g}%
}=:\mathbf{g}_{T}^{\prime}\quad\mbox{on $\Gamma$}.
\end{align*}
By construction, $\operatorname{div}\mathbf{f}^{\prime}=0$ and
$\operatorname{div}_{\Gamma}\mathbf{g}_{T}^{\prime}=0$. Furthermore, using
$\Vert\operatorname{div}\mathbf{f}\Vert_{\mathbf{H}^{m-1}(\Omega)}\leq
C\Vert\mathbf{f}\Vert_{\mathbf{H}^{m}(\Omega)}$ and $\Vert\operatorname{div}%
_{\Gamma}\mathbf{g}_{T}\Vert_{\mathbf{H}^{m-3/2}(\Gamma)}\leq C\Vert
\mathbf{g}_{T}\Vert_{\mathbf{H}^{m-1/2}(\Gamma)}$ we obtain%
\begin{align}
\Vert\mathbf{f}^{\prime}\Vert_{\mathbf{H}^{m}(\Omega)} &  \leq C\left(
\Vert\mathbf{f}\Vert_{\mathbf{H}^{m}(\Omega)}+C|k|\Vert\mathbf{g}_{T}%
\Vert_{\mathbf{H}^{m-1/2}(\Gamma)}\right)  ,\label{eq:estimate-f'}\\
\Vert\mathbf{g}_{T}^{\prime}\Vert_{\mathbf{H}^{m-1/2}(\Gamma)} &  \leq
C\Vert\mathbf{g}_{T}\Vert_{\mathbf{H}^{m-1/2}(\Gamma)}.\label{eq:estimate-g'}%
\end{align}

\textbf{Step 2:} (Analysis of $\mathcal{S}_{\Omega,k}^{\operatorname{MW}%
}(\mathbf{f}^{\prime},0)$ with $\operatorname{div}\mathbf{f}^{\prime}=0$) {We
claim that
\begin{equation}
\mathcal{S}_{\Omega,k}^{\operatorname{MW}}(\mathbf{f}^{\prime},0)={\mathbf{z}%
}_{H^{2},{\mathbf{f}}}+{\mathbf{z}}_{{\mathcal{A}},{\mathbf{f}}}%
\end{equation}
for some functions ${\mathbf{z}}_{H^{2},{\mathbf{f}}}$ and ${\mathbf{z}%
}_{{\mathcal{A}},{\mathbf{f}}}$ satisfying the estimates (\ref{zH2estimate1})
(and therefore also (\ref{zH2estimate0}) since we focus on the case ${\mathbf g}_T=  0$ and (\ref{zaregsplitest}). }We have
$\operatorname{div}\mathbf{f}^{\prime}=0$ and assume $\mathbf{g}_{T}=0$, which
implies $\mathbf{g}_{T}^{\prime}=0$. Set $\mathbf{f}_{0}^{\prime}%
:=\widetilde{\mathbf{f}}_{0}:=\mathbf{f}^{\prime}$ and define, with the
mapping $\mathbf{f}\mapsto\varphi_{\mathbf{f}}$ of
Lemma~\ref{lemma:make-div-free}, recursively for $n=0,1,\ldots,$
\begin{align}
\mathbf{z}_{H^{2},n} &  :=\mathcal{S}_{\Omega,k}^{+}(H_{\Omega}^{0}%
\widetilde{\mathbf{f}}_{n},0)+\mathcal{S}_{\Omega,k}^{+}(H_{\Omega}%
\mathbf{S}\widetilde{\mathbf{f}}_{n},0)\label{zh2defplus}\\
\mathbf{z}_{\mathcal{A},n} &  :=\mathcal{S}_{\Omega,k}^{\operatorname{MW}%
}(L_{\Omega}^{0}\widetilde{\mathbf{f}}_{n},0)+\mathcal{S}_{\Omega
,k}^{\operatorname{MW}}(L_{\Omega}\mathbf{S}\widetilde{\mathbf{f}}%
_{n},0)\nonumber\\
\mathbf{f}_{n+1}^{\prime} &  :=2k^{2}\mathbf{z}_{H^{2},n},\nonumber\\
\widetilde{\mathbf{f}}_{n+1} &  :=\mathbf{f}_{n+1}^{\prime}+ \nabla
\varphi_{\mathbf{f}_{n+1}^{\prime}}.\nonumber
\end{align}
We note that $\operatorname{div}\widetilde{\mathbf{f}}_{n}=0$ for all $n$. 
From Lemma~\ref{lemma:estimateS+H}, we get: 
if $\widetilde{\mathbf{f}}_{n}\in\mathbf{H}^{\ell}(  \Omega)  $, then 
\begin{equation}
|k|^{\ell+2}\Vert\mathcal{S}_{\Omega,k}^{+}({H}_{\Omega}^{0}\widetilde
{\mathbf{f}}_{n},0)\Vert_{\mathbf{H}^{\ell}(\Omega),k}\leq C\lambda^{-1/2}%
\Vert\widetilde{\mathbf{f}}_{n}\Vert_{\mathbf{H}^{\ell}(\Omega)}%
.\label{Splusauxest}%
\end{equation}
Next, we obtain from Lemma~\ref{lemma:make-div-free} and the above defined recurrence relation 
%
%
\begin{align}
\Vert\nabla\varphi_{\mathbf{f}_{n}^{\prime}}\Vert_{\mathbf{H}^{\ell}(\Omega)}
&  \leq C\Vert\operatorname{div}\mathbf{f}_{n}^{\prime}\Vert_{\mathbf{H}%
^{\ell-1}(\Omega)}\leq C\Vert\mathbf{f}_{n}^{\prime}\Vert_{\mathbf{H}^{\ell
}(\Omega)},\qquad\ell=0,\ldots,m,
\label{estnablaphifnprime}\\
\Vert\mathbf{f}_{n+1}^{\prime}\Vert_{\mathbf{H}^{\ell}(\Omega)} &
\overset{\text{(\ref{Splusauxest})}}{\leq}C\lambda^{-1/2}\Vert\widetilde
{\mathbf{f}}_{n}\Vert_{\mathbf{H}^{\ell}(\Omega)}\leq C\lambda^{-1/2}%
\Vert\mathbf{f}_{n}^{\prime}\Vert_{\mathbf{H}^{\ell}(\Omega)}, \
\ell=0,\ldots,m,
\label{eq:7.21}
\\
\Vert\widetilde{\mathbf{f}}_{n+1}\Vert_{\mathbf{H}^{m}(\Omega)} &  \leq
C\Vert\mathbf{f}_{n+1}^{\prime}\Vert_{\mathbf{H}^{m}(\Omega)}\leq
C\lambda^{-1/2}\Vert\widetilde{\mathbf{f}}_{n}\Vert_{\mathbf{H}^{m}(\Omega)}.
\label{eq:7.22}
\end{align}
{}From the equation that defines $\mathbf{z}_{H^{2},n}$ and since
$\operatorname*{div}H_{\Omega}^{0}\widetilde{\mathbf{f}}_{n}=0$ we get%
\begin{equation}
(  2k^{2})  ^{-1}\operatorname*{div}\mathbf{f}_{n+1}^{\prime
}=\operatorname*{div}\mathbf{z}_{H^{2},n}=k^{-2}\operatorname*{div}H_{\Omega
}\mathbf{S}\widetilde{\mathbf{f}}_{n}.\label{rep_divfnprime}%
\end{equation}
Since $\mathbf{S}$ is a smoothing operator, this implies that 
$\operatorname*{div} \mathbf{f}_{n+1}^{\prime}$ is smooth and 
the first estimate in
(\ref{estnablaphifnprime}) actually holds for any $\ell\in\mathbb{N}_{0}$. 
The bounds (\ref{eq:7.21}), (\ref{eq:7.22}) show that the
functions $\mathbf{f}_{n}$ and $\widetilde{\mathbf{f}}_{n}$ decay in geometric
progression as $n$ increases if $\lambda>1$ is chosen such that $C\lambda
^{-1/2}=:q<1$. Fixing such a $\lambda>1$, a geometric series argument implies
for any $\mu\in\left\{  0,1,\ldots,m\right\}  $%
\begin{equation}
\sum_{n=0}^{\infty}\Vert\widetilde{\mathbf{f}}_{n}\Vert_{\mathbf{H}^{\mu
}(\Omega)}\leq C\Vert\mathbf{f}^{\prime}\Vert_{\mathbf{H}^{\mu}(\Omega
)}.\label{eq:geometric-series}%
\end{equation}
We also get from Theorem~\ref{lemma:apriori-with-good-sign} and
Lemma~\ref{lemma:estimateS+H} and the smoothing property of ${\mathbf S}$ 
(recall that $\operatorname*{div}\widetilde{\mathbf f} = 0$) 
\begin{align}
\nonumber 
\Vert\mathbf{z}_{H^{2},n}\Vert_{\mathbf{H}^{m+1}(\operatorname{curl}%
,\Omega),k} &\!\!  \overset{\text{Thm.~\ref{lemma:apriori-with-good-sign}}}{\leq
}\!\!C\left\vert k\right\vert ^{-2}\left(  \Vert H_{\Omega}^{0}\widetilde
{\mathbf{f}}_{n}\Vert_{\mathbf{H}^{m}(\operatorname{div},\Omega),k}+\Vert
H_{\Omega}\mathbf{S}\widetilde{\mathbf{f}}_{n}\Vert_{\mathbf{H}^{m}%
(\operatorname{div},\Omega),k}\right)  \\
&  \overset{(\ref{eq:lemma:estimateS+H-3})}{\leq}C|k|^{-(m+1)}\Vert
\widetilde{\mathbf{f}}_{n}\Vert_{\mathbf{H}^{m}(\Omega)}.
\label{eq:zh2-estimate}
\end{align}
Lemma~\ref{LemLowFreqEst} shows that $L_{\Omega}^{0}\widetilde{\mathbf{f}}%
_{n}$, $L_{\Omega}\mathbf{S}\widetilde{\mathbf{f}}_{n}\in\mathcal{A}%
(C_{1}\Vert\widetilde{\mathbf{f}}_{n}\Vert_{\mathbf{L}^{2}(\Omega)}%
,C_{2}\lambda|k|,\Omega)$ for some $C_{1}$, $C_{2}$ depending only on $\Omega
$. {}From Theorem~\ref{ThmAnaRegSum}, we infer
\begin{equation}
|\mathbf{z}_{\mathcal{A},n}|_{\mathbf{H}^{p}(\Omega)}\leq C_{\mathbf{z}}|k|^{ \theta-1  }%
\Vert\widetilde{\mathbf{f}}_{n}\Vert_{\mathbf{L}^{2}(\Omega)}\gamma^{p}%
\max\,(p,|k|)^{p}\qquad\forall p\in\mathbb{N}_{p\geq2}\label{eq:zA-H1curl}%
\end{equation}
for some $C_{\mathbf{z}}$, $\gamma$ independent of $k$ and $n$; $\gamma$
depends on $\lambda$, which has been fixed above. Upon setting
\[
\mathbf{z}_{H^{2},{\mathbf{f}^{\prime}}}:=\sum_{n=0}^{\infty}\mathbf{z}%
_{H^{2},n},\qquad\nabla\varphi:=\sum_{n=1}^{\infty}\nabla\varphi
_{\mathbf{f}_{n}^{\prime}},\qquad\mathbf{z}_{\mathcal{A},\mathbf{f}}%
:=\sum_{n=0}^{\infty}\mathbf{z}_{\mathcal{A},n},
\]
we have by (\ref{eq:zh2-estimate}) and (\ref{eq:geometric-series})%
\begin{align*}
\left\Vert \mathbf{z}_{H^{2},{\mathbf{f}^{\prime}}}\right\Vert _{\mathbf{H}%
^{m+1}(\operatorname{curl},\Omega),k} &  \leq C|k|^{-(m+1)}\Vert
\mathbf{f}^{\prime}\Vert_{\mathbf{H}^{m}(\Omega)}\leq C|k|^{-(m+1)}%
\Vert\mathbf{f}\Vert_{\mathbf{H}^{m}(\Omega)},\\
\mathbf{z}_{\mathcal{A},\mathbf{f}} &  \in\mathcal{A}(C|k|^{ \theta-1 }%
\Vert\mathbf{f}\Vert_{\mathbf{L}^{2}(\Omega)},\gamma,\Omega).
\end{align*}
For the term $\nabla\varphi$, we get%
\begin{align*}
\left\Vert k^{-2}\nabla\varphi\right\Vert _{\mathbf{H}^{m+1}%
(\operatorname{curl},\Omega),k} &  =\left\vert k\right\vert ^{-1}\left\Vert
\nabla\varphi\right\Vert _{\mathbf{H}^{m+1}(\Omega),k}\leq C\left\vert
k\right\vert ^{-1}\sum_{n=1}^{\infty}\left\Vert \operatorname*{div}%
\mathbf{f}_{n}^{\prime}\right\Vert _{\mathbf{H}^{m}\left(  \Omega\right)
,k}\\
&  \overset{\text{(\ref{rep_divfnprime})}}{=}C\left\vert k\right\vert
^{-1}\sum_{n=0}^{\infty}\Vert2\operatorname*{div}H_{\Omega}\mathbf{S}%
\widetilde{\mathbf{f}}_{n}\Vert_{\mathbf{H}^{m}\left(  \Omega\right)  ,k}\\
&  \leq C\sum_{n=0}^{\infty}\Vert H_{\Omega}\mathbf{S}\widetilde{\mathbf{f}%
}_{n}\Vert_{\mathbf{H}^{m+1}\left(  \Omega\right)  ,k}\leq C\left\vert
k\right\vert ^{-\left(  m+1\right)  }\sum_{n=0}^{\infty}\Vert\widetilde
{\mathbf{f}}_{n}\Vert_{\mathbf{L}^{2}\left(  \Omega\right)  },
\end{align*}
where we used (\ref{eq:lemma:estimateS+H-25}) with $n\leftarrow m+1$ for the
last estimate. The combination with (\ref{eq:geometric-series}) shows%
\[
\Vert k^{-2}\nabla\varphi\Vert _{\mathbf{H}^{m+1}%
(\operatorname{curl},\Omega),k}\leq C\left\vert k\right\vert ^{-\left(
m+1\right)  }\Vert \mathbf{f}^{\prime}\Vert _{\mathbf{L}^{2}( \Omega)  }.
\]
{We set
\[
{\mathbf{z}}_{H^{2},{\mathbf{f}}}:={\mathbf{z}}_{H^{2},{\mathbf{f}}^{\prime}%
}+ k^{-2}\nabla\varphi.
\]
} That is, the terms $\mathbf{z}_{H^{2},\mathbf{f}}$ and $\mathbf{z}%
_{\mathcal{A},\mathbf{f}}$ satisfy the estimates 
(\ref{zH2estimate0}), (\ref{zH2estimate1}) given in the statement of the theorem 
for the present case ${\mathbf g}_T = 0$. 
We compute%
\[
\mathcal{L}_{\Omega,k}(\mathbf{z}_{H^{2},\mathbf{f}}+\mathbf{z}_{\mathcal{A}%
,\mathbf{f}})=\sum_{n=0}^{\infty}\mathbf{f}_{n}^{\prime}-\mathbf{f}%
_{n+1}^{\prime}=\mathbf{f}_{0}^{\prime}=\mathbf{f}^{\prime},\qquad
\mathcal{B}_{ \Gamma ,k}(\mathbf{z}_{H^{2},{\mathbf{f}}}+\mathbf{z}_{\mathcal{A},\mathbf{f}})=0.
\]
By the uniqueness assertion of Proposition~\ref{lemma:apriori-homogeneous-rhs}%
, we have identified $\mathbf{z}_{H^{2},\mathbf{f}}+\mathbf{z}_{\mathcal{A}%
,\mathbf{f}}=\mathcal{S}_{\Omega,k}^{\operatorname{MW}}(\mathbf{f}^{\prime
},0)$.

\textbf{Step 3:} (Analysis of $\mathcal{S}_{\Omega,k}^{\operatorname{MW}%
}(0,\mathbf{g}_{T}^{\prime})$ with $\operatorname{div}_{\Gamma}\mathbf{g}%
_{T}^{\prime}=0$) We define
\[
\mathbf{z}_{H^{2},\mathbf{g}}:=\mathcal{S}_{\Omega,k}^{+}(0,\mathbf{H}%
_{\Gamma}\mathbf{g}_{T}^{\prime}),\qquad\mathbf{z}_{\mathcal{A},\mathbf{g}%
}:=\mathcal{S}_{\Omega,k}^{\operatorname{MW}}(0,\mathbf{L}_{\Gamma}%
\mathbf{g}_{T}^{\prime}).
\]
{}From Theorem~\ref{lemma:apriori-with-good-sign} and the properties of
$\mathbf{H}_{\Gamma}$ given in
Proposition~\ref{lemma:properties-of-H-via-lifting} we get
\begin{align}
\nonumber 
\Vert\mathbf{z}_{H^{2},\mathbf{g}}\Vert_{\mathbf{H}^{m+1}(\Omega),k}  &  \leq
C|k|^{-3}\Vert\mathbf{H}_{\Gamma}\mathbf{g}_{T}^{\prime}\Vert_{\mathbf{H}%
^{m-1/2}(\operatorname{div}_{\Gamma},\Gamma),k} \\
& \overset
{(\ref{eq:lemma:estimateS+H-15}),\operatorname{div}_{\Gamma}\mathbf{g}%
_{T}^{\prime}=0}{\leq}C|k|^{-1-m}\Vert\mathbf{g}_{T}\Vert_{\mathbf{H}%
^{m-1/2}(\Gamma)},\\
\Vert\mathbf{z}_{H^{2},\mathbf{g}}\Vert_{\mathbf{H}^{m+1}(\operatorname{curl}%
,\Omega),k}  &  \overset{\text{Prop.~\ref{lemma:properties-of-H-via-lifting}}%
}{\leq}C|k|^{-(m+1)}\Vert\mathbf{g}_{T}\Vert_{\mathbf{H}^{m+1/2}(\Gamma)}.
\label{ZH2gsecond}%
\end{align}
That is, $\mathbf{z}_{H^{2},\mathbf{g}}$ satisfies the estimates 
(\ref{zH2estimate0}), (\ref{zH2estimate1}) given in the
statement of the theorem for the present case ${\mathbf f} = 0$. 
For $\mathbf{z}_{\mathcal{A},\mathbf{g}}$ we observe
that Lemma~\ref{LemLowFreqEst} ensures\footnote{We write $\mathbf{L}_{\Gamma
}\mathbf{g}_{T}^{\prime}$ instead of introducing a new symbol $\mathbf{Z}$
with $\mathbf{Z}|_{\Gamma}=\mathbf{L}_{\Gamma}\mathbf{g}_{T}^{\prime}$}
$\mathbf{L}_{\Gamma}\mathbf{g}_{T}^{\prime}\in\mathcal{A}(C\Vert\mathbf{g}%
_{T}^{\prime}\Vert_{\mathbf{H}^{-1/2}(\Gamma)},\gamma,\Omega)$ for some $C$
depending only on $\Omega$ and $\gamma>0$ depending on $\Omega$ and $\lambda$.
We note $\Vert\mathbf{g}_{T}^{\prime}\Vert_{\mathbf{H}^{-1/2}(\Gamma)}\leq
C\Vert\mathbf{g}_{T}\Vert_{\mathbf{H}^{-1/2}(\Gamma)}$ by
(\ref{eq:estimate-g'}). From Theorem~\ref{ThmAnaRegSum} we obtain
\[
\mathbf{z}_{\mathcal{A},\mathbf{g}}\in\mathcal{A}(C|k|^{ \theta-1/2}  \Vert\mathbf{g}_{T}\Vert_{\mathbf{H}^{-1/2}(\Gamma)},\gamma,\Omega).
\]
Since $\theta-1/2\leq\theta$, the function $\mathbf{z}_{\mathbf{A},\mathbf{g}}$ satisfies the estimate
stated in the theorem. Finally, we observe that $\widehat{\mathbf{z}%
}:=\mathcal{S}_{\Omega,k}^{\operatorname{MW}}(0,\mathbf{g}_{T}^{\prime
})-(\mathbf{z}_{H^{2},\mathbf{g}}+\mathbf{z}_{\mathcal{A},\mathbf{g}})$
satisfies%
\[
\mathcal{L}_{\Omega,k}\widehat{\mathbf{z}}=2k^{2}\mathcal{S}_{\Omega,k}%
^{+}(0,\mathbf{H}_{\Gamma}\mathbf{g}_{T}^{\prime})=:\widehat{\mathbf{f}%
},\qquad\mathcal{B}_{ \Gamma  ,k}\widehat{\mathbf{z}}=0.
\]
{}From Lemma~\ref{lemma:estimateS+H} we get using $\operatorname{div}_{\Gamma
}\mathbf{g}_{T}^{\prime}=0$
\begin{subequations}
\label{eq:estimate-fhat}%
\begin{align}
\Vert\widehat{\mathbf{f}}\Vert_{\mathbf{H}^{m}(\Omega)}& \leq C|k|\Vert
\mathbf{g}_{T}^{\prime}\Vert_{\mathbf{H}^{m-1/2}(\Gamma)},
& 
\Vert \widehat{\mathbf{f}}\Vert_{\mathbf{H}^{m}(\Omega)}& \leq C\Vert\mathbf{g}%
_{T}^{\prime}\Vert_{\mathbf{H}^{m+1/2}(\Gamma)},
\\
\Vert\widehat{\mathbf{f}%
}\Vert_{\mathbf{L}^{2}(\Omega)}& \leq C|k|\Vert\mathbf{g}_{T}^{\prime}%
\Vert_{\mathbf{H}^{-1/2}(\Gamma)}. 
\end{align}
\end{subequations}
We note that $\operatorname{div}\widehat{\mathbf{f}}=0$ and that Step~2
provides a decomposition of $\widehat{\mathbf{z}}$ in the form {$\widehat
{\mathbf{z}}=\mathbf{z}_{H^{2},\widehat{\mathbf{f}}}+\mathbf{z}_{\mathcal{A}%
,\widehat{\mathbf{f}}}$.} By Step~2, the term $\mathbf{z}_{H^{2}%
,\widehat{\mathbf{f}}}$ can be controlled in terms of $\Vert\widehat
{\mathbf{f}}\Vert_{\mathbf{H}^{m}(\Omega)}$ in thus in the required form. For
$\mathbf{z}_{\mathcal{A},\widehat{\mathbf{f}}}$, we note that Step~2 yields
\[
\mathbf{z}_{\mathcal{A},\mathbf{\widehat{\mathbf{f}}}}\in\mathcal{A}(C|k|^{ \theta-1 }%
\Vert\widehat{\mathbf{f}}\Vert_{\mathbf{L}^{2}(\Omega)},\gamma,\Omega
)\subset\mathcal{A}(C|k|^{ \theta} 
\Vert\mathbf{g}_{T}\Vert_{\mathbf{H}^{-1/2}(\Gamma)},\gamma,\Omega),
\]
which is an analytic function with the desired estimate. We summarize that
$\mathbf{z}_{H^{2}}$, $\mathbf{z}_{\mathcal{A}}$ in the statement of the
theorem are given by%
\[
\mathbf{z}_{H^{2}}=\mathbf{z}_{H^{2},\mathbf{f}}
+\mathbf{z}_{H^{2},\mathbf{g}}+\mathbf{z}_{H^{2},\widehat{\mathbf{f}}}%
\quad\text{and\quad}\mathbf{z}_{\mathcal{A}}=\mathbf{z}_{\mathcal{A}%
,\mathbf{f}}+\mathbf{z}_{\mathcal{A},\mathbf{g}}+\mathbf{z}_{\mathcal{A}%
,\widehat{\mathbf{f}}}%
\]
and the summands have been estimated in Step 1-3.

\textbf{Step 4:} The proof is now complete with the exception of the statement
in (\ref{item:TheoMainIt-i}) that $|k|\Vert\mathbf{g}_{T}\Vert_{\mathbf{H}%
^{m-1/2}(\Gamma)}$ can be replaced with $\Vert\mathbf{g}_{T}\Vert
_{\mathbf{H}^{m+1/2}(\Gamma)}$. However, this follows directly from
(\ref{zH2estimate1}) via%
\[
\left\vert k\right\vert \left\Vert \mathbf{z}_{H^{2},\mathbf{g}}\right\Vert
_{\mathbf{H}^{m+1}\left(  \Omega\right)  ,k}\leq\left\Vert \mathbf{z}%
_{H^{2},\mathbf{g}}\right\Vert _{\mathbf{H}^{m+1}\left(  \operatorname*{curl}%
,\Omega\right)  ,k}\overset{\text{(\ref{ZH2gsecond})}}{\leq}C\left\vert
k\right\vert ^{-m-1}\left\Vert \mathbf{g}_{T}\right\Vert _{\mathbf{H}%
^{m+1/2}\left(  \Gamma\right)  }
\]
and for the control of $\widehat{\mathbf{z}}$ in Step~3 via the bound
$\Vert\widehat{\mathbf{f}}\Vert_{{\mathbf{H}^{m}(\Omega)}}\lesssim
\Vert{\mathbf{g}}_{T}^{\prime}\Vert_{{\mathbf{H}^{m+1/2}(\Gamma)}}$ in
(\ref{eq:estimate-fhat}).
\end{proof}

\bigskip

\section{Discretization}

\label{sec:discretization}
In this section, we describe the $hp$-FEM based on N\'ed\'elec elements 
and discuss the approximation properties of various $hp$-approximation operators. 
These operators made their appearance already in \cite{MelenkSauterMaxwell_I}. Here, 
we strengthen the results of \cite[Sec.~{8}]{MelenkSauterMaxwell_I} in that we additionally 
control the error on the boundary of the elements, which is required due to the 
impedance boundary conditions considered here.

\subsection{Meshes and N\'{e}d\'{e}lec elements}

\label{sec:nedelec-elements}

The classical example of curl-conforming FE spaces are the N\'{e}d\'{e}lec
elements, \cite{nedelec80}. We restrict our attention here to so-called
\textquotedblleft type I\textquotedblright\ elements (sometimes also referred
to as the N\'{e}d\'{e}lec-Raviart-Thomas element) on tetrahedra. These spaces
are based on a conforming (no hanging nodes), shape-regular triangulation ${\mathcal{T}}_{h}$ of $\Omega\subset
\mathbb{R}^{3}$. That is, ${\mathcal{T}}_{h}$ satisfies:

\begin{enumerate}
[(i)]

\item The (open) elements $K\in{\mathcal{T}}_{h}$ cover $\Omega$, i.e.,
$\overline{\Omega}=\cup_{K\in{\mathcal{T}}_{h}}\overline{K}$.

\item Associated with each element $K$ is the \emph{element map}, a $C^{1}%
$-diffeomorphism $F_{K}:\overline{\widehat{K}} \rightarrow\overline{K}$. The
set $\widehat{K}$ is the \emph{reference tetrahedron}.

\item Denoting $h_{K}=\operatorname*{diam}K$, there holds, with some
\emph{shape-regularity constant $\gamma_{\mathcal{T}}$},
\begin{equation}
h_{K}^{-1}\Vert F_{K}^{\prime}\Vert_{L^{\infty}(\widehat{K})}+h_{K}\Vert
(F_{K}^{\prime})^{-1}\Vert_{L^{\infty}(\widehat{K})}\leq\gamma_{\mathcal{T}}.
\label{defhkloc}%
\end{equation}

\item The intersection of two elements is only empty, a vertex, an edge, a
face, or they coincide (here, vertices, edges, and faces are the images of the
corresponding entities on the reference tetrahedron $\widehat{K}$). The
parametrization of common edges or faces are compatible. That is, if two
elements $K$, $K^{\prime}$ share an edge (i.e., $F_{K}(e)=F_{K^{\prime}%
}(e^{\prime})$ for edges $e$, $e^{\prime}$ of $\widehat{K}$) or a face (i.e.,
$F_{K}(f)=F_{K^{\prime}}(f^{\prime})$ for faces $f$, $f^{\prime}$ of
$\widehat{K}$), then $F_{K}^{-1}\circ F_{K^{\prime}}:f^{\prime}\rightarrow f$
is an affine isomorphism.
\end{enumerate}

The maximal mesh width is denoted by%
\begin{equation}
h:=\max\left\{  h_{K}:K\in\mathcal{T}_{h}\right\}  .\label{eq:hmax}%
\end{equation}
The following assumption
requires
that the element map $F_{K}$ can be decomposed as a composition of an affine
scaling with an $h$-independent mapping. We adopt the setting of
\cite[Sec.~{5}]{MelenkSauterMathComp} and assume that the element maps $F_{K}$
of the conforming, $\gamma$-shape regular triangulation ${\mathcal{T}}_{h}$ satisfy the
following additional requirements:

\begin{assumption}
[normalizable regular triangulation]\label{def:element-maps} Each element map
$F_{K}$ can be written as $F_{K}=R_{K}\circ A_{K}$, where $A_{K}$ is an
\emph{affine} map and the maps $R_{K}$ and $A_{K}$ satisfy for constants
$C_{\operatorname*{affine}}$, $C_{\operatorname{metric}}$, $\gamma>0$
independent of $K$:
\begin{align*}
&  \Vert A_{K}^{\prime}\Vert_{L^{\infty}(\widehat{K})}\leq
C_{\operatorname*{affine}}h_{K},\qquad\Vert(A_{K}^{\prime})^{-1}%
\Vert_{L^{\infty}(\widehat{K})}\leq C_{\operatorname*{affine}}h_{K}^{-1},\\
&  \Vert(R_{K}^{\prime})^{-1}\Vert_{L^{\infty}(\widetilde{K})}\leq
C_{\operatorname{metric}},\qquad\Vert\nabla^{n}R_{K}\Vert_{L^{\infty
}(\widetilde{K})}\leq C_{\operatorname{metric}}\gamma^{n}n!\qquad\forall
n\in{\mathbb{N}}_{0}.
\end{align*}
Here, $\widetilde{K}=A_{K}(\widehat{K})$ and $h_{K}>0$ is the element diameter.
\end{assumption}

\begin{remark}
A prime example of meshes that satisfy Assumption~\ref{def:element-maps} are
those patchwise structured meshes as described, for example, in
\cite[Ex.~{5.1}]{MelenkSauterMathComp} or \cite[Sec.~{3.3.2}]{MelenkHabil}.
These meshes are obtained by first fixing a macro triangulation of $\Omega$;
the actual triangulation is then obtained as images of affine triangulations
of the reference element. 
\eremk
\end{remark}

On the reference tetrahedron $\widehat{K}$ we introduce the classical
N\'{e}d\'{e}lec type I and Raviart-Thomas elements of degree $p\geq0$ (see,
e.g., \cite{Monkbook}):
\begin{align}
{\mathcal{P}}_{p}(\widehat{K})  &  :=\operatorname*{span}\{\mathbf{x}^{\alpha
}\,|\,|\alpha|\leq p\},\label{eq:Pp}\\
\mathbf{RT}_{p}(\widehat{K})  &  :=\{\mathbf{p}(\mathbf{x})+\mathbf{x}%
{q}(\mathbf{x})\,|\,\mathbf{p}\in({\mathcal{P}}_{p}(\widehat{K}))^{3},{q}%
\in{\mathcal{P}}_{p}(\widehat{K})\},\\
\boldsymbol{\mathcal{N}}_{p}^{\operatorname*{I}}(\widehat{K})  &
:=\{\mathbf{p}(\mathbf{x})+\mathbf{x}\times\mathbf{q}(\mathbf{x}%
)\,|\,\mathbf{p},\mathbf{q}\in({\mathcal{P}}_{p}(\widehat{K}))^{3}\}.
\label{eq:Np}%
\end{align}
The spaces $S_{p+1}({\mathcal{T}}_{h})$, $\mathbf{RT}_{p}\left(
\mathcal{T}_{h}\right)  $, $\boldsymbol{\mathcal{N}}_{p}^{\operatorname*{I}%
}({\mathcal{T}}_{h})$ are then defined as in \cite[(3.76)]{Monkbook} by
transforming covariantly $\boldsymbol{\mathcal{N}}_{p}^{\operatorname*{I}
}(\widehat{K})$ and contravariantly $\mathbf{RT}_p({\widehat K})$:
\begin{subequations}
\label{eq:curl-conforming-hp-spaces}%
\begin{align}
S_{p+1}({\mathcal{T}}_{h})  &  :=\{u\in H^{1}(\Omega)\,|\,u|_{K}\circ F_{K}%
\in{\mathcal{P}}_{p+1}(\widehat{K})\},\\
\mathbf{RT}_{p}({\mathcal{T}}_{h})  &  :=\{\mathbf{u}\in\mathbf{H}%
(\operatorname*{div},\Omega)\,|\,({\operatorname*{det}F_{K}^{\prime}}%
)(F_{K}^{\prime})^{-1}\mathbf{u}|_{K}\circ F_{K}\in\mathbf{RT}_{p}(\widehat
{K})\},\\
\boldsymbol{\mathcal{N}}_{p}^{\operatorname*{I}}({\mathcal{T}}_{h})  &
:=\{\mathbf{u}\in\mathbf{H}(\operatorname*{curl},\Omega)\,|\,(F_{K}^{\prime
})^{T}\mathbf{u}|_{K}\circ F_{K}\in\boldsymbol{\mathcal{N}}_{p}%
^{\operatorname*{I}}(\widehat{K})\}.
\end{align}
\end{subequations}
We set\footnote{%
Note that $\mathbf{X}_{h}=\boldsymbol{\mathcal{N}}_{p}^{\operatorname*{I}%
}({\mathcal{T}}_{h})$ and $S_{h}=S_{p+1}({\mathcal{T}}_{h})$ since
$\boldsymbol{\mathcal{N}}_{p}^{\operatorname*{I}}({\mathcal{T}}_{h}%
)\subset\mathbf{X}_{\operatorname*{imp}}$ and $S_{p+1}({\mathcal{T}}%
_{h})\subset H_{\operatorname*{imp}}^{1}(  \Omega)  $.
 }%
\begin{align}
\label{defXh} 
\mathbf{X}_{h}  &  :=\boldsymbol{\mathcal{N}}_{p}^{\operatorname*{I}%
}({\mathcal{T}}_{h})\cap\mathbf{X}_{\operatorname*{imp}},
& 
S_{h}  &  :=S_{p+1}({\mathcal{T}}_{h})\cap H_{\operatorname*{imp}}^{1}\left(
\Omega\right) 
\end{align}
and recall the well-known exact sequence property%
\begin{equation}
S_{h}\overset{\nabla}{\longrightarrow}\mathbf{X}_{h}\overset
{\operatorname*{curl}}{\longrightarrow}\operatorname*{curl}\mathbf{X}_{h}.
\label{exdiscseq_rm}%
\end{equation}

The $hp$-FEM Galerkin discretization for the electric Maxwell problem
(\ref{weakformulation}) is given by:%
\begin{equation}
\text{find }\mathbf{E}_{h}\in\mathbf{X}_{h}\text{\quad such that }A_{k}(
\mathbf{E}_{h},\mathbf{v})  =\left(  \mathbf{j,v}\right)  +\left(
\mathbf{g}_{T},\mathbf{v}\right)  _{\mathbf{L}^{2}\left(  \Gamma\right)
}\quad\forall\mathbf{v}\in\mathbf{X}_{h}. \label{discMW}%
\end{equation}


\subsection{$hp$-Approximation operators}

{We will use polynomial approximation operators that are constructed
elementwise, i.e., for an operator $\widehat{I}_{p}$ on the reference element
$\widehat{K}$, a global operator $I_{p}$ is defined 
by setting $(I_{p}u)|_{K}:=\widehat{I}_{p}(u\circ F_{K}))\circ F_{K}^{-1}%
$. If $\widehat{I}_{p}$ maps into ${\mathcal{P}}_{p+1}(\widehat{K})$, we say
$\widehat{I}_{p}$ \emph{admits an element-by-element construction}, if the
operator $I_{p}$ defined in this way maps into $S_{p+1}({\mathcal{T}}_{h})$.
Analogously, if $\widehat{I}_{p}$ maps into $\boldsymbol{\mathcal{N}}%
_{p}^{\operatorname*{I}}(\widehat{K})$, then we say that $\widehat{I}_{p}$
\emph{admits an element-by-element construction} if the resulting operator
$I_{p}$ maps into $\boldsymbol{\mathcal{N}}_{p}^{\operatorname*{I}%
}(\mathcal{T}_{h})$. }

For scalar functions (or gradient fields), we have elemental approximation
operators with the optimal convergence in $L^{2}$ and $H^{1}$:

\begin{lemma}
\label{lemma:element-by-element-approximation} Let $\widehat{K}\subset
\mathbb{R}^{d}$, $d\in\{2,3\}$, be the reference triangle or reference
tetrahedron and
$\mathbb{R} \ni m%
\geq(d+1)/2$. Then, for every $p\in\mathbb{N}_{0}$,  there exists a linear
operator $\widehat{\Pi}_{p}:H^{%
m%
}(\widehat{K})\rightarrow\mathcal{P}_{p+1}$ that permits an element-by-element
construction
such that if $p\geq m-2$
\begin{equation}
\begin{split}
&  \Vert u-\widehat{\Pi}_{p}u\Vert_{L^{2}(\widehat{K})}
+\frac{1}{p+1}\Vert u-\widehat{\Pi}_{p}u\Vert_{H^{1}(\widehat{K})}
+ (p+1)^{-1/2}\Vert u-\widehat{\Pi}_{p}u\Vert_{L^{2}(\partial \widehat{K})}  
\\ 
\label{eq:lemma:element-by-element-approximation}
& \qquad \mbox{}
+(p+1)^{-3/2}\Vert u-\widehat{\Pi}_{p}u\Vert_{H^{1}%
(\partial\widehat{K})}\leq C{(p+1)^{-m}}|u|_{H^{m}(\widehat{K})}
\end{split}
\end{equation}
for a constant $C>0$ that depends only on
$m$, $d$, and the choice of reference triangle/tetrahedron.

{For the case $d=3$, the condition on
$m$
can be relaxed to
$m%
>d/2$. }
\end{lemma}

\begin{proof}
The operator $\widehat{\Pi}_{p}$ may be taken as the operators $\widehat{\Pi
}^{\operatorname{grad},3d}_{p+1}$ \ for $d = 3$ or $\widehat{\Pi
}^{\operatorname{grad},2d}_{p+1}$ for $d = 2$ of \cite{melenk_rojik_2018}. The
volume estimates follow from \cite[Cor.~{2.12}]{melenk_rojik_2018} for the case $d= 3$ 
and \cite[Thm.~{2.13}]{melenk_rojik_2018} for the case $d = 2$. 
For the estimates on $\partial\widehat{K}$, one notices that the restriction
of $\widehat{\Pi}^{\operatorname{grad},3d}_{p+1}$ to a boundary face
$\widehat{f}$ is the operator $\widehat{\Pi}^{\operatorname{grad},2d}_{p+1}$
on that face and that the restriction of $\widehat{\Pi}^{\operatorname{grad}%
,2d}_{p+1}$ to an edge of the reference triangle is the operator $\widehat
{\Pi}^{\operatorname{grad},1d}_{p+1}$ discussed in \cite[Lem.~{4.1}%
]{melenk_rojik_2018}.

For $d=3$ an operator $\widehat{\Pi}_{p}$ with the stated approximation
properties is constructed in \cite[Thm.~{B.4}]{MelenkSauterMathComp} for the
case
$m
>d/2=3/2$. The statement about the approximation on $\partial\widehat{K}$
follows by a more careful analysis of the proof of \cite[Thm.~{B.4}%
]{MelenkSauterMathComp}. For the reader's convenience, the proof is reproduced
in 
\ifarxiv Theorem~\ref{thm:finite-regularity-constrained-approximation}.
\else
\cite[Thm.~{B.5}]{MelenkSauterMaxwell_II}. 
\fi
\end{proof}

The fact that $\widehat{\Pi}_{p}$ in
Lemma~\ref{lemma:element-by-element-approximation} has the element-by-element
construction property means that an elementwise definition of the operator
$\Pi_{p}^{\nabla,s}:H^{m}(\Omega)\rightarrow S_{p+1}(\mathcal{T}_{h})$ by
$(\Pi_{p}^{\nabla,s}\varphi)|_{K}=(\widehat{\Pi}_{p}(\varphi\circ F_{K}))\circ
F_{K}^{-1}$ maps indeed into $S_{p+1}(\mathcal{T}_{h})\subset H^{1}(\Omega)$.

In the following we always assume for the spatial dimension $d=3$. By scaling
arguments we get the following result:

\begin{corollary}
\label{cor:Pinabla}Let $d=3$. For $m\in\mathbb{N}_{>3/2}$ and $p\geq m-2$ the
operator $\Pi_{p}^{\nabla,s}:H^{m}(\Omega)\rightarrow S_{p+1}(\mathcal{T}%
_{h})$ has following the approximation properties for all $K\in\mathcal{T}_{h}$:%
\begin{align}
&\!\!\! \Vert\varphi-\Pi_{p}^{\nabla,s}\varphi\Vert_{L^{2}(K)}+\frac{h_{K}}{p+1}%
\Vert\varphi-\Pi_{p}^{\nabla,s}\varphi\Vert_{H^{1}(K)}  
  \leq C\left(
\frac{h_{K}}{p+1}\right)  ^{m}\Vert\varphi\Vert_{H^{m}(K)},\\
& \frac{h_{K}}{p+1}\Vert\varphi-\Pi_{p}^{\nabla,s}\varphi\Vert_{H^{1}(\partial
K)}    \leq C\left(  \frac{h_{K}}{p+1}\right)  ^{m-1/2}\Vert\varphi
\Vert_{H^{m}(K)}.
\end{align}
\end{corollary}%

In \cite[Lem.~{8.2}]{MelenkSauterMaxwell_I} approximation operators
$\widehat{\Pi}_{p}^{\operatorname*{curl},s}:{\mathbf{H}}^{1}%
(\operatorname*{curl},\widehat{K})\rightarrow\boldsymbol{\mathcal{N}}%
_{p}^{\operatorname*{I}}(\widehat{K})$ and $\widehat{\Pi}_{p}%
^{\operatorname*{div},s}:{\mathbf{H}}^{1}(\operatorname*{div},\widehat
{K})\rightarrow\mathbf{RT}_{p}(\widehat{K})$ on the reference tetrahedron
$\widehat{K}$ are defined with certain elementwise approximation properties.
Global versions of these operators $\Pi_{p}^{\operatorname*{curl}%
,s}:\mathbf{H}(\operatorname{curl},\Omega)\cap{\prod_{K\in\mathcal{T}_{h}}}{\mathbf{H}}^{1}%
(\operatorname*{curl},%
K%
)\rightarrow\boldsymbol{\mathcal{N}}_{p}^{\operatorname*{I}}(\mathcal{T}_{h})$
and $\Pi_{p}^{\operatorname*{div},s}:\mathbf{H}\left(  \operatorname*{div}%
,\Omega\right)  \mathbf{\cap}{\prod_{K\in\mathcal{T}_{h}}}{\mathbf{H}}^{1}(%
K%
,\operatorname*{div},\widehat{K})\rightarrow\mathbf{RT}_{p}\left(
\mathcal{T}_{h}\right)  $ are characterized by lifting the operators on the
reference element by (cf.~\cite[Def.~{8.1}]{MelenkSauterMaxwell_I})%
\begin{align}
& 
(\Pi_{p}^{\operatorname{curl},s}\mathbf{u})|_{K}\circ F_{K}   :=(F_{K}%
^{\prime})^{-T}\widehat{\Pi}_{p}^{\operatorname{curl},s}((F_{K}^{\prime
})^{\top}\mathbf{u}\circ F_{K}),
\label{DefPiE}\\
& (\Pi_{p}^{\operatorname*{div},s}\mathbf{u})|_{K}   :=(\operatorname*{det}%
(F_{K}^{\prime}))^{-1}F_{K}^{\prime}(\widehat{\Pi}_{p}^{\operatorname*{div}%
,s}(\operatorname*{det}F_{K}^{\prime})(F_{K}^{\prime})^{-1}\mathbf{u}\circ
F_{K}))\circ F_{K}^{-1}.
\label{DefPiF}%
\end{align}
The approximation properties of $\Pi_{p}^{\operatorname{curl},s}$ are inferred
from those of $\widehat{\Pi}_{p}^{\operatorname{curl},s}$ given in
\cite[Lem.~{8.2}]{MelenkSauterMaxwell_I}. We obtain:

\begin{lemma}
\label{lemma:Picurls} Let $m\in\mathbb{N}_{>3/2}$ and $p\geq m-1$. Let
$\widetilde{C}$, $B>0$. Then there are constants $C$, $\sigma>0$ depending
only on $\widetilde{C}$, $B$, $m$, and the constants of
Assumption~\ref{def:element-maps} such that the following holds for the
operator $\Pi_{p}^{\operatorname{curl},s}:\mathbf{H}^{m}(\Omega)\rightarrow%
{\boldsymbol{\mathcal{N}}}_{p}^{\operatorname{I}}%
(\mathcal{T}_{h})$ and all $K\in\mathcal{T}_{h}$:
\begin{enumerate}
[(i)]

\item \label{item:lemma:Picurls-i} If $\mathbf{u}\in\mathbf{H}^{m}(K)$ then%
\begin{align}
\label{eq:lemma:Picurls-10}
& \Vert\mathbf{u}-\Pi_{p}^{\operatorname{curl},s}\mathbf{u}\Vert_{\mathbf{L}%
^{2}(K)}+\frac{h_{K}}{p+1}\Vert\mathbf{u}-\Pi_{p}^{\operatorname{curl}%
,s}\mathbf{u}\Vert_{\mathbf{H}^{1}(K)}    
\leq C\left(\frac{h_{K}}{p+1}\right)^{m}
\Vert\mathbf{u}\Vert_{\mathbf{H}^{m}(K)},
\\
\label{eq:lemma:Picurls-20}%
& \Vert\mathbf{u}-\Pi_{p}^{\operatorname{curl},s}\mathbf{u}\Vert_{\mathbf{L}%
^{2}(\partial K)}  
  \leq C\left(  \frac{h_{K}}{p+1}\right)  ^{m-1/2}%
\Vert\mathbf{u}\Vert_{\mathbf{H}^{m}(K)}. 
\end{align}
\item \label{item:lemma:Picurls-ii} 
If $\mathbf{u}\in\mathcal{A}%
(C_{\mathbf{u}}(K),B,K)$ for some $C_{\mathbf{u}}(K)>0$ and if
\begin{equation}
h_{K}+|k|h_{K}/p\leq\widetilde{C} \label{eq:cond-lemma:Picurls}%
\end{equation}
then
\begin{align}
&  h_{K}^{1/2}\Vert\mathbf{u}-\Pi_{p}^{\operatorname{curl},s}\mathbf{u}%
\Vert_{\mathbf{L}^{2}(\partial{K})}+\Vert\mathbf{u}-\Pi_{p}%
^{\operatorname{curl},s}\mathbf{u}\Vert_{\mathbf{L}^{2}({K})}+h_{K}%
\Vert\mathbf{u}-\Pi_{p}^{\operatorname{curl},s}\mathbf{u}\Vert_{\mathbf{H}%
^{1}({K})}\nonumber\\
&  \qquad\leq CC_{\mathbf{u}}(K)\left(  \left(  \frac{h_{K}}{h_{K}+\sigma
}\right)  ^{p+1}+\left(  \frac{|k|h_{K}}{\sigma p}\right)  ^{p+1}\right)  .
\end{align}

\item \label{item:lemma:Picurls-iii} 
If $\mathbf{u}\in\mathcal{A}%
(C_{\mathbf{u}},B,\Omega)$ for some $C_{\mathbf{u}}>0$ and if
(\ref{eq:cond-lemma:Picurls}) holds, then
\[
\Vert\mathbf{u}-\Pi_{p}^{\operatorname{curl},s}\mathbf{u}\Vert
_{\operatorname{imp},k}\leq C_{\mathbf{u}}|k|\left(  \left(  \frac{h}%
{h+\sigma}\right)  ^{p}+\left(  \frac{|k|h}{\sigma p}\right)  ^{p}\right)  .
\]
\end{enumerate}
\end{lemma}

\begin{proof}
The result follows from modifications of the procedure in \cite[Sec.~{8.3}%
]{MelenkSauterMaxwell_I}. We recall the structure $F_{K}=R_{K}\circ A_{K}$ of
the element maps by Assumption~\ref{def:element-maps}. For $K\in
\mathcal{T}_{h}$ we define $\widetilde{K}:=A_{K}(K)$ and the transformed
functions $\widehat{\mathbf{v}}:=(F_{K}^{\prime})^{\top}\mathbf{v}\circ F_{K}$
on $\widehat{K}$ and $\widetilde{\mathbf{v}}:=(R_{K}^{\prime})^{\top
}\mathbf{v}\circ R_{K}$ on $\widetilde{K}$. We note that $\widehat{\mathbf{v}%
}=(A_{K}^{\prime})^{\top}\widetilde{\mathbf{v}}\circ A_{K}$. By
Assumption~\ref{def:element-maps} and the fact that $A_{K}$ is affine, we have%
\begin{align}
\Vert\widetilde{\mathbf{v}}\Vert_{\mathbf{H}^{j}(\widetilde{K})}\sim
\Vert{\mathbf{v}}\Vert_{\mathbf{H}^{j}({K})},\qquad &  \Vert\widetilde
{\mathbf{v}}\Vert_{\mathbf{L}^{2}(\partial\widetilde{K})}\sim\Vert{\mathbf{v}%
}\Vert_{\mathbf{L}^{2}(\partial{K})},\label{eq:scaling-1}\\
|\widehat{\mathbf{v}}|_{\mathbf{H}^{j}(\widehat{K})}\sim h_{K}^{1+j-3/2}%
|\widetilde{\mathbf{v}}|_{\mathbf{H}^{j}(\widetilde{K})},\quad &
\Vert\widehat{\mathbf{v}}\Vert_{\mathbf{L}^{2}(\partial\widehat{K})}\sim
h_{K}^{1-1}\Vert\widetilde{\mathbf{v}}\Vert_{\mathbf{L}^{2}(\partial
{\widetilde{K}})}, \label{eq:scaling-2}%
\end{align}
where the implied constant depends only on $j$ and the constants of
Assumption~\ref{def:element-maps}.

\emph{Proof of (\ref{item:lemma:Picurls-i}):} {}From \cite[Lem.~{8.2}%
]{MelenkSauterMaxwell_I}, we have for
\begin{equation}
p\Vert\widehat{\mathbf{u}}-\widehat{\Pi}_{p}^{\operatorname{curl},s}%
\widehat{\mathbf{u}}\Vert_{\mathbf{L}^{2}(\widehat{K})}+\Vert\widehat
{\mathbf{u}}-\widehat{\Pi}_{p}^{\operatorname{curl},s}\widehat{\mathbf{u}%
}\Vert_{\mathbf{H}^{1}(\widehat{K})}\leq Cp^{-(m-1)}|\mathbf{u}|_{\mathbf{H}%
^{m}(\widehat{K})}. \label{eq:lemma:Picurls-100}%
\end{equation}
This approximation result and the scaling argument expressed in
(\ref{eq:scaling-1}), (\ref{eq:scaling-2}) produce (\ref{eq:lemma:Picurls-10}%
). The multiplicative trace inequality $\Vert\widehat{\mathbf{v}%
}\Vert_{\mathbf{L}^{2}(\partial\widehat{K})}^{2}\leq C\Vert\widehat
{\mathbf{v}}\Vert_{\mathbf{L}^{2}(\widehat{K})}\Vert\widehat{\mathbf{v}}%
\Vert_{\mathbf{H}^{1}(\widehat{K})}$ applied to (\ref{eq:lemma:Picurls-100})
and similar scaling arguments produce (\ref{eq:lemma:Picurls-20}).

\emph{Proof of (\ref{item:lemma:Picurls-ii}):} By \cite[Lem.~{8.4}%
]{MelenkSauterMaxwell_I}, the pull-back $\widehat{\mathbf{u}}
\in \mathcal{A}(CC_{\mathbf{u}}(K)h_{K}^{1-3/2},h_{K}B^{\prime},\widehat{K})$ for
some $B^{\prime}$ depending only on $B$ and the constants of
Assumption~\ref{def:element-maps}. By \cite[Lem.~{8.2}]{MelenkSauterMaxwell_I}
there are constants depending only on $B$ and the constants of
Assumption~\ref{def:element-maps} such that
\[
\Vert\widehat{\mathbf{u}}-\widehat{\Pi}_{p}^{\operatorname{curl},s}%
\widehat{\mathbf{u}}\Vert_{W^{2,\infty}(\widehat{K})}\leq Ch_{K}%
^{1-3/2}C_{\mathbf{u}}(K)\left(  \left(  \frac{h_{K}}{h_{K}+\sigma}\right)
^{p+1}+\left(  \frac{|k|h_{K}}{\sigma p}\right)  ^{p+1}\right)  .
\]
With similar scaling arguments as in the proof of (\ref{item:lemma:Picurls-i}%
), we obtain the stated estimate.

\emph{Proof of (\ref{item:lemma:Picurls-iii}):} For each $K\in\mathcal{T}_{h}%
$, we define
\[
\widetilde{C}_{\mathbf{u}}^{2}(K):=\sum_{n=0}^{\infty}\frac{|\mathbf{u}%
|_{\mathbf{H}^{n}(K)}^{2}}{(2B)^{2n}\max(n+1,|k|)^{2n}}%
\]
and note
\[
\mathbf{u}\in\mathcal{A}(\widetilde{C}_{\mathbf{u}}(K),2B,K) \qquad\mbox{with}
\qquad\sum_{K\in\mathcal{T}_{h}}\widetilde{C}_{\mathbf{u}}^{2}(K)\leq
2C_{\mathbf{u}}^{2}.
\]
We then sum the elementwise error estimates provided by
(\ref{item:lemma:Picurls-ii}).
\end{proof}


\subsection{An interpolating projector on the finite element space}


For the error analysis, the following subspaces of $\mathbf{H}^{1}( \Omega) $
will play an important role:

\begin{equation}
\mathbf{V}_{k,0}:=\left\{  \mathbf{u}\in\mathbf{X}_{\operatorname*{imp}}%
\mid
\innerprod{
\mathbf{u},\nabla\varphi
}
=0\quad\forall\varphi\in H_{\operatorname*{imp}}^{1}\left(
\Omega\right)  \right\}  . \label{DefV0}%
\end{equation}


\begin{proposition}
\label{PropV0incl}Let $\Omega$ be a bounded Lipschitz domain with simply
connected, analytic  boundary. The space $\mathbf{V}_{k,0}$ can alternatively be characterized by%
\begin{equation}
\mathbf{V}_{k,0}=\left\{  \mathbf{u}\in\mathbf{X}_{\operatorname*{imp}}%
\mid\operatorname*{div}\mathbf{u}=0\wedge\operatorname*{i}k\left\langle
\mathbf{u},\mathbf{n}\right\rangle +\operatorname*{div}\nolimits_{\Gamma
}\mathbf{u}_{T}=0\quad\text{on }\Gamma\right\}  . \label{PropVcharact}%
\end{equation}

\end{proposition}

The proof of this proposition is standard and uses the same arguments as,
e.g., \cite[Lem.~{4.10}]{MelenkSauterMaxwell_I}.

\begin{proposition}
\label{PropEmbVk0}Let $\Omega$ be a bounded Lipschitz domain with simply
connected analytic boundary. It holds $\displaystyle\mathbf{V}_{k,0}\subset
\mathbf{H}^{1}(\Omega)$, and there exists $c>0$ independent of $k$ such that%
\[
c\left\vert k\right\vert \left\Vert \mathbf{v}\right\Vert _{\mathbf{H}%
^{1}\left(  \Omega\right)  ,k}\leq\left\Vert \mathbf{v}\right\Vert
_{\mathbf{H}\left(  \operatorname*{curl},\Omega\right)  ,k}\leq\left\Vert
\mathbf{v}\right\Vert _{\operatorname*{imp},k}\qquad\forall\mathbf{v}%
\in\mathbf{V}_{k,0}.
\]

\end{proposition}

\begin{proof}
The estimate $\Vert\mathbf{v}\Vert_{\mathbf{H}(\operatorname{curl},\Omega
),k}\leq\Vert\mathbf{v}\Vert_{\operatorname{imp},k}$ follows directly from the
definition of the norms. For the lower bound, we employ the Helmholtz
decomposition of $\mathbf{v}\in\mathbf{V}_{k,0}$ 
as in Lemma~\ref{lemma:helmholtz-a-la-schoeberl}%
(\ref{item:lemma:helmholtz-a-la-schoeberl-i}) 
and take into account $\operatorname*{div}\mathbf{v}=0$. That is, there
exist $\mathbf{w}\in\mathbf{H}^{1}(\Omega)$ and $\varphi\in H^{1}(\Omega)$
with
\begin{align}
\mathbf{v}& =\nabla\varphi+\mathbf{w} , 
& \Vert \mathbf{w}\Vert _{\mathbf{H}^{1}(\Omega)}& \leq C\left\Vert
\operatorname*{curl}\mathbf{v}\right\Vert . \label{wcurlv}%
\end{align}
Since $\operatorname*{div}\mathbf{v}=0$ we conclude $\Vert \mathbf{v}%
\Vert _{\mathbf{H}(\operatorname*{div},\Omega)}=\Vert
\mathbf{v}\Vert $ so that a trace theorem gives us%
\[
\Vert \langle \mathbf{v},\mathbf{n}\rangle \Vert
_{H^{-1/2}(  \Gamma)  }\leq C\Vert \mathbf{v}\Vert
_{\mathbf{H}(\operatorname*{div},\Omega)}=C\Vert \mathbf{v}\Vert .
\]
It holds%
\[
\Delta_{\Gamma}(  \varphi\vert _{\Gamma})
=\operatorname*{div}\nolimits_{\Gamma}(  \mathbf{v}_{T}-\mathbf{w}%
_{T})  \overset{\text{Prop.~\ref{PropV0incl}}}{=}-\operatorname*{i}%
k\langle \mathbf{v},\mathbf{n}\rangle-\operatorname*{div}%
\nolimits_{\Gamma}\mathbf{w}_{T}=:\tilde{v}\text{.}%
\]
By the smoothness of the closed manifold $\Gamma$ and the shift
properties of the Laplace-Beltrami operator we get 
\begin{align*}
\Vert \varphi\Vert _{H^{3/2}(  \Gamma)  } 
& \leq
C\Vert \tilde{v}\Vert _{H^{-1/2}(  \Gamma)  }
\leq C \left[ \Vert \operatorname*{div}\nolimits_{\Gamma}\mathbf{w}_{T}\Vert
_{H^{-1/2}\left(  \Gamma\right)  }+\left\vert k\right\vert \Vert
\langle \mathbf{v},\mathbf{n}\rangle \Vert _{H^{-1/2}( \Gamma)  } 
\right]\\
& 
\leq C\left[  \Vert \mathbf{w}_{T}\Vert
_{\mathbf{H}^{1/2}(  \Gamma)  }+|k|  \Vert
\mathbf{v}\Vert \right] 
  \leq C\left(  \left\Vert \mathbf{w}\right\Vert _{\mathbf{H}^{1}\left(
\Omega\right)  }+\left\vert k\right\vert \left\Vert \mathbf{v}\right\Vert
\right)  
\\
& \overset{\text{(\ref{wcurlv})}}{\leq}C\left\Vert \mathbf{v}%
\right\Vert _{\mathbf{H}\left(  \operatorname*{curl},\Omega\right)  ,k}.
\end{align*}
%
%
Since $\varphi$ solves
\[
-\Delta\varphi=\operatorname*{div}\mathbf{w}\text{\textbf{\quad}in }\Omega,
\]
the shift theorem for the Laplace operator on smooth domains leads to%
\begin{equation}
\Vert \nabla\varphi\Vert _{\mathbf{H}^{1}(  \Omega)
}\leq C\!\left(  \Vert \operatorname*{div}\mathbf{w}\Vert +\Vert
\varphi\Vert _{H^{3/2}(  \Gamma)  }\right)  \leq C\!\left(
\Vert \operatorname*{curl}\mathbf{v}\Vert +\Vert
\mathbf{v}\Vert _{\mathbf{H}(  \operatorname*{curl},\Omega)
,k}\right)  . \label{nablaphiest}%
\end{equation}
The combination of (\ref{wcurlv}) and (\ref{nablaphiest}) shows that
$\mathbf{v}\in\mathbf{H}^{1}(  \Omega)  $ and%
\[
\left\Vert \mathbf{v}\right\Vert _{\mathbf{H}^{1}\left(  \Omega\right)  }\leq
C\left\Vert \mathbf{v}\right\Vert _{\mathbf{H}\left(  \operatorname*{curl}%
,\Omega\right)  ,k}\leq C\left\Vert \mathbf{v}\right\Vert
_{\operatorname*{imp},k}.
\]
Since we have trivially $\left\vert k\right\vert \left\Vert \mathbf{v}%
\right\Vert \leq\left\Vert \mathbf{v}\right\Vert _{\mathbf{H}\left(
\operatorname*{curl},\Omega\right)  ,k}$, the assertion follows.%
\end{proof}

We also need the following subspace of $\mathbf{V}_{k,0}$  
given by%
\begin{equation}
{\mathbf{V}}_{k,0,h}:=\left\{  {\mathbf{v}}\in{\mathbf{V}}_{k,0}%
\mid\operatorname*{curl}{\mathbf{v}}\in\operatorname*{curl}{\mathbf{X}}%
_{h}\right\}  . \label{eq:Vk0h}%
\end{equation}

The operator $\Pi_{p}^{\operatorname*{curl},s}$ in (\ref{DefPiE}),
(\ref{DefPiF}) has ($p$-optimal) approximation properties in $\Vert\cdot
\Vert_{\operatorname*{curl},\Omega,k}$ as it has simultaneously $p$-optimal
approximation properties in $L^{2}$ and $H^{1}$. However, it is not a
projection and does not have the commuting diagram property. Since this is
needed for the estimate of the consistency term in Section \ref{SecConsAna} we
employ operators, $\Pi_{p}^{\operatorname{curl},c}$, $\Pi_{p}%
^{\operatorname*{div},c}$, which enjoy these properties. They were constructed
in \cite{melenk_rojik_2018} in an element-by-element fashion and used in
\cite[Thm.~{8.2}]{MelenkSauterMaxwell_I}. The choice $\Pi_{h}^{E}%
:\mathbf{V}_{k,0,h}+\mathbf{X}_{h}\rightarrow\mathbf{X}_{h}$ as $\Pi
_{p}^{\operatorname{curl},c}$ and the companion operator $\Pi_{h}%
^{F}:H(\operatorname{div},\Omega)\cap\prod_{K\in{\mathcal{T}}_{h}}{\mathbf{H}%
}^{1}(\operatorname{div},K)\rightarrow\operatorname*{curl}\mathbf{X}_{h}$ as
$\Pi_{p}^{\operatorname*{div},c}$ allows us to derive quantitative convergence
estimates in Section~\ref{SecStabConv}.

\begin{lemma}
\label{AdiscSp}
The operators $\Pi_{h}^{E}:=\Pi_{p}^{\operatorname*{curl},c}$ and $\Pi_{h}%
^{F}:=\Pi_{p}^{\operatorname*{div},c}$ satisfy the following properties:
$\Pi_{h}^{E}:\mathbf{V}_{k,0,h}+\mathbf{X}_{h}\rightarrow{\mathbf{X}}_{h}$ and
$\Pi_{h}^{F}:H(\operatorname{div},\Omega)\cap\prod_{K\in{\mathcal{T}}_{h}%
}{\mathbf{H}}^{1}(\operatorname{div},K)\rightarrow\operatorname*{curl}%
\mathbf{X}_{h}$ are linear mappings with

\begin{enumerate}
[(i)] 

\item \label{item:AdiscSp-a} $\Pi_{h}^{E}$ is a projection, i.e., the
restriction $\left.  \Pi_{h}^{E}\right\vert _{\mathbf{X}_{h}}$ is the identity
on $\mathbf{X}_{h}$.

\item \label{item:AdiscSp-b}The operators $\Pi_{h}^{E}$ have the commuting
property: $\operatorname*{curl}\Pi_{h}^{E}=\Pi_{h}^{F}\operatorname*{curl}$.
\end{enumerate}
\end{lemma}

\begin{proof}
Since $\Pi_{p}^{\operatorname*{curl},c}$ is based on an element-by-element
construction it is well defined on $\mathbf{H}(\operatorname*{curl}%
,\Omega)\cap{\prod_{K\in\mathcal{T}_{h}}}\mathbf{H}^{1}(\operatorname*{curl}%
,K)$. Since $\mathbf{V}_{k,0,h}+\mathbf{X}_{h}$ is a subspace of this space,
the mapping properties follow.
The projection property of $\Pi_{h}^{E}$ and the commuting property of
$\Pi_{h}^{E}$ and $\Pi_{h}^{F}$ are proved in \cite[Thm.~{2.10}, Rem.~{2.11}%
]{melenk_rojik_2018}.%
\end{proof}


\section{%
Stability and convergence of the Galerkin discretization\label{SecStabConv}}

%

The wavenumber-explicit stability and convergence analysis for Maxwell's
equations with transparent boundary conditions has been developed recently in
\cite{MelenkSauterMaxwell_I} and generalizes the theory in \cite[Sec.~{7.2}%
]{Monkbook}. A ``roadmap'' for the convergence proof of \cite{MelenkSauterMaxwell_I}
is given in \cite[Sec.~{1.1}--{1.3}]{MelenkSauterMaxwell_I}. In the present analysis, 
we follow this ``roadmap'' taking into account the change in boundary conditions
from transparent boundary conditions to impedance boundary conditions. 
A key role is played by the term 
$\innerprod{
\mathbf{u},\mathbf{v}%
} = 
A_{k}( \mathbf{u} ,\mathbf{v}) - 
\left(  \operatorname*{curl}%
\mathbf{u,}\operatorname*{curl}\mathbf{v}\right)$  from 
(\ref{eq:def:(())}), which includes the boundary conditions. This sesquilinear form  
determines the space $\mathbf{V}_{k,0}$ (see (\ref{DefV0})) and the regular decomposition
in Def.~\ref{DefHelmsplit} ahead and its properties differentiate the present case of 
impedance boundary conditions from the transparent boundary condition case. 
Compared to the case of transparent boundary conditions, the present impedance boundary conditions
case is simpler in that fewer approximation quantities $\eta^{\operatorname{alg}}_j$, $\tilde \eta_j^{\operatorname{alg}}$ 
are required in the analysis.

In this section, we develop a stability and convergence theory for
Maxwell's equations with impedance boundary conditions, see
Sect.~\ref{SecMW_Imp}. Recall the definition of the sesquilinear form 
$
\innerprod{
\cdot,\cdot%
}$ of (\ref{eq:def:(())}) and of the norm $\left\Vert
\cdot\right\Vert _{k,+}$ in Definition \ref{def:Ximp}.

We introduce the quantity $\delta_{k}:\mathbf{X}_{\operatorname*{imp}%
}\rightarrow\mathbb{R}$ by $\delta_{k}\left(  \mathbf{0}\right)  :=0\ $and for
$\mathbf{w}\in\mathbf{X}_{\operatorname*{imp}}\backslash\left\{
\mathbf{0}%
\right\}  $ by
\begin{equation}
\delta_{k}(\mathbf{w}):=\sup_{\mathbf{v}_{h}\in\mathbf{X}_{h}\backslash
\left\{
\mathbf{0}%
\right\}  }\left(  2\frac{\vert 
\innerprod{
\mathbf{w},\mathbf{v}_{h}%
}\vert }{\Vert\mathbf{w}\Vert_{\operatorname*{imp},k}%
\Vert\mathbf{v}_{h}\Vert_{\operatorname*{imp},k}}\right)  ,
\label{delta_k_def}%
\end{equation}
which will play the important role of a consistency term.

\begin{proposition}
[quasi-optimality]\label{prop:error-est} Let $\mathbf{E}\in\mathbf{X}%
_{\operatorname*{imp}}$ and $\mathbf{E}_{h}\in\mathbf{X}_{h}$ satisfy%
\[
A_{k}( \mathbf{E}-\mathbf{E}_{h},\mathbf{v}_{h}) =0\qquad\forall
\,\mathbf{v}_{h}\in\mathbf{X}_{h}.
\]
Assume that $\delta_{k}( \mathbf{e}_{h}) <1$ for $\mathbf{e}_{h}%
:=\mathbf{E}-\mathbf{E}_{h}$. Then, $\mathbf{e}_{h}$ satisfies, for all
$\mathbf{w}_{h}\in\mathbf{X}_{h}$, the quasi-optimal error estimate%
\[
\left\Vert \mathbf{e}_{h}\right\Vert _{\operatorname*{imp},k}\leq
\frac{1+\delta_{k}( \mathbf{e}_{h}) }{1-\delta_{k}( \mathbf{e}_{h})
}\left\Vert \mathbf{E}-\mathbf{w}_{h}\right\Vert _{\operatorname*{imp},k}.
\]

\end{proposition}

\begin{proof}
The definitions of the sesquilinear forms $A_{k}$ and $
\innerprod{
\cdot,\cdot
}$ imply%
\begin{equation}
\left\Vert \mathbf{e}_{h}\right\Vert _{_{\operatorname*{imp},k}}%
^{2}=\left\vert A_{k}( \mathbf{e}_{h},\mathbf{e}_{h}) +2
\innerprod{
\mathbf{e}_{h},\mathbf{e}_{h}%
}\right\vert . \label{eq:prop:error-est-10}%
\end{equation}%
We employ Galerkin orthogonality for the first term in
(\ref{eq:prop:error-est-10}) to obtain for any $\mathbf{w}_{h}\in
\mathbf{X}_{h}$%
\[
\left\Vert \mathbf{e}_{h}\right\Vert _{\operatorname*{imp},k}^{2}%
\leq\bigl\vert A_{k}( \mathbf{e}_{h},\mathbf{E}-\mathbf{w}_{h}) +2
\innerprod{
\mathbf{e}_{h},\mathbf{E}-\mathbf{w}_{h}%
}\bigr\vert +\delta_{k}\left(  {\mathbf{e}}_{h}\right)
\bigl\Vert \mathbf{e}_{h}\bigr\Vert _{\operatorname*{imp},k}\!\!\!\!
\underset
{\leq\left\Vert \mathbf{e}_{h}\right\Vert _{\operatorname*{imp},k}+\left\Vert
\mathbf{E}-\mathbf{w}_{h}\right\Vert _{\operatorname*{imp},k}}{\underbrace
{\left\Vert \mathbf{E}_{h}-\mathbf{w}_{h}\right\Vert _{\operatorname*{imp},k}%
}}.
\]
We write $A_{k}$ in the form (\ref{eq:Ak-alternative}) so that%
\begin{align}
\left(  1-\delta_{k}( \mathbf{e}_{h}) \right)  \left\Vert \mathbf{e}%
_{h}\right\Vert _{\operatorname*{imp},k}^{2}  &  \leq\left\vert \left(
\operatorname{curl}\mathbf{e}_{h},\operatorname{curl}\left(  \mathbf{E}%
-\mathbf{w}_{h}\right)  \right)  +
\innerprod{
\mathbf{e}_{h},\mathbf{E}-\mathbf{w}_{h}%
} \right\vert \label{errest1mdelta}\\
&\quad   +\delta_{k}( {\mathbf{e}}_{h}) \left\Vert \mathbf{e}_{h}\right\Vert
_{\operatorname*{imp},k}\left\Vert \mathbf{E}-\mathbf{w}_{h}\right\Vert
_{\operatorname*{imp},k}.\nonumber
\end{align}

The sesquilinear form $
\innerprod{
\cdot,\cdot%
} $ is continuous, and we have%
\begin{equation}
\bigl|
\innerprod{
\mathbf{u},\mathbf{v}%
}
\bigr| \leq\left\Vert \mathbf{u}\right\Vert _{k,+}\left\Vert
\mathbf{v}\right\Vert _{k,+}\qquad\forall\mathbf{u},\mathbf{v}\in
\mathbf{X}_{\operatorname*{imp}}. \label{Cont2prod}%
\end{equation}
Hence,
\begin{align*}
\left(  1-\delta_{k}( \mathbf{e}_{h}) \right)  \left\Vert \mathbf{e}%
_{h}\right\Vert _{\operatorname*{imp},k}^{2}
&\leq 
\left\Vert \mathbf{e}%
_{h}\right\Vert _{\operatorname*{imp},k}\left\Vert \mathbf{E}-\mathbf{w}%
_{h}\right\Vert _{\operatorname*{imp},k}
\\
& \qquad \mbox{} +\delta_{k}( {\mathbf{e}}_{h})
\left\Vert \mathbf{e}_{h}\right\Vert _{\operatorname*{imp},k}\left\Vert
\mathbf{E}-\mathbf{w}_{h}\right\Vert _{\operatorname*{imp},k},
\end{align*}
and the assertion follows.%
\end{proof}


\subsection{Splitting of the consistency term}


We introduce continuous and discrete Helmholtz decompositions that are adapted 
to the problem under consideration.

\begin{definition}
\label{DefHelmsplit} On $\mathbf{v}\in\mathbf{X}_{\operatorname*{imp}}$ the
\emph{Helmholtz splittings}%
\begin{subequations}
\label{Helmtot}
\begin{align}
\mathbf{v}  &  =\Pi_{k}^{\operatorname*{curl}}\mathbf{v}+\Pi_{k}^{\nabla
}\mathbf{v},
\label{Helmtota}\\
\mathbf{v}  &  =\Pi_{k,h}^{\operatorname*{curl}}\mathbf{v}+\Pi_{k,h}^{\nabla
}\mathbf{v} 
\label{Helmtotb}%
\end{align}
\end{subequations}%
are given via operators $\Pi_{k}^{\nabla}$, $\Pi_{k}^{\operatorname*{curl}}$
and their discrete counterparts $\Pi_{k,h}^{\nabla}$, $\Pi_{k,h}%
^{\operatorname*{curl}}$ by seeking $\Pi_{k}^{\nabla}\mathbf{v}\in\nabla
H_{\operatorname*{imp}}^{1}( \Omega) $ and $\Pi_{k,h}^{\nabla}\mathbf{v}%
\in\nabla S_{h}$ such that%
\begin{subequations}
\label{HelmDecomp}
\begin{align}
\bdinnerprod{
\Pi_{k}^{\nabla}\mathbf{v},\nabla\psi%
}  &  =
\innerprod{
\mathbf{v},\nabla\psi%
}\qquad\forall\psi\in H_{\operatorname*{imp}}^{1}( \Omega) ,
\label{HelmDecompa}\\
\bdinnerprod{
\Pi_{k,h}^{\nabla}\mathbf{v},\nabla\psi%
}  &  =\innerprod{
\mathbf{v},\nabla\psi%
}\qquad\forall\psi\in S_{h}. 
\label{HelmDecompb}%
\end{align}
\end{subequations}%
The operators $\Pi_{k}^{\operatorname*{curl}}\mathbf{v}$, $\Pi_{k,h}%
^{\operatorname*{curl}}\mathbf{v}$ are then given via the relations
(\ref{Helmtot}).
\end{definition}

It is easy to see (cf.\ (\ref{Akcommute})) that
\begin{subequations}
\label{sign_prop}
\begin{align}
\bdinnerprod{
\nabla\psi,\Pi_{-k}^{\nabla}\mathbf{v}%
}  &  =\innerprod{
\nabla\psi,\mathbf{v}%
}\qquad\forall\psi\in H_{\operatorname*{imp}}^{1}( \Omega) ,
\label{sign_propa}\\
\bdinnerprod{
\nabla\psi,\Pi_{-k,h}^{\nabla}\mathbf{v}%
}  &  =\innerprod{
\nabla\psi,\mathbf{v}%
}\qquad\forall\psi\in S_{h}. 
\label{sign_propb}%
\end{align}
\end{subequations}%

Solvability of these equations follows trivially from the Lax-Milgram lemma as
can be seen from the following lemma.

\begin{lemma}
\label{LemCurlGradProj}Problems (\ref{HelmDecomp}) have unique solutions,
which satisfy%
\begin{align*}
\bigl\Vert \Pi_{k}^{\nabla}\mathbf{v}\bigr\Vert _{\operatorname*{imp}%
,k}+\bigr\Vert \Pi_{k}^{\operatorname*{curl}}\mathbf{v}\bigr\Vert
_{\operatorname*{imp},k} &  \leq C\bigl\Vert \mathbf{v}\bigr\Vert
_{\operatorname*{imp},k},\\
\bigl\Vert \Pi_{k,h}^{\nabla}\mathbf{v}\bigr\Vert _{\operatorname*{imp}%
,k}+\bigl\Vert \Pi_{k,h}^{\operatorname*{curl}}\mathbf{v}\bigr\Vert
_{\operatorname*{imp},k} &  \leq C\bigl\Vert \mathbf{v}\bigr\Vert
_{\operatorname*{imp},k}.
\end{align*}

\end{lemma}

\begin{proof}
We first consider the continuous problem (\ref{HelmDecompa}). Taking 
$\sigma=\exp\left(  \left(  \operatorname*{sign}k\right)  \operatorname*{i}%
\frac{\pi}{4}\right)  $ we obtain coercivity as in the proof of
Theorem~\ref{lemma:apriori-with-good-sign}%
(\ref{item:lemma:apriori-with-good-sign-0}) via
\[
\operatorname{Re}\innerprod{
\nabla\psi,\sigma\nabla\psi%
}=2^{-1/2}\! \left[  k^{2}\left(  \nabla\psi,\nabla\psi\right)
+\left\vert k\right\vert \left(  \nabla_{\Gamma}\psi,\nabla_{\Gamma}%
\psi\right)  _{\mathbf{L}^{2}(  \Gamma)  }\right]  \! =2^{-1/2}%
\Vert\nabla\psi\Vert_{\operatorname{imp},k}^{2}.
\]
The continuity follows from (\ref{Cont2prod}):%
\begin{equation}
\left\vert \innerprod{
\nabla\varphi,\nabla\psi%
}\right\vert \leq\left\Vert \nabla\varphi\right\Vert
_{k,+}\left\Vert \nabla\psi\right\Vert _{k,+}=\left\Vert \nabla\varphi
\right\Vert _{\operatorname*{imp},k}\left\Vert \nabla\psi\right\Vert
_{\operatorname*{imp},k}. \label{2prodcont}%
\end{equation}
This implies existence, uniqueness, and the \textsl{a priori} estimate%
\[
\bigl\Vert \Pi_{k}^{\nabla}\mathbf{v}\bigr\Vert _{\operatorname*{imp},k}%
\leq\sqrt{2}\left\Vert \mathbf{v}\right\Vert _{\operatorname*{imp},k}.
\]
The estimate of $\Pi_{k}^{\operatorname*{curl}}\mathbf{v}$ follows by a
triangle inequality. Since the coercivity and continuity estimates are
inherited by the finite dimensional subspace $\nabla S_{h}$, well-posedness
also follows on the discrete level. The estimates for the other operators
follow verbatim.%
\end{proof}

The principal splitting of the consistency term $\delta_{k}$ in
(\ref{delta_k_def}) is introduced next. We write%

\begin{align}
\label{mainsplittingerror}%
& \innerprod{
 \mathbf{e}_{h},\mathbf{v}_{h}%
} = \\
\nonumber 
&  \quad 
\underset{=:T_{1}}{\underbrace{\bdinnerprod{
\mathbf{e}_{h},\left(  \Pi_{-k,h}^{\operatorname*{curl}}-\Pi_{-k}%
^{\operatorname*{curl}}\right)  \mathbf{v}_{h}%
}}}+\underset{=:T_{2}}{\underbrace{\bdinnerprod{
\mathbf{e}_{h},\Pi_{-k}^{\operatorname*{curl}}\mathbf{v}_{h}%
}}}+\underset{=:T_{3}}{\underbrace{\bdinnerprod{
\mathbf{e}_{h},\Pi_{-k,h}^{\nabla}\mathbf{v}_{h}%
}}}. 
\end{align}
Galerkin orthogonality implies $\innerprod{
\mathbf{e}_{h},\Pi_{-k,h}^{\nabla}\mathbf{v}_{h}%
}=0$, i.e., $T_{3}=0$.


\subsection{Consistency analysis: the term $\mathbf{T}_{1}$ in
(\ref{mainsplittingerror})\label{SecConsAna}}


The continuity of the sesquilinear form $\innerprod{
\cdot,\cdot%
}$ (cf.\ (\ref{Cont2prod})) implies%
\begin{equation}
\left\vert T_{1}\right\vert \leq\left\Vert \mathbf{e}_{h}\right\Vert
_{k,+}\bigl\Vert \bigl(  \Pi_{-k,h}^{\operatorname*{curl}}-\Pi_{-k}%
^{\operatorname*{curl}}\bigr)  \mathbf{v}_{h}\bigr\Vert _{k,+}.
\label{estT1}%
\end{equation}

The definition of the discrete and continuous Helmholtz decomposition applied
to a discrete function $\mathbf{v}_{h}$ leads to (cf.\ Def.~\ref{DefHelmsplit}%
, (\ref{sign_prop}))%
\begin{equation}
\bdinnerprod{
\nabla\psi_{h},\left(  \Pi_{-k,h}^{\operatorname*{curl}}-\Pi_{-k}%
^{\operatorname*{curl}}\right)  \mathbf{v}_{h}%
}=0\quad\forall\psi_{h}\in S_{h}. \label{eq:galerkin-Pi}%
\end{equation}
We use (\ref{Helmtot}) to get $\operatorname*{curl}\Pi_{-k}%
^{\operatorname*{curl}}=\operatorname*{curl}\Pi_{-k,h}^{\operatorname*{curl}%
}=\operatorname*{curl}$ on $\mathbf{X}_{h}$ and thus%
\begin{align}
\label{eq:intro-2000}%
& \operatorname*{curl}\bigl(  \Pi_{-k,h}^{\operatorname*{curl}}{\mathbf{v}}%
_{h}-\Pi_{h}^{E}\Pi_{-k}^{\operatorname*{curl}}{\mathbf{v}}_{h}\bigr)   
\overset{\text{Lem.~\ref{AdiscSp}(\ref{item:AdiscSp-b}) }}{=}%
\\ 
\nonumber 
& \quad \operatorname*{curl}\bigl(  \Pi_{-k,h}^{\operatorname*{curl}}{\mathbf{v}}%
_{h}\bigr)  -\Pi_{h}^{F}\operatorname*{curl}\bigl(  \Pi_{-k}%
^{\operatorname*{curl}}{\mathbf{v}}_{h}\bigr)  =
 \operatorname*{curl}%
{\mathbf{v}}_{h}-\Pi_{h}^{F}\operatorname*{curl}{\mathbf{v}}_{h} 
\overset{\text{Lem.~\ref{AdiscSp}(\ref{item:AdiscSp-b})}}{=}%
\\ \nonumber 
& \quad 
\operatorname*{curl}{\mathbf{v}}_{h}-\operatorname*{curl}\Pi_{h}%
^{E}{\mathbf{v}}_{h}\overset{\text{Lem.~\ref{AdiscSp}(\ref{item:AdiscSp-a}) }%
}{=}\operatorname*{curl}\left(  {\mathbf{v}}_{h}-{\mathbf{v}}_{h}\right)  =0.
\end{align}
By the exact sequence property (\ref{exdiscseq_rm}), the observation
(\ref{eq:intro-2000}) implies that $\Pi_{-k,h}^{\operatorname*{curl}%
}\mathbf{v}_{h}-\Pi_{h}^{E}\Pi_{-k}^{\operatorname*{curl}}\mathbf{v}%
_{h}=\nabla\psi_{h}$ for some $\psi_{h}\in S_{h}$ and therefore 
\begin{equation}
\bigl(  \Pi_{-k,h}^{\operatorname*{curl}}-\Pi_{-k}^{\operatorname*{curl}%
}\bigr)  \mathbf{v}_{h}=\nabla\psi_{h}+\bigl(  (  \Pi_{h}^{E}-I)
\Pi_{-k}^{\operatorname*{curl}}\bigr)  \mathbf{v}_{h}.
\label{eq:res-of-exact-sequence}%
\end{equation}

For the second factor in (\ref{estT1}) we get by the Galerkin orthogonality
(\ref{eq:galerkin-Pi}) and (\ref{eq:res-of-exact-sequence})%
\begin{align*}
\left\Vert \left(  \Pi_{-k,h}^{\operatorname*{curl}}-\Pi_{-k}%
^{\operatorname*{curl}}\right)  \mathbf{v}_{h}\right\Vert _{k,+}^{2} \!\!& = 
\operatorname{Re}\bdinnerprod{
  \bigl(  \Pi_{h}^{E}-I)  \Pi_{-k}^{\operatorname*{curl}}
\mathbf{v}_{h},(  \Pi_{-k,h}^{\operatorname*{curl}}-\Pi_{-k}%
^{\operatorname*{curl}}\bigr)  \mathbf{v}_{h}%
}\\
&  +\left(  \operatorname*{sign}k\right)  \operatorname{Im}\bdinnerprod{
\  
  \bigl(  \Pi_{h}^{E}-I\bigr)  \Pi_{-k}^{\operatorname*{curl}}
\mathbf{v}_{h},\bigl(  \Pi_{-k,h}^{\operatorname*{curl}}-\Pi_{-k}%
^{\operatorname*{curl}}\bigr)  \mathbf{v}_{h}%
}\\
& \leq   2\bigl\Vert \bigl(  \Pi_{h}^{E}-I\bigr)  \Pi_{-k}%
^{\operatorname*{curl}}\mathbf{v}_{h}\bigr\Vert _{k,+}\bigl\Vert
\bigl(  \Pi_{-k,h}^{\operatorname*{curl}}-\Pi_{-k}^{\operatorname*{curl}%
}\bigr)  \mathbf{v}_{h}\bigr\Vert _{k,+}%
\end{align*}
so that%
\[
\bigl\Vert \bigl(  \Pi_{-k,h}^{\operatorname*{curl}}-\Pi_{-k}%
^{\operatorname*{curl}}\bigr)  \mathbf{v}_{h}\bigr\Vert _{k,+}%
\leq2\bigr\Vert \bigl(  \Pi_{h}^{E}-I\bigr)  \Pi_{-k}%
^{\operatorname*{curl}}  \mathbf{v}_{h}\bigr\Vert _{k,+}.
\]
This leads to the estimate of $T_{1}$%
\[
\left\vert T_{1}\right\vert \leq2\left\Vert \mathbf{e}_{h}\right\Vert
_{k,+}\bigl\Vert \bigl(  I-\Pi_{h}^{E}\bigr)  \Pi_{-k}^{\operatorname*{curl}%
}\mathbf{v}_{h}\bigr\Vert _{k,+}.
\]
We set\footnote{%
Our choice of notation is motivated by the appearance of similar approximation
quantities in the companion paper \cite{MelenkSauterMaxwell_I} (for transparent boundary
conditions), where various measures of approximability $\eta_{j}^{\exp}$,
$j\in\left\{  1,3,4,5,7\right\}  $, and $\eta_{j}^{\operatorname*{alg}}$,
$j\in\left\{  2,6\right\}  $, were introduced and used in the
convergence analysis. Here, only the quantities $\eta_{6}^{\operatorname*{alg}%
}$ in (\ref{eq:eta-6}) and $\tilde{\eta}_{2}^{\operatorname*{alg}}$ in
(\ref{eq:eta2-alg}) are needed.
}%
\begin{equation}
\eta_{6}^{\operatorname*{alg}}:=\eta_{6}^{\operatorname{alg}}\bigl(
\mathbf{X}_{h},\Pi_{h}^{E}\bigr)  :=\sup_{\substack{\mathbf{w}\in
\mathbf{V}_{-k,0}\backslash\left\{
\mathbf{0}%
\right\}  \colon\\\operatorname*{curl}{\mathbf{w}}\in\operatorname*{curl}%
{\mathbf{X}}_{h}}}\frac{\bigl\Vert \mathbf{w}-\Pi_{h}^{E}\mathbf{w}\bigr\Vert
_{k,+}}{\Vert \mathbf{w}\Vert _{\mathbf{H}^{1}( \Omega)  }}\label{eq:eta-6}%
\end{equation}
and obtain%
\begin{align}
\nonumber 
\left\vert T_{1}\right\vert & \leq2\Vert \mathbf{e}_{h}\Vert
_{k,+}\eta_{6}^{\operatorname*{alg}}\bigl\Vert \Pi_{-k}^{\operatorname*{curl}%
}\mathbf{v}_{h}\bigr\Vert _{\mathbf{H}^{1}(  \Omega)  }%
\leq2C\left\Vert \mathbf{e}_{h}\right\Vert _{k,+}\eta_{6}^{\operatorname*{alg}%
}\bigl\Vert \Pi_{-k}^{\operatorname*{curl}}\mathbf{v}_{h}\bigr\Vert
_{\operatorname*{imp},k} \\
& \leq\tilde{C}\left\Vert \mathbf{e}_{h}\right\Vert
_{k,+}\eta_{6}^{\operatorname*{alg}}\left\Vert \mathbf{v}_{h}\right\Vert
_{\operatorname*{imp},k}.
\label{finalestT1}%
\end{align}

\subsubsection{$hp$-Analysis of $T_{1}$}

In \cite[(4.72)]{MelenkSauterMaxwell_I} it was proved that for our choice
$\Pi_{h}^{E}:=\Pi_{p}^{\operatorname*{curl},c}$ with $\Pi_{p}%
^{\operatorname*{curl},c}$ as in \cite{melenk_rojik_2018}, \cite[\S 8]%
{MelenkSauterMaxwell_I} (see Lem.~\ref{AdiscSp}), one has
\begin{equation}
\sup_{\substack{\mathbf{w}\in\mathbf{V}_{-k,0}\backslash\left\{
\mathbf{0}%
\right\}  \colon\\\operatorname*{curl}{\mathbf{w}}\in\operatorname*{curl}%
{\mathbf{X}}_{h}}}\frac{\vert k\vert \bigl\Vert \mathbf{w}-\Pi
_{h}^{E}\mathbf{w}\bigr\Vert }{\left\Vert \mathbf{w}\right\Vert
_{\mathbf{H}^{1}(  \Omega)  }}\leq C\frac{\left\vert k\right\vert
h}{p}. \label{esteta6alg1}%
\end{equation}
For the boundary term in the norm $\Vert\cdot\Vert_{k,+}$ we study the
approximation properties of the operator $\Pi_{p}^{\operatorname{curl},c}$ of
\cite{melenk_rojik_2018} on the boundary of the reference tetrahedron
$\widehat{K}$ more carefully.

\begin{lemma}
Let $\widehat{\Pi}_{p}^{\operatorname{curl},3d}$ be the operator introduced in
\cite{melenk_rojik_2018}. \label{LemRojikVar}For all $\mathbf{u}\in
\mathbf{H}^{1}( \widehat{K}) $ with $\operatorname*{curl}\mathbf{u}%
\in(  \mathcal{P}_{p}( \widehat{K}) )  ^{3}$, there holds with the
tangential component operator $\Pi_{T,\partial\widehat{K}}$
\[
\bigl\Vert \Pi_{T,\partial\widehat{K}} \bigl(  \mathbf{u}-\widehat{\Pi}%
_{p}^{\operatorname*{curl},3d}\mathbf{u}\bigr)  \bigr\Vert _{\mathbf{L}%
^{2}(\partial\widehat{K}) }\leq Cp^{-1/2}\left\Vert \mathbf{u}\right\Vert
_{H^{1}( \widehat{K}) }.
\]

\end{lemma}

\begin{proof}
We follow the proof of \cite[Lem.~{6.15}]{melenk_rojik_2018} and employ the
notation used there. {From \cite[proof of Lem.~{6.15}]{melenk_rojik_2018} and
\cite[(6.42)]{melenk_rojik_2018}, we can decompose ${\mathbf{u}}=\nabla
\varphi+{\mathbf{v}}$ with
\begin{equation}
\Vert\varphi\Vert_{H^{2}(\widehat{K})}+\Vert{\mathbf{v}}\Vert_{{\mathbf{H}%
}^{1}(\widehat{K})}\leq C\Vert{\mathbf{u}}\Vert_{{\mathbf{H}}^{1}(\widehat
{K})}. \label{StabRojikDecomp}%
\end{equation}
}Since $\operatorname{curl}\mathbf{u}\in(\mathcal{P}_{p}(\widehat
{K}))^{3}$ we have $\mathbf{v}-\widehat{\Pi}_{p}^{\operatorname*{curl}%
,3d}\mathbf{v}=0$. We conclude
\[
\mathbf{u}-\widehat{\Pi}_{p}^{\operatorname*{curl},3d}\mathbf{u}%
=\mathbf{v}+\nabla\varphi-\widehat{\Pi}_{p}^{\operatorname{curl}%
,3d}(\mathbf{v}+\nabla\varphi)=\nabla(\varphi-\widehat{\Pi}_{p+1}%
^{\operatorname*{grad},3d}\varphi)
\]
with $\widehat{\Pi}_{p+1}^{\operatorname*{grad},3d}$ as in
\cite{melenk_rojik_2018}. The construction of the projection-based
interpolation operators $\widehat{\Pi}_{p}^{\operatorname{curl},3d}$,
$\widehat{\Pi}_{p}^{\operatorname{grad},3d}$ is such that facewise, they
reduce to corresponding 2D operators. That is, for each face $f\subset
\partial\widehat{K}$ we have
\[
  \Pi_{T,\partial\widehat{K}}\bigl(  \mathbf{u}-\widehat{\Pi}%
_{p}^{\operatorname*{curl},3d}\mathbf{u}\bigr)  \vert _{f}=
\Pi_{T,\partial\widehat{K}}\bigl(  \nabla\bigl(  \varphi-\widehat{\Pi}%
_{p+1}^{\operatorname*{grad},3d}\varphi\bigr)  \bigr)  \vert
_{f}=\nabla_{{f}}\bigl(  \mathbf{I}-\widehat{\Pi}_{p+1}^{\operatorname*{grad}%
,2d}\bigr)  \bigl(  \varphi\vert _{f}\bigr)  .
\]
We apply \cite[Thm.~{2.13}]{melenk_rojik_2018} to obtain%
\begin{align*}
\bigl\Vert \Pi_{T,\partial\widehat{K}}\bigl(  \mathbf{u}-\widehat{\Pi}%
_{p}^{\operatorname*{curl},3d}\mathbf{u}\bigr)  \bigr\Vert _{\mathbf{L}%
^{2}(f)}  &  =\bigl\Vert \nabla_{{f}}\bigl(  \mathbf{I}-\widehat{\Pi}%
_{p+1}^{\operatorname*{grad},2d}\bigr)  \varphi\vert
_{f}\bigr)  \bigr\Vert _{\mathbf{L}^{2}(f)}\leq Cp^{-1/2}\Vert
\varphi\Vert _{H^{3/2}(f)}\\
&  \leq Cp^{-1/2}\Vert \varphi\Vert _{H^{2}(\widehat{K})}%
\overset{\text{(\ref{StabRojikDecomp})}}{\leq}Cp^{-1/2}\Vert
\mathbf{u}\Vert _{H^{1}(\widehat{K})}.
\end{align*}%
\end{proof}

For the boundary part of $\left\Vert \mathbf{w}\right\Vert
_{\operatorname*{imp},k}$ of a $\mathbf{w}\in\mathbf{V}_{-k,0}$ with
$\operatorname*{curl}{\mathbf{w}}\in\operatorname*{curl}{\mathbf{X}}_{h}$, we
get, by applying a scaling argument to Lemma~\ref{LemRojikVar}:%
\begin{align}
| k|  \bigl\Vert \Pi_{T}\bigl(  \mathbf{w}-\Pi_{h}%
^{E}\mathbf{w}\bigr)  \bigr\Vert _{\mathbf{L}^{2}(\Gamma)}^{2}  &
=| k|  \sum_{\substack{K\in\mathcal{T}_{h}\\ \left\vert
K\cap\Gamma\right\vert >0}}\bigl\Vert \Pi_{T}\bigl(  \mathbf{w}-\Pi_{h}%
^{E}\mathbf{w}\bigr)  \bigr\Vert _{\mathbf{L}^{2}(K\cap\Gamma)}%
^{2}
\label{esteta6alg2}\\
&  \leq C\frac{|k|  h}{p}\sum_{\substack{K\in
\mathcal{T}_{h}\\ \left\vert K\cap\Gamma\right\vert >0}}\Vert
\mathbf{w}\Vert _{\mathbf{H}^{1}(K)}^{2}\leq C\frac{ |k| h}{p}\Vert \mathbf{w}\Vert _{\mathbf{H}^{1}(\Omega
)}^{2}.\nonumber
\end{align}
The combination of (\ref{esteta6alg1}) with (\ref{esteta6alg2}) leads to%
\begin{equation}
\eta_{6}^{\operatorname*{alg}}\leq C\left(  \frac{|k | h}{p}\right)  ^{1/2}
\left(  1+\left(  \frac{\left\vert k\right\vert h}%
{p}\right)  ^{1/2}\right)  . 
\label{eta6final}%
\end{equation}


\subsection{Consistency analysis: the term $\mathbf{T}_{2}$ in
(\ref{mainsplittingerror})}

\label{SecT2}

Recall the definition of $T_{2}=\innerprod{
\mathbf{e}_{h},\mathbf{v}_{0}%
}$ with $\mathbf{v}_{0}:=\Pi_{-k}^{\operatorname*{curl}}%
\mathbf{v}_{h} = (\operatorname{I} - \Pi_{-k}^{\nabla}) \mathbf{v}_h$. The function $\mathbf{v}_{0}$ belongs to $\mathbf{X}%
_{\operatorname*{imp}}$ and by combining (\ref{Helmtota}) and
(\ref{HelmDecompa}) we find that $\mathbf{v}_{0}$ belongs to $\mathbf{V}%
_{-k,0}$. Proposition~\ref{PropEmbVk0} implies $\mathbf{v}_{0}\in
\mathbf{H}^{1}(\Omega)$ and%
\begin{equation}
\Vert \mathbf{v}_{0}\Vert _{\mathbf{H}^{1}(\Omega)}\leq |k| 
\Vert \mathbf{v}_{0}\Vert _{\mathbf{H}^{1}(\Omega
),k}\leq C\Vert \mathbf{v}_{0}\Vert _{\mathbf{H}%
(\operatorname*{curl},\Omega),k}\leq C\Vert \mathbf{v}_{0}\Vert
_{\operatorname*{imp},k}. \label{v0estH1}%
\end{equation}
The characterization (\ref{PropVcharact}) of $\mathbf{V}_{-k,0}$ implies%
\begin{equation}
\operatorname*{div}\nolimits_{\Gamma}(  \mathbf{v}_{0})
_{T}=\operatorname*{i}k\langle \mathbf{v}_{0},\mathbf{n}\rangle .
\label{divgammaest}%
\end{equation}

To estimate the term $T_{2}$, we consider the dual problem: Given
$\mathbf{v}_{0}\in\mathbf{V}_{-k,0}$, find $\mathbf{z}\in\mathbf{X}%
_{\operatorname*{imp}}$ such that%
\[
A_{k}(\mathbf{w},\mathbf{z})=\innerprod{
\mathbf{w},\mathbf{v}_{0}%
}\quad\forall\mathbf{w}\in\mathbf{X}_{\operatorname*{imp}}.
\]
The operator $\mathcal{N}_{-k}:\mathbf{V}_{-k,0}\rightarrow\mathbf{X}%
_{\operatorname*{imp}}$ is defined by $\mathcal{N}_{-k}\mathbf{v}%
_{0}:=\mathbf{z}$. The strong formulation is given by%
\begin{align}
\mathcal{L}_{\Omega,-k}\mathbf{z}& =k^{2}\mathbf{v}_{0} \quad  \text{in }\Omega, 
& 
\mathcal{B}_{ \Gamma ,-k}\mathbf{z}& =-\operatorname*{i}k\left(  \mathbf{v}_{0}\right)  _{T} \quad 
\text{on }\Gamma. 
\label{eq:dual-problem-N}%
\end{align}
Hence, $\mathcal{N}_{-k}\mathbf{v}_{0}=\mathcal{S}_{\Omega,-k}%
^{\operatorname*{MW}}\left(  k^{2}\mathbf{v}_{0},-\operatorname*{i}k\left(
\mathbf{v}_{0}\right)  _{T}\right)  $. By Galerkin orthogonality satisfied by
$\mathbf{e}_{h}$, we have for any $\mathbf{w}_{h}\in\mathbf{X}_{h}$
\begin{equation}
\left\vert \innerprod{
\mathbf{e}_{h},\mathbf{v}_{0}%
}\right\vert =\left\vert A_{k}(\mathbf{e}_{h},\mathcal{N}%
_{-k}\mathbf{v}_{0}-\mathbf{w}_{h})\right\vert \leq\left\Vert \mathbf{e}%
_{h}\right\Vert _{\operatorname*{imp},k}\left\Vert \mathcal{N}_{-k}%
\mathbf{v}_{0}-\mathbf{w}_{h}\right\Vert _{\operatorname*{imp},k}.
\label{ReT2}%
\end{equation}
We set%
\begin{equation}
\tilde{\eta}_{2}^{\operatorname*{alg}}\left(  \mathbf{X}_{h}\right)
:=\sup_{\mathbf{v}_{0}\in\mathbf{V}_{-k,0}\backslash\left\{
\mathbf{0}%
\right\}  }\inf_{\mathbf{w}_{h}\in\mathbf{X}_{h}}\frac{\left\Vert
\mathcal{N}_{-k}\mathbf{v}_{0}-\mathbf{w}_{h}\right\Vert _{\operatorname*{imp}%
,k}}{\left\Vert \mathbf{v}_{0}\right\Vert _{\operatorname*{imp},k}}
\label{eq:eta2-alg}%
\end{equation}
so that%
\begin{equation}
\left\vert T_{2}\right\vert =\left\vert \innerprod{
\mathbf{e}_{h},\mathbf{v}_{0}%
}\right\vert \leq\tilde{\eta}_{2}^{\operatorname*{alg}}\left(
\mathbf{X}_{h}\right)  \Vert \mathbf{e}_{h}\Vert
_{\operatorname*{imp},k}\Vert \mathbf{v}_{0}\Vert
_{\operatorname*{imp},k}. 
\label{estT2a}%
\end{equation}

\subsubsection{$hp$-Analysis of $T_{2}$}


Next, we gauge the approximation property $\tilde{\eta}_{2}%
^{\operatorname*{alg}}\left(  \mathbf{X}_{h}\right)  $. We employ the
splitting given by Theorem~\ref{TheoMainIt}, viz.,
\begin{equation}
\mathcal{N}_{-k}\mathbf{v}_{0}=\mathbf{z}=\mathbf{z}_{H^{2}}+\mathbf{z}%
_{\mathcal{A}}+k^{-2}\nabla\varphi_{\mathbf{f}} -  \operatorname{i}k^{-1}%
\nabla\varphi_{\mathbf{g}}.\label{Nmksplitting}%
\end{equation}
Note that these
five
functions $\mathbf{z}$, $\mathbf{z}_{H^{2}}$, $\mathbf{z}_{\mathcal{A}}$,
$\varphi_{\mathbf{f}}$, $\varphi_{\mathbf{g}}$ depend on $\mathbf{v}_{0}$ but
we suppress this in the notation. From Theorem~\ref{TheoMainIt} with
$m=m^{\prime}=1$ we have%
\begin{align}
\Vert\varphi_{\mathbf{f}}\Vert_{H^{2}(\Omega)} &  \leq C\left\vert
k\right\vert ^{2}\Vert\operatorname{div}\mathbf{v}_{0}\Vert_{L^{2}(\Omega
)}\overset{\operatorname{div}\mathbf{v}_{0}=0}{=}0,\nonumber\\
\Vert\varphi_{\mathbf{g}}\Vert_{H^{2}(\Omega)} &  \leq C|k|\Vert
\operatorname{div}_{\Gamma}(\mathbf{v}_{0})_{T}\Vert_{H^{-1/2}(\Gamma)}\leq
C|k|\Vert(\mathbf{v}_{0})_{T}\Vert_{H^{1/2}(\Gamma)}\overset{(\ref{v0estH1}%
)}{\leq}C|k|\Vert\mathbf{v}_{0}\Vert_{\operatorname{imp},k},\nonumber\\
\Vert\mathbf{z}_{H^{2}}\Vert_{\mathbf{H}^{2}(\Omega)} &  \leq\left\vert
k\right\vert ^{2}\Vert\mathbf{z}_{H^{2}}\Vert_{\mathbf{H}^{2}(\Omega
),k}\overset{\text{(\ref{zH2estimate0})}}{\leq}C|k|^{-1}\left(  \Vert
k^{2}\mathbf{v}_{0}\Vert_{\mathbf{H}^{1}(\Omega)}+|k|\Vert k(\mathbf{v}%
_{0})_{T}\Vert_{\mathbf{H}^{1/2}(\Gamma)}\right)  \nonumber\\
&  \overset{(\ref{v0estH1})}{\leq}C|k|\Vert\mathbf{v}_{0}\Vert
_{\operatorname{imp},k},\nonumber\\
\mathbf{z}_{\mathcal{A}} &  \in\mathcal{A}(CC_{\mathbf{z}},B,\Omega
),\nonumber\\
C_{\mathbf{z}} &  =|k|^{ \theta-1 }%
\bigl(  \Vert k^{2}\mathbf{v}_{0}\Vert+|k|\Vert k(\mathbf{v}_{0})_{T}%
\Vert_{\mathbf{H}^{-1/2}(\Gamma)}\bigr)  \nonumber\\
&  \overset{\operatorname{i}k(\mathbf{v}_{0})_{T}=\operatorname{div}_{\Gamma
}\mathbf{v}_{0},(\ref{v0estH1})}{\leq}C|k|^{ \theta  }%
\Vert\mathbf{v}_{0}\Vert_{\operatorname{imp},k}.\label{eq:regularity-zA}%
\end{align}
We note that $\varphi_{\mathbf{f}}=0$. For the approximation of $\nabla
\varphi_{\mathbf{g}}$, we use the elementwise defined operator $\Pi
_{p}^{\nabla,s}$ of Corollary~\ref{cor:Pinabla} with $m=2$ there to get%
\begin{align}
\bigl\Vert k^{-1}\bigl(  \nabla\varphi_{\mathbf{g}}-\nabla \Pi_{p}^{\nabla,s}%
\varphi_{\mathbf{g}}\bigr)  \bigr\Vert _{\operatorname{imp},k}  & \leq
C|k|^{-1}\left(  |k|\frac{h}{p}+|k|^{1/2}\left(  \frac{h}{p}\right)
^{1/2}\right)  \Vert\varphi_{\mathbf{g}}\Vert_{H^{2}(\Omega)}\nonumber\\
& \leq C\left(  \frac{|k|h}{p}\right)  ^{1/2}\Vert\mathbf{v}_{0}%
\Vert_{\operatorname{imp},k}.\label{eq:approx-varphi_g-contribution}%
\end{align}%
For the approximation of $\mathbf{z}_{H^{2}}$, we employ the elementwise
defined operator $\Pi_{p}^{\operatorname*{curl},s}:\mathbf{H}^{2}%
(\Omega)\rightarrow\mathbf{X}_{h}$ as in Lemma~\ref{lemma:Picurls}. By summing
over all elements the estimates of Lemma~\ref{lemma:Picurls}%
(\ref{item:lemma:Picurls-i}) we get
\begin{align}
&  \Vert \mathbf{z}_{H^{2}}-\Pi_{p}^{\operatorname*{curl},s}%
\mathbf{z}_{H^{2}}\Vert _{\mathbf{H}\left(  \operatorname*{curl}%
,\Omega\right)  ,k}^{2}=\sum_{K\in\mathcal{T}_{h}}\Vert \mathbf{z}%
_{H^{2}}-\Pi_{p}^{\operatorname*{curl},s}\mathbf{z}_{H^{2}}\Vert
_{\mathbf{H}(  \operatorname*{curl},K)  ,k}^{2}%
\label{zH2PiCurlsEst}\\
&  \qquad\quad\leq C\sum_{K\in\mathcal{T}_{h}}\frac{h_{K}^{2}}{p^{2}}\left(
1+\frac{\left\vert k\right\vert ^{2}h_{K}^{2}}{p^{2}}\right)  \Vert
\mathbf{z}_{H^{2}}\Vert _{\mathbf{H}^{2}(  K)  }^{2}\leq
C\frac{h^{2}}{p^{2}}\left(  1+\frac{\left\vert k\right\vert ^{2}h^{2}}{p^{2}%
}\right)  \Vert \mathbf{z}_{H^{2}}\Vert _{\mathbf{H}^{2}(
\Omega)  }^{2}\nonumber\\
&  \qquad\quad\leq C\frac{\left\vert k\right\vert ^{2}h^{2}}{p^{2}}\left(
1+\frac{\left\vert k\right\vert ^{2}h^{2}}{p^{2}}\right)  \left\Vert
\mathbf{v}_{0}\right\Vert _{\operatorname*{imp},k}^{2}.\nonumber
\end{align}
For the boundary part of the $\left\Vert \cdot\right\Vert
_{\operatorname*{imp},k}$ norm we proceed similarly using 
Lemma~\ref{lemma:Picurls}(\ref{item:lemma:Picurls-i}) to arrive at
\[
|k|^{1/2}\Vert\mathbf{z}_{H^{2}}-\Pi_{p}^{\operatorname{curl},s}%
\mathbf{z}_{H^{2}}\Vert_{\mathbf{L}^{2}(\Gamma)}\leq \! C|k|^{1/2}\left(
\frac{h}{p}\right)  ^{3/2}\!\!\!
\Vert\mathbf{z}_{H^{2}}\Vert_{\mathbf{H}^{2}%
(\Omega)}\leq \! C\! \left(  \frac{|k|h}{p}\right)  ^{3/2}\!\!\!
\Vert\mathbf{v}_{0}%
\Vert_{\operatorname{imp},k}.
\]
In summary, we have proved%
\begin{equation}
\Vert \mathbf{z}_{H^{2}}-\Pi_{p}^{\operatorname*{curl},s}\mathbf{z}%
_{H^{2}}\Vert _{\operatorname*{imp},k}\leq C\frac{ | k| h }{p}\left(  1+\left(  \frac{h\left\vert k\right\vert }{p}\right)
^{1/2}+\frac{h\left\vert k\right\vert }{p}\right)  \left\Vert \mathbf{v}%
_{0}\right\Vert _{\operatorname*{imp},k}.\label{zH2finalest}%
\end{equation}
Next, for the analytic part $\mathbf{z}_{\mathcal{A}}$ we get from
Lemma~\ref{lemma:Picurls}(\ref{item:lemma:Picurls-iii}) in view of
(\ref{eq:regularity-zA}) under the (mild) resolution condition
\begin{equation}
h+\frac{\left\vert k\right\vert h}{p}\leq\widetilde{C}\label{mildrescond}%
\end{equation}
that%
\begin{align}
\Vert\mathbf{z}_{\mathcal{A}}-\Pi_{p}^{\operatorname{curl},s}\mathbf{z}%
_{\mathcal{A}}\Vert_{\operatorname{imp},k} &  \leq CC_{\mathbf{z}}\left\vert
k\right\vert \left(  \left(  \frac{h}{h+\sigma}\right)  ^{p}+\left(
\frac{|k|h}{\sigma p}\right)  ^{p}\right)  \label{zAfinalest}\\
&  \leq C\left\Vert \mathbf{v}_{0}\right\Vert _{\operatorname{imp},k}|k|^{ \theta+1 }%
\left(  \left(  \frac{h}{h+\sigma}\right)  ^{p}+\left(  \frac{|k|h}{\sigma
p}\right)  ^{p}\right)  .\nonumber
\end{align}
This derivation is summarized in the following lemma.

\begin{lemma}
\label{Lemeta2alg} Assume hypothesis (\ref{AssumptionAlgGrowth}) and let
$\Omega$ be a bounded Lipschitz domain with simply connected, analytic
boundary. Let the mesh satisfy 
Assumption~\ref{def:element-maps}. Let $c_{2}$, $\varepsilon>0$ be given. Then
there exists $c_{1}>0$ (depending only on the constants of
(\ref{AssumptionAlgGrowth}), $\Omega$, the parameters of
Assumption~\ref{def:element-maps}, and $c_{2}$, $\varepsilon$) such that for
$h$, $k$, $p$ satisfying the resolution condition
\begin{equation}
\frac{\left\vert k\right\vert h}{p}\leq c_{1}\quad\text{and\quad}p\geq
\max\left\{  1,c_{2}\ln\left\vert k\right\vert \right\}  \label{rescond}%
\end{equation}
there holds
\begin{equation}
\tilde{\eta}_{2}^{\operatorname*{alg}}\left(  \mathbf{X}_{h}\right)
\leq\varepsilon. \label{eta2est}%
\end{equation}

\end{lemma}

\begin{proof}
We combine (\ref{eq:approx-varphi_g-contribution}), (\ref{zH2finalest}), and
(\ref{zAfinalest}) with the resolution condition to arrive at%
\begin{equation}
\tilde{\eta}_{2}^{\operatorname*{alg}}\left(  \mathbf{X}_{h}\right)  \leq
C\left(  \left(  \frac{h\left\vert k\right\vert }{p}\right)  ^{1/2}%
+{\left\vert k\right\vert ^{ \theta+1} }%
\left[  \left(  \frac{h}{\sigma+h}\right)  ^{p}+\left(  \frac{\left\vert
k\right\vert h}{\sigma p}\right)  ^{p}\right]  \right)
.\label{eta2algestproof}%
\end{equation}
Clearly, by selecting $c_1$ sufficiently small, we may ensure 
that the first term in (\ref{eta2algestproof}), $(|k| h/p)^{1/2}$, is smaller than $\varepsilon/3$. 
The second term in (\ref{eta2algestproof}), $|k|^{\theta+1} (h/(\sigma + h))^p$, can be made smaller than $\varepsilon/3$ 
for sufficiently small $c_1$ by appealing to \cite[Lem.~{8.7}]{MelenkSauterMaxwell_I}.  
For the last term in (\ref{eta2algestproof}), we may assume that $c_1 < \sigma$ and then estimate 
\begin{align*}
|k|^{\theta+1} \left(  \frac{|k| h}{\sigma p}\right)  ^{p}
& \leq |k|^{\theta+1} (c_1/\sigma)^p \leq |k|^{\theta+1} (c_1/\sigma)^{\max\{1,c_2 \ln |k|\}} \\
& = \min\{|k|^{\theta+1} (c_1/\sigma), |k|^{\theta+1 + c_2 \ln (c_1/\sigma)}\}. 
\end{align*}
This expression can be made smaller than $\varepsilon/3$ uniformly in $|k| \in [1,\infty)$  by selecting $c_1$ sufficiently small: 
the first term in the minimum tends to $0$ as $c_1\rightarrow 0$ uniformly in $|k| \in [1,2]$ and the second term in the minimum 
tends to zero as $c_1 \rightarrow 0$ uniformly in $|k| \ge  2$. 
\end{proof}


\subsection{$h$-$p$-$k$-explicit stability and convergence estimates for 
Maxwell's equations}

We begin with the estimate of the consistency term $\delta_{k}$.

\begin{lemma}
\label{LemResCond}Let the assumptions in Lemma~\ref{Lemeta2alg} hold. Let
$\varepsilon$, $c_{2}>0$ be given. Then, one can choose a constant $c_{1}>0$
sufficiently small such that the resolution condition (\ref{rescond}) implies%
\[
\delta_{k}(\mathbf{e}_{h})\leq\varepsilon.
\]
\end{lemma}%

\begin{proof}
We
combine estimates (\ref{2prodcont}), (\ref{mainsplittingerror}),
(\ref{finalestT1}), (\ref{eta6final}), (\ref{estT2a})
for $\mathbf{v}_{0}:=\Pi_{-k}^{\operatorname*{curl}}\mathbf{v}_{h}$, 
Lemma~\ref{LemCurlGradProj}, and (\ref{eta2est}) in a straightforward way to obtain%
\[
\left\vert \innerprod{
\mathbf{e}_{h},\mathbf{v}_{h}%
}\right\vert \leq C\bigl(  c_{1}^{1/2}+c_{1}+\varepsilon\bigr)
\left\Vert \mathbf{e}_{h}\right\Vert _{\operatorname*{imp},k}\left\Vert
\mathbf{v}_{h}\right\Vert _{\operatorname*{imp},k}%
\]
and thus $\displaystyle\delta_{k}(\mathbf{e}_{h})\leq2C(c_{1}^{1/2}%
+c_{1}+\varepsilon).$ We may assume that $\sqrt{c_{1}}\leq\varepsilon\leq1$ so
that $\displaystyle\delta_{k}(\mathbf{e}_{h})\leq6C\varepsilon$. Since the
constant $C>0$ does not depend on $\varepsilon$, the result follows by
adjusting constants.
\end{proof}

This estimate allows us to formulate the quasi-optimality of the $hp$-FEM
Galerkin discretization and to show $h$-$p$-$k$-explicit convergence rates
under suitable regularity assumptions.

\begin{theorem}
\label{thm:quasi-optimality} Let $\Omega\subset\mathbb{R}^{3}$ be a bounded
Lipschitz domain with a simply connected, analytic boundary. Let the stability
Assumption~\ref{AssumptionAlgGrowth} be satisfied.
Let the finite element mesh with mesh size $h$ satisfy
Assumption~\ref{def:element-maps}, and let $\mathbf{X}_{h}$ be defined by as
the space of N\'{e}d\'{e}lec-type-I elements of degree $p$ (cf.\ (\ref{defXh})).

Then, for any $\mathbf{j}$, $\mathbf{g}_{T}$ satisfying
(\ref{eq:condition-on-j-g}), the variational form of Maxwell's equations
(\ref{weakformulation}) has a unique solution $\mathbf{E}$.

For any fixed $c_{2}>0$ and $\eta\in(0,1) $ one can select $c_{1}>0$
(depending only on $\Omega$ and the constants of (\ref{AssumptionAlgGrowth})
and Assumption~\ref{def:element-maps}) such that the resolution condition
(\ref{rescond}) implies that the discrete problem (\ref{discMW}) has a unique
solution $\mathbf{E}_{h}$, which satisfies the quasi-optimal error estimate%
\begin{equation}
\label{eq:quasi-optimality}\left\Vert \mathbf{E}-\mathbf{E}_{h}\right\Vert
_{\operatorname*{imp},k}\leq\frac{1+\eta}{1-\eta}\inf_{\mathbf{w}_{h}%
\in\mathbf{X}_{h}}\left\Vert \mathbf{E}-\mathbf{w}_{h}\right\Vert
_{\operatorname*{imp},k}.
\end{equation}

\end{theorem}

\begin{proof}
Existence and uniqueness of the continuous variational Maxwell problem follow
from Proposition~\ref{lemma:apriori-homogeneous-rhs}. From
Lemma~\ref{LemResCond} we know that $c_{1}$ can be chosen sufficiently small
such that $\delta(\mathbf{e}_{h})<\eta$. As in the proof of Theorem
\cite[Thm.~{4.15}]{MelenkSauterMaxwell_I} (which goes back to \cite[Thm.~{3.9}%
]{MelenkLoehndorf}) existence, uniqueness, and quasi-optimality follows.%
\end{proof}

The quasi-optimality result (\ref{eq:quasi-optimality}) leads to quantitative,
$k$-explicit error estimates if a $k$-explicit regularity of the solution
$\mathbf{E}$ is available. In the following corollary, we draw on the
regularity assertions of Theorem~\ref{TheoMainIt}. We point out, however, that
due to our relying on the operator $\Pi_{p}^{\operatorname{curl},s}$ and the
regularity assertion Theorem~\ref{TheoMainIt}, the regularity requirements on
the data $\mathbf{j}$, $\mathbf{g}_{T}$ are not the weakest possible ones.

\begin{corollary}
\label{cor:convergence} Let the hypotheses of
Theorem~\ref{thm:quasi-optimality} be valid. Given $\eta\in(0,1)$ and
$c_{2}>0$ let $c_{1}$ be as in Theorem~\ref{thm:quasi-optimality}. Then, under
the scale resolution condition (\ref{rescond}) the following holds: Let $m$,
$m^{\prime}\in\mathbb{N}_{0}$ and $(\mathbf{j},\mathbf{g}_{T})\in
\mathbf{H}^{m}(\Omega)\times\mathbf{H}^{m-1/2}(\Gamma)$ together with
$(\operatorname{div}\mathbf{j},\operatorname{div}_{\Gamma}\mathbf{g}_{T}%
)\in\mathbf{H}^{m^{\prime}}(\Omega)\times\mathbf{H}^{m^{\prime}-1/2}(\Gamma)$.
If $p\geq\max(m,m^{\prime})$, then
\begin{align}
\Vert\mathbf{E}-\mathbf{E}_{h}\Vert_{\operatorname{imp},k}  &  \leq
C\frac{1+\eta}{1-\eta}\Biggl\{C_{\mathbf{j},\mathbf{g},m}\left(  \frac{h}%
{p}\right)  ^{m}+C_{\mathbf{j},\mathbf{g},m^{\prime}}|k|^{-3/2}\left(
\frac{h}{p}\right)  ^{m^{\prime}+1/2}\label{QuantErrEst}\\
&  \qquad\qquad\qquad\mbox{}+|k|C_{\mathbf{j},\mathbf{g},\mathcal{A}}\left(
\left(  \frac{h}{h+\sigma}\right)  ^{p}+\left(  \frac{|k|h}{\sigma p}\right)
^{p}\right)  \Biggr\},\nonumber
\end{align}
where%
\begin{align}
C_{\mathbf{j},\mathbf{g},m}  &  :=|k|^{-1}\Vert\mathbf{j}\Vert_{\mathbf{H}%
^{m}(\Omega)}+\Vert\mathbf{g}_{T}\Vert_{\mathbf{H}^{m-1/2}(\Gamma)},\\
C_{\mathbf{j},\mathbf{g},m^{\prime}}  &  :=\Vert\operatorname{div}%
\mathbf{j}\Vert_{\mathbf{H}^{m^{\prime}}(\Omega)}+|k|\Vert\operatorname{div}%
_{\Gamma}\mathbf{g}_{T}\Vert_{\mathbf{H}^{m^{\prime}-1/2}(\Gamma)},\\
C_{\mathbf{j},\mathbf{g},\mathcal{A}}  &  :=\left\vert k\right\vert ^{ \theta-1} 
\left(  \Vert\mathbf{j}\Vert_{\mathbf{L}^{2}(\Omega)}+|k|\Vert\mathbf{g}%
_{T}\Vert_{\mathbf{H}^{-1/2}(\Gamma)}\right)  .
\end{align}

\end{corollary}

\begin{proof}
For the error estimate (\ref{QuantErrEst}), we employ the solution
decomposition provided by Thm.~\ref{TheoMainIt}:
\[
\mathbf{E}=\mathbf{E}_{H^{2}}+\mathbf{E}_{\mathcal{A}}+k^{-2}\nabla
\varphi_{\mathbf{j}}+\operatorname{i}k^{-1}\nabla\varphi_{\mathbf{g}}%
\]
with
\begin{align*}
\Vert\mathbf{E}_{H^{2}}\Vert_{\mathbf{H}^{m+1}(\Omega)} &  \leq CC_{\mathbf{j}%
,\mathbf{g},m},\\
\Vert\varphi_{\mathbf{j}}\Vert_{H^{m^{\prime}+2}(\Omega)}+|k|\Vert
\varphi_{\mathbf{g}}\Vert_{H^{m^{\prime}+2}(\Omega)} &  \leq CC_{\mathbf{j}%
,\mathbf{g},m^{\prime}},\\
\mathbf{E}_{\mathcal{A}} &  \in\mathcal{A}(CC_{\mathbf{j},\mathbf{g}%
,\mathcal{A}},B,\Omega)
\end{align*}
for $k$-independent constants $C$, $B$. With the operators $\Pi_{p}^{\nabla
,p}$ and $\Pi_{p}^{\operatorname{curl},s}$ of Corollary~\ref{cor:Pinabla} and
Lemma~\ref{lemma:Picurls} we get
\begin{align*}
 \Vert\mathbf{E}_{H^{2}}-\Pi_{p}^{\operatorname{curl},s}\mathbf{E}_{H^{2}}%
\Vert_{\operatorname{imp},k} 
& \! \leq \! CC_{\mathbf{j},\mathbf{g},m}\!\left[\!
\left(  \frac{h}{p}\right)  ^{m}\!\!+|k|\left(  \frac{h}{p}\right)  ^{m+1}%
\!\!+|k|^{1/2}\left(  \frac{h}{p}\right)  ^{m+1/2}\right]  \\
&  \leq CC_{\mathbf{j},\mathbf{g},m}\left(  \frac{h}{p}\right)  ^{m},\\
\left\vert k\right\vert ^{-2}\Vert\nabla\varphi_{\mathbf{j}}-\nabla\Pi
_{p}^{\nabla,s}\varphi_{\mathbf{j}}\Vert_{\operatorname{imp},k} 
&\!  \leq\!
CC_{\mathbf{j},\mathbf{g},m^{\prime}}|k| ^{-2}\!\left[\!
|k|\left(  \frac{h}{p}\right)  ^{m^{\prime}+1}+|k|^{1/2}\left(  \frac{h}%
{p}\right)  ^{m^{\prime}+1/2}\right]  \\
&  \leq CC_{\mathbf{j},\mathbf{g},m^{\prime}}\left\vert k\right\vert
^{-2}|k|^{1/2}\left(  \frac{h}{p}\right)  ^{m^{\prime}+1/2},\\
|k|^{-1}\Vert\nabla\varphi_{\mathbf{g}}-\nabla\Pi
_{p}^{\nabla,s}\varphi_{\mathbf{g}}\Vert_{\operatorname{imp},k} &  \leq
CC_{\mathbf{j},\mathbf{g},m^{\prime}}\left\vert k\right\vert ^{-2}%
|k|^{1/2}\left(  \frac{h}{p}\right)  ^{m^{\prime}+1/2},\\
\Vert\mathbf{E}_{\mathcal{A}}-\Pi_{p}^{\operatorname{curl},s}\mathbf{E}%
_{\mathcal{A}}\Vert_{\operatorname{imp},k} &  \leq C|k|C_{\mathbf{j}%
,\mathbf{g},\mathcal{A}}\left(  \left(  \frac{h}{h+\sigma}\right)
^{p}+\left(  \frac{|k|h}{\sigma p}\right)  ^{p}\right)  .
\end{align*}
\end{proof}

\section{Numerical results}

\label{sec:numerics}
\begin{figure}[ptb]
\includegraphics[width=0.3\textwidth]{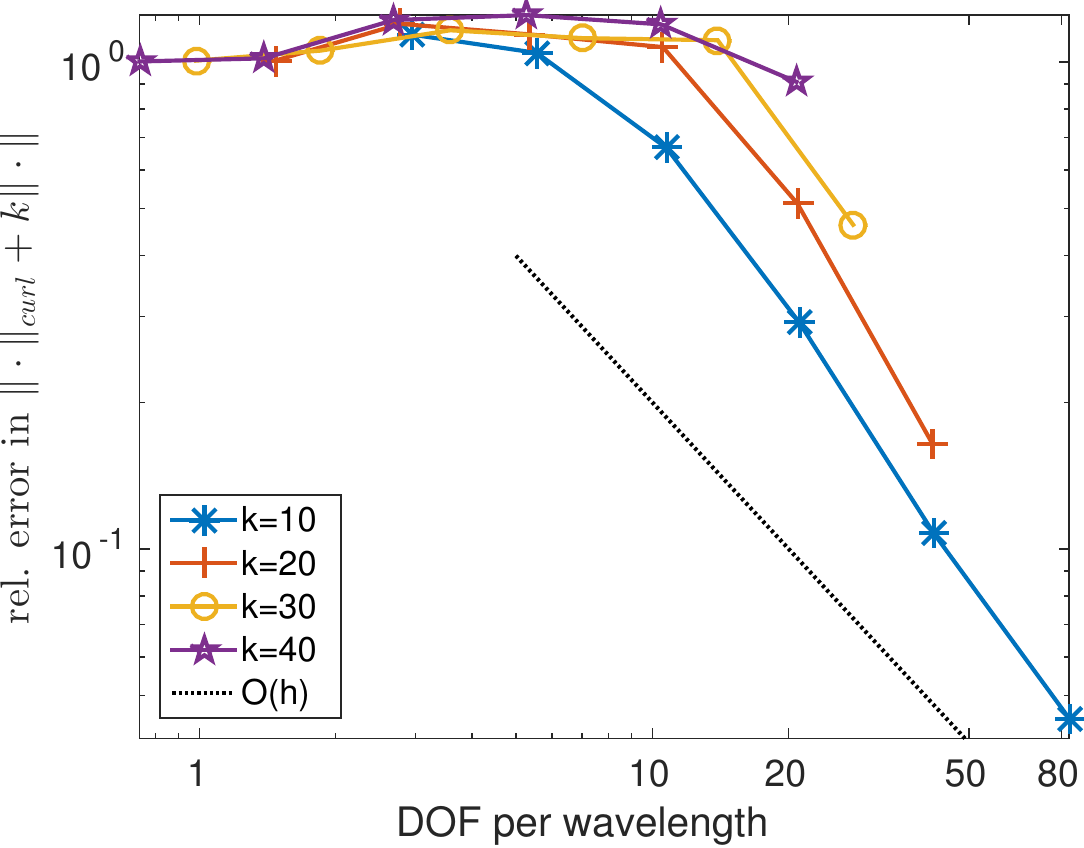}
\includegraphics[width=0.3\textwidth]{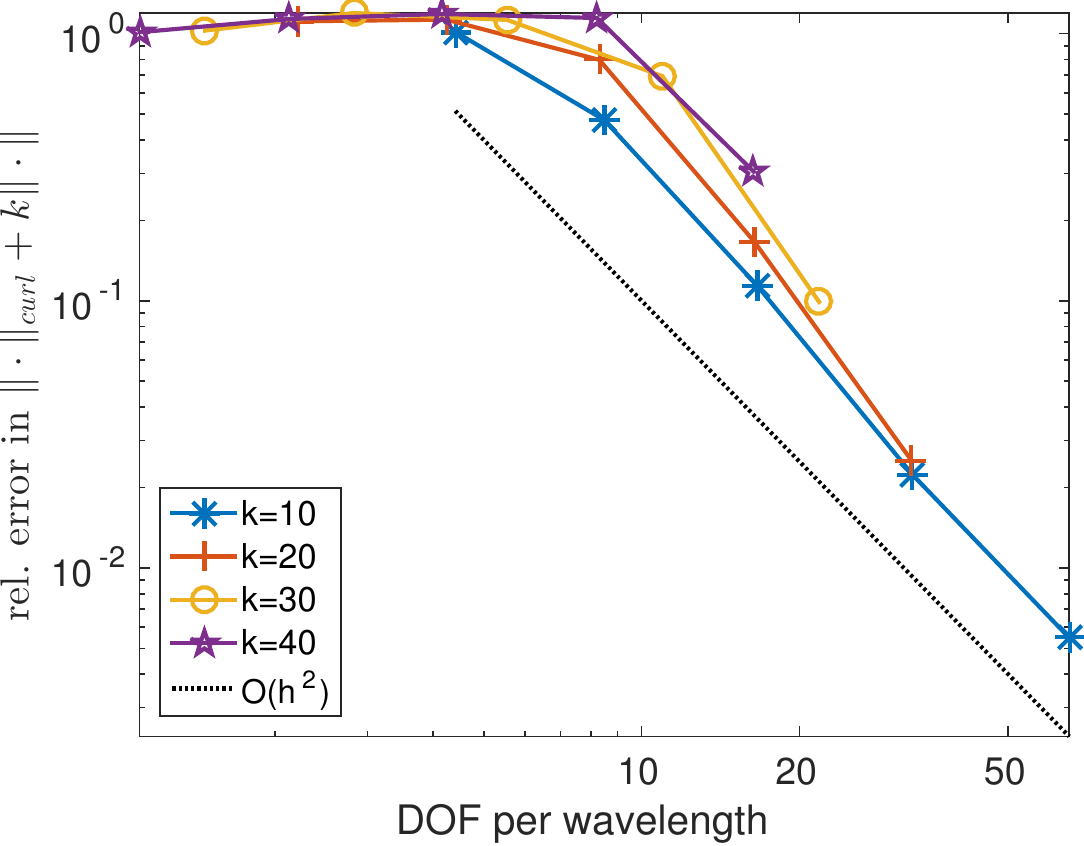}
\includegraphics[width=0.3\textwidth]{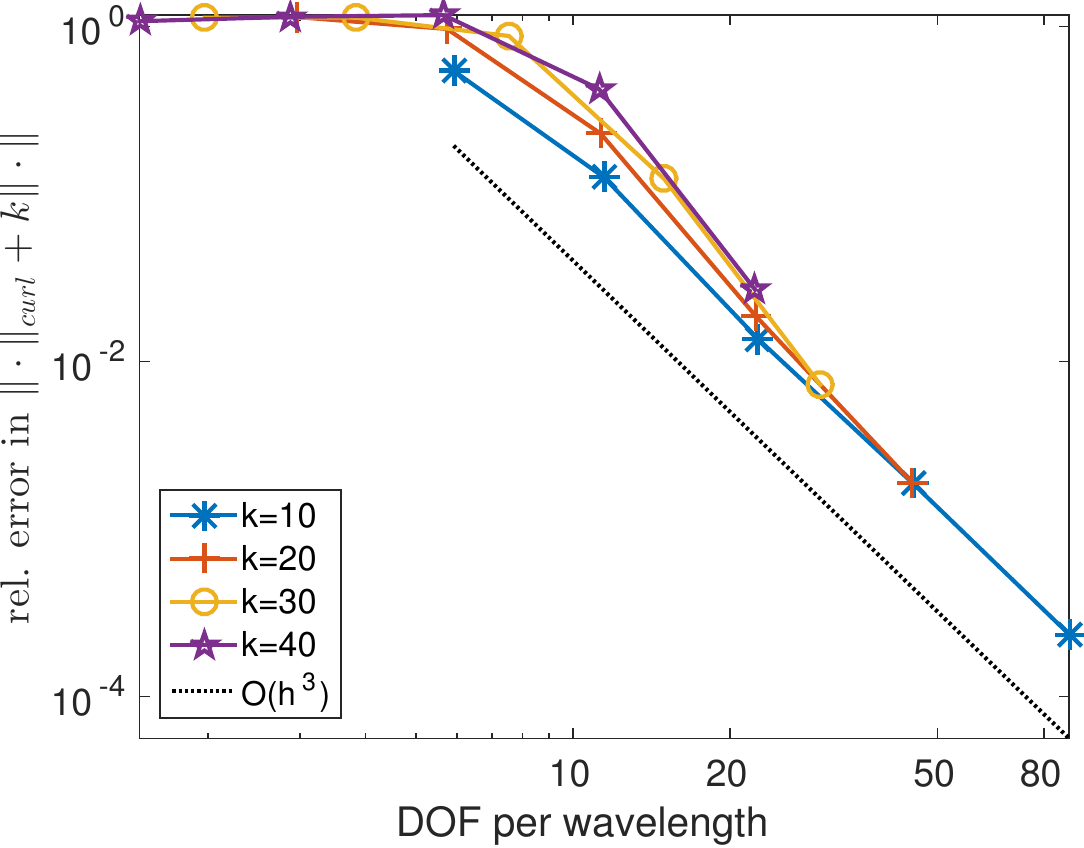}
\caption{
$\Omega=(-1,1)^{3}$, smooth solution; left to right: $p\in\{1,2,3\}$
        }%
\label{fig:cube}%
\end{figure}
\begin{figure}[ptb]
\includegraphics[width=0.3\textwidth]{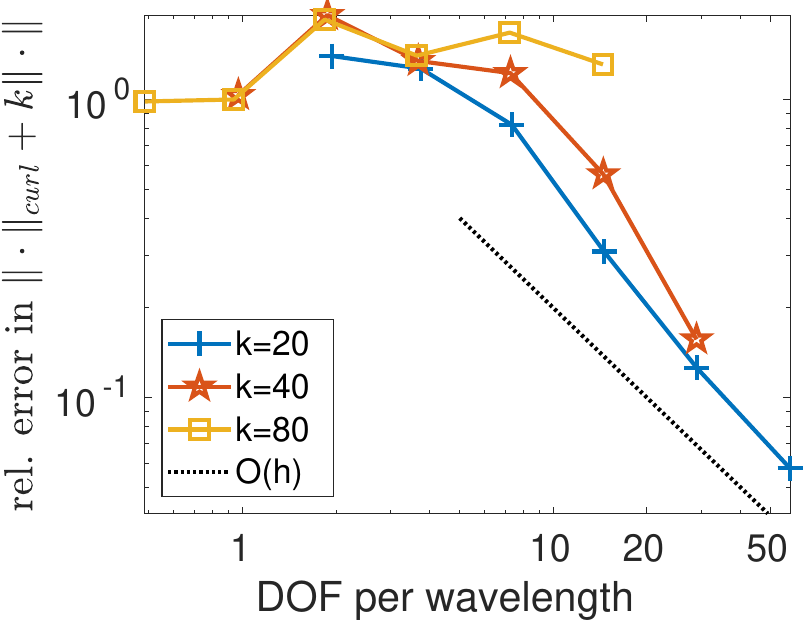}
\includegraphics[width=0.3\textwidth]{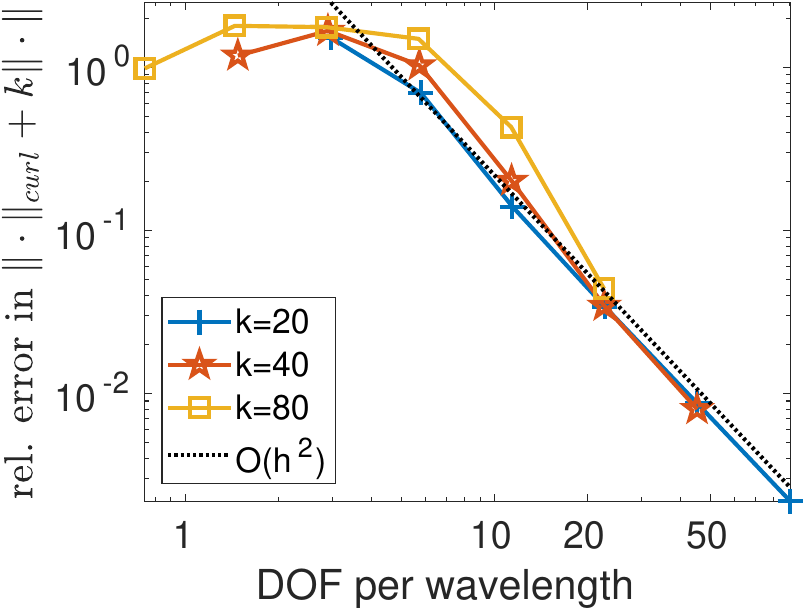}
\includegraphics[width=0.3\textwidth]{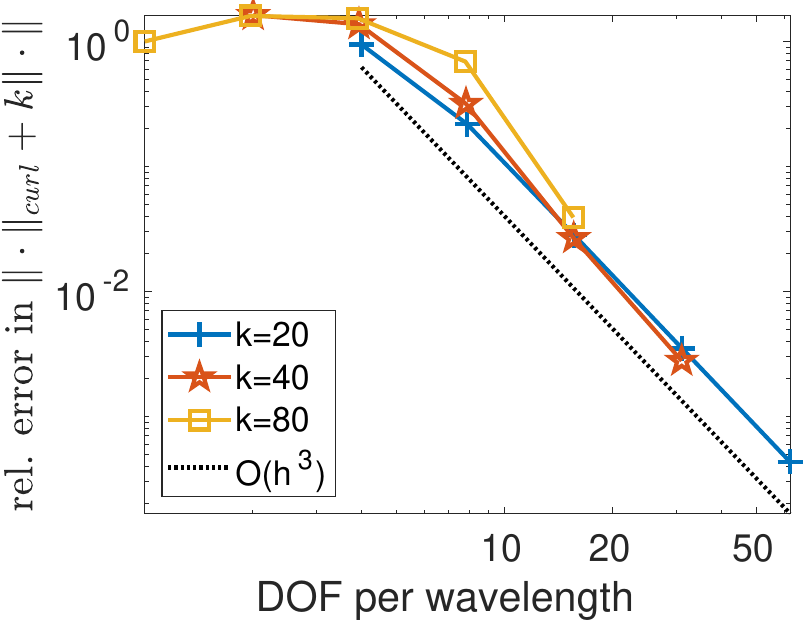}
\caption{
$\Omega= (-1,1)^{3}\setminus[-1/2,1/2]^{3}$, smooth solution; left to right:
$p \in \{1,2,3\}$}%
\label{fig:cube-hole}%
\end{figure}
We illustrate the theoretical findings of Theorem~\ref{thm:quasi-optimality}
and Corollary~\ref{cor:convergence} by two numerical experiments. All
computations are performed with NGSolve, \cite{schoeberlNGSOLVE,schoeberl97}
using N\'{e}d\'{e}lec type II elements, i.e., full polynomial spaces.
\begin{remark} 
While the analysis of the present paper is peformed in detail for 
N\'{e}d\'{e}lec type I elements, it can be extended to 
N\'{e}d\'{e}lec type II elements. Key is the observation that 
commuting diagram operators $\hatPigradcom$ and $\hatPicurlcom$ analogous
to the ones used in Sections~\ref{sec:discretization}, \ref{SecStabConv} for 
type I elements also exist for type II elements. This is discussed in 
\cite[Sec.~{4.8}]{rojikDiss}. 
%
\eremk
\end{remark}

We show in Figs.~\ref{fig:cube}, \ref{fig:cube-hole} the relative error in the
norm $\Vert\operatorname{curl}\cdot\Vert_{\mathbf{L}^{2}(\Omega)}+\left\vert
k\right\vert \Vert\cdot\Vert_{\mathbf{L}^{2}(\Omega)}\sim\Vert\cdot
\Vert_{\mathbf{H}(\operatorname{curl},\Omega),\left\vert k\right\vert }$
versus the number of degrees of freedom per wavelength
\[
N_{\lambda}=\frac{2\pi\operatorname{DOF}^{1/3}}{|k||\Omega|^{1/3}},
\]
where $\operatorname{DOF}$ stands for the dimension of the ansatz space.

\begin{example}
\label{example:cube} We consider $\Omega=(-1,1)^{3}$ and impose the right-hand
side and the impedance boundary conditions in such a way that the exact solution is
$\mathbf{E}\left(  \mathbf{x}\right)  =\operatorname{curl}\sin(kx_{1}%
)(1,1,1)^{\top}$. Fig.~\ref{fig:cube} shows the performance for the choices
$k\in\{10,20,30,40\}$ and $p\in\{1,2,3\}$ as the mesh is refined
quasi-uniformly.
The final problem sizes were $\operatorname{DOF}=18,609,324$ for $p=1$,
$\operatorname{DOF}=9,017,452$ for $p=2$, and $\operatorname{DOF}=23,052,940$
for $p=3$.

We observe the expected asymptotic $O(h^{p})$ convergence. We also observe
that the onset of asymptotic quasi-optimal convergence is reached for smaller
values of $N_{\lambda}$ for higher order methods. This is expected in view of
Theorem~\ref{thm:quasi-optimality}, although the present setting of a
piecewise analytic geometry is not covered by
Theorem~\ref{thm:quasi-optimality}. 
\eremk
\end{example}

\begin{example}
\label{example:cube_hole} We consider $\Omega=(-1,1)^{3}\setminus
\lbrack-1/2,1/2]^{3}$ and Maxwell's equations with impedance boundary
conditions on $\partial(-1,1)^{3}$ and perfectly conducting boundary
conditions on the inner boundary $\partial(-1/2,1/2)^{3}$. We prescribe an exact solution 
$\mathbf{E}(  \mathbf{x})  =k\cos(kx_{1})(x_{1}^{2}-1/4)(x_{2}%
^{2}-1/4)(x_{3}^{2}-1/4)(0,-1,1)^{\top}$. Fig.~\ref{fig:cube-hole} shows the
performance for the choices $k\in\{20,40,80\}$ and $p\in\{1,2,3\}$ as the mesh
is refined quasi-uniformly.
The final problem sizes were $\operatorname{DOF}=43,598,374$ for $p=1$,
$\operatorname{DOF}=168,035,046$ for $p=2$, and $\operatorname{DOF}%
=54,063,558$ for $p=3$.

We observe the expected asymptotic $O(h^{p})$ convergence. We also observe
that the onset of asymptotic quasi-optimal convergence is reached for smaller
values of $N_{\lambda}$ for higher order methods. 
\eremk
\end{example}

\ifarxiv
\appendix

\color{black}%

\section{Analytic regularity with bounds explicit in the wavenumber
(\cite{nicaise-tomezyk19,tomezyk19} revisited)\label{AppAnalyticity}}

In this appendix, we reproduce \cite[Appendix A]{nicaise-tomezyk19} and add
parts (marked in \color{red} red color \color{black} that are required to
account for the addition of right-hand sides in Theorem~\ref{ThmAnaRegSum}. We
will also follow the numbering in \cite[Appendix A]{nicaise-tomezyk19} and
indicate insertions by \textquotedblleft+1/2\textquotedblright,
\textquotedblleft+1/3\textquotedblright\ etc.

To be consistent with \cite{nicaise-tomezyk19}, we follow in this appendix the
\textquotedblleft French\textquotedblright\ convention and denote
\[
\mathbb{N}=\{0,1,2,\ldots\}\quad\mbox{ and }\quad\mathbb{N}^{\ast
}=\{1,2,\ldots\}.
\]
Furthermore, we assume as in \cite{nicaise-tomezyk19} that $k\geq k_{0}>0$,
while the estimates for negative $k\leq-k_{0}<0$ follow by replacing $k$ by
$\left\vert k\right\vert $ in
\color{black}%
all estimates.
\color{red}%
{For a bounded open set $\omega$, the seminorms $|\cdot|_{p,\omega}$ are
defined by \setcounter{equation}{-1}
\begin{equation}
|u|_{p,\omega}^{2}=\sum_{\alpha\in\mathbb{N}^{n}\colon|\alpha|=p}\frac
{|\alpha|!}{\alpha!}\Vert\partial^{\alpha}u\Vert_{L^{2}(\omega)}^{2}%
,\qquad\Vert u\Vert_{p,\omega}^{2}=\sum_{l=0}^{p}|u|_{l,\omega}^{2}.
\label{eq:A0}%
\end{equation}
}
%

\color{black}%
We assume that ${\mathbf{f}}$ is an analytic function satisfying%
\begin{subequations}
\label{eq:A1}
\end{subequations}%
%

\begin{equation}
|{\mathbf{f}}|_{p,\Omega}\leq C_{\mathbf{f}}\lambda_{\mathbf{f}}^{p}%
\max(p,k)^{p}\quad\forall p\in{\mathbb{N}}. \tag{%
\ref{eq:A1}%
}\label{eq:A10}%
\end{equation}
{%
\color{red}
Additionally, given a tubular neighborhood ${\mathcal{T}}$ of $\partial\Omega$
and (analytic) functions ${\mathbf{G}}_{D}:{\mathcal{T}}\rightarrow
{\mathbb{C}}^{N_{1}}$ ${\mathbf{G}}_{B}:{\mathcal{T}}\rightarrow{\mathbb{C}%
}^{N_{0}}$ that satisfy }
{%
\color{red}%
\begin{align}
\Vert{\mathbf{G}}_{D}\Vert_{p,{\mathcal{T}}}  &  \leq C_{{\mathbf{g}}_{D}%
}\lambda_{{\mathbf{g}}_{D}}^{p}\max(p,k)^{p}\qquad\forall p\in\mathbb{N},\tag{%
\ref{eq:A1}%
+1/3}\label{eq:A1a}\\
\Vert{\mathbf{G}}_{B}\Vert_{p,{\mathcal{T}}}  &  \leq C_{{\mathbf{g}}_{B}%
}\lambda_{{\mathbf{g}}_{B}}^{p}\max(p,k)^{p}\qquad\forall p\in\mathbb{N}.
\tag{%
\ref{eq:A1}%
+2/3}\label{eq:A1b}%
\end{align}
{We consider the two traces ${\mathbf{g}}_{D}:=\gamma\mathbf{G}_{D}${ and
}${\mathbf{g}}_{B}:=\gamma\mathbf{G}_{B}${. } Then, we can formulate in
analogy to \cite[Thm.~{A.1}]{nicaise-tomezyk19}} }

\begin{theorem}
\label{thm:A1}Let $\Omega\subset\mathbb{R}^{n}$, $n\leq3$, be a bounded domain
with an analytic boundary $\partial\Omega$, and let $(L,D,B)$ be an elliptic
system in the sense of {
\color{red}%
{ \cite[Def.~{2.2.31}, Def.~{2.2.33}]{CoDaNi} }}
with $L$ (resp.~D and B) an $N\times N$ (resp.~$N_{0}\times N$ and $N_{{%
\color{red}%
{1}}}\times N$ with $N_{0}$, $N_{1}\in{
\color{red}%
{\mathbb{N}} }
$ such that $N_{0}+N_{1}=N$) system of differential operators of order $2$
(resp.~$0$ and $1$) with analytic coefficients, $N\in\mathbb{N}^{\ast}$ and
$k>1$. Let ${\mathbf{f}}$, {
\color{red}%
{ ${\mathbf{g}}_{D}$, ${\mathbf{g}}_{B}$} }
be analytic functions verifying (\ref{eq:A10}) {
\color{red}%
{ and (\ref{eq:A1a}), (\ref{eq:A1b})} }
and ${\mathbf{G}}$ a matrix with analytic coefficients {
\color{red}%
{(in a tubular neighborhood of $\partial\Omega$)} }
. If ${\mathbf{u}}{
\color{red}%
{\in{\mathbf{H}}^{2}(\Omega)} }
$ is a solution of
\begin{subequations}
\label{eq:A2}
\end{subequations}%
\begin{equation}
\left\{
\begin{array}
[c]{ccll}%
L({\mathbf{u}}) & = & {\mathbf{f}}+k^{2}{\mathbf{u}} & \text{ in $\Omega$},\\
D({\mathbf{u}}) & = & {
\color{red}%
\mathbf{g}{_{D}}} & \text{ on $\partial\Omega$},\\
B({\mathbf{u}}) & = & k{\mathbf{G}}{\mathbf{u}}{
\color{red}%
{+{\mathbf{g}}_{B}}} & \text{ on $\partial\Omega$,}%
\end{array}
\right.  \tag{%
\ref{eq:A2}%
}\label{eq:A2neu}%
\end{equation}
then we have
\[
|{\mathbf{u}}|_{p,\Omega}\leq C_{\mathbf{u}}K^{p}\max(p,k)^{p},\quad\forall
p\in{\mathbb{N}},p\geq2
\]
with $C_{\mathbf{u}}=\left(  {
\color{red}%
{k^{-2}}}
C_{\mathbf{f}}+{
\color{red}%
{C_{{\mathbf{g}}_{D}}+k^{-1}C_{{\mathbf{g}}_{B}}}
+}\Vert{\mathbf{u}}\Vert_{\Omega}+k^{-1}\Vert{\mathbf{u}}\Vert_{1,\Omega
}\right)  $.
\end{theorem}


\begin{extraremark}{\Alph{section}.1+1/3}
The point of Theorem~\ref{thm:A1} is the tracking of the $k$-dependence. The key
ingredient of the proof below is the appeal to \cite[Prop.~{2.6.6}]{CoDaNi}
and \cite[Prop.~{2.6.7}]{CoDaNi}, which in turn rely on

\begin{enumerate}
[(a)]

\item local $H^{2}$-estimates of the form \cite[(2.47)]{CoDaNi}. These local
$H^{2}$-estimates result from the ellipticity of the operator $L$ and the
covering condition of the boundary conditions (cf.\ \cite[Cor.~{2.2.16}%
]{CoDaNi});

\item the condition \cite[(2.50)]{CoDaNi}. Condition \cite[(2.50)]{CoDaNi} is
satisfied by the ellipticity of the operator $L$ (see footnote $3$ in
\cite[p.~97]{CoDaNi}). \hbox{}\hfill%
\endproof

\end{enumerate}
\end{extraremark}

\begin{extraexample}{\Alph{section}.1+2/3}
We transform the problem
\begin{align}
\operatorname*{curl}\operatorname*{curl}{\mathbf{E}}-k^{2}{\mathbf{E}}  &
={\mathbf{f}}\quad\mbox{ in $\Omega$},\tag{%
\ref{eq:A2}%
+1/5}\label{EMWAPPa}\\
\gamma_{T}\operatorname*{curl}{\mathbf{E}}-\operatorname{i}k{\mathbf{E}}_{T}
&  ={\mathbf{g}}_{T}\quad\mbox{ on $\partial\Omega$}, \tag{%
\ref{eq:A2}%
+2/5}\label{EMWAPPb}%
\end{align}
for sufficiently smooth ${\mathbf{f}}$, ${\mathbf{g}}_{T}$ to the form
considered in Theorem~\ref{thm:A1}. We proceed as in \cite{nicaise-tomezyk19}
and write the Maxwell equations as a first order system for the electric field
$\mathbf{E}$ and the magnetic field $\mathbf{H}:=-\frac{\operatorname*{i}}%
{k}\operatorname*{curl}\mathbf{E}$ and set $\mathbf{u}=\left(  \mathbf{E}%
,\mathbf{H}\right)  $. The function ${\mathbf u}$ solves the elliptic system%
\begin{equation}%
\begin{array}
[c]{cclc}%
L\left(  \mathbf{u}\right)  := & \left(
\begin{array}
[c]{c}%
\operatorname*{curl}\operatorname*{curl}\mathbf{E}-\nabla\operatorname*{div}%
\mathbf{E}\\
\operatorname*{curl}\operatorname*{curl}\mathbf{H}-\nabla\operatorname*{div}%
\mathbf{H}%
\end{array}
\right)  & =\mathbf{F}+k^{2}\mathbf{u} & \text{in }\Omega,\\
D\left(  \mathbf{u}\right)  := & \mathbf{H}\times\mathbf{n}-\mathbf{E}_{T} &
=-\frac{\operatorname*{i}}{k}\mathbf{g}_{T} & \text{on }\Gamma,\\
B\left(  \mathbf{u}\right)  := & \left(
\begin{array}
[c]{c}%
\operatorname*{div}\mathbf{E}\\
\operatorname*{div}\mathbf{H}\\
\gamma_{T}\operatorname*{curl}\mathbf{H}+\left(  \operatorname*{curl}%
\mathbf{E}\right)  _{T}%
\end{array}
\right)  & =kG\mathbf{u}+\mathbf{G}_{\Gamma} & \text{on }\Gamma
\end{array}
\tag{%
\ref{eq:A2}%
+3/5}\label{eq:A2a}%
\end{equation}
for
\begin{equation}
\mathbf{F}:=\left(
\begin{array}
[c]{c}%
\mathbf{f}+\frac{1}{k^{2}}\nabla\operatorname*{div}\mathbf{f}\\
-\frac{\operatorname*{i}}{k}\operatorname*{curl}\mathbf{f}%
\end{array}
\right)  ,\text{\quad}G\mathbf{u}:=\left(
\begin{array}
[c]{c}%
0\\
0\\
\operatorname*{i}\left(  \mathbf{H}_{T}-\gamma_{T}\mathbf{E}\right)
\end{array}
\right)  \text{,\quad}\mathbf{G}_{\Gamma}:=\left(
\begin{array}
[c]{c}%
-\frac{1}{k^{2}}\left.  \left(  \operatorname*{div}\mathbf{f}\right)
\right\vert _{\Gamma}\\
0\\
-\frac{\operatorname*{i}}{k}\gamma_{T}\mathbf{f}%
\end{array}
\right)  . \tag{%
\ref{eq:A2}%
+4/5}\label{eq:A2aa}%
\end{equation}

Elliptic systems of the form (\ref{eq:A2a}) are considered in \cite[Thm.~{A1}%
]{nicaise-tomezyk19} for $\mathbf{g}_{T}=\mathbf{G}_{\Gamma}=\mathbf{0}$.

We now show that (\ref{eq:A2a}) is indeed an elliptic system in the sense of
\cite[Def.~{2.2.31}]{CoDaNi}. The operator $L$ is simply the block Laplacian
since $\operatorname{curl}\operatorname{curl}=\nabla\operatorname{div}-\Delta$
and clearly elliptic. To see that the boundary conditions are \emph{covering}
at a boundary point ${\mathbf{x}}_{0}$, we first consider the system
(\ref{eq:A2a}) in a half space $\{x_{3}>0\}$ with boundary $\{x_{3}=0\}$. We
assume that ${\mathbf{x}}_{0}=(0,0,0)^{\top}$. The principal part
$L^{\mathsf{pr}}$ of $L$ coincides with $L$. In the notation of
\cite[Notation~{2.2.1}]{CoDaNi}, the spaces $\mathfrak{M}[L;%
\mbox{\boldmath$ \xi$}%
^{\prime}]$ and $\mathfrak{M}_{+}[L;%
\mbox{\boldmath$ \xi$}%
^{\prime}]$ (with $%
\mbox{\boldmath$ \xi$}%
^{\prime}\in\mathbb{R}^{2}$) can be computed explicitly due to the simple
structure of $L$:
\begin{align*}
\mathfrak{M}[L;%
\mbox{\boldmath$ \xi$}%
^{\prime}]=\{{\mathbf{u}}(t)  &  =%
\mbox{\boldmath$ \eta$}%
e^{-|%
\mbox{\boldmath$ \xi$}%
^{\prime}|t}+%
\mbox{\boldmath$ \beta$}%
e^{|%
\mbox{\boldmath$ \xi$}%
^{\prime}|t}\,|\,%
\mbox{\boldmath$ \eta$}%
,%
\mbox{\boldmath$ \beta$}%
\in\mathbb{C}^{6}\},\\
\mathfrak{M}_{+}[L;%
\mbox{\boldmath$ \xi$}%
^{\prime}]=\{{\mathbf{u}}(t)  &  =%
\mbox{\boldmath$ \eta$}%
e^{-|%
\mbox{\boldmath$ \xi$}%
^{\prime}|t}\,|\,%
\mbox{\boldmath$ \eta$}%
\in\mathbb{C}^{6}\}.
\end{align*}
In particular, the operator $L$ is properly elliptic, \cite[Def.~{2.2.3}%
]{CoDaNi}. The boundary operators are obtained by inserting {$\mathbf{u}$%
}$({\mathbf{x}}^{\prime},t)=e^{\operatorname{i}{\mathbf{x}}^{\prime}\cdot%
\mbox{\boldmath$ \xi$}%
^{\prime}}e^{-|%
\mbox{\boldmath$ \xi$}%
^{\prime}|t}%
\mbox{\boldmath$ \eta$}%
$ into the boundary conditions and afterwards removing the factor
$e^{\operatorname{i}{\mathbf{x}}^{\prime}\cdot%
\mbox{\boldmath$ \xi$}%
^{\prime}}$ and setting $({\mathbf{x}}^{\prime},t)=(0,0,0)={\mathbf{x}}_{0}$,
which gives
\begin{align*}
D({\mathbf{u}})  &  =\left(
\begin{array}
[c]{c}%
\eta_{5}-\eta_{1}\\
-\eta_{4}-\eta_{2}%
\end{array}
\right) \\
B({\mathbf{u}})  &  =\left(
\begin{array}
[c]{c}%
\operatorname{i}\xi_{1}^{\prime}\eta_{1}+\operatorname{i}\xi_{2}^{\prime}%
\eta_{2}-|%
\mbox{\boldmath$ \xi$}%
^{\prime}|\eta_{3}\\
\operatorname{i}\xi_{1}^{\prime}\eta_{4}+\operatorname{i}\xi_{2}^{\prime}%
\eta_{5}-|%
\mbox{\boldmath$ \xi$}%
^{\prime}|\eta_{6}\\
-\operatorname{i}\xi_{1}^{\prime}\eta_{6}-|%
\mbox{\boldmath$ \xi$}%
^{\prime}|\eta_{4}+\operatorname{i}\xi_{2}^{\prime}\eta_{3}+|%
\mbox{\boldmath$ \xi$}%
^{\prime}|\eta_{2}\\
-\operatorname{i}\xi_{2}^{\prime}\eta_{6}-|\xi^{\prime}|\eta_{5}%
-\operatorname{i}\xi_{1}^{\prime}\eta_{3}-|%
\mbox{\boldmath$ \xi$}%
^{\prime}|\eta_{1}%
\end{array}
\right)  .
\end{align*}%
\color{red}%
In matrix form, we have
\[
\underbrace{\left(
\begin{array}
[c]{cccccc}%
-1 & 0 & 0 & 0 & 1 & 0\\
0 & -1 & 0 & -1 & 0 & 0\\
\operatorname{i}\xi_{1}^{\prime} & \operatorname{i}\xi_{2}^{\prime} & -|%
\mbox{\boldmath$ \xi$}%
^{\prime}| & 0 & 0 & 0\\
0 & 0 & 0 & \operatorname{i}\xi_{1}^{\prime} & \operatorname{i}\xi_{2}%
^{\prime} & -|%
\mbox{\boldmath$ \xi$}%
^{\prime}|\\
0 & |\xi^{\prime}| & \operatorname{i}\xi_{2}^{\prime} & -|%
\mbox{\boldmath$ \xi$}%
^{\prime}| & 0 & -\operatorname{i}\xi_{1}^{\prime}\\
-|%
\mbox{\boldmath$ \xi$}%
^{\prime}| & 0 & -\operatorname{i}\xi_{1}^{\prime} & 0 & -|%
\mbox{\boldmath$ \xi$}%
^{\prime}| & -\operatorname{i}\xi_{2}^{\prime}%
\end{array}
\right)  }_{=:{\mathbf{M}}}\left(
\begin{array}
[c]{c}%
\eta_{1}\\
\eta_{2}\\
\eta_{3}\\
\eta_{4}\\
\eta_{5}\\
\eta_{6}%
\end{array}
\right)  .
\]
A direct calculation shows $\operatorname{det}{\mathbf{M}}=|%
\mbox{\boldmath$ \xi$}%
^{\prime}|^{4}$. Hence, ${\mathbf{M}}$ is invertible for every $%
\mbox{\boldmath$ \xi$}%
^{\prime}\neq0$. This means that the set of boundary conditions satisfies the
covering condition. According to \cite[Def.~{2.2.31}]{CoDaNi}, this covering
condition has to be satisfied for every ${\mathbf{x}}_{0}\in\partial\Omega$
and the \textquotedblleft tangent\textquotedblright\ model system. Locally
flattening the boundary at ${\mathbf{x}}_{0}$ in such a way that the Jacobian
is orthogonal at ${\mathbf{x}}_{0}$ leads to \textquotedblleft
tangent\textquotedblright\ systems that are of the same form as above. Hence,
the system (\ref{eq:A2a}) is elliptic. \hbox{}\hfill%
\endproof

\end{extraexample}

{ \color{red} While \cite[Cor.~A.2]{nicaise-tomezyk19} focuses on the
regularity of $\mathbf{u}$, we formulate the following corollary differently
by giving a regularity assertion for for $\mathbf{E}$. }

\begin{corollary}
\label{CorA.2}Let $\Omega\subset{\mathbb{R}}^{3}$ be a bounded domain with an
analytic boundary. Let the function ${\mathbf{f}}$ satisfy (\ref{eq:A10}). {%
\color{red}
Let the function ${\mathbf{g}}:{\mathcal{T}}\rightarrow{\mathbb{C}}^{3}$
satisfy for a tubular neighborhood ${\mathcal{T}}$ of $\partial\Omega$
\[
|{\mathbf{g}}|_{p,{\mathcal{T}}}\leq C_{{\mathbf{g}}}\lambda_{{\mathbf{g}}%
}^{p}\max(p,k)^{p},\quad\forall p\in\mathbb{N}.
\]
Let ${\mathbf{E}}\in{\mathbf{H}}^{2}(\Omega,\operatorname*{curl})$ solve
(\ref{EMWAPPa}), (\ref{EMWAPPb}). Then, we have%
\[
|{\mathbf{E}}|_{p,\Omega}\leq C_{\mathbf{E}}K^{p}\max(p,k)^{p}\quad\forall
p\in{\mathbb{N}},\quad p\geq2,
\]
where, for some $C$, $K>0$ independent of $k$,
\[
C_{\mathbf{E}}=C\left(  C_{\mathbf{f}}k^{-2}+\left\vert k\right\vert
^{-1}C_{{\mathbf{g}}}+\frac{1}{\left\vert k\right\vert }\left\Vert
\mathbf{E}\right\Vert _{\mathbf{H}^{1}\left(  \Omega,\operatorname*{curl}%
\right)  ,k}\right)  .
\]
}

\end{corollary}

\begin{proof}
{
\color{red}%
We note that the analyticity of $\partial\Omega$ implies that the normal
vector $\mathbf{n}$ has an analytic extension $\mathbf{n}^{\star}$ to a
tubular neighborhood of $\partial\Omega$. Since the operators $\gamma_{T}$,
$\Pi_{T}$ have the form $\gamma_{T}\mathbf{v}=\gamma(\mathbf{v}\times
\mathbf{n}^{\star})$, $\Pi_{T}\mathbf{v}=\gamma(\mathbf{n}^{\star}%
\times(\mathbf{v}\times\mathbf{n}^{\star}))$, we see that $G\mathbf{u}$ and
$\mathbf{G}_{\Gamma}$ of (\ref{eq:A2aa}) can be written as
\[
G\mathbf{u}=\gamma\widetilde{G}\mathbf{u},\qquad\mathbf{G}_{\Gamma}%
=\gamma\widetilde{\mathbf{G}}_{\Gamma},\qquad\widetilde{\mathbf{G}}_{\Gamma
}:=\left(
\begin{array}
[c]{c}%
-\frac{1}{k^{2}}\operatorname{div}\mathbf{f}\\
0\\
-\frac{\operatorname{i}}{k}\mathbf{f}\times\mathbf{n}^{\star}%
\end{array}
\right)  ,
\]
where the matrix-valued function $\widetilde{G}$ is analytic in a tubular
neighborhood of $\partial\Omega$. Likewise, the boundary data $\mathbf{g}%
_{T}=\gamma(\mathbf{n}^{\star}\times(\mathbf{g}\times\mathbf{n}^{\star
}))=\gamma\widetilde{\mathbf{g}}$ for some $\widetilde{\mathbf{g}}$ that is
analytic in a tubular neighborhood of $\partial\Omega$. By, e.g.,
\cite[Lemma~{2.6}]{melenk2018wavenumber}, we get by setting $\mathbf{G}%
_{B}:=\widetilde{\mathbf{G}}_{\Gamma}$, $\mathbf{G}_{D}:=-\operatorname{i}%
k^{-1}\widetilde{\mathbf{g}}$ for suitable $K_{\mathbf{f}}$,
$\widetilde{\lambda}_{\mathbf{g}}$ independent of $k$
\begin{align*}
\Vert{\mathbf{F}}\Vert_{p,\Omega}  &  \leq CC_{\mathbf{f}}K_{\mathbf{f}}%
^{p}\max(p,k)^{p},\quad\forall p\in{\mathbb{N}},\\
\Vert{\mathbf{G}}_{B}\Vert_{p,\Omega}  &  \leq k^{-1}CC_{\mathbf{f}%
}K_{\mathbf{f}}^{p}\max(p,k)^{p},\quad\forall p\in{\mathbb{N}},\\
\Vert{\mathbf{G}}_{D}\Vert_{p,\Omega}  &  \leq k^{-1}CC_{\mathbf{g}%
}\widetilde{\lambda}_{\mathbf{g}}^{p}\max(p,k)^{p},\quad\forall p\in
{\mathbb{N}}.
\end{align*}%
\color{red}%
We may apply Theorem~\ref{thm:A1} with $G=\widetilde{G}$ and ${\mathbf{F}}$
given by (\ref{eq:A2a}). To that end, we observe that the regularity
assumption {${\mathbf{E}}\in{\mathbf{H}}^{2}(\Omega,\operatorname*{curl})$
implies }$\mathbf{u}=\left(  \mathbf{E},\mathbf{H}\right)  =\left(
\mathbf{E},-\frac{\operatorname*{i}}{k}\operatorname*{curl}\mathbf{E}\right)
\in\mathbf{H}^{2}\left(  \Omega\right)  $ so that Thm.~\ref{thm:A1} is indeed
applicable.}
\end{proof}

\color{black}%

\begin{remark}
\thinspace%
\color{red}%
See{ \cite[Rem.~{A.3}]{nicaise-tomezyk19}. \hbox{}\hfill%
\endproof
}
\end{remark}

\subsection{%
\color{black}%
Analytic regularity near the boundary}


\color{black}%
Denote $B_{R}^{+}=B_{R}(0,R)\cap\{\mathbf{x}|x_{n}>0\}$ and $\Gamma
_{R}=\{\mathbf{x}\in\overline{B_{R}^{+}}|x_{n}=0\}$, with $R\in(0,1]$. {%
\color{red}
We emphasize that in the following developments, most estimates will not be
sharp in their dependence on $R$. }

Let ${\mathbf{f}}$ be an analytic function and ${\mathbf{G}}$ be a matrix with
analytic coefficients defined on $\overline{B_{R}^{+}}$ such that
\begin{subequations}
\label{eq:A3}
\end{subequations}%
\begin{subequations}
\label{eq:A4}
\end{subequations}%
\begin{align}
\Vert\partial^{\alpha}{\mathbf{f}}\Vert_{B_{R}^{+}}  &  \leq C_{\mathbf{f}%
}\lambda_{\mathbf{f}}^{|\alpha|}\max(|\alpha|,k)^{|\alpha|}\qquad\forall
\alpha\in\mathbb{N}^{n},\tag{%
\ref{eq:A3}%
}\label{eq:A3a}\\
\Vert\partial^{\alpha}{\mathbf{G}}\Vert_{\infty,B_{R}^{+}}  &  \leq
C_{\mathbf{G}}\lambda_{\mathbf{G}}^{|\alpha|}\max(|\alpha|,k)^{|\alpha|}%
\qquad\forall\alpha\in\mathbb{N}^{n} \tag{%
\ref{eq:A4}%
}\label{eq:A4neu}%
\end{align}
{%
\color{red}
as well as the functions ${\mathbf{g}}_{D}:={\mathbf{G}}_{D}|_{\Gamma_{R}}$
and ${\mathbf{g}}_{B}:={\mathbf{G}}_{B}|_{\Gamma_{R}}$ with%
\begin{align}
\Vert\partial^{\alpha}{\mathbf{G}}_{D}\Vert_{B_{R}^{+}}  &  \leq
C_{{\mathbf{g}}_{D}}\lambda_{{\mathbf{g}}_{D}}^{|\alpha|}\max(|\alpha
|,k)^{|\alpha|}\qquad\forall\alpha\in\mathbb{N}^{n},\tag{%
\ref{eq:A4}%
+1/3}\label{eq:A4a}\\
\Vert\partial^{\alpha}{\mathbf{G}}_{B}\Vert_{B_{R}^{+}}  &  \leq
C_{{\mathbf{g}}_{B}}\lambda_{{\mathbf{g}}_{B}}^{|\alpha|}\max(|\alpha
|,k)^{|\alpha|}\qquad\forall\alpha\in\mathbb{N}^{n} \tag{%
\ref{eq:A4}%
+2/3}\label{eq:A4b}%
\end{align}
}
for some $k\geq1$ for some positive constants $C_{\mathbf{f}}$, $\lambda
_{\mathbf{f}}$, $C_{\mathbf{G}}$, $\lambda_{\mathbf{G}}$, {%
\color{red}
$C_{{\mathbf{g}}_{D}}$, $\lambda_{{\mathbf{g}}_{D}}$, $C_{{\mathbf{g}}_{B}}$,
$\lambda_{{\mathbf{g}}_{B}}$. }
Let ${\mathbf{u}}\in{\mathbf{H}}^{2}(B_{R}^{+})$ be a solution of
\begin{equation}
\left\{
\begin{array}
[c]{rcll}%
L({\mathbf{u}}) & = & {\mathbf{f}}+k^{2}{\mathbf{u}} & \text{ in $B_{R}^{+}$%
},\\
D({\mathbf{u}}) & = & {%
\color{red}%
{\mathbf{g}_{D}}} & \text{ on $\Gamma_{R}$},\\
B({\mathbf{u}}) & = & k{\mathbf{G}}{\mathbf{u}{%
\color{red}%
+{\mathbf{g}}_{B}}} & \text{ on $\Gamma_{R}$,}%
\end{array}
\right.  \label{eq:A5}%
\end{equation}
where $(L,D,B)$ is an elliptic system with analytic coefficients (in the sense
above), with $D$ (resp.~$B$) an operator of order $0$ (resp.~$1$). For further
purposes, we define a few norms\footnote{cf.\ \cite[Notation~{2.6.5}]{CoDaNi}
and \cite[Notation~{2.5.1}]{CoDaNi}} \label{eq:nicaise-seminorms}
\begin{align*}
|{\mathbf{u}}|_{p,q,B_{R}^{+}}  &  :=\max_{\substack{|\alpha|=p\\\alpha
_{n}\leq q}}\Vert\partial^{\alpha}{\mathbf{u}}\Vert_{B_{R}^{+}},\\
{[\![{\mathbf{u}}]\!]}_{p,q,B_{R}^{+}}  &  :=\max_{0\leq\rho\leq R/(2p)}%
\rho^{p}|{\mathbf{u}}|_{p,q,B_{R-p\rho}^{+}},\qquad\mbox{ for all $p > 0$},\\
{[\![{\mathbf{u}}]\!]}_{0,0,B_{R}^{+}}  &  :=\Vert{\mathbf{u}}\Vert_{B_{R}%
^{+}},\\
\rho_{\ast}^{2}{[\![{\mathbf{u}}]\!]}_{p,q,B_{R}^{+}}  &  :=\max_{0\leq
\rho\leq R/(2(p+1))}\rho^{p+2}|{\mathbf{u}}|_{p,q,B_{R-(p+1)\rho}^{+}},\\
|{\mathbf{u}}|_{p,\frac{1}{2},\Gamma_{R}}  &  :=\max_{\alpha^{\prime}%
\in\mathbb{N}^{n-1}\colon|\alpha^{\prime}|=p}\Vert\partial^{\alpha^{\prime}%
}{\mathbf{u}}\Vert_{\frac{1}{2},\Gamma_{R}},\\
\rho_{\ast}^{\frac{3}{2}}{[\![{\mathbf{u}}]\!]}_{p,\frac{1}{2},\Gamma_{R}}  &
:=\max_{0\leq\rho\leq R/(2(p+1))}\rho^{p+\frac{3}{2}}|{\mathbf{u}}%
|_{p,\frac{1}{2},\Gamma_{R-(p+1)\rho}},\\
{%
\color{red}%
|{\mathbf{u}}|_{p,\frac{3}{2},\Gamma_{R}}}  &  {%
\color{red}%
:=\max_{\alpha^{\prime}\in{\mathbb{N}}^{n-1}\colon|\alpha^{\prime}|=p}%
\Vert\partial^{\alpha^{\prime}}{\mathbf{u}}\Vert_{\frac{3}{2},\Gamma_{R}},}\\
{%
\color{red}%
\rho_{\ast}^{\frac{1}{2}}{[\![{\mathbf{u}}]\!]}_{p,\frac{3}{2},\Gamma_{R}}}
&  {%
\color{red}%
:=\max_{0\leq\rho\leq R/(2(p+1))}\rho^{\frac{1}{2}+p}|{\mathbf{u}}%
|_{p,\frac{3}{2},\Gamma_{R-(p+1)\rho)}},}%
\end{align*}
for all $p$, $q\in\mathbb{N}$ with $q\leq p$.

We will first estimate the norm of the tangential derivatives (and the normal
derivative up to 2) by using the standard analytic regularity of elliptic
systems. Then, we will be able to estimate the complete norm ${[\![
{\mathbf{u}}]\!]}_{p,q,B_{R}^{+}}$. So we start with an estimation of the norm
of tangential derivatives ${[\![ {\mathbf{u}}]\!]}_{p,2,B_{R}^{+}}$. Before,
let us prove the next technical results that allow to pass from a sum on the
multi-indices to a sum on their lengths.

\begin{lemma}
\label{lemma:A4} Let $h$ be a mapping from $\mathbb{N}$ into $[0,\infty)$ and
a multi-index $\alpha^{\prime}\in\mathbb{N}^{n-1}$, for $n = 2$ or $n = 3$.
Then we have
\begin{equation}
\sum_{\beta^{\prime}\in\mathbb{N}^{n-1}\colon\beta^{\prime}\leq\alpha^{\prime
}} h(|\beta^{\prime}|) \leq2 \sum_{p=0}^{|\alpha^{\prime}|} h(p)
e^{|\alpha^{\prime}| - p}.
\end{equation}

\end{lemma}

\begin{proof}
See \cite[Lemma~{A.4}]{nicaise-tomezyk19}.
\end{proof}

\begin{corollary}
Let $h$ be a mapping from $\mathbb{N}$ into $[0,\infty)$ and a multi-index
$\alpha\in\mathbb{N}^{n}$ with $\alpha_{n}\leq1$. Then we have
\setcounter{equation}{7}
\begin{equation}
\sum_{\beta\in\mathbb{N}^{n}\colon\beta\leq\alpha}h(|\beta|)\leq2(1+\frac
{1}{e})\sum_{p=0}^{|\alpha|}h(p)e^{|\alpha|-p}.
\end{equation}

\end{corollary}

\begin{proof}
See \cite[Cor.~{A.5}]{nicaise-tomezyk19}.
\end{proof}

\begin{lemma}
\label{lemma:A6} There exists a positive constant $C$ (depending on $n$), a
positive constant $C_{tr,R}$ (depending only on $R\leq1$), and a positive
constant $\lambda_{{\mathbf{G}}}^{\prime}\geq\lambda_{{\mathbf{G}}}$ such that
for all $l\in\mathbb{N}$ and any ${\mathbf{u}}\in{\mathbf{H}}^{l+1}(B_{R}%
^{+})$, we have
\begin{equation}
\rho_{\ast}^{\frac{3}{2}}{[\![{\mathbf{G}}{\mathbf{u}}]\!]}_{l,\frac{1}%
{2},\Gamma_{R}}\leq CC_{tr,R}C_{{\mathbf{G}}}\sum_{p=0}^{l+1}(\lambda
_{{\mathbf{G}}}^{\prime})^{l+1-p}\max(l+1,k)^{l+1-p}{[\![{\mathbf{u}}%
]\!]}_{p,2,B_{R}^{+}}. \label{eq:A9}%
\end{equation}
{%
\color{red}
Additionally, the constant $C_{tr,R}$ may be assumed to satisfy the trace
estimate
\[
\Vert{\mathbf{v}}\Vert_{\frac{1}{2},\Gamma_{R}}\leq C_{tr,R}\Vert{\mathbf{v}%
}\Vert_{1,B_{R}^{+}}\qquad\forall{\mathbf{v}}\in{\mathbf{H}}^{1}(B_{R}^{+}).
\]
}

\end{lemma}

\begin{proof}
See \cite[Lemma~{A.6}]{nicaise-tomezyk19}.
\end{proof}

{%
\color{red}
Noting the trace estimate (cf.\ \cite[(A.10)]{nicaise-tomezyk19}), we obtain
from (\ref{eq:A4a}), (\ref{eq:A4b}) for the traces ${\mathbf{g}}_{D}$ and
${\mathbf{g}}_{B}$ for suitable $C_{tr,R}^{\prime}$ \setcounter{equation}{12}%
\begin{subequations}
\label{eq:A13}
\end{subequations}%
\begin{align}
\rho_{\ast}^{\frac{1}{2}}{[\![{\mathbf{g}}_{D}]\!]}_{l,\frac{3}{2},\Gamma
_{R}}  &  \leq C_{tr,R}^{\prime}\left[  {[\![{\mathbf{G}}_{D}]\!]}%
_{l,0,B_{R}^{+}}+{[\![{\mathbf{G}}_{D}]\!]}_{l+2,2,B_{R}^{+}}\right] \tag{%
\ref{eq:A13}%
+1/4}\label{eq:A13a}\\
&  \overset{(\ref{eq:A4a})}{\leq}2C_{tr,R}^{\prime}C_{{\mathbf{g}}_{D}}%
\lambda_{{\mathbf{g}}_{D}}^{l+2}\max(l+2,k)^{l+2},\nonumber\\
\rho_{\ast}^{\frac{3}{2}}{[\![{\mathbf{g}}_{B}]\!]}_{l,\frac{1}{2},\Gamma
_{R}}  &  \leq C_{tr,R}\left[  {[\![{\mathbf{G}}_{B}]\!]}_{l,0,B_{R}^{+}%
}+{[\![{\mathbf{G}}_{B}]\!]}_{l+1,1,B_{R}^{+}}\right] \tag{%
\ref{eq:A13}%
+1/2}\label{eq:A13b}\\
&  \overset{(\ref{eq:A4b})}{\leq}2C_{tr,R}C_{{\mathbf{g}}_{B}}\lambda
_{{\mathbf{g}}_{B}}^{l+1}\max(l+2,k)^{l+1},\nonumber\\
\Vert{\mathbf{v}}\Vert_{\frac{3}{2},\Gamma_{R}}  &  \leq C_{tr,R}^{\prime
}\Vert{\mathbf{v}}\Vert_{2,B_{R}^{+}}\qquad\forall{\mathbf{v}}\in{\mathbf{H}%
}^{2}(B_{R}^{+}). \tag{%
\ref{eq:A13}%
+3/4}\label{eq:A13c}%
\end{align}
Here, we have implicitly assumed $\lambda_{{\mathbf{g}}_{D}}\geq1$ and
$\lambda_{{\mathbf{g}}_{B}}\geq1$. We also emphasize that the constants
$C_{tr,R}$ and $C_{tr,R}^{\prime}$ depend on $R$, which could be tracked by a
more careful application of the trace estimates. }

\begin{lemma}
\label{lemma:A7} Let ${\mathbf{u}}\in{\mathbf{H}}^{2}(B_{R}^{+})$ be a
solution of (\ref{eq:A5}) with ${\mathbf{f}}$, {%
\color{red}
${\mathbf{g}}_{D}$, ${\mathbf{g}}_{B}$, }
${\mathbf{G}}$ analytic and satisfying (\ref{eq:A3})---(\ref{eq:A4b}). Then
there exist $K>1$ and $C_{R}>1$ such that for all $p\geq2$
\[
{[\![{\mathbf{u}}]\!]}_{p,2,B_{R}^{+}}\leq C_{{\mathbf{u}}}(B_{R}^{+}%
)K^{p}\max(p,k)^{p}%
\]
with {%
\color{red}
$C_{{\mathbf{u}}}(B_{R}^{+})=C_{R}(C_{\mathbf{f}} {%
\color{red}%
k^{-2} }
+C_{{\mathbf{g}}_{D}}+k^{-1}C_{{\mathbf{g}}_{B}}+\Vert{\mathbf{u}}\Vert
_{B_{R}^{+}}+k^{-1}\Vert{\mathbf{u}}\Vert_{1,B_{R}^{+}})$. }

\end{lemma}

\begin{proof}
We will prove this result by induction, by applying a standard analytic
regularity result (i.e., \cite[Prop.~{2.6.6}]{CoDaNi}), which gives us a real
number $A\geq1$ such that for all $p\geq2$ \setcounter{equation}{13}
\begin{equation}
{[\![{\mathbf{u}}]\!]}_{p,2,B_{R}^{+}}\leq\sum_{l=0}^{p-2}A^{p-1-l}\left(
\rho_{\ast}^{2}{[\![L({\mathbf{u}})]\!]}_{l,0,B_{R}^{+}}+\rho_{\ast}^{\frac
{3}{2}}{[\![B({\mathbf{u}})]\!]}_{l,\frac{1}{2},\Gamma_{R}}+{%
\color{red}%
\rho_{\ast}^{\frac{1}{2}}{[\![D({\mathbf{u}})]\!]}_{l,\frac{3}{2},\Gamma_{R}}%
}\right)  +A^{p-1}\sum_{l=0}^{1}{[\![{\mathbf{u}}]\!]}_{l,l,B_{R}^{+}}.
\label{eq:A14}%
\end{equation}
\textbf{Initialization:} For $p=2$, by (\ref{eq:A14}) we have
\begin{align*}
&  {[\![{\mathbf{u}}]\!]}_{2,2,B_{R}^{+}}\leq A\left(  \rho_{\ast}%
^{2}{[\![L({\mathbf{u}})]\!]}_{0,0,B_{R}^{+}}+\rho_{\ast}^{\frac{3}{2}%
}{[\![B({\mathbf{u}})]\!]}_{0,\frac{1}{2},\Gamma_{R}}{%
\color{red}%
+\rho_{\ast}^{\frac{1}{2}}{[\![D({\mathbf{u}})]\!]}_{0,\frac{3}{2},\Gamma_{R}%
}}\right)  +A\sum_{l=0}^{1}{[\![{\mathbf{u}}]\!]}_{l,l,B_{R}^{+}}\\
&  \leq A\left(  \rho_{\ast}^{2}{[\![{\mathbf{f}}+k^{2}{\mathbf{u}}%
]\!]}_{0,0,B_{R}^{+}}+\rho_{\ast}^{\frac{3}{2}}{[\![k{\mathbf{G}}{\mathbf{u}}{%
\color{red}%
+{\mathbf{g}}_{B}}]\!]}_{0,\frac{1}{2},\Gamma_{R}}{%
\color{red}%
+\rho_{\ast}^{\frac{1}{2}}{[\![{\mathbf{g}}_{D}]\!]}_{0,\frac{3}{2},\Gamma
_{R}}}\right)  +A\sum_{l=0}^{1}{[\![{\mathbf{u}}]\!]}_{l,l,B_{R}^{+}}\\
&  \overset{R\leq1}{\leq}A\left(  \Vert{\mathbf{f}}\Vert_{B_{R}^{+}}%
+k^{2}\Vert{\mathbf{u}}\Vert_{B_{R}^{+}}+k\Vert{\mathbf{G}}{\mathbf{u}}%
\Vert_{\frac{1}{2},\Gamma_{R}}{%
\color{red}%
+\Vert{\mathbf{g}}_{B}\Vert_{\frac{1}{2},\Gamma_{R}}}{%
\color{red}%
+\Vert{\mathbf{g}}_{D}\Vert_{\frac{3}{2},\Gamma_{R}}}\right)  +A\sum_{l=0}%
^{1}{[\![{\mathbf{u}}]\!]}_{l,l,B_{R}^{+}}\\
&  \overset{(\ref{eq:A13a}),(\ref{eq:A13b})}{\leq}A\left(  \Vert{\mathbf{f}%
}\Vert_{B_{R}^{+}}+(k^{2}+1)\Vert{\mathbf{u}}\Vert_{B_{R}^{+}}+kC_{tr,R}%
\Vert{\mathbf{G}}{\mathbf{u}}\Vert_{1,B_{R}^{+}}{%
\color{red}%
+C_{tr,R}\Vert{\mathbf{G}}_{B}\Vert_{1,B_{R}^{+}}}\right. \\
&  \quad\left.  {%
\color{red}%
+C_{tr,R}^{\prime}\Vert{\mathbf{G}}_{D}\Vert_{2,B_{R}^{+}}}+\Vert{\mathbf{u}%
}\Vert_{1,B_{R}^{+}}\right)
\end{align*}
with the constants $C_{tr,R}$, {%
\color{red}%
$C_{tr,R}^{\prime}$ }
introduced before. By noticing that
\begin{align*}
kC_{tr,R}\Vert{\mathbf{G}}{\mathbf{u}}\Vert_{1,B_{R}^{+}}  &  \leq{%
\color{red}%
CkC_{tr,R}C_{\mathbf{G}}}\left(  \Vert{\mathbf{u}}\Vert_{1,B_{R}^{+}}%
+\lambda_{\mathbf{G}}\Vert{\mathbf{u}}\Vert_{B_{R}^{+}}\right)  ,\\
{%
\color{red}%
\Vert{\mathbf{G}}_{B}\Vert_{1,B_{R}^{+}}}  &  {%
\color{red}%
\leq CC_{{\mathbf{g}}_{B}}\lambda_{{\mathbf{g}}_{B}}k},\\
{%
\color{red}%
\Vert{\mathbf{G}}_{D}\Vert_{2,B_{R}^{+}}}  &  {\color{red}\leq CC_{{\mathbf{g}%
}_{D}}\lambda_{{\mathbf{g}}_{D}}^{2}k^{2}}%
\end{align*}
we then have
\begin{align*}
{[\![{\mathbf{u}}]\!]}_{2,2,B_{R}^{+}}  &  \leq A\bigl(\Vert{\mathbf{f}}%
\Vert_{B_{R}^{+}}+(k^{2}+1+CC_{tr,R}C_{\mathbf{G}}\lambda_{\mathbf{G}}%
k)\Vert{\mathbf{u}}\Vert_{B_{R}^{+}}+(CC_{tr,R}C_{\mathbf{G}}k+1)\Vert
{\mathbf{u}}\Vert_{1,B_{R}^{+}}\bigr)\\
&  \qquad{%
\color{red}%
+A\bigl(CC_{tr,R}C_{{\mathbf{g}}_{B}}k\lambda_{{\mathbf{g}}_{B}}%
+CC_{tr,R}^{\prime}C_{{\mathbf{g}}_{D}}k^{2}\lambda_{{\mathbf{g}}_{D}}%
^{2}\bigr)}\\
&  \leq Ak^{2}\Bigl[{%
\color{red}%
k^{-2}}C_{\mathbf{f}}{%
\color{red}%
+k^{-1}CC_{tr,R}C_{{\mathbf{g}}_{B}}\lambda_{{\mathbf{g}}_{B}}+CC_{tr,R}%
^{\prime}C_{{\mathbf{g}}_{D}}\lambda_{{\mathbf{g}}_{D}}^{2}+}\\
&  \qquad(2+CC_{tr,R}C_{\mathbf{G}}\lambda_{\mathbf{G}}{%
\color{red}%
k^{-1}})\Vert{\mathbf{u}}\Vert_{B_{R}^{+}}+(CC_{tr,R}C_{\mathbf{G}}+{%
\color{red}%
k^{-1}})k^{-1}\Vert{\mathbf{u}}\Vert_{1,B_{R}^{+}}\Bigr]\\
&  \leq Ak^{2}\max(2+CC_{tr,R}C_{\mathbf{G}}{%
\color{red}%
k^{-1}}\lambda_{\mathbf{G}},CC_{tr,R}C_{\mathbf{G}}+{%
\color{red}%
k^{-1}},{%
\color{red}%
CC_{tr,R}\lambda_{{\mathbf{g}}_{B}},CC_{tr,R}^{\prime}\lambda_{{\mathbf{g}%
}_{D}}^{2}})\times\\
&  \qquad\left(  C_{\mathbf{f}}{%
\color{red}%
k^{-2}}+{%
\color{red}%
k^{-1}C_{{\mathbf{g}}_{B}}+C_{{\mathbf{g}}_{D}}}+\Vert{\mathbf{u}}\Vert
_{B_{R}^{+}}+k^{-1}\Vert{\mathbf{u}}\Vert_{1,B_{R}^{+}}\right) \\
&  \leq{%
\color{red}%
C_{R}}C_{\mathbf{u}}(B_{R}^{+})\max(2,k)^{2}\leq C_{\mathbf{u}}(B_{R}%
^{+})K^{2}\max(2,k)^{2},
\end{align*}%
\color{black}%
with $C_{R}\geq A\max(2+CC_{tr,R}C_{\mathbf{G}}{%
\color{red}%
k^{-1}}
\lambda_{\mathbf{G}},CC_{tr,R}C_{\mathbf{G}}+{%
\color{red}%
k^{-1}}
,{%
\color{red}%
CC_{tr,R}\lambda_{{\mathbf{g}}_{B}},CC_{tr,R}^{\prime}\lambda_{{\mathbf{g}%
}_{D}}^{2}}
)\geq1$ and since $K\geq1$.

\textbf{Induction hypothesis:} For all $2\leq p^{\prime}\leq p$, we have
\begin{equation}
{[\![{\mathbf{u}}]\!]}_{p^{\prime},2,B_{R}^{+}}\leq C_{\mathbf{u}}(B_{R}%
^{+})K^{p^{\prime}}\max(p^{\prime},k)^{p^{\prime}}. \label{eq:A15}%
\end{equation}
We will show this estimate for $p+1$: Using (\ref{eq:A14}), we may write
\begin{equation}
{[\![{\mathbf{u}}]\!]}_{p+1,2,B_{R}^{+}}\leq\sum_{l=0}^{p-1}A^{p-l}\left(
\rho_{\ast}^{2}{[\![L({\mathbf{u}})]\!]}_{l,0,B_{R}^{+}}+\rho_{\ast}^{\frac
{3}{2}}{[\![B({\mathbf{u}})]\!]}_{l,\frac{1}{2},\Gamma_{R}}+{%
\color{red}%
\rho_{\ast}^{\frac{1}{2}}{[\![D({\mathbf{u}})]\!]}_{l,\frac{3}{2},\Gamma_{R}}%
}\right)  +A^{p}\sum_{l=0}^{1}{[\![\mathbf{u}]\!]}_{l,l,B_{R}^{+}}.
\label{eq:A16}%
\end{equation}
We now need to estimate each term of this right-hand side. We start by
estimating $\rho_{\ast}^{2}{[\![L({\mathbf{u}})]\!]}_{l,0,B_{R}^{+}}$ for
$l\leq p-1$: First we notice that
\[
\rho_{\ast}^{2}{[\![L({\mathbf{u}})]\!]}_{l,0,B_{R}^{+}}\overset{{%
\color{red}%
R\leq1}}{\leq}{[\![{\mathbf{f}}+k^{2}{\mathbf{u}}]\!]}_{l,0,B_{R}^{+}}%
\leq{[\![{\mathbf{f}}]\!]}_{l,0,B_{R}^{+}}+k^{2}{[\![{\mathbf{u}}%
]\!]}_{l,2,B_{R}^{+}}.
\]
By the induction hypothesis (\ref{eq:A15}), we then have
\begin{align*}
\rho_{\ast}^{2}{[\![L({\mathbf{u}})]\!]}_{l,0,B_{R}^{+}}  &  \leq
C_{\mathbf{f}}\lambda_{\mathbf{f}}^{l}\max(l,k)^{l}+k^{2}C_{\mathbf{u}}%
(B_{R}^{+})K^{l}\max(l,k)^{l}\\
&  \leq k^{2}\max(l,k)^{l}C_{\mathbf{u}}(B_{R}^{+})K^{l}\left(  {%
\color{red}%
\frac{\lambda_{\mathbf{f}}^{l}}{K^{l}}}+1\right)  .
\end{align*}
As $l+2\leq p+1$, this estimate directly implies that
\begin{align*}
\rho_{\ast}^{2}{[\![L({\mathbf{u}})]\!]}_{l,0,B_{R}^{+}}  &  \leq
\max(p+1,k)^{p+1}C_{\mathbf{u}}(B_{R}^{+})K^{l}\left(  \left(  \frac
{\lambda_{\mathbf{f}}^{l}}{K}\right)  ^{l}+1\right) \\
&  \leq2\max(p+1,k)^{p+1}C_{\mathbf{u}}(B_{R}^{+})K^{l},
\end{align*}
for $K>\lambda_{\mathbf{f}}$. Multiplying this estimate by $A^{p-l}$ and
summing on $l$, one gets
\begin{align*}
\sum_{l=0}^{p-1}A^{p-l}\rho_{\ast}^{2}{[\![L({\mathbf{u}})]\!]}_{l,0,B_{R}%
^{+}}  &  \leq C_{\mathbf{u}}(B_{R}^{+})K^{p+1}\max(p+1,k)^{p+1}\frac{2}%
{K}\sum_{l=0}^{p-1}A^{p-l}K^{l-p}\\
&  \leq C_{\mathbf{u}}(B_{R}^{+})K^{p+1}\max(p+1,k)^{p+1}\frac{2}{K}\sum
_{l=0}^{p-1}\left(  \frac{A}{K}\right)  ^{p-l}.
\end{align*}
If $K\geq2A$, then $\sum_{l=0}^{p-1}(\frac{A}{K})^{p-l}\leq\sum_{l=1}^{\infty
}(\frac{A}{K})^{l}\leq1$, which yields
\begin{equation}
\sum_{l=0}^{p-1}A^{p-l}\rho_{\ast}^{2}{[\![L({\mathbf{u}})]\!]}_{l,0,B_{R}%
^{+}}\leq C_{\mathbf{u}}(B_{R}^{+})K^{p+1}\max(p+1,k)^{p+1}\frac{2}{K}.
\label{eq:A17}%
\end{equation}
\emph{Estimation of $\rho_{\ast}^{\frac{3}{2}}{[\![B({\mathbf{u}}%
)]\!]}_{l,\frac{1}{2},\Gamma_{R}}$:} By the boundary condition on
${\mathbf{u}}$, we have
\[
\rho_{\ast}^{\frac{3}{2}}{[\![B({\mathbf{u}})]\!]}_{l,\frac{1}{2},\Gamma_{R}%
}\leq k\rho_{\ast}^{\frac{3}{2}}{[\![{\mathbf{G}}{\mathbf{u}}]\!]}_{l,\frac
{1}{2},\Gamma_{R}}{%
\color{red}%
+\rho_{\ast}^{\frac{3}{2}}{[\![{\mathbf{g}}_{B}]\!]}_{l,\frac{1}{2},\Gamma
_{R}}}
\]%
\color{black}%
and by the estimate (\ref{eq:A9}) {
\color{red}%
and (\ref{eq:A13b}) }
, we get
\begin{align}
\rho_{\ast}^{\frac{3}{2}}{[\![B({\mathbf{u}})]\!]}_{l,\frac{1}{2},\Gamma_{R}}
&  \leq kCC_{tr,R}C_{\mathbf{G}}\sum_{p^{\prime}=0}^{l+1}(\lambda_{\mathbf{G}%
}^{\prime})^{l+1-p^{\prime}}\max(l+1,k)^{l+1-p^{\prime}}{[\![{\mathbf{u}}%
]\!]}_{p^{\prime},2,B_{R}^{+}}\label{eq:A18}\\
&  \qquad{%
\color{red}%
+2C_{tr,R}C_{{\mathbf{g}}_{B}}\lambda_{{\mathbf{g}}_{B}}^{l+1}\max
(l+2,k)^{l+1}=:I+II.}\nonumber
\end{align}
{
\color{red}%
For the first term, $I$, }the induction hypothesis (\ref{eq:A15}) then leads
to
\begin{align*}
{%
\color{red}%
I}  &  \leq kCC_{tr,R}C_{\mathbf{G}}C_{\mathbf{u}}(B_{R}^{+})\sum_{p^{\prime
}=0}^{l+1}(\lambda_{\mathbf{G}}^{\prime})^{l+1-p^{\prime}}K^{p^{\prime}}%
\max(l+1,k)^{l+1-p^{\prime}}\max(p^{\prime},k)^{p^{\prime}}\\
&  \leq kCC_{tr,R}C_{\mathbf{G}}C_{\mathbf{u}}(B_{R}^{+})\max(l+1,k)^{l+1}%
K^{l+1}\sum_{p^{\prime}=0}^{l+1}\left(  \frac{\lambda_{\mathbf{G}}^{\prime}%
}{K}\right)  ^{l+1-p^{\prime}}.
\end{align*}
Hence, for $K\geq2\lambda_{\mathbf{G}}^{\prime}$ {%
\color{red}%
and $K\geq\lambda_{{\mathbf{g}}_{B}}$ }
(recalling that $l+2\leq p+1$ and that $\sum_{p^{\prime}=0}^{l+1}\left(
\frac{\lambda_{\mathbf{G}}^{\prime}}{K}\right)  ^{l+1-p^{\prime}}\leq2$), we
deduce%
\begin{align*}
\rho_{\ast}^{\frac{3}{2}}{[\![B({\mathbf{u}})]\!]}_{l,\frac{1}{2},\Gamma_{R}}
&  \leq{%
\color{red}%
I+II}\leq2CC_{tr}C_{\mathbf{G}}C_{\mathbf{u}}(B_{R}^{+})K^{l+1}\max
(p+1,k)^{p+1}\\
&  \quad{%
\color{red}%
+2C_{tr,R}k^{-1}C_{{\mathbf{g}}_{B}}\lambda_{{\mathbf{g}}_{B}}^{l+1}%
\max(l+2,k)^{l+1%
\color{red}%
}}\\
&  \leq2C_{tr,R}(CC_{\mathbf{G}}%
\color{red}%
{+1)}%
\color{black}%
C_{\mathbf{u}}(B_{R}^{+})K^{l+1}\max(p+1,k)^{p+1}.
\end{align*}
Multiplying this estimate by$A^{p-l}$ and summing on $l$, we get
\begin{align*}
\sum_{l=0}^{p-1}A^{p-l}\rho_{\ast}^{\frac{3}{2}}{[\![B({\mathbf{u}}%
)]\!]}_{l,\frac{1}{2},\Gamma_{R}}  &  \leq C_{\mathbf{u}}(B_{R}^{+}%
)K^{p+1}\max(p+1,k)^{p+1}2C_{tr,R}(CC_{\mathbf{G}}%
\color{red}%
{+1)}%
\color{black}%
\sum_{l=0}^{p-1}A^{p-l}K^{l-p}\\
&  \leq C_{\mathbf{u}}(B_{R}^{+})K^{p+1}\max(p+1,k)^{p+1}2C_{tr,R}%
(CC_{\mathbf{G}}{%
\color{red}%
+1})\frac{A}{K}\sum_{l=0}^{\infty}\left(  \frac{A}{K}\right)  ^{l}.
\end{align*}
Again, for $K\geq2A$, we arrive at%
\begin{subequations}
\label{eq:A19}
\end{subequations}%
\begin{equation}
\sum_{l=0}^{p-1}A^{p-l}\rho_{\ast}^{\frac{3}{2}}{[\![B({\mathbf{u}}%
)]\!]}_{l,\frac{1}{2},\Gamma_{R}}\leq C_{\mathbf{u}}(B_{R}^{+})K^{p+1}%
\max(p+1,k)^{p+1}\frac{4C_{tr,R}(CC_{\mathbf{G}}{%
\color{red}%
+1})A}{K}. \tag{%
\ref{eq:A19}%
}\label{eq:A19neu}%
\end{equation}
{
\color{red}%
Estimation of $\rho_{\ast}^{\frac{1}{2}}{[\![D({\mathbf{u}})]\!]}_{l,\frac
{3}{2},\Gamma_{R}}$:
\begin{align*}
\rho_{\ast}^{\frac{1}{2}}{[\![D({\mathbf{u}})]\!]}_{l,\frac{3}{2},\Gamma_{R}}
&  =\rho_{\ast}^{\frac{1}{2}}{[\![{\mathbf{g}}_{D}]\!]}_{l,\frac{3}{2}%
,\Gamma_{R}}\overset{(\ref{eq:A13a})}{\leq}2C_{tr,R}^{\prime}C_{{\mathbf{g}%
}_{D}}\lambda_{{\mathbf{g}}_{D}}^{l+2}\max(l+2,k)^{l+2}\\
&  \overset{l+2\leq p+1}{\leq}2C_{tr,R}^{\prime}C_{{\mathbf{g}}_{D}}%
\lambda_{{\mathbf{g}}_{D}}^{l+2}\max(p+1,k)^{p+1}.
\end{align*}%
\color{red}%
Multiplying this estimate by $A^{p-l}$ and summing on $l$, we get for
$K\geq\lambda_{{\mathbf{g}}_{D}}$ and $K\geq2A$
\begin{subequations}
\begin{align}
\sum_{l=0}^{p-1}A^{p-l}\rho_{\ast}^{\frac{1}{2}}{[\![D({\mathbf{u}}%
)]\!]}_{l,\frac{3}{2},\Gamma_{R}}  &  \leq C_{{\mathbf{g}}_{D}}\max
(p+1,k)^{p+1}2C_{tr,R}^{\prime}\lambda_{{\mathbf{g}}_{D}}^{2}\sum_{l=0}%
^{p-1}A^{p-l}\lambda_{{\mathbf{g}}_{D}}^{l}\nonumber\\
&  \overset{\lambda_{{\mathbf{g}}_{D}}\leq K}{\leq}C_{{\mathbf{g}}_{D}}%
\max(p+1,k)^{p+1}2C_{tr,R}^{\prime}\lambda_{{\mathbf{g}}_{D}}^{2}\sum
_{l=0}^{p-1}A^{p-l}K^{l}\nonumber\\
&  \leq C_{{\mathbf{g}}_{D}}\max(p+1,k)^{p+1}2C_{tr,R}^{\prime}\lambda
_{{\mathbf{g}}_{D}}^{2}K^{p}\sum_{l=0}^{p-1}A^{p-l}K^{l-p}\nonumber\\
&  \leq C_{{\mathbf{g}}_{D}}\max(p+1,k)^{p+1}2C_{tr,R}^{\prime}\lambda
_{{\mathbf{g}}_{D}}^{2}K^{p}\frac{A}{K}\sum_{l=0}^{\infty}\left(  \frac{A}%
{K}\right)  ^{l}\nonumber\\
&  \leq C_{{\mathbf{g}}_{D}}\max(p+1,k)^{p+1}K^{p+1}\frac{4C_{tr,R}^{\prime
}A\lambda_{{\mathbf{g}}_{D}}^{2}}{K^{2}}\nonumber\\
&  \leq C_{{\mathbf{u}}}(B_{R}^{+})\max(p+1,k)^{p+1}K^{p+1}\frac
{4C_{tr,R}^{\prime}A\lambda_{{\mathbf{g}}_{D}}^{2}}{K^{2}}. \tag{%
\ref{eq:A19}%
+1/3}\label{eq:A19a}%
\end{align}
\color{black}%
Finally using the definition $C_{\mathbf{u}}(B_{R}^{+})$, we directly check
that
\end{subequations}
\begin{equation}%
\color{black}%
\sum_{l=0}^{1}{[\![{\mathbf{u}}]\!]}_{l,l,B_{R}^{+}}\leq\frac{k}{C_{R}%
}C_{\mathbf{u}}(B_{R}^{+}) \tag{%
\ref{eq:A19}%
+2/3}\label{eq:A20}%
\end{equation}%
\color{black}%
and therefore (since we assume that $K\geq2A$)
\[
A^{p}\sum_{l=0}^{1}{[\![{\mathbf{u}}]\!]}_{l,l,B_{R}^{+}}\leq\frac{1}{C_{R}%
}C_{\mathbf{u}}(B_{R}^{+})K^{p}\max(p+1,k)^{p+1}.
\]
In summary, using this estimate (\ref{eq:A17}), (\ref{eq:A19neu}), and {
\color{red}%
(\ref{eq:A19a})}%
\color{black}%
, in (\ref{eq:A16}) we have obtained
\[
{[\![{\mathbf{u}}]\!]}_{p+1,2,B_{R}^{+}}\leq C_{\mathbf{u}}(B_{R}^{+}%
)K^{p+1}\max(p+1,k)^{p+1}\left(  \frac{2+4C_{tr,R}(CC_{\mathbf{G}}{%
\color{red}%
+1})A+\frac{1}{C_{R}}}{K}{%
\color{red}%
+\frac{4C_{tr,R}^{\prime}A\lambda_{{\mathbf{g}}_{D}}^{2}}{K^{2}}}
\right)  .
\]
This yields (\ref{eq:A15}) for $p+1$ if
\[
{%
\color{red}%
K\geq\max\left(  \lambda_{\mathbf{f}},\lambda_{{\mathbf{g}}_{D}}%
,2A,\lambda_{{\mathbf{G}}}^{\prime},2+4C_{tr,R}(CC_{\mathbf{G}}+1)A+\frac
{1}{C_{R}}+4C_{tr,R}^{\prime}\lambda_{{\mathbf{g}}_{D}}^{2}A\right)  .}
\]
}
\end{proof}

Now, we will show an equivalent lemma but which also estimates the norm of the
normal derivatives of higher order.

\begin{lemma}
\label{lemma:A8} Let ${\mathbf{u}}\in{\mathbf{H}}^{2}(B_{R}^{+})$ be a
solution of (\ref{eq:A5}) with ${\mathbf{f}}$, ${\mathbf{G}}$, {%
\color{red}
${\mathbf{G}}_{D}$, and ${\mathbf{G}}_{B}$ }
analytic and satisfying (\ref{eq:A3}), (\ref{eq:A4neu}), (\ref{eq:A4a}),
(\ref{eq:A4b}). Then there exist $K_{1}$, $K_{2}\geq1$ such that for all $p$,
$q\geq2$ with $q\leq p$, we have
\[
{[\![{\mathbf{u}}]\!]}_{p,q,B_{R}^{+}}\leq C_{\mathbf{u}}(B_{R}^{+})K_{1}%
^{p}K_{2}^{q}\max(p,k)^{p},
\]
with $C_{\mathbf{u}}(B_{R}^{+})=C_{R}\left(  C_{\mathbf{f}} {%
\color{red}%
k^{-2} +k^{-1}C_{{\mathbf{g}}_{B}}+C_{{\mathbf{g}}_{D}} }
+\Vert{\mathbf{u}}\Vert_{B_{R}^{+}}+k^{-1}\Vert{\mathbf{u}}\Vert_{1,B_{R}^{+}%
}\right)  $.
\end{lemma}

\begin{proof}
Again, we will show this lemma by induction and by using a standard analytic
regularity result for elliptic problems (i.e., \cite[Prop.~{2.6.7}]{CoDaNi}),
which gives us \setcounter{equation}{20}%
\begin{align}
{[\![{\mathbf{u}}]\!]}_{p,q,B_{R}^{+}}  &  \leq\sum_{l=0}^{p-2}A^{p-1-l}%
\left\{  \sum_{\nu=0}^{\min(l,q-2)}B^{q-1-\nu}\rho_{\ast}^{2}%
{[\![L({\mathbf{u}})]\!]}_{l,\nu,B_{R}^{+}}+B^{q-1}\left(  \rho_{\ast}%
^{\frac{3}{2}}{[\![B({\mathbf{u}})]\!]}_{l,\frac{1}{2},\Gamma_{R}}{%
\color{red}%
+\rho_{\ast}^{\frac{1}{2}}{[\![D({\mathbf{u}})]\!]}_{l,\frac{3}{2},\Gamma_{R}%
}}\right)  \right\} \nonumber\label{eq:A21}\\
&  \qquad+A^{p-1}B^{q-1}\sum_{l=0}^{1}{[\![{\mathbf{u}}]\!]}_{l,l,B_{R}^{+}}%
\end{align}
for some positive constants $A$ and $B\geq1$. The induction is done on $q$,
the initialization step $q=2$ is obtained from Lemma~\ref{lemma:A7} by taking
$K_{1}\geq K$ and $K_{2}\geq1$. The induction hypothesis is: For all $p\geq3$,
$2\leq q^{\prime}\leq p-1$, it holds \setcounter{equation}{21}
\begin{equation}
{[\![{\mathbf{u}}]\!]}_{p,q^{\prime},B_{R}^{+}}\leq C_{{\mathbf{u}}}(B_{R}%
^{+})K_{1}^{p}K_{2}^{q^{\prime}}\max(p,k)^{p}. \label{eq:A22}%
\end{equation}
We use the estimate (\ref{eq:A21}) to get \setcounter{equation}{22}
\begin{align}
{[\![{\mathbf{u}}]\!]}_{p,q+1,B_{R}^{+}}  &  \leq\sum_{l=0}^{p-2}%
A^{p-1-l}\left\{  \sum_{\nu=0}^{\min(l,q-1)}B^{q-\nu}\rho_{\ast}%
^{2}{[\![L({\mathbf{u}})]\!]}_{l,\nu,B_{R}^{+}}+B^{q}\rho_{\ast}^{\frac{3}{2}%
}{[\![B({\mathbf{u}})]\!]}_{l,\frac{1}{2},\Gamma_{R}}+{%
\color{red}%
B^{q}\rho_{\ast}^{\frac{1}{2}}{[\![D({\mathbf{u}})]\!]}_{l,\frac{3}{2}%
,\Gamma_{R}}}\right\} \nonumber\label{eq:A23}\\
&  +A^{p-1}B^{q}\sum_{l=0}^{1}{[\![{\mathbf{u}}]\!]}_{l,l,B_{R}^{+}}.
\end{align}
\emph{Estimate of $\rho_{\ast}^{2}{[\![L({\mathbf{u}})]\!]}_{l,\nu,B_{R}^{+}}%
$}. By the induction hypothesis (\ref{eq:A22}) we may write
\begin{align*}
\rho_{\ast}^{2}{[\![L({\mathbf{u}})]\!]}_{l,\nu,B_{R}^{+}}  &  \leq
{[\![{\mathbf{f}}]\!]}_{l,\nu,B_{R}^{+}}+k^{2}{[\![{\mathbf{u}}]\!]}%
_{l,\nu,B_{R}^{+}}\\
&  \leq C_{\mathbf{f}}\lambda_{\mathbf{f}}^{l}\max(l,k)^{l}+k^{2}%
C_{{\mathbf{u}}}(B_{R}^{+})K_{1}^{l}K_{2}^{\nu}\max(l,k)^{l}\\
&  \leq C_{{\mathbf{u}}}(B_{R}^{+})K_{1}^{l}K_{2}^{\nu}k^{2}\max
(l,k)^{l}\left(  \left(  \frac{\lambda_{\mathbf{f}}}{K_{1}}\right)  ^{l}{%
\color{red}%
\frac{1}{K_{2}^{\nu}}}+1\right) \\
&  \overset{l\leq p-2}{\leq}C_{\mathbf{u}}(B_{R}^{+})K_{1}^{p}K_{2}^{q+1}%
\max(p,k)^{p}\frac{2}{K_{1}K_{2}}K_{1}^{l-p+1}K_{2}^{\nu-q}%
\end{align*}
if $K_{1}\geq\lambda_{\mathbf{f}}$. Multiplying this estimate by
$A^{p-1-l}B^{q-\nu}$ and summing on $\nu$ and $l$, one gets
\begin{align*}
&  \sum_{l=0}^{p-2}A^{p-1-l}\sum_{\nu=0}^{\min(l,q-1)}B^{q-\nu}\rho_{\ast}%
^{2}{[\![L({\mathbf{u}})]\!]}_{l,\nu,B_{R}^{+}}\\
&  \quad\leq C_{\mathbf{u}}(B_{R}^{+})K_{1}^{p}K_{2}^{q+1}\max(p,k)^{p}%
\frac{2}{K_{1}K_{2}}\sum_{l=0}^{p-2}A^{p-1-l}\sum_{\nu=0}^{\min(l,q-1)}%
B^{q-\nu}K_{1}^{l-p+1}K_{2}^{\nu-q}\\
&  \leq C_{\mathbf{u}}(B_{R}^{+})K_{1}^{p}K_{2}^{q+1}\max(p,k)^{p}\frac
{2}{K_{1}K_{2}}\sum_{l=0}^{p-2}\left(  \frac{A}{K_{1}}\right)  ^{p-1-l}%
\sum_{\nu=0}^{\min(l,q-1)}\left(  \frac{B}{K_{2}}\right)  ^{q-\nu}.
\end{align*}
Choosing $K_{1}\geq2A$ and $K_{2}\geq2B$, we conclude that%
\begin{equation}
\sum_{l=0}^{p-2}A^{p-1-l}\sum_{\nu=0}^{\min(l,q-1)}B^{q-\nu}\rho_{\ast}%
^{2}{[\![L({\mathbf{u}})]\!]}_{l,\nu,B_{R}^{+}}\leq C_{\mathbf{u}}(B_{R}%
^{+})K_{1}^{p}K_{2}^{q+1}\max(p,k)^{p}\frac{8}{K_{1}K_{2}}. \label{eq:A24}%
\end{equation}
\emph{Estimation of $\rho_{\ast}^{\frac{3}{2}}{[\![B({\mathbf{u}}%
)]\!]}_{l,\frac{1}{2},\Gamma_{R}}$ for $l\leq p-2$:} We use the estimate
(\ref{eq:A9}) and the induction hypothesis (\ref{eq:A22}) {%
\color{red}%
and (\ref{eq:A13b}) }to get
\begin{align*}
\rho_{\ast}^{\frac{3}{2}}{[\![B({\mathbf{u}})]\!]}_{l,\frac{1}{2},\Gamma_{R}}
&  {%
\color{red}%
\leq\rho_{\ast}^{\frac{3}{2}}{[\![k{\mathbf{G}}{\mathbf{u}}]\!]}_{l,\frac
{1}{2},\Gamma_{R}}+\rho_{\ast}^{\frac{3}{2}}{[\![{\mathbf{g}}_{B}%
]\!]}_{l,\frac{1}{2},\Gamma_{R}}}\\
& \\
&  \leq kCC_{tr,R}C_{\mathbf{G}}C_{\mathbf{u}}(B_{R}^{+})K_{2}^{2}%
\sum_{p^{\prime}=0}^{l+1}(\lambda_{{\mathbf{G}}}^{\prime})^{l+1-p^{\prime}%
}K_{1}^{p^{\prime}}\max(l+1,k)^{l+1-p^{\prime}}\max(p^{\prime},k)^{p^{\prime}%
}\\
&  \quad{%
\color{red}%
+2C_{tr,R}C_{{\mathbf{g}}_{B}}\lambda_{{\mathbf{g}}_{B}}^{l+1}\max
(l+2,k)^{l+1}.}%
\end{align*}
In the above right-hand side as $l+2\leq p$ and $p^{\prime}\leq p-1$, we
obtain
\begin{align*}
\rho_{\ast}^{\frac{3}{2}}{[\![B({\mathbf{u}})]\!]}_{l,\frac{1}{2},\Gamma_{R}}
&  \leq CC_{tr,R}C_{\mathbf{G}}C_{\mathbf{u}}(B_{R}^{+})K_{2}^{2}\max
(p,k)^{p}K_{1}^{l+1}\sum_{p^{\prime}=0}^{l+1}\left(  \frac{\lambda
_{\mathbf{G}}^{\prime}}{K_{1}}\right)  ^{l+1-p^{\prime}}\\
&  \quad{%
\color{red}%
+2C_{tr,R}k^{-1}C_{{\mathbf{g}}_{B}}\lambda_{{\mathbf{g}}_{B}}^{l+1}%
\max(p,k)^{p}.}%
\end{align*}
For $K_{1}\geq2\lambda_{\mathbf{G}}^{\prime}$ {%
\color{red}%
and $K_{1}\geq\lambda_{{\mathbf{g}}_{B}}$ }
we deduce that%
\begin{subequations}
\label{eq:A25}
\end{subequations}%
\begin{equation}
\rho_{\ast}^{\frac{3}{2}}{[\![B({\mathbf{u}})]\!]}_{l,\frac{1}{2},\Gamma_{R}%
}\leq2C_{tr,R}C_{\mathbf{u}}(B_{R}^{+}){%
\color{red}%
(CC_{{\mathbf{G}}}K_{2}^{2}+1)}\max(p,k)^{p}K_{1}^{l+1}. \tag{%
\ref{eq:A25}%
}\label{eq:A25neu}%
\end{equation}
Multiplying this estimate by $A^{p-1-l}B^{q}$ and summing on $l$, as before
one gets (since $K_{1}\geq2A$ and {%
\color{red}%
$K_{2}\geq B$}
)
\begin{align*}
\sum_{l=0}^{p-2}A^{p-1-l}B^{q}\rho_{\ast}^{\frac{3}{2}}{[\![B({\mathbf{u}%
})]\!]}_{l,\frac{1}{2},\Gamma_{R}}  &  {%
\color{red}%
\leq C_{\mathbf{u}}(B_{R}^{+})K_{1}^{p}B^{q}\max(p,k)^{p}2C_{tr,R}%
(CC_{{\mathbf{G}}}K_{2}^{2}+1)\sum_{l=0}^{p-2}A^{p-1-l}K_{1}^{l+1}K_{1}^{-p}%
}\\
&  \overset{{%
\color{red}%
K_{2}\geq B}}{\leq}C_{\mathbf{u}}(B_{R}^{+})K_{1}^{p}K_{2}^{q+1}\max(p,k)^{p}{%
\color{red}%
\left(  \frac{4C_{tr,R}A(CC_{\mathbf{G}}K_{2}^{2}+1)}{K_{1}K_{2}}\right)  .}%
\end{align*}
{%
\color{red}%
Estimation of $\rho_{\ast}^{\frac{1}{2}}{[\![D({\mathbf{u}})]\!]}_{l,\frac
{3}{2},\Gamma_{R}}$: From (\ref{eq:A13a}) and $l+2\leq p$ we have%
\begin{align}
\rho_{\ast}^{\frac{1}{2}}{[\![D({\mathbf{u}})]\!]}_{l,\frac{3}{2},\Gamma_{R}}
&  \leq2C_{tr,R}^{\prime}C_{{\mathbf{g}}_{D}}\lambda_{{\mathbf{g}}_{D}}%
^{l+2}\max(l+2,k)^{l+2}\nonumber\\
&  \leq C_{\mathbf{u}}(B_{R}^{+})2C_{tr,R}^{\prime}\lambda_{{\mathbf{g}}_{D}%
}^{p}\max(p,k)^{p}. \tag{%
\ref{eq:A25}%
+1/2}\label{eq:A25a}%
\end{align}
}
Finally, using (\ref{eq:A20}), one has
\[
A^{p-1}B^{q}\sum_{l=0}^{1}{[\![\mathbf{u}]\!]}_{l,l,B_{R}^{+}}\leq
C_{\mathbf{u}}(B_{R}^{+})K_{1}^{p}K_{2}^{q+1}\max(p,k)^{p}\frac{1}{C_{R}%
K_{1}K_{2}}.
\]
Inserting this estimate and the estimates (\ref{eq:A24}), (\ref{eq:A25neu}), {%
\color{red}%
(\ref{eq:A25a}) }
into (\ref{eq:A16}), we can conclude that
\begin{align*}
{[\![\mathbf{u}]\!]}_{p,q+1,B_{R}^{+}}  &  \leq C_{\mathbf{u}}(B_{R}^{+}%
)K_{1}^{p}K_{2}^{q+1}{%
\color{red}%
\max(p,k)^{p}}\left(  \frac{8}{K_{1}K_{2}}+{%
\color{red}%
\frac{4C_{tr,R}A(CC_{\mathbf{G}}K_{2}^{2}+1)}{K_{1}K_{2}}}+\frac{1}{C_{R}%
K_{1}K_{2}}\right. \\
&  \quad\left.  +{%
\color{red}%
\frac{2C_{tr,R}^{\prime}}{K_{2}^{q+1}}\left(  \frac{\lambda_{{\mathbf{g}}_{D}%
}}{K_{1}}\right)  ^{p}}\right) \\
&  \leq C_{\mathbf{u}}(B_{R}^{+})K_{1}^{p}K_{2}^{q+1}\max(p,k)^{p}%
\end{align*}
for $K_{1}$ and $K_{2}$ large enough.
\end{proof}

\begin{remark}
\label{rem:A9} In Lemma~\ref{lemma:A8}, if we take $p=q$, we obtain
\[
{[\![{\mathbf{u}}]\!]}_{p,p,B_{R}^{+}}\leq C_{\mathbf{u}}(B_{R}^{+})K^{p}%
\max(p,k)^{p}%
\]
with $K=K_{1}K_{2}$. \hbox{}\hfill%
\endproof

\end{remark}



\subsection{Interior analytic regularity}


Let $B_{R} = B(0,R)$, $L$ an elliptic system of order $2$ defined in $B_{R}$,
and $k > 1$. Here, we consider a solution ${\mathbf{u}}$ of
\setcounter{equation}{25}
\begin{align}
\label{eq:A26}L({\mathbf{u}}) = {\mathbf{f}} + k^{2} {\mathbf{u}}
\quad\mbox{ in $B_R$}.
\end{align}
We now define the following semi-norms
\begin{align*}
{[\![ \mathbf{u}]\!]}_{p,B_{R}}  &  := \max_{0 < \rho< \frac{R}{2 p} }%
\max_{|\alpha| = p} \rho^{p} \|\partial^{\alpha}{\mathbf{u}}\|_{B_{R - p \rho
}},\\
\rho_{\ast}^{2} {[\![ \mathbf{u}]\!]}_{p,B_{R}}  &  := \max_{0 < \rho<
\frac{R}{2 p} }\max_{|\alpha| = p} \rho^{p+2} \|\partial^{\alpha}{\mathbf{u}%
}\|_{B_{R - p \rho}}.
\end{align*}
We suppose that ${\mathbf{f}}$ is analytic with \setcounter{equation}{26}
\begin{align}
\label{eq:A27}\|\partial^{\alpha}{\mathbf{f}}\|_{B_{R}}  &  \leq
C_{\mathbf{f}} \lambda^{p}_{\mathbf{f}} \max(|\alpha|,k)^{|\alpha|},
\quad\forall\alpha\in{\mathbb{N}}^{n},
\end{align}
for some positive constants $C_{\mathbf{f}}$ and $\lambda_{\mathbf{f}}$
independent of $k$.

\begin{lemma}
\label{lemma:A10} Let ${\mathbf{u}}\in{\mathbf{H}}^{2}(B_{R})$ be a solution
of (\ref{eq:A26}) with ${\mathbf{f}}$ satisfying (\ref{eq:A27}). Then there
exists $K\geq1$ such that
\[
{[\![\mathbf{u}]\!]}_{p,B_{R}}\leq C{\mathbf{u}}(B_{R}^{+})K^{p}\max(p,k)^{p}%
\]
with $C_{\mathbf{u}}(B_{R}^{+})=C_{R}(C_{\mathbf{f}} {%
\color{red}%
k^{-2} }
+\Vert\mathbf{u}\Vert_{B_{R}}+k^{-1}\Vert\mathbf{u}\Vert_{1,B_{R}})$ for {%
\color{red}
suitable }
$C_{R}\geq1$.
\end{lemma}

\begin{proof}
The proof is exactly the same as the one of Lemma~\ref{lemma:A7} when we use
\cite[Prop.~{1.6.3}]{CoDaNi} (a standard interior regularity result) instead
of \cite[Prop.~{2.6.6}]{CoDaNi}.
\end{proof}



\subsection{Proof of Theorem~\ref{thm:A1}}


According to a standard procedure, see for instance \cite[p.~{105}]{CoDaNi},
the first step of the proof is to consider a covering of $\Omega$ by some open
sets, which verifies
\begin{align*}
\Omega\subset\cup_{j=1}^{N} \hat B_{j} \subset\cup_{j=1}^{N} B_{j},
\end{align*}
where

\begin{enumerate}
\item $B_{j}=B(\mathbf{x}_{j},\xi_{j})$ and $\hat{B}_{j}=B(\mathbf{x}%
_{j},\frac{\xi_{j}}{2})$, with $\xi_{j}>0$ small enough such that
$\overline{B}(\mathbf{x}_{j},\xi_{j})\subset\Omega$, if $\mathbf{x}_{j}%
\in\Omega$,

\item while in the case $\mathbf{x}_{j}\in\partial\Omega$, $B_{j}$ is a
sufficiently small neighborhood of $\mathbf{x}_{j}$ such that there exists an
analytic map $\phi_{j}$ from $B_{j}$ onto the ball $B(0,\xi_{j})$ (for some
$\xi_{j}>0$ such that $\phi_{j}(B_{j}\cap\Omega)=B_{j}^{+}$, and $\phi
_{j}(B_{j}\cap\partial\Omega)=\Gamma_{\xi_{j}}$. In that case, we set $\hat
{B}_{j}=\phi_{j}^{-1}(B_{\xi_{j}/2}^{+})$.
\end{enumerate}

This yields
\[
|\mathbf{u}|_{p,\Omega}\lesssim\sum_{i=1}^{N}|\mathbf{u}|_{p,\hat{B}_{j}%
\cap\Omega}\lesssim\sum_{1\leq i\leq N\colon\mathbf{x}_{i}\in\Omega
}|\mathbf{u}|_{p,\hat{B}_{j}}+\sum_{1\leq i\leq N\colon\mathbf{x}_{i}%
\in\partial\Omega}|\mathbf{u}|_{p,\hat{B}_{j}\cap\Omega}.
\]
In the case of an interior ball, namely, for $i$ such that $\mathbf{x}_{i}%
\in\Omega$, we simply perform a translation to apply Lemma~\ref{lemma:A10}.
Hence, the operator $L$ does not change, and we directly
have\footnote{{\color{red}We employ the multinomial formula and the definition
of the norm $|\cdot|_{p,\hat{B}_{i}}$ given in (\ref{eq:A0}) so that the
formula deviates from \cite{nicaise-tomezyk19}}}
\[
|\mathbf{u}|_{p,\hat{B}_{i}}\lesssim{%
\color{red}%
n^{p/2}}
{[\![\mathbf{u}]\!]}_{p,{%
\color{red}%
p},B_{i}}\lesssim{%
\color{red}%
n^{p/2}}
C_{\mathbf{u}}(B_{i})K^{p}\max(p,k)^{p}.
\]
By the definition of $C_{\mathbf{u}}(B_{i})$, we then arrive at%
\begin{equation}
|\mathbf{u}|_{p,\hat{B}_{i}}\lesssim\left(  C_{\mathbf{f}}{%
\color{red}%
k^{-2}}+\Vert\mathbf{u}\Vert_{B_{i}}+k^{-1}\Vert\mathbf{u}\Vert_{1,B_{i}%
}\right)  ({%
\color{red}%
n^{1/2}}K)^{p}\max(p,k)^{p}. \label{eq:A28}%
\end{equation}
In the case when the open set intersects the boundary of $\Omega$, namely, for
each $i$ such that $\mathbf{x}_{i}\in\partial\Omega$, we apply the change of
variables $\mathbf{\hat{x}}=\phi_{i}(\mathbf{x})$, which allows us to pass
from $B_{i}\cap\Omega$ to $B_{\xi_{i}}^{+}$ and transforms the system
(\ref{eq:A2}) restricted to $B_{i}\cap\Omega$ into an elliptic system with
analytic coefficients of the form (\ref{eq:A5}) with other operators $\hat{L}%
$, $\hat{D}$, $\hat{B}$, which $\hat{D}$ (resp.\ $\hat{B}$) an operator of
order $0$ (resp.\ $1$). First thanks to a Fa\`{a}-di-Bruno formula, we obtain
(see \cite[(1.b)]{CoDaNi})
\[
|\mathbf{u}|_{p,\hat{B}_{i}\cap\Omega}\lesssim c_{i}^{p+1}\sum_{l=0}^{p}%
\frac{{%
\color{red}%
p!}}{l!}|\hat{\mathbf{u}}|_{l,B_{\xi_{i}/2}^{+}},
\]
with a positive constant $c_{i}$ which depends only on the transformation that
allows us to pass from $B_{i}\cap\Omega$ to $B_{\xi_{i}}^{+}$. Then, we can
apply Lemma~\ref{lemma:A8} (see Remark~\ref{rem:A9}) and get
\[
|\mathbf{u}|_{p,\hat{B}_{i}\cap\Omega}\lesssim{%
\color{red}%
n^{p/2}}
c_{i}^{p+1}C_{\hat{\mathbf{u}}}(B_{\xi_{i}}^{+})\sum_{l=0}^{p}\frac{{%
\color{red}%
p!}}{l!}K^{l}\max(l,k)^{l}.
\]
Using \cite[(A.11)]{nicaise-tomezyk19}, {%
\color{red}
i.e., \setcounter{equation}{10}
\[
\frac{p!}{q!}\leq p^{p-q},\quad\forall p,q,\in{\mathbb{N}}\colon q\leq p,
\]
}
and a change of variables (in $C_{\hat{\mathbf{u}}}(B_{\xi_{i}}^{+})$ and
again Fa\`{a}-di-Bruno formula) we obtain
\[
|\mathbf{u}|_{p,\hat{B}_{i}\cap\Omega}\lesssim{%
\color{red}%
n^{p/2}}
c_{i}^{p+1}\left(  C_{\mathbf{f}}{%
\color{red}%
k^{-2}}
+\Vert\mathbf{u}\Vert_{B_{i}\cap\Omega}+k^{-1}\Vert\mathbf{u}\Vert
_{1,B_{i}\cap\Omega}\right)  \max(p,k)^{p}\sum_{l=0}^{p}K^{l}.
\]
This yields
\[
|\mathbf{u}|_{p,\hat{B}_{i}\cap\Omega}\lesssim\frac{c_{i}K}{K-1}\left(
C_{\mathbf{f}}{%
\color{red}%
k^{-2}}
+\Vert\mathbf{u}\Vert_{B_{i}\cap\Omega}+k^{-1}\Vert\mathbf{u}\Vert
_{1,B_{i}\cap\Omega}\right)  (c_{i}{%
\color{red}%
n^{1/2}}
K)^{p}\max(p,k)^{p}.
\]
The combination of this estimate with (\ref{eq:A28}) yields the result.

\endproof

\section{Details of the proof of
Lemma~\ref{lemma:element-by-element-approximation}}

In this section, we give some details of the proof of
Lemma~\ref{lemma:element-by-element-approximation}. The estimates of the
volume terms $\left\Vert \cdot\right\Vert _{L^{2}(\widehat{K})}$, $\left\Vert
\cdot\right\Vert _{H^{1}(\widehat{K})}$ in
(\ref{eq:lemma:element-by-element-approximation}) are proved in
\cite[Thm.~{B.4}]{MelenkSauterMathComp} and it remains to estimate the
boundary norms in (\ref{eq:lemma:element-by-element-approximation}). For this,
we first reproduce the arguments given in \cite[Appendix~B]%
{MelenkSauterMathComp} to make our proof for the boundary norms
(Theorem~\ref{thm:finite-regularity-constrained-approximation}) self-contained.%

\color{black}%
Before proceeding we recall the definition of the Sobolev space $H_{00}%
^{1/2}(\Omega)$. If $\Omega$ is an edge or a face of a triangle or a
tetrahedron, then the Sobolev norm $\Vert\cdot\Vert_{H_{00}^{1/2}(\Omega)}$ is
defined by
\begin{equation}
\Vert u\Vert_{H_{00}^{1/2}(\Omega)}^{2}:=\Vert u\Vert_{H^{1/2}(\Omega)}%
^{2}+\left\Vert \frac{u}{\sqrt{\operatorname*{dist}(\cdot,\partial\Omega)}%
}\right\Vert _{L^{2}(\Omega)}^{2}, \label{eq:H1200}%
\end{equation}
and the space $H_{00}^{1/2}(\Omega)$ is the completion of $C_{0}^{\infty
}(\Omega)$ under this norm.

\begin{definition}
[{element-by-element construction, {\cite[Def.~{5.3}]{MelenkSauterMathComp}}}%
]\label{def:element-by-element} Let $\widehat{K}$ be the reference simplex in
${\mathbb{R}}^{d}$, $d\in\{2,3\}$. A polynomial $\pi$ is said to permit
\emph{an element-by-element} construction of polynomial degree $p$ for $u\in
H^{s}(\widehat{K})$, $s>d/2$, if:

\begin{enumerate}
[(i)]

\item \label{item:vertex} $\pi(V)=u(V)$ for all $d+1$ vertices $V$ of
$\widehat{K},$

\item \label{item:edge} for every edge $e$ of $\widehat{K}$, the restriction
$\pi|_{e}\in{\mathcal{P}}_{p}$ is the unique minimizer
of
\begin{equation}
\pi\mapsto p^{1/2}\Vert u-\pi\Vert_{L^{2}(e)}+\Vert u-\pi\Vert_{H_{00}%
^{1/2}(e)} \label{eq:edge-minimizer}%
\end{equation}
under the constraint that $\pi$ satisfies (\ref{item:vertex}); here the
Sobolev norm $H_{00}^{1/2}$ is defined in (\ref{eq:H1200}).

\item \label{item:face} (for $d=3$) for every face $f$ of $\widehat{K}$, the
restriction $\pi|_{f}\in{\mathcal{P}}_{p}$ is the unique minimizer of
\begin{equation}
\pi\mapsto p\Vert u-\pi\Vert_{L^{2}(f)}+\Vert u-\pi\Vert_{H^{1}(f)}
\label{eq:face-minimizer}%
\end{equation}
under the constraint that $\pi$ satisfies (\ref{item:vertex}),
(\ref{item:edge}) for all vertices and edges of the face $f$.
\end{enumerate}
\end{definition}

\begin{lemma}
[{{\cite[Lemma~{B.1}]{MelenkSauterMathComp}}}]\label{lemma:lifting-2d} Let
$\widehat{K}^{2D}$ be the reference triangle in 2D. Vertex and edge lifting
operators can be constructed with the following properties:

\begin{enumerate}
\item \label{item:vertex-lifting-2D} For each vertex $V$ of $\widehat{K}^{2D}$
there exists a polynomial $L_{V,p}\in{\mathcal{P}}_{p}$ that attains the value
$1$ at the vertex $V$ and vanishes on the edge of $\widehat{K}^{2D}$ opposite
to $V$. Additionally, for every $s\geq0$, there exists $C_{s}>0$ such that
$\Vert L_{V,p}\Vert_{H^{s}(\widehat{K}^{2D})}\leq C_{s}p^{-1+s}$.

\item \label{item:edge-lifting-2D} For every edge $e$ of $\widehat{K}^{2D}$
there exists a bounded linear operator $\pi_{e}: H^{1/2}_{00}(e) \rightarrow
H^{1}(\widehat{K}^{2D})$ with the following properties:

\begin{enumerate}
\item $\forall u\in{\mathcal{P}}_{p}\cap H_{00}^{1/2}(e):\quad\pi_{e}%
u\in{\mathcal{P}}_{p},$

\item $\forall u\in H_{00}^{1/2}\left(  e\right)  :\quad\pi_{e}u|_{\partial
\widehat{K}^{2D}\setminus e}=0,$

\item $\forall u\in H_{00}^{1/2}\left(  e\right)  :\quad p\Vert\pi_{e}%
u\Vert_{L^{2}(\widehat{K}^{2D})}+\Vert\pi_{e}u\Vert_{H^{1}(\widehat{K}^{2D}%
)}\leq C\left(  \Vert u\Vert_{H_{00}^{1/2}(e)}+p^{1/2}\Vert u\Vert_{L^{2}%
(e)}\right)  $.
\end{enumerate}
\end{enumerate}
\end{lemma}

\begin{lemma}
[{{\cite[Lemma~{B.2}]{MelenkSauterMathComp}}}]\label{lemma:lifting-3d} Let
$\widehat{K}^{3D}$ be the reference tetrahedron in 3D. Vertex, edge, and face
lifting operators can be constructed with the following properties:

\begin{enumerate}
[(i)]

\item \label{item:vertex-lifting} For each vertex $V$ of $\widehat{K}^{3D}$
there exists a polynomial $L_{V,p} \in{\mathcal{P}}_{p}$ that attains the
value $1$ at the vertex $V$ and vanishes on the face opposite $V$.
Additionally, for every $s \ge0$ there exists $C_{s} > 0$ such that
$\|L_{V,p}\|_{H^{s}(\widehat{K}^{3D})} \leq C_{s} p^{-3/2+s}$.

\item \label{item:edge-lifting} For every edge $e$ of $\widehat{K}^{3D}$ there
exists a bounded linear operator $\pi_{e}: H^{1/2}_{00}(e) \rightarrow
H^{1}(\widehat{K}^{3D})$ with the following properties:

\begin{enumerate}
\item $\pi_{e} u \in{\mathcal{P}}_{p}$ if $u \in{\mathcal{P}}_{p} \cap
H^{1/2}_{00}(e)$

\item $(\pi_{e}u)|_{f}=0$ for the two faces $f$ with $\overline{f}\cap
e=\emptyset$

\item for the two faces $f$ adjacent to $e$ (i.e., $\overline{f}\cap e=e$)
\begin{align*}
p\Vert\pi_{e}u\Vert_{L^{2}(f)}+\Vert\pi_{e}u\Vert_{H^{1}(f)}  &  \leq C\Vert
u\Vert_{H_{00}^{1/2}(e)}+p^{1/2}\Vert u\Vert_{L^{2}(e)},\\
\Vert\pi_{e}u\Vert_{H^{1/2}(\partial\widehat{K}^{3D})}  &  \leq C\left(
p^{-1/2}\Vert u\Vert_{H_{00}^{1/2}(e)}+\Vert u\Vert_{L^{2}(e)}\right)  ,\\
p\Vert\pi_{e}u\Vert_{L^{2}(\widehat{K}^{3D})}+\Vert\pi_{e}u\Vert
_{H^{1}(\widehat{K}^{3D})}  &  \leq C\left(  p^{-1/2}\Vert u\Vert
_{H_{00}^{1/2}(e)}+\Vert u\Vert_{L^{2}(e)}\right)  .
\end{align*}

\end{enumerate}

\item \label{item:face-lifting} For every face $f$ of $\widehat{K}^{3D}$ there
exists a bounded linear operator $\pi_{f}:H^{1/2}_{00}(f) \rightarrow
H^{1}(\widehat{K}^{3D})$ with the following properties:

\begin{enumerate}
\item $\pi_{f} u \in{\mathcal{P}}_{p}$ if $u \in{\mathcal{P}}_{p} \cap
H^{1/2}_{00}(f)$

\item $(\pi_{e} u)|_{f^{\prime}} = 0$ for the faces $f^{\prime}\ne f$
\begin{align*}
p \|\pi_{f} u\|_{L^{2}(\widehat{K}^{3D})} + \|\pi_{f} u\|_{H^{1}%
(\widehat{K}^{3D})}  &  \leq C \left(  \|u\|_{H^{1/2}_{00}(f)} + p^{1/2}
\|u\|_{L^{2}(f)}\right)  .
\end{align*}

\end{enumerate}
\end{enumerate}
\end{lemma}

\begin{lemma}
[{{\cite[Lemma~{B.3}]{MelenkSauterMathComp}}}]%
\label{lemma:finite-regularity-approximation} Let $\widehat{K}$ be the
reference triangle or the reference tetrahedron. Let $s>d/2$. Then there
exists for every $p$ a bounded linear operator $\pi_{p}:H^{s}(\widehat{K}%
)\rightarrow{\mathcal{P}}_{p}$ and for each $t\in\lbrack0,s]$ a constant $C>0$
(depending only on $s$ and $t$) such that
\begin{equation}
\Vert u-\pi_{p}u\Vert_{H^{t}(\widehat{K})}\leq Cp^{-(s-t)}|u|_{H^{s}%
(\widehat{K})},\qquad p\geq s-1.
\label{eq:lemma:finite-regularity-approximation-1000}%
\end{equation}
Additionally, we have $\Vert u-\pi_{p}u\Vert_{L^{\infty}(\widehat{K})}\leq
Cp^{-(s-d/2)}|u|_{H^{s}(\widehat{K})}$. For the case $d=2$ we furthermore have
$\Vert u-\pi_{p}u\Vert_{H^{t}(e)}\leq Cp^{-(s-1/2-t)}|u|_{H^{s}(\widehat{K})}$
for $0\leq t\leq s-1/2$ for every edge. For the case $d=3$ we have $\Vert
u-\pi_{p}u\Vert_{H^{t}(f)}\leq Cp^{-(s-1/2-t)}|u|_{H^{s}(\widehat{K})}$ for
$0\leq t\leq s-1/2$ for every face $f$ and $\Vert u-\pi_{p}u\Vert_{H^{t}%
(e)}\leq Cp^{-(s-1-t)}|u|_{H^{s}(\widehat{K})}$ for $0\leq t\leq s-1$ for
every edge.
\end{lemma}

The following theorem is a slight modification of \cite[Thm.~{B.4}%
]{MelenkSauterMathComp} and expresses the statement of
Lemma~\ref{lemma:element-by-element-approximation}.

\begin{theorem}
[{{Lemma~\ref{lemma:element-by-element-approximation}/\cite[Thm.~{B.4}%
]{MelenkSauterMathComp}}}]%
\label{thm:finite-regularity-constrained-approximation} Let $\widehat{K}%
\subset{\mathbb{R}}^{3}$ be the reference tetrahedron. Let $s>3/2$. Then there
exists $C>0$ (depending only on $s$) and for every $p$ a linear operator
$\pi:H^{s}(\widehat{K})\rightarrow{\mathcal{P}}_{p}$ that permits an
element-by-element construction in the sense of
Definition~\ref{def:element-by-element} such that
\begin{align}
p\Vert u-\pi u\Vert_{L^{2}(\widehat{K})}+\Vert u-\pi u\Vert_{H^{1}%
(\widehat{K})}  &  \leq Cp^{-(s-1)}|u|_{H^{s}(\widehat{K})}\qquad\forall p\geq
s-1,\label{eq:lemma:finite-regularity-constrained-approximation-100}\\
p^{1/2} \|u - \pi u\|_{L^{2}(\partial\widehat{K})} + p^{-1/2} \|u - \pi
u\|_{H^{1}(\partial\widehat{K})}  &  \leq C p^{-(s-1)} |u|_{H^{s}%
(\widehat{K})} \qquad\forall p \ge s-1.
\label{eq:lemma:finite-regularity-constrained-approximation-110}%
\end{align}

\end{theorem}

\begin{proof}
We also mention that
(\ref{eq:lemma:finite-regularity-constrained-approximation-100}) is shown in
\cite[Thm.~{B.4}]{MelenkSauterMathComp} so that our reproducing the proof
focuses on ensuring that
(\ref{eq:lemma:finite-regularity-constrained-approximation-110}) holds.

We will construct $\pi u$ for a given $u$; inspection of the construction
shows that $u\mapsto\pi u$ is in fact a linear operator.

Let $\pi^{1}\in{\mathcal{P}}_{p}$ be given by
Lemma~\ref{lemma:finite-regularity-approximation}. Then, for $p\geq s-1$ there
holds
\begin{align}
\Vert u-\pi^{1}\Vert_{H^{t}(\widehat{K})}  &  \leq Cp^{-(s-t)}|u|_{H^{s}%
(\widehat{K})},\qquad0\leq t\leq s\label{lmsc1}\\
\Vert u-\pi^{1}\Vert_{H^{t}(f)}  &  \leq Cp^{-(s-t-1/2)}|u|_{H^{s}%
(\widehat{K})},\qquad\forall\mbox{ faces $f$},\quad0\leq t\leq
s-1/2\label{lmsc2}\\
\Vert u-\pi^{1}\Vert_{H^{t}(e)}  &  \leq Cp^{-(s-t-1)}|u|_{H^{s}(\widehat{K}%
)},\qquad\forall\mbox{ edges $e$},\quad0\leq t\leq s-1\label{lmsc3}\\
\Vert u-\pi^{1}\Vert_{L^{\infty}(\widehat{K})}  &  \leq Cp^{-(s-3/2)}%
|u|_{H^{s}(\widehat{K})}. \label{lmsc4}%
\end{align}
From (\ref{lmsc4}) and the vertex-lifting properties given in
Lemma~\ref{lemma:lifting-3d}, we may adjust $\pi^{1}$ by vertex liftings to
obtain a polynomial $\pi^{2}$ satisfying (\ref{lmsc1})--(\ref{lmsc4}) and
additionally the condition (\ref{item:vertex}) of
Definition~\ref{def:element-by-element}. We next adjust the edge values. The
polynomial $\pi^{2}$ coincides with $u$ in the vertices and satisfies
(\ref{lmsc3}). By fixing a $t\in(1/2,s-1)$, we get from an interpolation
inequality:
\begin{align*}
p^{1/2}\Vert u-\pi^{2}\Vert_{L^{2}(e)}+\Vert u-\pi^{2}\Vert_{H_{00}^{1/2}(e)}
&  \leq p^{1/2}\Vert u-\pi^{2}\Vert_{L^{2}(e)}+C\Vert u-\pi^{2}\Vert
_{L^{2}(e)}^{1-1/(2t)}\Vert u-\pi^{2}\Vert_{H^{t}(e)}^{1/(2t)}\\
&  \leq Cp^{-(s-3/2)}|u|_{H^{s}(\widehat{K})}.
\end{align*}
Hence, for an edge $e$, the minimizer $\pi^{e}$ of the functional
(\ref{eq:edge-minimizer}) satisfies $p^{1/2}\Vert u-\pi^{e}\Vert_{L^{2}%
(e)}+\Vert u-\pi^{e}\Vert_{H_{00}^{1/2}(e)}\leq Cp^{-(s-3/2)}|u|_{H^{s}%
(\widehat{K})}$; the triangle inequality therefore gives that the correction
$\pi^{e}-\pi^{2}$ needed to obtain condition (\ref{item:edge}) of
Def.~\ref{def:element-by-element} likewise satisfies $p^{1/2}\Vert\pi^{e}%
-\pi^{2}\Vert_{L^{2}(e)}+\Vert\pi_{e}-\pi^{2}\Vert_{H_{00}^{1/2}(e)}\leq
Cp^{-(s-3/2)}|u|_{H^{s}(\widehat{K})}$. We conclude that the edge lifting of
Lemma~\ref{lemma:lifting-3d} allows us to adjust $\pi^{2}$ to get a polynomial
$\pi^{3}\in{\mathcal{P}}_{p}$ that satisfies the conditions (\ref{item:vertex}%
) and (\ref{item:edge}) of Def.~\ref{def:element-by-element}. Additionally, we
have
\begin{align*}
p\Vert u-\pi^{3}\Vert_{L^{2}(\widehat{K})}+\Vert u-\pi^{3}\Vert_{H^{1}%
(\widehat{K})}  &  \leq Cp^{-(s-1)}|u|_{H^{s}(\widehat{K})},\\
p\Vert u-\pi^{3}\Vert_{L^{2}(f)}+\Vert u-\pi^{3}\Vert_{H^{1}(f)}  &  \leq
Cp^{-(s-3/2)}|u|_{H^{s}(\widehat{K})}\qquad\mbox{ for all faces $f$}.
\end{align*}
Since $\pi^{3}|_{e}=\pi^{e}$ for the edges, the minimizer $\pi^{f}$ of the
functional (\ref{eq:face-minimizer}) for each face $f$ has to satisfy $p\Vert
u-\pi^{f}\Vert_{L^{2}(f)}+\Vert u-\pi^{f}\Vert_{H^{1}(f)}\leq p\Vert u-\pi
^{3}\Vert_{L^{2}(f)}+\Vert u-\pi^{3}\Vert_{H^{1}(f)}\leq Cp^{-(s-3/2)}%
|u|_{H^{s}(\widehat{K})}$. From the triangle inequality, we conclude
\[
p\Vert\pi^{3}-\pi^{f}\Vert_{L^{2}(f)}+\Vert\pi^{3}-\pi^{f}\Vert_{H^{1}(f)}\leq
Cp^{-(s-3/2)}|u|_{H^{s}(\widehat{K})},\quad\mbox{ together with }\pi^{3}%
-\pi^{f}\in H_{0}^{1}(f).
\]
Hence, the face lifting of Lemma~\ref{lemma:lifting-3d} allows us to correct
the face values to achieve also condition (\ref{item:face}) of
Definition~\ref{def:element-by-element}. Lemma~\ref{lemma:lifting-3d} also
implies that the correction is such that
(\ref{eq:lemma:finite-regularity-constrained-approximation-100}) is true.
\end{proof}


\fi


\subsection*{Acknowledgements}

\addcontentsline{toc}{section}{Acknowledgements}
We cordially thank Claudio Rojik (TU Wien) for assistance with the numerical
computations in Section~\ref{sec:numerics}. Financial support by the Austrian
Science Fund FWF (through grants P 28367-N35 and F65) is gratefully
acknowledged.
\addcontentsline{toc}{section}{References}
\bibliographystyle{alpha}
\bibliography{maxwell}

\newcommand{\noopsort}[1]{} \newcommand{\printfirst}[2]{#1}
  \newcommand{\singleletter}[1]{#1} \newcommand{\switchargs}[2]{#2#1}
  \def\cprime{$'$} \def\cprime{$'$} \def\cprime{$'$}
\begin{thebibliography}{MMPR20}

\bibitem[ABDG98]{amrouche-bernardi-dauge-girault98}
C.~Amrouche, C.~Bernardi, M.~Dauge, and V.~Girault.
\newblock Vector potentials in three-dimensional non-smooth domains.
\newblock {\em Math. Methods Appl. Sci.}, 21(9):823--864, 1998.

\bibitem[Ain04a]{Ainsworth2004}
Mark Ainsworth.
\newblock Discrete dispersion relation for $hp$-{F}inite {E}lement
  approximation at high wave number.
\newblock {\em SIAM J. Numer. Anal.}, 42(2):553--575, 2004.

\bibitem[Ain04b]{ainsworth04b}
Mark Ainsworth.
\newblock Dispersive properties of high-order {N}\'ed\'elec/edge element
  approximation of the time-harmonic {M}axwell equations.
\newblock {\em Philos. Trans. R. Soc. Lond. Ser. A Math. Phys. Eng. Sci.},
  362(1816):471--491, 2004.

\bibitem[BCFM22]{bernkopf-chaumont-melenk21}
Maximilian Bernkopf, Th{\'e}ophile Chaumont-Frelet, and Jens~M. Melenk.
\newblock Wavenumber-explicit convergence analysis for heterogeneous
  {H}elmholtz problems, 2022.
\newblock arXiv:2209.03601.

\bibitem[BCS02]{BuffaCostabelSheen}
Annalisa Buffa, Martin Costabel, and D.~Sheen.
\newblock On traces for {${\bf H}({\bf curl},\Omega)$} in {L}ipschitz domains.
\newblock {\em J. Math. Anal. Appl.}, 276(2):845--867, 2002.

\bibitem[Buf05]{Buffa2005}
Annalisa Buffa.
\newblock Remarks on the discretization of some noncoercive operator with
  applications to heterogeneous {M}axwell equations.
\newblock {\em SIAM J. Numer. Anal.}, 43(1):1--18, 2005.

\bibitem[CDN10]{CoDaNi}
Martin Costabel, Monique Dauge, and Serge Nicaise.
\newblock {Corner Singularities and Analytic Regularity for Linear Elliptic
  Systems. Part I: Smooth domains}.
\newblock Technical Report https://hal.archives-ouvertes.fr/hal-00453934, HAL
  arxives-ouvertes.fr, 2010.

\bibitem[Ces96]{Cessenat_book}
Michel Cessenat.
\newblock {\em Mathematical methods in electromagnetism}.
\newblock World Scientific Publishing Co., Inc., River Edge, NJ, 1996.

\bibitem[CF19]{chaumont19}
Th{\'e}ophile Chaumont-Frelet.
\newblock Mixed finite element discretizations of acoustic {H}elmholtz problems
  with high wavenumbers.
\newblock {\em Calcolo}, 56(4):Paper No. 49, 27, 2019.

\bibitem[CFN20]{chaumont-frelet-nicaise20}
Th{\'e}ophile Chaumont-Frelet and Serge Nicaise.
\newblock Wavenumber explicit convergence analysis for finite element
  discretizations of general wave propagation problems.
\newblock {\em IMA J. Numer. Anal.}, 40(2):1503--1543, 2020.

\bibitem[CFNP18]{chaumont-frelet-nicaise-pardo18}
Th\'{e}ophile Chaumont-Frelet, Serge Nicaise, and David Pardo.
\newblock Finite element approximation of electromagnetic fields using
  nonfitting meshes for geophysics.
\newblock {\em SIAM J. Numer. Anal.}, 56(4):2288--2321, 2018.

\bibitem[CFV22]{chaumont-frelet-vega21}
Th\'{e}ophile Chaumont-Frelet and Patrick Vega.
\newblock Frequency-explicit approximability estimates for time-harmonic
  {M}axwell's equations.
\newblock {\em Calcolo}, 59(2):Paper No. 22, 15, 2022.

\bibitem[CK92]{coltonkress_inverse}
David Colton and Rainer Kress.
\newblock {\em Inverse acoustic and electromagnetic scattering theory},
  volume~93 of {\em Applied Mathematical Sciences}.
\newblock Springer-Verlag, Berlin, 1992.

\bibitem[CM10]{costabel-mcintosh10}
Martin Costabel and Alan McIntosh.
\newblock On {B}ogovski\u\i \ and regularized {P}oincar\'e integral operators
  for de {R}ham complexes on {L}ipschitz domains.
\newblock {\em Math. Z.}, 265(2):297--320, 2010.

\bibitem[Cos90]{costabel90}
Martin Costabel.
\newblock A remark on the regularity of solutions of {M}axwell's equations on
  {L}ipschitz domains.
\newblock {\em Math. Methods Appl. Sci.}, 12(4):365--368, 1990.

\bibitem[EG18]{ern-guermond18}
Alexandre Ern and Jean-Luc Guermond.
\newblock Analysis of the edge finite element approximation of the {M}axwell
  equations with low regularity solutions.
\newblock {\em Comput. Math. Appl.}, 75(3):918--932, 2018.

\bibitem[EM12]{MelenkHelmStab2010}
Sofi Esterhazy and Jens~M. Melenk.
\newblock On stability of discretizations of the {H}elmholtz equation.
\newblock In I.G. Graham, T.Y. Hou, O.~Lakkis, and R.~Scheichl, editors, {\em
  Numerical {A}nalysis of {M}ultiscale {P}roblems}, volume~83 of {\em Lect.
  Notes Comput. Sci. Eng.}, pages 285--324. Springer, Berlin, 2012.

\bibitem[FW14]{feng-wu14}
Xiaobing Feng and Haijun Wu.
\newblock An absolutely stable discontinuous {G}alerkin method for the
  indefinite time-harmonic {M}axwell equations with large wave number.
\newblock {\em SIAM J. Numer. Anal.}, 52(5):2356--2380, 2014.

\bibitem[GM12]{gatica_MW_imped}
Gabriel~N. Gatica and Salim Meddahi.
\newblock Finite element analysis of a time harmonic {M}axwell problem with an
  impedance boundary condition.
\newblock {\em IMA J. Numer. Anal.}, 32(2):534--552, 2012.

\bibitem[Hip02]{hiptmair-acta}
Ralf Hiptmair.
\newblock Finite elements in computational electromagnetism.
\newblock {\em Acta Numer.}, 11:237--339, 2002.

\bibitem[Hip15]{hipt_reg_decomp}
Ralf Hiptmair.
\newblock Maxwell's equations: continuous and discrete.
\newblock In {\em Computational electromagnetism}, volume 2148 of {\em Lecture
  Notes in Math.}, pages 1--58. Springer, Cham, 2015.

\bibitem[HLZ12]{hiptmair-li-zou12}
Ralf Hiptmair, Jingzhi Li, and Jun Zou.
\newblock Universal extension for {S}obolev spaces of differential forms and
  applications.
\newblock {\em J. Funct. Anal.}, 263(2):364--382, 2012.

\bibitem[HMP11]{hiptmair-moiola-perugia11}
Ralf Hiptmair, Andrea Moiola, and Ilaria Perugia.
\newblock Stability results for the time-harmonic {M}axwell equations with
  impedance boundary conditions.
\newblock {\em Math. Models Methods Appl. Sci.}, 21(11):2263--2287, 2011.

\bibitem[HP19]{hipt_pech}
Ralf Hiptmair and Clemens Pechstein.
\newblock Discrete regular decompositions of tetrahedral discrete 1-forms.
\newblock In {\em Maxwell's equations---analysis and numerics}, volume~24 of
  {\em Radon Ser. Comput. Appl. Math.}, pages 199--258. De Gruyter, Berlin,
  2019.

\bibitem[LCQ17]{lu-chen-qiu17}
Peipei Lu, Huangxin Chen, and Weifeng Qiu.
\newblock An absolutely stable {$hp$}-{HDG} method for the time-harmonic
  {M}axwell equations with high wave number.
\newblock {\em Math. Comp.}, 86(306):1553--1577, 2017.

\bibitem[LM11]{MelenkLoehndorf}
Maike L\"ohndorf and Jens~M. Melenk.
\newblock Wavenumber-explicit {$hp$}-{BEM} for high frequency scattering.
\newblock {\em SIAM J. Numer. Anal.}, 49(6):2340--2363, 2011.

\bibitem[LSW21]{lafontaine2021decompositions}
David Lafontaine, Euan~A. Spence, and Jared Wunsch.
\newblock Decompositions of high-frequency {H}elmholtz solutions via functional
  calculus, and application to the finite element method, 2021.
\newblock arXiv:2102.13081.

\bibitem[LSW22]{lafontaine2021wavenumberexplicit}
David Lafontaine, Euan~A. Spence, and Jared Wunsch.
\newblock Wavenumber-explicit convergence of the {$hp$}-{FEM} for the
  full-space heterogeneous {H}elmholtz equation with smooth coefficients.
\newblock {\em Comput. Math. Appl.}, 113:59--69, 2022.

\bibitem[McL00]{Mclean00}
William McLean.
\newblock {\em Strongly Elliptic Systems and Boundary Integral Equations}.
\newblock Cambridge, Univ. Press, 2000.

\bibitem[Mel95]{MelenkDiss}
Jens~M. Melenk.
\newblock {\em On {G}eneralized {F}inite {E}lement {M}ethods}.
\newblock PhD thesis, University of Maryland at College Park, 1995.

\bibitem[Mel02]{MelenkHabil}
Jens~M. Melenk.
\newblock {\em hp-{F}inite {E}lement {M}ethods for {S}ingular {P}erturbations}.
\newblock Springer, Berlin, 2002.

\bibitem[Mel12]{MelenkStab}
Jens~M. Melenk.
\newblock Mapping properties of combined field {H}elmholtz boundary integral
  operators.
\newblock {\em SIAM J. Math. Anal.}, 44(4):2599--2636, 2012.

\bibitem[MMPR20]{mascotto-melenk-perugia-rieder20}
Lorenzo Mascotto, Jens~M. Melenk, Ilaria Perugia, and Alexander Rieder.
\newblock F{EM}-{BEM} mortar coupling for the {H}elmholtz problem in three
  dimensions.
\newblock {\em Comput. Math. Appl.}, 80(11):2351--2378, 2020.

\bibitem[Mon03]{Monkbook}
Peter Monk.
\newblock {\em Finite element methods for {M}axwell's equations}.
\newblock Oxford University Press, New York, 2003.

\bibitem[MPS13]{MPS13}
J.~M. Melenk, A.~Parsania, and S.~A. Sauter.
\newblock General {DG}-methods for highly indefinite {H}elmholtz problems.
\newblock {\em J. Sci. Comput.}, 57(3):536--581, 2013.

\bibitem[MR20]{melenk_rojik_2018}
Jens~M. Melenk and Claudio Rojik.
\newblock On commuting {$p$}-version projection-based interpolation on
  tetrahedra.
\newblock {\em Math. Comp.}, 89(321):45--87, 2020.

\bibitem[MR21]{melenk-rieder21}
Jens~M. Melenk and Alexander Rieder.
\newblock On superconvergence of {R}unge-{K}utta convolution quadrature for the
  wave equation.
\newblock {\em Numer. Math.}, 147(1):157--188, 2021.

\bibitem[MS10]{MelenkSauterMathComp}
Jens~M. Melenk and Stefan~A. Sauter.
\newblock Convergence {A}nalysis for {F}inite {E}lement {D}iscretizations of
  the {H}elmholtz equation with {D}irichlet-to-{N}eumann boundary condition.
\newblock {\em Math. Comp}, 79:1871--1914, 2010.

\bibitem[MS11]{mm_stas_helm2}
Jens~M. Melenk and Stefan~A. Sauter.
\newblock Wave-{N}umber {E}xplicit {C}onvergence {A}nalysis for {G}alerkin
  {D}iscretizations of the {H}elmholtz {E}quation.
\newblock {\em SIAM J. Numer. Anal.}, 49(3):1210--1243, 2011.

\bibitem[MS21a]{MelenkSauterMaxwell_I}
Jens~M. Melenk and Stefan~A. Sauter.
\newblock Wavenumber-explicit {$hp$}-{FEM} analysis for {M}axwell's equations
  with transparent boundary conditions.
\newblock {\em Found. Comput. Math.}, 21(1):125--241, 2021.

\bibitem[MS21b]{melenk2018wavenumber}
Jens~Markus Melenk and Stefan Sauter.
\newblock Wavenumber-explicit $ hp $-{FEM} analysis for {M}axwell's equations
  with transparent boundary conditions.
\newblock {\em Found. Comput. Math.}, 21:125--241, 2021.

\bibitem[N{\'e}d80]{nedelec80}
Jean-Claude N{\'e}d{\'e}lec.
\newblock Mixed finite elements in {${\bf R}^{3}$}.
\newblock {\em Numer. Math.}, 35(3):315--341, 1980.

\bibitem[N{\'e}d01]{Nedelec01}
Jean-Claude N{\'e}d{\'e}lec.
\newblock {\em Acoustic and {E}lectromagnetic {E}quations}.
\newblock Springer, New York, 2001.

\bibitem[NT19]{nicaise-tomezyk17}
Serge Nicaise and J{\'e}r{\^o}me Tomezyk.
\newblock The time-harmonic {M}axwell equations with impedance boundary
  conditions in polyhedral domains.
\newblock In {U}. {L}anger, D.~Pauly, and S.~Repin, editors, {\em {M}axwell{'}s
  {E}quations: {A}nalysis and {N}umerics}, Radon Series on Computational and
  Applied Mathematics 24, pages 285--340, Berlin, 2019. De Gruyter.

\bibitem[NT20]{nicaise-tomezyk19}
Serge Nicaise and J{\'e}r{\^o}me Tomezyk.
\newblock {Convergence analysis of a hp-finite element approximation of the
  time-harmonic Maxwell equations with impedance boundary conditions in domains
  with an analytic boundary}.
\newblock {\em Numer. Methods Partial Differential Eq.}, 36:1868--1903, 2020.

\bibitem[Roj20]{rojikDiss}
Claudio Rojik.
\newblock {\em $p$-version projection-based interpolation}.
\newblock PhD thesis, Institut of Analysis and Scientific Computing, Technische
  Universit{\"{a}}t Wien, 2020.
\newblock https://doi.org/10.34726/hss.2019.65840.

\bibitem[Sch]{schoeberlNGSOLVE}
Joachim Sch\"oberl.
\newblock Finite {E}lement {S}oftware {NETGEN}/{NGS}olve version 6.2.
\newblock {\tt https://ngsolve.org/}.

\bibitem[Sch97]{schoeberl97}
Joachim Sch{\"o}berl.
\newblock {NETGEN} - {A}n advancing front 2{D}/3{D}-mesh generator based on
  abstract rules.
\newblock {\em Computing and Visualization in Science}, 1(1):41--52, Jul 1997.

\bibitem[Sch09]{schoeberlscript}
Joachim Sch{\"{o}}berl.
\newblock Numerical {m}ethods for {M}axwell {E}quations.
\newblock Technical Report Via WWW-address:
  http://www.asc.tuwien.ac.at/$\sim$schoeberl/wiki/lva/notes/maxwell.pdf,
  Technische Universit{\"{a}}t Wien, 2009.

\bibitem[Sch18]{Schweizer-Friedrichs-2016}
Ben Schweizer.
\newblock {F}riedrichs inequality, {H}elmholtz decomposition, vector
  potentials, and the div-curl lemma.
\newblock {\em INdAM-Springer series}, Trends on Applications of Mathematics to
  Mechanics, 2018.

\bibitem[Spe14]{spence14}
Euan~A. Spence.
\newblock Wavenumber-explicit bounds in time-harmonic acoustic scattering.
\newblock {\em SIAM J. Math. Anal.}, 46(4):2987--3024, 2014.

\bibitem[Ste70]{emstein}
Elias~M. Stein.
\newblock {\em Singular Integrals and Differentiability Properties of
  Functions}.
\newblock Princeton, University Press, Princeton, N.J., 1970.

\bibitem[Tom19]{tomezyk19}
J\'er{\^o}me Tomezyk.
\newblock {\em R\'esolution num\'erique de quelques probl\`emes du type
  Helmholtz avec conditions au bord d'imp\'edance ou des couches absorbantes
  (PML)}.
\newblock PhD thesis, Universit\'e Polytechnique Hauts-de-France, 2019.

\bibitem[Tri95]{triebel95}
Hans Triebel.
\newblock {\em Interpolation theory, function spaces, differential operators}.
\newblock Johann Ambrosius Barth, Heidelberg, second edition, 1995.

\bibitem[Ver19]{verfuerth19}
Barbara Verf\"{u}rth.
\newblock Heterogeneous multiscale method for the {M}axwell equations with high
  contrast.
\newblock {\em ESAIM Math. Model. Numer. Anal.}, 53(1):35--61, 2019.

\bibitem[ZSWX09]{zhong-shu-wittum-xu09}
Liuqiang Zhong, Shi Shu, Gabriel Wittum, and Jinchao Xu.
\newblock Optimal error estimates for {N}edelec edge elements for time-harmonic
  {M}axwell's equations.
\newblock {\em J. Comput. Math.}, 27(5):563--572, 2009.

\end{thebibliography}

\end{document}